\documentclass[11pt]{article}

\usepackage[margin=30mm]{geometry}
\usepackage{amsmath}
\usepackage{amssymb}
\usepackage{amsthm}
\usepackage{arydshln}
\usepackage{bm}
\usepackage{enumitem}
\usepackage{float}
\usepackage{graphicx}
\usepackage{mathtools}
\usepackage{subfiles}
\usepackage{xcolor,colortbl}
\usepackage{tikz}
\usepackage{tikz-cd}
\usepackage[titletoc]{appendix}
\usetikzlibrary{decorations.pathmorphing}
\usepackage{verbatimbox}
\usepackage[colorinlistoftodos]{todonotes}
\usepackage[numbers]{natbib}
\usepackage{hyperref}

\setlist[enumerate]{topsep=0pt}
\setlist[itemize]{topsep=0pt}

\newtheorem{corollary}{Corollary}
\newtheorem{definition}{Definition}
\newtheorem{proposition}{Proposition}
\newtheorem{theorem}{Theorem}

\newcommand{\norm} [1] {\left\lVert #1 \right\rVert}
\newcommand{\card} [1] {\left\vert #1 \right\vert}
\newcommand{\cM} {\mathcal{M}}
\newcommand{\cN} {\mathcal{N}}
\newcommand{\cW} {\mathcal{W}}
\newcommand{\cY} {\mathcal{Y}}
\newcommand{\union} {\ \cup \ }
\newcommand{\spacecomma}{\hspace{0.08em},\,}

\def\colOne {orange!90!black}
\def\colTwo {cyan!75!black}
\def\colThree {red!75!black}
\def\opac{0.5}
\def\opacTen{50}

\graphicspath{03Figures/}

\allowdisplaybreaks

\makeatletter
\newcommand\tableofcontentsA{%
    \@starttoc{toca}%
}
\makeatother

\def \nd {circle (1/\scl * 2.5pt)}
\def \grd[#1][#2] {\draw[dotted] (0,0) grid (#1,#2);}
\def \dummyNodes[#1][#2][#3][#4] {\node at (#1,#2) {}; \node at (#3,#4) {};}
\def \dummyNodesPolar[#1][#2] {\node at (#1) {}; \node at (#2) {};} 
\def \SSYTscl {0.85\width}

\def\P[#1][#2] {\node at (#1,#2) {\color{blue}{$P$}};}
\def\polarP[#1][#2] {\node at (#1:#2) {\color{blue}{$P$}};}
\def\leaf[#1][#2] {\node at (#1,#2) {\tikz \fill[blue] circle (2pt);};}

\def \onetile {-- ++(0,1) -- ++(1,0) ++(-1,-1) -- ++(1,0)}
\def \twotiles {-- ++(0,1) -- ++(2,0) ++(-2,-1) -- ++(2,0) ++(-1,1) edge[dash pattern = on 3pt off 2.5pt] ++(0,-1) ++(1,-1)}
\def \finaledge {-- ++(0,1)}

\def \pathStep {++(1,0)} 
\def \pathEdge {-- ++(1,0)}
\def \pathNoEdge {++(1,0)}

\def \ladderStep {++(0,-0.75)}
\def \ladderSS {+(1,0)}
\def \ladderNodes{\ladderStep \halfnd \ladderSS \halfnd}
\def \hor {-- ++(1,0) ++(-1,-0.75)}
\def \finalhor {-- ++(1,0)}
\def \ver {-- ++(0,-0.75) +(1,0.75) -- +(1,0) ++(0,-0.75)}
\def \finalver {-- ++(0,-0.75) +(1,0.75) -- +(1,0)}
\def \ladderDummyNodes[#1] {\node at (0.5,0.2) (A) {}; \node at (0.5,#1) (B) {};} 

\def \glassDummyNodes {\node at (1,1.2) (A) {}; \node at (1,-3.2) (B) {};}
\def \glassUp {-- ++(0.5,1)}
\def \glassDn {-- ++(0.5,-1)}
\def \plates [#1] {
    \draw (0,0) -- (#1,0);
    \draw (0,-1) -- (#1,-1);
    \draw (0,-2) -- (#1,-2);
}

\def \un {-- ++(0,-1)}
\def \bin {-- ++(-1,-1) ++(1,1) -- ++(1,-1)}
\def \tern {-- ++(-1,-1) ++(1,1) -- ++(0,-1) ++(0,1) -- ++(1,-1)}
\def \widebin {-- ++(-1.5,-1) ++(1.5,1) -- ++(1.5,-1)}
\def \widetern {-- ++(-1.5,-1) ++(1.5,1) -- ++(0,-1) ++(0,1) -- ++(1.5,-1)}
\def \verywidetern {-- ++(-2.5,-1) ++(2.5,1) -- ++(0,-1) ++(0,1) -- ++(2.5,-1)}

\def \smlunnd {++(0,-1) \smlnd}

\def \smlbinnd {++(-1,-1) \smlnd ++(2,0) \smlnd}

\def \smlwidebinnd {++(-1.5,-1) \smlnd ++(3,0) \smlnd}
\def \ternnd {++(-1,-1) \nd ++(1,0) \nd ++(1,0) \nd}
\def \smlternnd {++(-1,-1) \smlnd ++(1,0) \smlnd ++(1,0) \smlnd}
\def \wideternnd {++(-1.5,-1) \nd ++ (1.5,0) \nd ++(1.5,0) \nd}
\def \smlwideternnd {++(-1.5,-1) \smlnd ++ (1.5,0) \smlnd ++(1.5,0) \smlnd}
\def \verywideternnd {++(-2.5,-1) \nd ++ (2.5,0) \nd ++(2.5,0) \nd}
\def \unlabel[#1][#2] {
    \node at ([shift={(-0.4,-0.05)}]#1,#2) {$f($};
    \node at ([shift={(0.3,-0.05)}]#1,#2) {$)$};
}
\def \binlabel[#1][#2] {
    \node at ([shift={(-0.4,0.05)}]#1,#2) {$($};
    \node at ([shift={(0.4,0.05)}]#1,#2) {$)$};
    \node at ([shift={(0,-0.3)}]#1,#2) {$\star$};
}
\def\ternlabel[#1][#2] {
    \node at ([shift={(-0.4,0.1)}]#1,#2) {$t($};
    \node at ([shift={(0.4,0.1)}]#1,#2) {$)$};
    \node at ([shift={(-0.2,-0.5)}]#1,#2) {$,$};
    \node at ([shift={(0.2,-0.5)}]#1,#2) {$,$};
}
\def \leaflabel[#1][#2] {
    \node at ([shift={(0,-0.3)}]#1,#2) {$\epsilon$};
}
\def \root {\tikz [scale = 0.15] \fill (0,0) -- (1,1.5) -- (-1,1.5) -- (0,0);}
\def \smlroot {\tikz [scale = 0.12] \fill (0,0) -- (1,1.5) -- (-1,1.5) -- (0,0);}
\def \smlnd {circle (0.8/\scl * 2.5pt)}
\def \ndlabel[#1][#2][#3] {\node at ([shift={(-0.3,0)}]#1,#2) {\scriptsize #3};}

\def \dyckUp {-- ++(1,1)}
\def \dyckDn {-- ++(1,-1)}
\def \dyckTwoDn {-- ++(1,-2)}
\def \motUp{-- ++(1,1)}
\def \motAc{-- ++(1,0)}
\def \motDn{-- ++(1,-1)}
\def \schUp {-- ++(0,1)}
\def \schAc {-- ++(1,0)}
\def \schDiag {-- ++(1,1)}
\def \whtNd[#1] {\draw[fill = white] #1 \smlnd;}
\def \blkNd[#1] {\fill #1 \smlnd;}

\def \V {\tikz [scale = 0.07] \draw (-1,2) -- (0,0) -- (1,2);}
\def \drawDashedChord[#1][#2] {\draw[dashed] (#1) to [bend right = 70] (#2);}

\def \dissectionArc[#1][#2] {\draw[red] (#1) to (#2);}
\def \trulyBlankPolygonSetup {
    \foreach \i in {1,...,\numPoints}
        {
        \coordinate (\i) at (-90 - 180/\numPoints + 360/\numPoints * \i:5);
        }
    \draw[ultra thick] (1) to (\numPoints);   
    
    \dummyNodes[0][-4.5][0][6]
}
\def \blankPolygonSetup {
    \foreach \i in {1,...,\numPoints}
        {
        \coordinate (\i) at (-90 - 180/\numPoints + 360/\numPoints * \i:5);
        }
    \draw[ultra thick] (1) to (\numPoints);   
    \draw[dashed] (-67.5:5) arc (-67.5:247.5:5);
    \dummyNodes[0][6][0][-4.5]
}
\def\polygonSetup {
    \foreach \i in {1,...,\numPoints}
        {
        \coordinate (\i) at (-90 - 180/\numPoints + 360/\numPoints * \i:5);
        }    
    \draw[ultra thick] (1) to (\numPoints);    
    \foreach \i [evaluate = \i as \ieval using \i - 1] in {2,...,\numPoints}
        {
        \draw (\i) to (\ieval);
        }
    \dummyNodes[0][6][0][-4.5]
}

\def \halfnd {circle (0.6/\scl * 2.5pt)}

\def \blankPartitionSetup {\draw (0,0) circle [radius = 5];
    \foreach \i in {1,...,\numPoints}
        {
        \coordinate (\i) at (180 - 360/\numPoints * \i + 360/\numPoints:5);
        }
    \dummyNodes[0][-6][0][6]
}
\def \partitionSetupOne {\draw (0,0) circle [radius = 5];
    \foreach \i in {1,...,\numPoints}
        {
        \draw[fill] (180 - 360/\numPoints * \i + 360/\numPoints:5) circle (10pt);
        \coordinate (\i) at (180 - 360/\numPoints * \i + 360/\numPoints:5);
        \node at (180 - 360/\numPoints * \i + 360/\numPoints:6.5) (\i\i) {\scriptsize $\i$};
        }
    \dummyNodes[0][-6][0][6]
}
\def \partitionSetup {\draw (0,0) circle [radius = 5];
    \foreach \i in {1,...,\numPoints}
        {
        \draw[fill] (180 - 360/\numPoints * \i + 360/\numPoints:5) circle (10pt);
        \coordinate (\i) at (180 - 360/\numPoints * \i + 360/\numPoints:5);
        \node at (180 - 360/\numPoints * \i + 360/\numPoints:7) (\i\i) {\scriptsize $\i$};
        }
    \dummyNodes[0][-6][0][6]
}
\def \partitionDrawArc[#1][#2] {\draw (#1) to [bend right = 30] (#2);}

\newcounter{fibFamily} 
\newenvironment{fibFamily}[1][]{
    \refstepcounter{fibFamily} \par\medskip
    \textbf{$\mathbb{F}_{\thefibFamily}:\ $ #1}
    }{\medskip}
\newcommand{\fibFamRef}[1]{$\mathbb{F}_{\ref{#1}}$}
\newcommand{\fibBaseRef}[1]{$\mathcal{F}_{\ref{#1}}$}

\newcommand{\mathFibBaseRef}[1]{\mathcal{F}_{\ref{#1}}}
\newcommand{\fibMag}[1]{$\left( \mathcal{F}_{\ref{#1}}, f_{\ref{#1}}, g_{\ref{#1}} \right)$}
    
\newcounter{motFamily} 
\newenvironment{motFamily}[1][]{
    \refstepcounter{motFamily} \par\medskip
    \textbf{$\mathbb{M}_{\themotFamily}:\ $ #1}
    }{\medskip}
\newcommand{\motFamRef}[1]{$\mathbb{M}_{\ref{#1}}$}
\newcommand{\motBaseRef}[1]{$\mathcal{M}_{\ref{#1}}$}

\newcommand{\mathMotBaseRef}[1]{\mathcal{M}_{\ref{#1}}}
\newcommand{\motMag}[1]{$\left( \mathMotBaseRef{#1}, f_{\ref{#1}}, \star_{\ref{#1}} \right)$}
    
\newcounter{schFamily} 
\newenvironment{schFamily}[1][]{
    \refstepcounter{schFamily} \par\medskip
    \textbf{$\mathbb{S}_{\theschFamily}:\ $ #1}
    }{\medskip}
\newcommand{\schFamRef}[1]{$\mathbb{S}_{\ref{#1}}$}
\newcommand{\schBaseRef}[1]{$\mathcal{S}_{\ref{#1}}$}

\newcommand{\mathSchBaseRef}[1]{\mathcal{S}_{\ref{#1}}}
\newcommand{\schMag}[1]{$\left( \mathSchBaseRef{#1}, f_{\ref{#1}}, \star_{\ref{#1}} \right)$}
 
\newcounter{fcFamily} 
\newenvironment{fcFamily}[1][]{
    \refstepcounter{fcFamily} \par\medskip
    \textbf{$\mathbb{T}_{\thefcFamily}:\ $ #1}
    }{\medskip}
\newcommand{\fcFamRef}[1]{$\mathbb{T}_{\ref{#1}}$}

\newcommand{\mathfcBaseRef}[1]{\mathcal{T}_{\ref{#1}}}
\newcommand{\fcMag}[1]{$\left( \mathfcBaseRef{#1}, t_{\ref{#1}} \right)$}

\begin{document}

\setcounter{page}{1}

\title{Fibonacci, Motzkin, Schr\"oder, Fuss-Catalan and other Combinatorial Structures: Universal and Embedded  Bijections}

\author{R. Brak\thanks{rb1@unimelb.edu.au}\hspace{1ex}   and N.\ Mahony\thanks{nedm@unimelb.edu.au}
    \vspace{0.15 in} \\
    School of Mathematics and Statistics,\\
    The University of Melbourne\\
    Parkville,  Victoria 3052,\\
    Australia\\
 }
 
  \newcommand{\fami}{\mathcal{F}}
  \newcommand{\bn}{\bm{n}}

\maketitle

\begin{abstract} 
A combinatorial structure, $\fami$, with counting sequence $\{a_n\}_{n\ge 0}$ and ordinary generating function $G_\fami=\sum_{n\ge0} a_n x^n$, is positive algebraic if $G_\fami$ satisfies a polynomial equation $G_\fami=\sum_{k=0}^N p_k(x)\,G_\fami^k $ and $p_k(x)$ is a polynomial in $x$ with \emph{non-negative integer coefficients}. 
We show that every such family is associated with a normed $\bn$-magma. An $\bn$-magma with $\bn=(n_1,\dots, n_k)$ is a pair $\cM$ and $\fami$ where $\cM$ is a set of combinatorial structures and $\fami$ is a tuple of $n_i$-ary maps $f_i\,:\,\cM^{n_i}\to \cM$. A norm is a super-additive size map $\norm{\cdot}\,:\, \cM\to \mathbb{N} $. 

If the normed $\bn$-magma is free then we show there exists a recursive, norm preserving, universal bijection between all positive algebraic families $\fami_i$ with the same counting sequence. 
A free $\bn$-magma is defined using a universal mapping principle.  We state a theorem which provides a combinatorial method of proving if a particular $\bn$-magma is free. 
We illustrate this by defining  several  $\bn$-magmas:  eleven $(1,1)$-magmas (the Fibonacci families), seventeen $(1,2)$-magmas (nine Motzkin and eight Schr\"oder families) and seven $(3)$-magmas  (the Fuss-Catalan families). 
We prove they are all free and hence obtain a universal bijection for each   $\bn$. 
We also show how the $\bn$-magma structure  manifests as an embedded bijection.



 \end{abstract}

\vfill
\paragraph{Keywords:}  

\newpage
{\small
\tableofcontents
}
\newpage

\section*{List of combinatorial families}

{\small 
\tableofcontentsA

}
\newpage 

 \section{Definitions and  general results}

 We generalise the Catalan results of Brak \cite{BRAK} to arbitrary positive algebraic combinatorial families. A combinatorial structure, $\fami$, with counting sequence $\{a_n\}_{n\ge 0}$ and ordinary generating function $G_\fami(x)=\sum_{n\ge0} a_n x^n$ is positive algebraic if $G_\fami(x)$ satisfies a polynomial equation   
\begin{equation}\label{eq_posalg}
     \sum_{k=0}^N p_k(x)\,G_\fami^k =G_\fami
\end{equation}
 and $p_k(x)=\sum_{i=0}^{m_k} b^{(k)}_{i}x^i$  is a degree $m_k$ polynomial in $x$ with \emph{non-negative  integer coefficients} (ie.\ all $ b^{(k)}_{i}\ge0$). 
 This case contains a large  number of well known combinatorial families such as Fibonacci, Catalan, Motzkin, Schr\"oder and Fuss-Catalan.
 
We show that  positive algebraic families are  associated with   particular normed $\bn$-magmas (a norm is a super-additive size map). An $\bn$-magma with $\bn=(n_1,\dots, n_k)$ is a pair $(\cM,\fami)$ where $\cM$ is a set  and $\fami$ is a tuple of $n_i$-ary maps $f_i\,:\,\cM^{n_i}\to \cM$. 
To each monomial $ b^{(k)}_{i}x^i G_\fami^k$ in \eqref{eq_posalg} we associate $b^{(k)}_{i}$ unique $k$-ary maps. 
These  are the maps that constitute the  $n_i$-ary maps of the  $\bn$-magma. 
These maps have to be carefully defined for each combinatorial family to ensure they satisfy certain required properties.

\subsection{Magma definitions}

In this section we generalise a number  of definitions from \cite{BRAK}.

A magma is an algebraic structure defined in \cite{bourbaki} as a pair \mbox{($\cM$, $\star$)} where $\cM$ is a non-empty countable set called the \textbf{base set} and $\star$ is a \textbf{product map} 
\begin{equation*}
    \star: \cM \times \cM \rightarrow \cM.
\end{equation*}
If \mbox{($\cN$, $\bullet$)} is a magma, then a \textbf{magma morphism} $\theta$ from $\cM$ to $\cN$ is a map $\theta: \cM \rightarrow \cN$ such that for all $m, m' \in \cM$, 
\begin{equation*}
    \theta(m \star m') = \theta(m) \bullet \theta(m').
\end{equation*}
We generalise these definitions to allow for arbitrary $n$-ary maps.

\begin{definition}[$\bm{n}$-magma] \label{def:generalMagma}
Let $\cM$ be a non-empty countable set called the \textbf{base set}. An \textbf{$\bm{n}$-magma} defined on $\cM$, where $\bm{n} = (n_1, \hdots, n_k)$ with $n_1 \leq \cdots \leq n_k$, is a $(k + 1)$-tuple \mbox{($\cM$, $f_1$, $\hdots$, $f_k$)} where 
\begin{equation*}
    f_i: \cM^{n_i} \rightarrow \cM
\end{equation*} is an $n_i$-ary map, for each $i = 1, \hdots, k$. If \mbox{($\cM$, $f_1$, $\hdots$, $f_k$)} and \mbox{($\cN$, $g_1$, $\hdots$, $g_k$)} are two \mbox{$\bm{n}$-magma}s, then an \textbf{$\bm{n}$-magma morphism} from $\cM$ to $\cN$ is a map $\theta: \cM \rightarrow \cN$ such that for all $i \in \{ 1, \hdots, k \}$ and all $m_1, \hdots, m_{n_i} \in \cM$, 
\begin{equation*}
    \theta( f_i(m_1, \hdots,m_{n_i}) ) = g_i( \theta(m_1), \hdots, \theta(m_{n_i}) ).
\end{equation*}
\end{definition}

In this definition, we adopt the convention that we write the maps in the same order for all distinct \mbox{$\bm{n}$-magma}s for the same $\bm{n}$. 
Hence if \mbox{($\cM$, $f_1$, $\hdots$, $f_k$)} and \mbox{($\cN$, $g_1$, $\hdots$, $g_k$)} are two \mbox{$\bm{n}$-magma}s, then $f_i$ and $g_i$ have the same arity, for each $i = 1, \hdots, k$. 
We also require that if we have more than one map with the same arity, then the order in which we write the maps is significant and thus permuting the maps defines a different \mbox{$\bm{n}$-magma}. 
For example, the two $(n_1, n_1)$-magmas ($\cM$, $f_1$, $f_2$) and ($\cM'$, $f_1'$, $f_2'$) are equal if and only if $\cM' = \cM$, $f_1' = f_1$ and $f_2' = f_2$. 
This distinction is important later when we prove that the Cartesian $(1,1)$-magma is free.

We are interested in when an \mbox{$\bm{n}$-magma} is free. 
To this end we start with the universal mapping definition of free, but will later  give a theorem --  Theorem \ref{thm:uniqFactNormedThenFree} -- enabling us to give a combinatorial, rather than a universal mapping, proof of when an \mbox{$\bm{n}$-magma} is free.

\begin{definition}[Free \mbox{$\bm{n}$-magma} Universal Mapping Principle] \label{def:freenmagma}
Let $\bm{n} = (n_1, \hdots, n_k)$ and let \mbox{($\cM$, $f_1$, $\hdots$, $f_k$)} be an \mbox{$\bm{n}$-magma}. Then \mbox{($\cM$, $f_1$, $\hdots$, $f_k$)} is \textbf{free} if the following is true: There exists a set $Y$ and a map $i: Y \rightarrow \cM$ such that for all \mbox{$\bm{n}$-magma}s \mbox{($\cN$, $f_1'$, $\hdots$, $f_k'$)} and for all maps $\varphi: Y \rightarrow \cN$, there exists a unique (up to isomorphism) \mbox{$\bm{n}$-magma} morphism $\theta: \cM \rightarrow \cN$ such that $\varphi = \theta \circ i$, that is, the diagram
\begin{center}
    \begin{tikzcd}
        \cM \arrow[r, "\theta"] & \cN \\
        Y \arrow[u, "i" left] \arrow[ru, "\varphi" below]
    \end{tikzcd}
\end{center}
commutes. The image of the set $Y$ in $\cM$, $X = \text{Img}(Y)$, will be called the \textbf{set of generators} of $\cM$.
\end{definition}

We now define a size function on the base set of the \mbox{$\bm{n}$-magma}. We will call the function a  norm as we will require it to be ``super-additive'' (defined below) with respect to each of the $\bm{n}$-magma maps. 
The norm is used in two essential ways. Firstly, it partitions the base set $\cM$ into sets of elements which have the same  size. The size of the sets in this partition defines the counting sequence. Secondly, the super-additivity of the norm map is used, along with unique factorisation, to give a combinatorial characterisation of free \mbox{$\bm{n}$-magma}s.  Let $\mathbb{N}=\{1,2,3,\dots\}$ denote the set of positive integers.

\begin{definition}[Norm] \label{def:norm}
Let \mbox{($\cM$, $f_1$, $\hdots$, $f_k$)} be an $\bm{n}$-magma, where $\bm{n} = (n_1, \hdots, n_k)$. A \textbf{norm} is a map $\norm{\cdot}: \cM \rightarrow \mathbb{N}$ that satisfies the \textbf{super-additive} conditions:
\begin{enumerate}
    \item All unary maps $f_i: \cM \rightarrow \cM$ must satisfy
    \begin{equation*}
    \norm{f_i(m)} > \norm{m}\,.
\end{equation*}
 
\item  All   maps $f_i: \cM^{n_i} \rightarrow \cM$, $n_i>1$, must satisfy
\begin{equation*}
    \norm{f_i(m_1, \hdots, m_{n_i})} \geq \sum_{j = 1}^{n_i} \norm{m_j}\,.
\end{equation*}

\end{enumerate}
If \mbox{($\cM$, $f_1$, $\hdots$, $f_k$)} has a norm, then it will be called a \textbf{normed $\bm{n}$-magma}. 
\end{definition}

Note, it is important that there are \emph{no} elements in $\cM$ of size zero, ie.\ we use the set $\mathbb{N}$ in the norm definition and \emph{not} the set $\mathbb{N}_0=\{0\}\,\cup\,\mathbb{N}$. 
Furthermore,  any norm defined on an \mbox{$\bm{n}$-magma} is not necessarily unique.
The non-uniqueness is significant as for example   there exist Motzkin and Schr\"oder families given by the same $(1,2)$-magma. 
Both families have the same base set, as well as the same unary and binary maps, however they have different norm maps.

We now consider the problem of factorisation in \mbox{$\bm{n}$-magma}s and discuss how it is related to the existence of a norm. 
We show that the existence of a norm guarantees that recursive factorisation of any \mbox{$\bm{n}$-magma} element will terminate. 
We begin by defining reducible and irreducible elements.

\begin{definition}[Reducible, irreducible elements] \label{def:reducibleIrreducible}
Let \mbox{($\cM$, $f_1$, $\hdots$, $f_k$)}  be an \mbox{$\bm{n}$-magma}, where \mbox{$\bm{n} = (n_1, \hdots, n_k)$}. The image of the maps $f_i$ in $\cM$, is called the \textbf{set of reducible elements}. 
Letting $\cM_i^+ = \text{Img}(f_i)$ for each  $i = 1, \hdots, k$, we can define the set of reducible elements as $\cM^+ = \bigcup_{i = 1}^k \cM_i^+$. The elements of the set $\cM^0 = \cM \backslash \cM^+$ are called \textbf{irreducible elements} and the set $\cM^0$ is called the \textbf{set of irreducibles}.
\end{definition}

A unique factorisation \mbox{$\bm{n}$-magma}   describes when every element of the base set can be written uniquely in terms of the irreducible elements and the \mbox{$\bm{n}$-magma} maps. 
 
\begin{definition}[Unique factorisation] \label{def:uniqFact}
Let \mbox{($\cM$, $f_1$, $\hdots$, $f_k$)}  be an \mbox{$\bm{n}$-magma}, where \mbox{$\bm{n} = (n_1, \hdots, n_k)$}. 
If 
\begin{enumerate}[label=(\roman*), topsep=0pt]
    \item every map $f_i: \cM^{n_i} \rightarrow \cM$ is injective, and
    \item $\cM_i^+ \cap \cM_j^+ = \emptyset$ for all $i, j \in \{ 1, \hdots, k \}$ such that $i \neq j$,
\end{enumerate}
then we will call \mbox{($\cM$, $f_1$, $\hdots$, $f_k$)} a \textbf{unique factorisation \mbox{$\bm{n}$-magma}}. 
\end{definition}

Note that (ii) requires that the images of the maps forms a partition of the set of reducible elements $\cM^+$.

We will be interested only in unique factorisation \mbox{$\bm{n}$-magma}s since this property holds for all combinatorial structures we consider.   

\emph{Thus in the remainder of this paper we assume all \mbox{$\bm{n}$-magma} are  unique factorisation \mbox{$\bm{n}$-magma}s.}

\subsection{General \texorpdfstring{$\bm{n}$}{n}-magma theorems}

In this section we present a number of general results about \mbox{$\bm{n}$-magma}s. 
These results will prove useful in later sections and allow us to describe certain properties which \mbox{$\bm{n}$-magma}s may possess. 
Propositions \ref{prop:magmaMin}, \ref{prop:normedIFF} and \ref{prop:freenMagmaIsomorphism}, along with Theorem \ref{thm:uniqFactNormedThenFree}, are generalisations of results stated and proven in \cite{BRAK}. 

\begin{proposition} \label{prop:magmaMin}
Let \mbox{($\cM$, $f_1$, $\hdots$, $f_k$)} be a normed $(n_1, \hdots, n_k)$-magma with non-empty base set $\cM$, and let the set of elements with minimal norm be $\cM_{\text{min}} \subset \cM$. Then $\cM_{\text{min}}$ is non-empty and all elements of $\cM_{\text{min}}$ are irreducible.
\end{proposition}

\begin{proof}

Clearly we have that $\cM_{\text{min}}$ is non-empty since $\cM$ is non-empty and $\cM_{\text{min}}$ is taken to be the subset of $\cM$ whose elements have minimal norm. 

To prove that all elements of $\cM_{\text{min}}$ are irreducible, proceed by contradiction. Assume that there exists some $m \in \cM_{\text{min}}$ such that $m$ is reducible. Therefore $m \in \text{Img}(f_i)$ for some $i \in \{1, \hdots, k \}$, so there exists $m_1, \hdots, m_{n_i}$ such that $m = f_i(m_1, \hdots, m_{n_i})$. If $n_i = 1$ (that is, $f_i$ is a unary map) then $m = f_i(m_1)$, and so $\norm{m} > \norm{m_1}$. This contradicts the fact that $m \in \cM_{\text{min}}$. If $n_i > 1$, then $\norm{m} \geq \sum_{j = 1}^{n_i} \norm{m_j}$, and since $\text{Img}(\norm{\cdot}) \subseteq \mathbb{N}$, we must have that $\norm{m} > \norm{m_j}$ for each $j \in \{ 1, \hdots, n_i \}$. This again contradicts the fact that $m \in \cM_{\text{min}}$. Thus we conclude that every element of $\cM_{\text{min}}$ is irreducible.
\end{proof}

Now let $\bm{n} = (n_1, \hdots, n_k)$ and consider an arbitrary unique factorisation \mbox{$\bm{n}$-magma} \linebreak \mbox{($\cM$, $f_1$, $\hdots$, $f_k$)}. We have that for all reducible elements $m \in \cM^+$, there exists a unique $i \in \{ 1, \hdots, k \}$ and unique $m_1, \hdots, m_{n_i} \in \cM$ such that
\begin{equation*}
    m = f_i(m_1, \hdots, m_{n_i}).
\end{equation*}

We can recursively define a function $\pi$ as
\begin{equation} \label{eq:pi}
    \pi(m) =
    \begin{cases}
        \left[ \pi(m_1), \hdots, \pi(m_{n_i}) \right], & \text{if $m \in \cM^+$ with $m = f_i(m_1, \hdots, m_{n_i})$}, \\
        m, & \text{if $m \in \cM^0$}.
    \end{cases}
\end{equation}
The bracketed expression $\left[ \pi(m_1), \hdots, \pi(m_{n_i}) \right]$ is considered an $n_i$-tuple. The square parentheses are used to avoid possible ambiguity arising from the use of round parentheses later.

We are interested in when the recursion \eqref{eq:pi} terminates. This motivates the following definition.

\begin{definition}[Finite decomposition \mbox{$\bm{n}$-magma}] \label{def:generalFiniteDecomposition} 
Let \mbox{($\cM$, $f_1$, $\hdots$, $f_k$)} be an \mbox{$\bm{n}$-magma}. If, for all elements $m \in \cM$, the recursive function \eqref{eq:pi} terminates then \mbox{($\cM$, $f_1$, $\hdots$, $f_k$)} will be called a \textbf{finite decomposition \mbox{$\bm{n}$-magma}}. If \mbox{($\cM$, $f_1$, $\hdots$, $f_k$)} is a finite decomposition \mbox{$\bm{n}$-magma}, then $\pi(m)$ will be called the \textbf{decomposition} of $m$.
\end{definition}

If the unique factorisation \mbox{$\bm{n}$-magma} has a norm, then the recursive function \eqref{eq:pi} will always terminate, as given by the following proposition.  

\begin{proposition} \label{prop:normedIFF}
Let \mbox{($\cM$, $f_1$, $\hdots$, $f_k$)} be a unique factorisation \mbox{$\bm{n}$-magma}. Then it is a normed \mbox{$\bm{n}$-magma} if and only if it is a finite decomposition \mbox{$\bm{n}$-magma}.
\end{proposition}

\begin{proof}

\textit{Forward:} Since \mbox{($\cM$, $f_1$, $\hdots$, $f_k$)} is normed, there exists a function $\norm{\cdot}: \cM \rightarrow \mathbb{N}$ which satisfies Definition \ref{def:norm}. 
Now taking any $m \in \cM$, either $m \in \cM^0$ and so the recursion terminates immediately or there exists a unique $i \in \{ 1, \hdots, k \}$ and unique $m_1, \hdots, m_{n_i} \in \cM$ such that $m = f_i(m_1, \hdots, m_{n_i})$. 
In this case, $\norm{m} \geq \sum_{j = 1}^{n_i} \norm{m_j}$, and so $\norm{m} > \norm{m_j}$ for each $j \in \{ 1, \hdots, n_i \}$ because $\norm{\cdot}$ takes values in $\mathbb{N}$. 
This recursive procedure continues until a factor is in $\cM^0$ at which point it terminates from \eqref{eq:pi}, or until a factor has minimal norm. 
In this case Proposition \ref{prop:magmaMin} states that this factor must be in $\cM^0$ and thus the recursion terminates. 
This must occur in a finite number of steps since $\text{Img}(\norm{\cdot}) \subseteq \mathbb{N}$ and, since $\mathbb{N}$ is a well ordered set,  any subset of $\mathbb{N}$ has a least element. 
Thus \mbox{($\cM$, $f_1$, $\hdots$, $f_k$)} is a finite decomposition \mbox{$\bm{n}$-magma}.

\textit{Reverse:}  Since \mbox{($\cM$, $f_1$, $\hdots$, $f_k$)} is a finite decomposition \mbox{$\bm{n}$-magma}, the recursion \eqref{eq:pi} must terminate for all $m \in \cM$. We can define a norm $\norm{\cdot}: \cM \rightarrow \mathbb{N}$ on \mbox{($\cM$, $f_1$, $\hdots$, $f_k$)} as follows. For each $m \in \cM^0$, define $\norm{m} = 1$. For each $m \in \cM^+$, we know that $m = f_i(m_1, \hdots, m_{n_i})$ for unique $i \in \{1, \hdots, k \}$ and unique $m_1, \hdots, m_{n_i} \in \cM$. Defining $\norm{m} = \sum_{j = 1}^{n_i} \norm{m_j}$ for such $m$, we have that $\norm{\cdot}$ is well-defined since $\pi(m)$ terminates and thus contains a finite number of occurrences of elements of $\cM^0$ (for which we have already defined the value of $\norm{\cdot}$). Further, we have that $\norm{\cdot}$ is a norm since it satisfies the conditions of Definition \ref{def:norm}.
\end{proof}

Note that the above proposition holds even if an \mbox{$\bm{n}$-magma} is \emph{not} a unique factorisation \mbox{$\bm{n}$-magma}. 
We chose to prove it only for the case of a unique factorisation \mbox{$\bm{n}$-magma} as this is the only result we require. 
Making this assumption also simplifies the proof considerably.

We now prove the following result, which gives us a combinatorial way to characterise when an \mbox{$\bm{n}$-magma} is free.

\begin{theorem} \label{thm:uniqFactNormedThenFree}
Let $\bm{n} = (n_1, \hdots, n_k)$. If \mbox{($\cM$, $f_1$, $\hdots$, $f_k$)} is a unique factorisation normed \mbox{$\bm{n}$-magma} with non-empty finite set of irreducibles, then \mbox{($\cM$, $f_1$, $\hdots$, $f_k$)} is a free \mbox{$\bm{n}$-magma} generated by the irreducible elements.
\end{theorem}

\begin{proof}
 
Let \mbox{($\cM$, $f_1$, $\hdots$, $f_k$)} be a unique factorisation normed \mbox{$\bm{n}$-magma} with non-empty finite set of irreducibles, $\cM^0$. Take  any set $Y$ such that $\card{Y} = \card{\cM^0}$. 
Let \mbox{($\cN$, $f_1'$, $\hdots$, $f_k'$)} be an arbitrary \mbox{$\bm{n}$-magma} and let $\varphi: Y \rightarrow \cN$ be any map. 
We are required to show that there exists a map $i: Y \rightarrow \cM^0$ such that there exists a unique \mbox{$\bm{n}$-magma} morphism $\theta: \cM \rightarrow \cN$ with the property that $\varphi = \theta \circ i$.
We take $i: Y \rightarrow \cM^0$ to be any bijection from $Y$ to $\cM^0 \subset \cM$ (such a bijection exists since $\varphi$ was chosen so that $\card{Y} = \card{\cM^0}$). We define $\theta: \cM \rightarrow \cN$ as follows:
\begin{enumerate}[label=(\roman*)]
    \item For all $m \in \cM^0$, we have $m = i(y)$ for some $y \in Y$. For such $m$, define \begin{equation} \label{eq:free1}
        \theta(m) = \varphi(y).
    \end{equation}
    \item For all $m \in \cM^+$, where $\cM^+ = \cM \backslash \cM^0$, we recursively define the image under $\theta$ as follows: if $m = f_i(m_1, \hdots, m_{n_i})$, then define
    \begin{equation} \label{eq:free2}
        \theta(m) = f_i'(\theta(m_1), \hdots, \theta(m_{n_i})).
    \end{equation}
\end{enumerate}
From \eqref{eq:free2}, we have that $\theta$ is an \mbox{$\bm{n}$-magma} morphism, while \eqref{eq:free1} ensures that \mbox{$\varphi = \theta \circ i$}. We have that $\theta$ is unique (given the choice of the maps $i$ and $\varphi$). This is because the value of $\theta$ on the set of irreducibles $\cM^0$ along with the recursive definition of $\theta$ for all other elements of $\cM$ uniquely specifies the value of $\theta$ for all elements of $\cM$. This also ensures that $\theta$ is the only map which satisfies both \eqref{eq:free2} and $\varphi = \theta \circ i$.
\end{proof}

Theorem \ref{thm:uniqFactNormedThenFree} shows that the following is sufficient to show that a set $\cM$ along with $k$ maps $f_1, \hdots, f_k$, where $f_i: \cM^{n_i} \rightarrow \cM$ for each $i \in \{ 1, \hdots, k \}$, is a free \mbox{$(n_1, \hdots, n_k)$-magma}:
\begin{enumerate}[label=(\roman*)]
    \item Show that all maps $f_i$ are injective.
    \item Show that the images of the maps are disjoint, that is, for all $i, j \in \{ 1, \hdots, k \}$ such that $i \neq j$, we have $\text{Img}(f_i) \cap \text{Img}(f_j) = \emptyset$.
    \item Show that there exists a map $\norm{\cdot}: \cM \rightarrow \mathbb{N}$ which satisfies Definition \ref{def:norm}.
    \item Determine the set of generators (usually by first determining the range of the maps $\cM^+ = \bigcup_{i = 1}^k \cM_i^+$ and then taking its complement, $\cM^0 = \cM \setminus \cM^+$).
\end{enumerate}

Note, \emph{any} norm will suffice to  prove a particular  $\bm{n}$-magma is  free (using Theorem \ref{thm:uniqFactNormedThenFree}).

The next result is well-known for free structures. 
It states that there exists an isomorphism between any pair of \mbox{$\bm{n}$-magma}s for the same $\bm{n}$. 
Moreover, it states that this isomorphism is unique up to the choice of bijection between the generators. 
We will use this result later to define universal bijections between our combinatorial families.

\begin{proposition} \label{prop:freenMagmaIsomorphism}
Let $Y$ be a set and let \mbox{($\cM$, $f_1$, $\hdots$, $f_k$)} and ($\cN$, $f_1'$, $\hdots$, $f_k'$) be free \mbox{$\bm{n}$-magma}s satisfying free \mbox{$\bm{n}$-magma} universal mapping diagrams as follows:
\begin{equation} \label{eq:commutingDiagrams}
    \begin{tikzcd}
        \cM \arrow[r, "\theta"] & \cM' \\
        Y \arrow[u, "i" left] \arrow[ru, "f" below]
    \end{tikzcd}
    \qquad
    \begin{tikzcd}
        \cN \arrow[r, "\phi"] & \cN' \\
        Y \arrow[u, "j" left] \arrow[ru, "g" below]
    \end{tikzcd}
\end{equation}
Then there exists a unique \mbox{$\bm{n}$-magma} isomorphism $\Gamma: \cM \rightarrow \cN$ such that $\Gamma \circ i = j$ and $\Gamma^{-1} \circ j = i$.
\end{proposition}

We repeat the standard proof since it is a constructive proof and hence  provides the basis for the universal bijection algorithm stated later.

\begin{proof}

Since \mbox{($\cM$, $f_1$, $\hdots$, $f_k$)} is a free \mbox{$\bm{n}$-magma}, the left diagram of \eqref{eq:commutingDiagrams} commutes for all $i: Y \rightarrow \cM$ and all $f: Y \rightarrow \cM'$. Taking $\cM' = \cN$ and $f = j$ gives
\begin{equation} \label{eq:jCommutes}
    j = \theta \circ i.
\end{equation}
Similarly, since ($\cN$, $f_1'$, $\hdots$, $f_k'$) is a free \mbox{$\bm{n}$-magma}, the right diagram of \eqref{eq:commutingDiagrams} commutes for all $j: Y \rightarrow \cN$ and all $g: Y \rightarrow \cN'$. Taking $\cN' = \cM$ and $j = i$ gives 
\begin{equation} \label{eq:iCommutes}
    i = \phi \circ j.
\end{equation}
Together, \eqref{eq:jCommutes} and \eqref{eq:iCommutes} give $j = \theta \circ \phi \circ j$ and $i = \phi \circ \theta \circ i$. Therefore $\theta \circ \phi = \text{id}_{\cM}$ and $\phi \circ \theta = \text{id}_{\cN}$, and thus $\Gamma = \theta$ and $\Gamma^{-1} = \phi$ are isomorphisms.
\end{proof}

Suppose that \mbox{($\cM$, $f_1$, $\hdots$, $f_k$)} and ($\cN$, $f_1'$, $\hdots$, $f_k'$) are free \mbox{$\bm{n}$-magma}s with one generator, for the same $\bm{n}$. Since $\card{Y} = 1$, we have that $i$ and $j$ are unique. From this it follows that the \mbox{$\bm{n}$-magma} isomorphism $\Gamma: \cM \rightarrow \cN$ from Proposition \ref{prop:freenMagmaIsomorphism} is unique.

Proposition \ref{prop:freenMagmaIsomorphism} proves to be very useful since all of the  the combinatorial families considered in the appendix  are free \mbox{$\bm{n}$-magma}s generated by a single element for some $\bm{n}$. 
Suppose that we have two combinatorial families which are both free \mbox{$\bm{n}$-magma}s generated by a single element for the same $\bm{n}$. 
Then this unique isomorphism gives us a one-to-one map between the objects of the two families. 
This map preserves the recursive structure of the objects. 
We will also show that this unique isomorphism preserves the norm of the objects it maps. 
Thus we have a size-preserving recursive bijection between the two combinatorial families. 
This will be used to define the universal bijections.
  
 \section{Combinatorial structures} 

We apply the above results to several well known combinatorial structures counted by Fibonacci, Motzkin, Schr\"oder and Fuss-Catalan numbers.


We begin by defining  these combinatorial families   which we will  generalise to those counted by the \mbox{``$p$-analogue''} of the sequences. 
This was done in \cite{BRAK} for the Catalan numbers by defining the $p$-Catalan numbers, where $p \in \mathbb{N} = \{ 1, 2, 3, \hdots \}$, by $C_0(p) = p$ and
\begin{equation} \label{eq:pCatalan}
    C_n(p) = p^{n + 1} \frac{1}{n + 1} \binom{2n}{n}, \qquad n \geq 1.
\end{equation}

A natural interpretation of these $p$-analogue sequences (in most instances) is as a variation of the combinatorial family where we allow some part of the object to be coloured in any of $p$ colours. 
In particular, we will see that it is the part which corresponds to the generator of the family. 
The $p$-analogue sequences   arise naturally when using certain set constructions which give the original sequences. 
The well known cases correspond to $p=1$.

\begin{definition} \label{def:pSequences}
Let $p \in \mathbb{N}$, and define the following sequences:
\begin{enumerate}[label=(\roman*)]
    \item Define the \textbf{$\bm{p}$-Fibonacci numbers} $F_n(p)$ by the recurrence relation
    \begin{equation} \label{eq:pFibonacciRecurrence}
        F_n(p) = F_{n - 1}(p) + F_{n - 2}(p), \qquad n \geq 2,
    \end{equation}
    with $F_0(p) = 0$ and $F_1(p) = p$.
    \item Define the \textbf{$\bm{p}$-Motzkin numbers} $M_n(p)$ by the recurrence relation
    \begin{equation} \label{eq:pMotzkinRecurrence}
        M_n(p) = M_{n - 1}(p) + \sum_{k = 0}^{n - 2} M_k(p) \, M_{n - k - 2}(p), \qquad n \geq 2,
    \end{equation}
    with $M_0(p) = M_1(p) = p$.
    \item Define the (little) \textbf{$\bm{p}$-Schr\"oder numbers} $S_n(p)$ by the recurrence relation
    \begin{equation} \label{eq:pSchroderRecurrence}
        S_n(p) = S_{n - 1}(p) + \sum_{k = 0}^{n - 1} S_k(p) \, S_{n - k - 1}(p), \qquad n \geq 1,
    \end{equation}
    with $S_0(p) = p$.
    \item Define the \textbf{order 3 $\bm{p}$-Fuss-Catalan numbers} (from here onwards, $p$-Fuss-Catalan numbers) $T_n(p)$ by the recurrence relation
    \begin{equation} \label{eq:pTernaryRecurrence}
        T_n(p) = \sum_{i = 0}^{n - 1} \sum_{j = 0}^{n - 1 - i} T_i(p) \, T_j(p) \, T_{n - 1 - i -j}(p), \qquad n \geq 1,
    \end{equation}
    with $T_0(p) = p$. 
\end{enumerate}
\end{definition}

We will refer to the order 3 Fuss-Catalan numbers simply as the \textit{Fuss-Catalan numbers}. 
We do not consider any other order of Fuss-Catalan number, although one could easily extend these ideas to higher orders in an obvious way. 
See \cite{AVAL20084660} for further discussion of the general Fuss-Catalan numbers.

For each of these four $p$-sequences we state propositions giving the algebraic equation satisfied by their generating functions and an expression for the counting sequences.  
The former is derived in the standard way from the above defining recurrence relations and the latter are derived using the Lagrange inversion formula. 
Since these are standard computations we do not provide any details.

\begin{proposition}
The generating function $F(x) = \sum_{n \geq 0} F_n(p) \, x^n$ for the $p$-Fibonacci numbers satisfies the algebraic equation
\begin{equation*}
    F(x) = p x + x F(x) + x^2 F(x).
\end{equation*}
and
\begin{equation*}
    F_n(p) = p \sum_{k = 0}^{\left\lfloor \frac{n - 1}{2} \right\rfloor} \binom{n - k - 1}{k}.
\end{equation*}
\end{proposition}

\begin{proposition}
The generating function $M(x) = \sum_{n \geq 0} M_n(p) \, x^n$ for the $p$-Motzkin numbers satisfies the algebraic equation
\begin{equation*}
    M(x) = p + x M(x) + x^2 M(x)^2.
\end{equation*}
and
\begin{equation*}
    M_n(p) = 
    \sum_{k = 0}^{\left\lfloor \frac{n}{2} \right\rfloor} \binom{n}{2k} C_k(p),
\end{equation*}
where   $C_k(p)$ is the $k^\text{th}$ $p$-Catalan number \eqref{eq:pCatalan}.
\end{proposition}

\begin{proposition}
The generating function $S(x) = \sum_{n \geq 0} S_n(p) \, x^n$ for the $p$-Schr\"oder numbers satisfies the algebraic equation
\begin{equation*}
    S(x) = p + x S(x) + x S(x)^2.
\end{equation*}
and
\begin{equation*}
    S_n(p) = 
    \sum_{k = 0}^{n} \binom{2n - k}{k} C_{n - k}(p),
\end{equation*}
where   $C_k(p)$ is the $k^\text{th}$ $p$-Catalan number \eqref{eq:pCatalan}.
\end{proposition}

\begin{proposition}
The generating function $T(x) = \sum_{n \geq 0} T_n(p) \, x^n$ for the $p$-Fuss-Catalan numbers satisfies the algebraic equation
\begin{equation*}
    T(x) = p + x T(x)^3.
\end{equation*}
and 
\begin{equation*}
    T_n(p) = \frac{1}{3n + 1} \binom{3n + 1}{n} p^{2n + 1}.
\end{equation*}
\end{proposition}

\section{Fibonacci normed (1,1)-magmas} 
\label{section:fib11magmas}
 
In this section, we discuss \mbox{(1,1)-magmas} and show if an appropriate norm is defined  they are related to combinatorial families which are enumerated by the Fibonacci numbers. 
We first define the Cartesian \mbox{(1,1)-magma} and show its relationship with the $p$-Fibonacci numbers. 
We then define  a universal bijection between any two Fibonacci normed \mbox{(1,1)-magmas}.

 \subsection{Cartesian \mbox{(1,1)-magma}} \label{section:standard11magma}

We   define arguably the simplest  \mbox{(1,1)-magma}, which we call the Cartesian \mbox{(1,1)-magma}. 

\begin{definition}[Cartesian \mbox{(1,1)-magma}] \label{def:standard11magma}
\begin{subequations}
    Let $X$ be a non-empty finite set. Define the sequence $\cW_n(X)$ of sets of words in the alphabet $\{ u_1, u_2 \} \cup X$ as follows:
    \begin{align}
        \cW_1(X) & = X, \\
        \cW_2(X) & = \{ u_1 w : \, w \in \cW_1(X) \}, \\
        \cW_n(X) & = \{ u_1 w : \, w \in \cW_{n - 1}(X) \} \cup \{ u_2 w : \, w \in \cW_{n - 2}(X) \}, \qquad n \geq 3, \label{eq:cart11magma}
    \end{align}
\end{subequations}
where $u_i w$ is the concatenation of the symbol $u_i$ with the word $w$, for $i = 1, 2$.

Let $\cW_X = \bigcup_{i \geq 1} \cW_i(X)$ and define two unary maps,
\begin{align*}
    & \mu_1: \cW_X \rightarrow \cW_X, \\
    & \mu_2: \cW_X \rightarrow \cW_X,
\end{align*} 
as follows:
\begin{align*}
    \mu_1(w) & = u_1 w, \\
    \mu_2(w) & = u_2 w.
\end{align*}
The triple \mbox{$(\cW_X, \mu_1, \mu_2)$} will be called the \textbf{Cartesian \mbox{(1,1)-magma} generated by $\bm{X}$}.
\end{definition}

If $X = \{ \epsilon \}$, the sequence of sets $\cW_n(X)$ defining the base set $\cW_X$ of the Cartesian \mbox{(1,1)-magma} begins as follows:
\begin{align*}
    \cW_1(X) & = \{ \epsilon \}, \\
    \cW_2(X) & = \{ u_1 \epsilon \}, \\
    \cW_3(X) & = \{ u_1 u_1 \epsilon, \  u_2 \epsilon \}, \\
    \cW_4(X) & = \{ u_1 u_1 u_1 \epsilon, \  u_1 u_2 \epsilon, \ u_2 u_1 \epsilon \}, \\
    \cW_5(X) & = \{ u_1 u_1 u_1 u_1 \epsilon, \  u_1 u_1 u_2 \epsilon, \ u_1 u_2 u_1 \epsilon, \ u_2 u_1 u_1 \epsilon, \ u_2 u_2 \epsilon \}, \\
    & \ \, \vdots
\end{align*}

Having defined the Cartesian \mbox{(1,1)-magma}, it is possible to  prove that it is free directly from Definition \ref{def:freenmagma} without introducing any norm. However we will take the shorter route by showing there exists a norm, hence along with the other conditions required  by Theorem \ref{thm:uniqFactNormedThenFree}, we will prove it is free.
 
\begin{theorem}
The Cartesian \mbox{(1,1)-magma} of Definition \ref{def:standard11magma}  is a free \mbox{(1,1)-magma}.
\end{theorem}

We will show that \mbox{($\mathcal{W}_X$, $\mu_1$, $\mu_2$)} is a unique factorisation normed \mbox{(1,1)-magma} with set of irreducibles equal to $X$. Then, by Theorem \ref{thm:uniqFactNormedThenFree}, we will have that \mbox{($\mathcal{W}_X$, $\mu_1$, $\mu_2$)} is a free \mbox{(1,1)-magma} generated by $X$.

\begin{proof}
Suppose that $w_1, w_2 \in \cW_X$ are such that $\mu_1(w_1) = \mu_1(w_2)$. We have $\mu_1(w_1) = u_1 w_1$ and $\mu_1(w_2) = u_1 w_2$ and hence $u_1 w_1 = u_1 w_2$  
thus have $w_1 = w_2$. Similarly we have that $\mu_2$ is injective since $\mu_2(w_1) = \mu_2(w_2)$ for $w_1, w_2 \in \cW_X$ implies $u_2 w_1 = u_2 w_2$ and hence $w_1 = w_2$. Thus we have that both $\mu_1$ and $\mu_2$ are injective.
 To show $\text{Img}(\mu_1) \cap \text{Img}(\mu_2) = \emptyset$, proceed by contradiction. Suppose that $$\text{Img}(\mu_1) \cap \text{Img}(\mu_2) \neq \emptyset$$ and take $w \in \text{Img}(\mu_1) \cap \text{Img}(\mu_2)$. Since $w \in \text{Img}(\mu_1)$, we have $w = \mu_1(w_1) = u_1 w_1$ for some $w_1 \in \cW_X$. Similarly, since $w \in \text{Img}(\mu_2)$, we must have $w = \mu_2(w_2) = u_2 w_2$ for some $w_2 \in \cW_X$. Therefore $u_1 w_1 = u_2 w_2$ and hence $u_1 = u_2$, thus giving a contradiction.

Thus the maps $\mu_1$ and $\mu_2$ are injective and $\text{Img}(\mu_1) \cap \text{Img}(\mu_2) = \emptyset$, so \mbox{($\mathcal{W}_X$, $\mu_1$, $\mu_2$)} is a unique factorisation \mbox{(1,1)-magma}.

The set of irreducibles is $X$ since this is the complement of $\text{Img}(\mu_1) \cup \text{Img}(\mu_2)$.

Take $w \in \cW_X$, supposing that $w \in \cW_n(X)$ and hence that $\norm{w} = n$. 
We have from \eqref{eq:cart11magma} that $\mu_1(w) \in \cW_{n + 1}(X)$ and $\mu_2(w) \in \cW_{n + 2}(X)$. Therefore $\norm{\mu_1(w)} = n + 1$ and $\norm{\mu_2(w)} = n + 2$. Thus $\norm{\mu_1(w)} > \norm{w}$ and $\norm{\mu_2(w)} > \norm{w}$ and so by Definition \ref{def:norm}, $\norm{\cdot}$ is a norm.
Therefore \mbox{($\mathcal{W}_X$, $\mu_1$, $\mu_2$)} is a normed \mbox{(1,1)-magma}.
\end{proof}


\begin{proposition} \label{prop:fib11magma}
Let ($\mathcal{W}_{X}$, $\mu_1$, $\mu_2$) be the Cartesian \mbox{(1,1)-magma} generated by the set $X$, where $\card{X} = p$, and define the map $\norm{\cdot}_{F}: \mathcal{W}_{X} \rightarrow{ \mathbb{N}}$ by $\norm{m}_F = n$ when $m \in \mathcal{W}_n(X)$. If
\begin{equation*}
    N_n = \{ m \in \mathcal{W}_X: \, \norm{m}_F = n \}, \quad n \geq 1,
\end{equation*}
then
\begin{equation*}
    \card{N_n} = F_n(p), \qquad n \geq 1,
\end{equation*}
where $F_n(p)$ is the $n$th $p$-Fibonacci number of Definition \ref{def:pSequences}.
\end{proposition}

\begin{proof}
First, note that $N_n = \cW_n(X)$ since $\norm{m}_F = n$ if and only if $m \in \cW_n(X)$. Now, we have $\cW_1(X) = X$ so $\card{N_1} = \card{\cW_1(X)} = \card{X} = p$, and $\cW_2(X) = \{ u_1 w : \ w \in \cW_1(X) \}$ so
\begin{equation*}
    \card{N_2} = \card {\cW_2(X)} = \card{ \{ u_1 w : \, w \in \cW_1(X) \} } = \card{\cW_1(X)} = p.
\end{equation*}
For $n \geq 3$, \eqref{eq:cart11magma} gives
\begin{align*}
    \card{N_n} & = \card{ \{ u_1 w : \, w \in \cW_{n - 1}(X) \} } + \card{ \{ u_2 w : \, w \in \cW_{n - 2}(X) \} } \\
    & = \card{\cW_{n - 1}(X)} + \card{\cW_{n - 2}(X)} \\
    & = \card{N_{n - 1}} + \card{N_{n - 2}}
\end{align*}
This is exactly the $p$-Fibonacci recurrence \eqref{eq:pFibonacciRecurrence}.
\end{proof}

\begin{corollary}
Let ($\mathcal{W}_{\epsilon}$, $\mu_1$, $\mu_2$) be the Cartesian \mbox{(1,1)-magma} generated by the single element $\epsilon$. If $N_n$ and $\norm{\cdot}_{F}: \mathcal{W}_{\epsilon} \rightarrow{ \mathbb{N}}$ are as defined in Proposition \ref{prop:fib11magma}, then
\begin{equation*}
    \lvert N_n \rvert = F_n, \qquad n \geq 1,
\end{equation*}
where $F_n$ is the $n$th Fibonacci number.
\end{corollary}

Take the Cartesian \mbox{(1,1)-magma} generated by $\{ \epsilon \}$, ($\mathcal{W}_{\epsilon}$, $\mu_1$, $\mu_2$) and  note that the norm \mbox{$\norm{\cdot}_F: \mathcal{W}_{\epsilon} \rightarrow \mathbb{N}$} from above could equivalently be defined by requiring that 
\begin{subequations} \label{eq:fibNorm0}
    \begin{align}
            \norm{\epsilon}_F & = 1, \\
            \norm{\mu_1(w)}_F & = \norm{w}_F + 1, \\
            \norm{\mu_2(w)}_F & = \norm{w}_F + 2,
    \end{align}
\end{subequations}
for all $w \in \mathcal{W}_{\epsilon}$. 


This motivates the following definition. We seek to characterise when a \mbox{(1,1)-magma} is associated with the Fibonacci numbers. As we have noted, we can define many norms on the same (1,1)-magma. Defining a norm $\norm{\cdot}_F$ which satisfies \eqref{eq:fibNorm0} gives us the Fibonacci numbers. Thus we define a Fibonacci normed \mbox{(1,1)-magma} to be a free (1,1)-magma with a single generator \textit{along with} a particular norm function which satisfies \eqref{eq:fibNorm0} for the relevant \mbox{(1,1)-magma} generator and unary maps.

\begin{definition}[Fibonacci normed \mbox{(1,1)-magma}] \label{def:fibonacci11magma}
Let \mbox{($\cM$, $f_1$, $f_2$)} be a unique factorisation normed \mbox{(1,1)-magma} with only one irreducible element, $\epsilon$. Let $\norm{\cdot}_F: \cM \rightarrow \mathbb{N}$ be a norm satisfying
\begin{subequations} \label{eq:fibNorm}
    \begin{align}
        \norm{\epsilon}_F & = 1, \\
        \norm{f_1(m)}_F & = \norm{m}_F + 1, \\
        \norm{f_2(m)}_F & = \norm{m}_F + 2,
    \end{align}
\end{subequations}
for all $m \in \cM$. Then \mbox{($\cM$, $f_1$, $f_2$)} with the norm $\norm{\cdot}_F$ is called a \textbf{Fibonacci normed \mbox{(1,1)-magma}}.
\end{definition}

In the remainder of this section, we reference a number of combinatorial structures which are counted by the Fibonacci numbers. Further details of these are provided in Appendix \ref{appendix:Fibonacci}. We take the convention that the two unary maps are called $f$ and $g$. We make it clear that we are using the maps specific to a certain family by placing a subscript on each of the maps. This subscript contains the number assigned to that family in the appendix. We choose to call the unique generator in each family $\epsilon$, and make it clear which family it comes from via the subscript. We name the maps in such a way that our \mbox{(1,1)-magma} is \mbox{\mbox{($\cM$, $f$, $g$)}}, and hence
\begin{equation}
    \norm{f(m)} = \norm{m} + 1, \qquad \norm{g(m)} = \norm{m} + 2,
\end{equation}
for all $m \in \cM$.

We have seen that any Fibonacci normed \mbox{(1,1)-magma} is such that the number of elements of the base set with norm $n$ is given by the $n$th Fibonacci number. We now present a simple example of a Fibonacci normed \mbox{(1,1)-magma}, taking the family of Fibonacci tilings \fibFamRef{fibFamily:tilings}. This family arises by considering the number of ways to tile a $1 \times n$ board using $1 \times 1$ squares and $1 \times 2$ dominoes. The number of ways to tile such a board is given by $F_n$. We can define the Fibonacci tiling \mbox{(1,1)-magma} \fibMag{fibFamily:tilings} as follows:
\begin{itemize}
    \item Take the base set \fibBaseRef{fibFamily:tilings} to be the set of all tilings of a $1 \times n$ board using $1 \times 1$ squares and $1 \times 2$ dominoes, for all $n \in \mathbb{N}_0$. Note that we consider the trivial empty tiling to be the only way to tile a $1 \times 0$ board (i.e. an empty board).
    \item Define one unary map to take a tiling of a $1 \times n$ board and add a single $1 \times 1$ square to the right to give a tiling of a $1 \times (n + 1)$ board. Call this map $f_{\ref{fibFamily:tilings}}$. Schematically:
    \begin{align*} \def\scl{0.5}
        f_{\ref{fibFamily:tilings}} \left(
        \begin{tikzpicture}[baseline={([yshift=-.5ex]current bounding box.center)}, scale = \scl]
            \dummyNodes[0.5][-0.2][0.5][1.2]
            \draw[fill = \colOne, opacity = \opac] (0,0) rectangle (6,1);
            \node at (3,0.5) {\footnotesize $t$};
        \end{tikzpicture}
        \right)
        =
        \begin{tikzpicture}[baseline={([yshift=-.5ex]current bounding box.center)}, scale = \scl]
            \dummyNodes[0.5][-0.2][0.5][1.2]
            \draw[fill = \colOne, opacity = \opac] (0,0) rectangle (6,1);
            \node at (3,0.5) {\footnotesize $t$};
            \draw (6,0) \onetile \finaledge;
        \end{tikzpicture}
    \end{align*}
    \item Define the other unary map to take a tiling of a $1 \times n$ board and add a $1 \times 2$ domino to the right to give a tiling of a $1 \times (n + 2)$ board. Call this map $g_{\ref{fibFamily:tilings}}$. Schematically:
    \begin{align*} \def\scl{0.5}
        g_{\ref{fibFamily:tilings}} \left(
        \begin{tikzpicture}[baseline={([yshift=-.5ex]current bounding box.center)}, scale = \scl]
            \dummyNodes[0.5][-0.2][0.5][1.2]
            \draw[fill = \colOne, opacity = \opac] (0,0) rectangle (6,1);
            \node at (3,0.5) {\footnotesize $t$};
        \end{tikzpicture}
        \right)
        =
        \begin{tikzpicture}[baseline={([yshift=-.5ex]current bounding box.center)}, scale = \scl]
            \dummyNodes[0.5][-0.2][0.5][1.2]
            \draw[fill = \colOne, opacity = \opac] (0,0) rectangle (6,1);
            \node at (3,0.5) {\footnotesize $t$};
            \draw (6,0) \twotiles \finaledge;
        \end{tikzpicture}
    \end{align*}
    \item The only generator is the trivial empty tiling of an empty board: $\epsilon_{\ref{fibFamily:tilings}} = \emptyset$. This is the only element in the base set which is not in the image of one of the two maps.
\end{itemize}

It is simple to see that \fibMag{fibFamily:tilings} is a unique factorisation (1,1)-magma. This is because we can form any tiling of a $1 \times n$ board by applying a unique sequence of compositions of the two maps to the generator. Take for example the following tiling:
\begin{align*}
    \begin{tikzpicture}[baseline={([yshift=-.5ex]current bounding box.center)}, scale = 0.5]
        \draw (0,0) \onetile \twotiles \onetile \finaledge;
    \end{tikzpicture}
\end{align*}
We see that this can be constructed as follows:
\begin{align*}
    f_{\ref{fibFamily:tilings}} \left( g_{\ref{fibFamily:tilings}} \left( f_{\ref{fibFamily:tilings}} \left(
    \epsilon
    \right) \right) \right)
    = 
    f_{\ref{fibFamily:tilings}} \left( g_{\ref{fibFamily:tilings}} \left( f_{\ref{fibFamily:tilings}} \left( 
    \emptyset
    \right) \right) \right)
    =
    f_{\ref{fibFamily:tilings}} \left( g_{\ref{fibFamily:tilings}} \left( \ 
    \begin{tikzpicture}[baseline={([yshift=-.5ex]current bounding box.center)}, scale = 0.5]
        \draw (0,0) \onetile \finaledge;
        \dummyNodes[0.5][0.2][0.5][0.8]
    \end{tikzpicture}
    \ \right) \right)
    =
    f_{\ref{fibFamily:tilings}} \left(
    \begin{tikzpicture}[baseline={([yshift=-.5ex]current bounding box.center)}, scale = 0.5]
        \draw (0,0) \onetile \twotiles \finaledge;
        \dummyNodes[0.2][0.2][2.8][0.8]
    \end{tikzpicture}
    \right)
    =
    \begin{tikzpicture}[baseline={([yshift=-.5ex]current bounding box.center)}, scale = 0.5]
        \draw (0,0) \onetile \twotiles \onetile \finaledge;
    \end{tikzpicture}
\end{align*}

We define the norm $\norm{\cdot}_F$ of a tiling of a $1 \times n$ board to be $n + 1$. Thus we see that the norm satisfies \eqref{eq:fibNorm} of Definition \ref{def:fibonacci11magma}. That is,
\begin{align*}
    \norm{\epsilon_{\ref{fibFamily:tilings}}}_F & = 1, \\
    \norm{f_{\ref{fibFamily:tilings}}(t)}_F & = \norm{t}_F + 1, \\
    \norm{g_{\ref{fibFamily:tilings}}(t)}_F & = \norm{t}_F + 2,
\end{align*}
for all $t \in \mathFibBaseRef{fibFamily:tilings}$.

Therefore the \mbox{(1,1)-magma} \fibMag{fibFamily:tilings} along with the norm just defined is indeed a Fibonacci normed \mbox{(1,1)-magma}.

Some Fibonacci families are such that there is no natural interpretation of the generator $\epsilon$. This is usually the case when there are $F_{n + 2}$ objects with traditional size parameter $n$. This is because the generator then corresponds to an object with traditional size parameter equal to -1. In such cases, we define the \mbox{(1,1)-magma} \mbox{($\cM$, $f$, $g$)} by the following procedure. We simply take $\epsilon$ to be an arbitrary symbol denoting the generator and define the two maps $f$ and $g$ as follows:
\begin{enumerate}[label=(\roman*)]
    \item Define $f(\epsilon)$.
    \item Define $g(\epsilon)$.
    \item For all $m \in \cM \backslash \{ \epsilon \}$, define $f(m)$.
    \item For all $m \in \cM \backslash \{ \epsilon \}$, define $g(m)$.
\end{enumerate}
This process completely specifies the two maps and the base set, so we have a well-defined free \mbox{(1,1)-magma}. For examples of when this procedure is required, see the following Fibonacci families in Appendix \ref{appendix:Fibonacci}:
\begin{itemize}
    \item \fibFamRef{fibFamily:binarySeqs}: Binary sequences with no consecutive 1's,
    \item \fibFamRef{fibFamily:subsets}: Subsets with no consecutive integers.
\end{itemize}

\subsection{Free \mbox{(1,1)-magma} isomorphisms and a universal bijection }

We have seen that any Fibonacci normed \mbox{(1,1)-magma} is such that the number of elements of the base set with norm $n$ is given by the $n$th Fibonacci number. Using the result of Proposition \ref{prop:freenMagmaIsomorphism}, there exists a unique \mbox{(1,1)-magma} isomorphism between any two Fibonacci normed \mbox{(1,1)-magmas}. As we will see, this isomorphism preserves the norms of the objects in the case that both \mbox{(1,1)-magmas} are Fibonacci normed \mbox{(1,1)-magmas}. Thus this gives a size-preserving bijection between the two Fibonacci families. We therefore make explicit just how this isomorphism is defined so that we are able to make use of it for specific families.

\begin{definition}[Universal bijection] \label{def:11univBij}
Let \mbox{($\cM$, $f$, $g$)} and \mbox{($\cN$, $f'$, $g'$)} be free \mbox{(1,1)-magmas} with generating sets $X_{\cM}$ and $X_{\cN}$ respectively, with $\card{X_{\cM}} = \card{X_{\cN}}$. Let $\sigma: X_{\cM} \rightarrow X_{\cN}$ be any bijection, and define the map $\Upsilon: \cM \rightarrow \cN$ as follows: for all $m \in \cM \setminus X_{\cM}$,
\begin{enumerate}[label=(\roman*)]
    \item Decompose $m$ into an expression in terms of generators $\epsilon_i \in X_{\cM}$ and the unary maps $f$ and $g$.
    \item In the decomposition of $m$, replace every occurrence of $\epsilon_i$ with $\sigma(\epsilon_i)$, every occurence of $f$ with $f'$ and every occurrence of $g$ with $g'$. Call this expression $\upsilon(m)$.
    \item Define $\Upsilon(m)$ to be $\upsilon(m)$, that is, evaluate all maps in $\upsilon(m)$ to give an element of $\cN$.
\end{enumerate}
\end{definition}

This leads to the following proposition. The proposition follows immediately from the fact that $\Upsilon$ is equal to the map $\Gamma$ from Proposition \ref{prop:freenMagmaIsomorphism}.

\begin{proposition} \label{cor:11univBij}
Let $\Upsilon: \cM \rightarrow \cN$ be the map of Definition \ref{def:11univBij}. Then $\Upsilon$ is a free \mbox{(1,1)-magma} isomorphism.
\end{proposition}

Schematically, we can write $\Upsilon$ as follows:
\begin{equation}
    m \qquad \overset{\text{decompose}}{\xrightarrow{\hspace*{1.5cm}}} \qquad \underset{\substack{\epsilon_i \rightarrow \, \sigma(\epsilon_i), \ f \rightarrow \, f', \ g \rightarrow \, g'}}{\text{substitute}} \qquad \overset{\text{evaluate}}{\xrightarrow{\hspace*{1.5cm}}} \qquad n.
\end{equation}

Since $\Upsilon$ is an isomorphism between the free \mbox{(1,1)-magmas} \mbox{($\cM$, $f$, $g$)} and \mbox{($\cN$, $f'$, $g'$)}, we have that $\Upsilon$ defines a bijection between the base sets $\cM$ and $\cN$. Further to this, it gives us that this bijection is recursive: if $m = f(m_0)$, then
\begin{equation*}
    \Upsilon(m) = f'(\Upsilon(m_0)),
\end{equation*}
and if $m = g(m_0)$, then
\begin{equation*}
    \Upsilon(m) = g'(\Upsilon(m_0)).
\end{equation*}

It is also important to note that if the free \mbox{(1,1)-magmas} have norms satisfying certain conditions, then the norm is preserved under the map $\Upsilon$. Suppose that \mbox{($\cM$, $f$, $g$)} and \mbox{($\cN$, $f'$, $g'$)} have norms \mbox{$\norm{\cdot}_{\cM}: \cM \rightarrow \mathbb{N}$} and \mbox{$\norm{\cdot}_{\cN}: \cN \rightarrow \mathbb{N}$} respectively. If:
\begin{enumerate}[label=(\roman*)]
    \item \label{cond:fib1} $\norm{m}_{\cM} = \norm{\sigma(m)}_{\cN}$ for all $m \in X_{\cM}$,
    \item \label{cond:fib2} $\norm{f(m)}_{\cM} = \norm{m}_{\cM} + \kappa_1$ for all $m \in \cM$ and $\norm{f'(n)}_{\cN} = \norm{n}_{\cN} + \kappa_1$ for all $n \in \cN$, where $\kappa_1 \in \mathbb{N}$, and
    \item \label{cond:fib3} $\norm{g(m)}_{\cM} = \norm{m}_{\cM} + \kappa_2$ for all $m \in \cM$ and $\norm{g'(n)}_{\cN} = \norm{n}_{\cN} + \kappa_2$ for all $n \in \cN$, where $\kappa_2 \in \mathbb{N}$,
\end{enumerate}
then we have
\begin{equation*}
    \norm{m}_{\cM} = \norm{\Upsilon(m)}_{\cN}, \qquad m \in \cM.
\end{equation*}
This result follows immediately by considering the decomposed expressions for $m$ and $\Upsilon(m)$.

Combinatorially we are primarily interested in bijections between structures of the same ``size'' and thus we are interested in bijections which preserve the norm. This is the case for Fibonacci normed \mbox{(1,1)-magmas} which are invariant under the map $\Upsilon$. 

We now present a number of examples illustrating the universal bijection of Definition \ref{def:11univBij}. We will consider the following Fibonacci normed \mbox{(1,1)-magmas}:
\begin{itemize}
    \item Fibonacci tilings \fibMag{fibFamily:tilings},
    \item Reflections through two plates of glass \fibMag{fibFamily:glassReflections},
    \item Binary sequences with no consecutive 1's \fibMag{fibFamily:binarySeqs},
    \item Compositions with no 1's \fibMag{fibFamily:compositionsNo1s}.
\end{itemize}
See Appendix \ref{appendix:Fibonacci} for the definitions and details of each of these \mbox{(1,1)-magmas}.

We begin by demonstrating the bijection from the family of Fibonacci tilings to the family of reflections through two plates of glass. First, take a Fibonacci tiling, and decompose it down into its factorised form:
\begin{equation*}    
    \begin{tikzpicture}[baseline={([yshift=-.5ex]current bounding box.center)}, scale = 0.5]
        \draw (0,0) \onetile \twotiles \onetile \finaledge;
    \end{tikzpicture}
    =
    f_{\ref{fibFamily:tilings}} \left( \
    \begin{tikzpicture}[baseline={([yshift=-.5ex]current bounding box.center)}, scale = 0.5]
        \draw (0,0) \onetile \twotiles \finaledge;
    \end{tikzpicture}
    \ \right)
    =
    f_{\ref{fibFamily:tilings}} \left( g_{\ref{fibFamily:tilings}} \left( \ 
    \begin{tikzpicture}[baseline={([yshift=-.5ex]current bounding box.center)}, scale = 0.5]
        \draw (0,0) \onetile \finaledge;
    \end{tikzpicture}
    \ \right) \right)
    =
    f_{\ref{fibFamily:tilings}} \left( g_{\ref{fibFamily:tilings}} \left( f_{\ref{fibFamily:tilings}} \left( 
    \emptyset
    \right) \right) \right)
    =
    f_{\ref{fibFamily:tilings}} \left( g_{\ref{fibFamily:tilings}} \left( f_{\ref{fibFamily:tilings}} \left( 
    \epsilon_{\ref{fibFamily:tilings}}
    \right) \right) \right)
\end{equation*}
Next, we replace every occurrence of the generator $\epsilon_{\ref{fibFamily:tilings}}$ with the generator of the reflections through two plates of glass, $\epsilon_{\ref{fibFamily:glassReflections}}$. We also replace each map $f_{\ref{fibFamily:tilings}}$ with $f_{\ref{fibFamily:glassReflections}}$ and each map $g_{\ref{fibFamily:tilings}}$ with $g_{\ref{fibFamily:glassReflections}}$. After making these substitutions, evaluate the resulting expression to obtain an element of the second family:
\def\scl{0.3}
\begin{equation*}
    f_{\ref{fibFamily:glassReflections}} \left( g_{\ref{fibFamily:glassReflections}} \left( f_{\ref{fibFamily:glassReflections}} \left( 
    \epsilon_{\ref{fibFamily:glassReflections}}
    \right) \right) \right) 
    = f_{\ref{fibFamily:glassReflections}} \left( g_{\ref{fibFamily:glassReflections}} \left(
        \begin{tikzpicture}[baseline={([yshift=-.5ex]current bounding box.center)}, scale = \scl]
        \plates[3]
        \dummyNodes[1][-3.2][1][1.2]
        \draw[red, ->] (0.5,1) -- (2.5,-3);
    \end{tikzpicture}
    \right) \right)
    =
    f_{\ref{fibFamily:glassReflections}} \left(
    \begin{tikzpicture}[baseline={([yshift=-.5ex]current bounding box.center)}, scale = \scl]
        \plates[4]
        \dummyNodes[1][-3.2][1][1.2]
        \draw[red] (1,0) \glassDn \glassDn \glassUp \glassDn;
        \draw[red, ->] (0.5,1) -- (1,0) (3,-2) -- (3.5,-3);
    \end{tikzpicture}
    \right)
   = \begin{tikzpicture}[baseline={([yshift=-.5ex]current bounding box.center)}, scale = \scl]
        \plates[5]
        \dummyNodes[1][-3.2][1][1.2]
        \draw[red] (1,0) \glassDn \glassDn \glassUp \glassDn \glassUp \glassUp;
        \draw[red, ->] (0.5,1) -- (1,0) (4,0) -- (4.5,1);
    \end{tikzpicture}
\end{equation*}
Thus the universal bijection maps
\begin{equation*}
    \begin{tikzpicture}[baseline={([yshift=-.5ex]current bounding box.center)}, scale = 0.5]
        \draw (0,0) \onetile \twotiles \onetile \finaledge;
    \end{tikzpicture}
    \quad \mapsto \quad
    \begin{tikzpicture}[baseline={([yshift=-.5ex]current bounding box.center)}, scale = \scl]
        \plates[5]
        \dummyNodes[1][-3.2][1][1.2]
        \draw[red] (1,0) \glassDn \glassDn \glassUp \glassDn \glassUp \glassUp;
        \draw[red, ->] (0.5,1) -- (1,0) (4,0) -- (4.5,1);
    \end{tikzpicture}
\end{equation*}
If instead we were seeking a bijection between Fibonacci tilings and binary sequences with no consecutive 1's, then we would simply replace all parts of the factorised expression for the Fibonacci tiling with the parts corresponding to binary sequences with no consecutive 1's, as follows:
\begin{equation*}
    f_{\ref{fibFamily:binarySeqs}} \left( g_{\ref{fibFamily:binarySeqs}} \left( f_{\ref{fibFamily:binarySeqs}} \left( 
        \epsilon_{\ref{fibFamily:binarySeqs}}
    \right) \right) \right)
    =
    f_{\ref{fibFamily:binarySeqs}} \left( g_{\ref{fibFamily:binarySeqs}} \left(
        \emptyset
    \right) \right)
    =
    f_{\ref{fibFamily:binarySeqs}} \left(
    01
    \right)
    =
    010.
\end{equation*}

Similarly for the Fibonacci family of compositions containing no 1's:
\begin{equation*}
    f_{\ref{fibFamily:compositionsNo1s}} \left( g_{\ref{fibFamily:compositionsNo1s}} \left( f_{\ref{fibFamily:compositionsNo1s}} \left( 
        \epsilon_{\ref{fibFamily:compositionsNo1s}}
    \right) \right) \right) 
    =
    f_{\ref{fibFamily:compositionsNo1s}} \left( g_{\ref{fibFamily:compositionsNo1s}} \left( f_{\ref{fibFamily:compositionsNo1s}} \left(
    2
    \right) \right) \right)
    =
    f_{\ref{fibFamily:compositionsNo1s}} \left( g_{\ref{fibFamily:compositionsNo1s}} \left( 
    3
    \right) \right)
    =
    f_{\ref{fibFamily:compositionsNo1s}} \left(
    3 + 2
    \right)
    =
    3 + 3.
\end{equation*}

So we see that the universal bijection gives each of the following bijections:
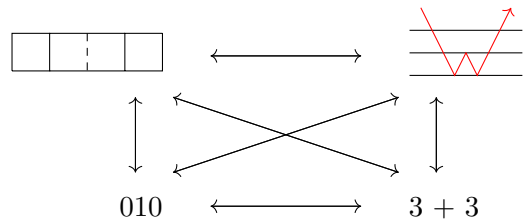
\begin{equation*}
    \begin{tikzpicture}
        \node[anchor = south east] at (0,0) (1) { \begin{tikzpicture}[baseline={([yshift=-4ex]current bounding box.center)}, scale = 0.5]
                \draw (0,0) \onetile \twotiles \onetile \finaledge;
            \end{tikzpicture} };
        \node[anchor = north east] at (0,-1) (2) {010};
        \node[anchor = south west] at (3,0) (3) {\begin{tikzpicture}[baseline={([yshift=-.5ex]current bounding box.center)}, scale = \scl]
            \plates[5]
            \dummyNodes[1][-3.2][1][1.2]
            \draw[red] (1,0) \glassDn \glassDn \glassUp \glassDn \glassUp \glassUp;
            \draw[red, ->] (0.5,1) -- (1,0) (4,0) -- (4.5,1);
        \end{tikzpicture}};
        \node[anchor = north west] at (3,-1) (4) {3 + 3};
        \path[->]
            (-0.5,0.2) edge (-0.5,-0.8)
            (-0.5,-0.8) edge (-0.5,0.2)
            (3.5,0.2) edge (3.5,-0.8)
            (3.5,-0.8) edge (3.5,0.2)
            (0.5,0.75) edge (2.5,0.75)
            (2.5,0.75) edge (0.5,0.75)
            (0.5,-1.25) edge (2.5,-1.25)
            (2.5,-1.25) edge (0.5,-1.25)
            (0,0.2) edge (3,-0.8)
            (3,-0.8) edge (0,0.2)
            (0,-0.8) edge (3,0.2)
            (3,0.2) edge (0,-0.8);
    \end{tikzpicture}
\end{equation*}

This demonstrates how useful this universal bijection is. Rather than simply obtaining bijections between families one by one, we see that each time we determine the \mbox{(1,1)-magma} structure of a Fibonacci family, it immediately gives us a bijection to each other Fibonacci family whose \mbox{(1,1)-magma} structure is known. As a result of this, we are able to very quickly build up a large number of bijections.

 \section{Motzkin and Schr\"oder normed (1,2)-magmas} \label{section:motzkinSchroder}

In this section we consider \mbox{(1,2)-magmas} and show how these relate to Motzkin numbers and Schr\"oder numbers.
We adopt the convention that the binary map is always written as an in-fix operator. 
We begin  by constructing an example of a free \mbox{(1,2)-magma}, which we call the Cartesian \mbox{(1,2)-magma}. 

 \subsection{Cartesian \mbox{(1,2)-magma}} \label{sec:cartesian12magma}

First, we introduce and discuss some notation that will be used throughout this section. 
We will use square parentheses when writing $n$-tuples to avoid possible ambiguity arising from the use of round parentheses later. 
Thus we take the $n$-ary Cartesian product to be the following:
\begin{equation*}
    X_1 \times \cdots \times X_n = \left\{ [x_1, \hdots, x_n]: \, x_i \in X_i, \ i \in \{ 1, \hdots, n \} \right\}.
\end{equation*}
and the notation $[X]$ to mean the set
\begin{equation*}
    [X] = \{ [x]: \, x \in X \}.
\end{equation*}
This gives us a set notation for unary maps which we will use to define an explicit \mbox{(1,2)-magma}.


\begin{definition}[Cartesian \mbox{(1,2)-magma}] 
\label{def:cartesian12magma}
Let $X$ be a non-empty finite set. Define the sequence $\cW_n(X)$ of sets of nested 1- and 2-tuples by
\begin{subequations} \label{eq:cart12magmaAll}
    \begin{align}
        \cW_1(X) & = X, \\
        \cW_2(X) & = \left[ \cW_1(X) \right], \\
        \cW_n(X) & = \left[ \cW_{n-1}(X) \right] \union \bigcup_{k = 1}^{n - 2} \left( \cW_k(X) \times \cW_{n - k - 1}(X) \right), \qquad n \geq 3. \label{eq:cart12magma}
    \end{align}
\end{subequations}
Let $\cW_X = \bigcup_{n \geq 1} \cW_n(X)$ and $\cW_X^+ = \cW_X \backslash X$. Define the unary map $\mu: \cW_X \rightarrow \cW_X$ by
\begin{equation}
    \mu(w) = [w], \qquad w \in \cW_X,
\end{equation}
and the binary map $\diamond: \cW_X \times \cW_X \rightarrow \cW_X$ by
\begin{equation}
    w_1 \diamond w_2 = [w_1, w_2], \qquad w_1, w_2 \in \cW_X.
\end{equation}
The triple \mbox{($\cW_X$, $\mu$, $\diamond$)} is called the \textbf{Cartesian \mbox{(1,2)-magma} generated by $\mathbf{X}$}.
\end{definition}

If $X = \{ \epsilon \}$, the sequence of sets $\cW_n(X)$ defining the base set $\cW_X$ of the Cartesian \mbox{(1,2)-magma} begins as follows:
\vspace{-1ex}
\begin{align*}
    \cW_1(X) & = \{ \epsilon \}, \\
    \cW_2(X) & = \{ [ \epsilon ] \}, \\
    \cW_3(X) & = \{ [ [ \epsilon ] ], \quad  [ \epsilon, \epsilon ] \}, \\
    \cW_4(X) & = \{ [ [ [ \epsilon ] ] ], \quad [ [ \epsilon, \epsilon ] ], \quad [ \epsilon, [ \epsilon ] ], \quad [ [ \epsilon ], \epsilon ] \}, \\
    & \ \, \vdots
\end{align*} 


We now prove that the Cartesian \mbox{(1,2)-magma} is free.

\begin{theorem} \label{thm:free12magma}
The Cartesian \mbox{(1,2)-magma} \mbox{($\cW_X$, $\mu$, $\diamond$)} is a free \mbox{(1,2)-magma}.
\end{theorem}

We will show that \mbox{($\cW_X$, $\mu$, $\diamond$)} is a unique factorisation normed \mbox{(1,2)-magma} with set of irreducibles equal to $X$. Then, by Theorem \ref{thm:uniqFactNormedThenFree}, we will have that \mbox{($\cW_X$, $\mu$, $\diamond$)} is a free \mbox{(1,2)-magma} generated by $X$.

\begin{proof}
Suppose that $w, w' \in \cW_X$ are such that $\mu(w) = \mu(w')$. Then $[w] = [w']$ and hence $w = w'$. Thus $\mu$ is injective. Now suppose $w_1, w_2, w_1', w_2' \in \cW_X$ are such that $w_1 \diamond w_2 = w_1' \diamond w_2'$. Then $[w_1, w_2] = [w_1', w_2']$ and hence $w_1 = w_1', w_2 = w_2'$ and so $\diamond$ is injective. Clearly we have $\text{Img}(\mu) \cap \text{Img}(\diamond) = \emptyset$ since $\text{Img}(\mu)$ contains only 1-tuples and $\text{Img}(\diamond)$ contains only 2-tuples. Thus \mbox{($\cW_X$, $\mu$, $\diamond$)} is a unique factorisation \mbox{(1,2)-magma}.

The set of irreducibles is $X$ since this is the complement of $\text{Img}(\mu) \cup \text{Img}(\diamond)$.


For each $w \in \cW_X$, define $\norm{w} = n$ if $w \in \cW_n(X)$. Now take $w \in \cW_n(X)$. Since $\mu(w) = [w] \in [\cW_n(X)]$, we have $\mu(w) \in \cW_{n + 1}(X)$ from \eqref{eq:cart12magmaAll}. Therefore 
\begin{equation*}
    \norm{\mu(w)} = n + 1 > \norm{w}.
\end{equation*}
Now consider $w_1 \in \cW_{n_1}(X)$ and $w_2 \in \cW_{n_2}(X)$. We have $w_1 \diamond w_2 = [w_1, w_2]$ and hence $w_1 \diamond w_2 \in \cW_{n_1 + n_2 + 1}(X)$ from \eqref{eq:cart12magmaAll}. Therefore
\begin{equation*}
    \norm{w_1 \diamond w_2} = n_1 + n_2 + 1 > \norm{w_1} + \norm{w_2}.
\end{equation*}
Therefore $\norm{\cdot}: \cW_X \rightarrow \mathbb{N}$ is a norm.
\end{proof}

\subsection{Motzkin normed \mbox{(1,2)-magmas}} \label{subsection:motzkin12}

For the purpose of proving Theorem \ref{thm:free12magma} we were required to demonstrate that there exists a norm on the (1,2)-magma \mbox{($\cW_X$, $\mu$, $\diamond$)}. While we could have chosen any function which satisfies Definition \ref{def:norm}, this particular norm was chosen since it gives rise to the Motzkin numbers. 
In Section \ref{sec_schroder} we will define a different norm  on the \emph{same} base set which will give rise to the Schr\"oder numbers. 

\begin{proposition} \label{prop:motzin12magma}
Let \mbox{($\cW_X$, $\mu$, $\diamond$)} be the Cartesian \mbox{(1,2)-magma} generated by the set $X$, where $\card{X} = p$. Define the map $\norm{\cdot}_M:  \cW_X \rightarrow \mathbb{N}$ by $\norm{m}_M = n$ when $m \in \cW_n(X)$, where $\cW_n(X)$ is as defined in Definition \ref{def:cartesian12magma} for $n \in \mathbb{N}$. If
\begin{equation*}
    N_n = \{ m \in \cW_{X}: \, \norm{m}_M = n \}, \qquad n \geq 1,
\end{equation*}
then
\begin{equation*}
    \vert N_n \vert = M_{n - 1}(p), \qquad n \geq 1,
\end{equation*}
where $M_n(p)$ are the $p$-Motzkin numbers from Definition \ref{def:pSequences}.
\end{proposition}

\begin{proof}
Since $\norm{m}_M = n$ if and only if $m \in \cW_n(X)$, we have $N_n = \cW_n(X)$. We have $N_1 = \cW_1(X) = X$ so $\card{N_1} = \card{X} = p$, and $N_2 = \cW_2(X) = [X]$ so $\card{N_2} = \card{[X]} = p$. Now, for $n \geq 3$, \eqref{eq:cart12magma} gives
\begin{align*}
    \card{N_n} & = \card{\left[ \cW_{n-1}(X) \right] \union \bigcup_{k = 1}^{n - 2} \left( \cW_k(X) \times \cW_{n - k - 1}(X) \right)} \\
    & = \card{N_{n - 1}} + \sum_{k = 1}^{n - 2} \card{N_k} \cdot \card{N_{n - k - 1}} 
\end{align*}
This is equivalent to the $p$-Motzkin recurrence \eqref{eq:pMotzkinRecurrence}.
%
\end{proof}

\begin{corollary}
Let \mbox{($\cW_{\epsilon}$, $\mu$, $\diamond$)} be the Cartesian \mbox{(1,2)-magma} generated by the single element, $\epsilon$. If $N_n$ and $\norm{\cdot}_M: \cW_{\epsilon} \rightarrow \mathbb{N}$ are as defined in Proposition \ref{prop:motzin12magma}, then
\begin{equation*}
    \vert N_n \vert = M_{n - 1}, \qquad n \geq 1,
\end{equation*}
where $M_n$ is the Motzkin number from Definition \ref{def:pSequences}.
\end{corollary}

The norm function $\norm{\cdot}_M: \cW_{\epsilon} \rightarrow \mathbb{N}$ from the above corollary could equivalently be defined recursively as follows: Let $\norm{\cdot}_M: \cW_{\epsilon} \rightarrow \mathbb{N}$ be such that
\begin{subequations} \label{eq:motNorm0}
    \begin{align}
        \norm{\epsilon}_M & = 1, \\
        \norm{\mu(w)}_M & = \norm{w}_M + 1, \\
        \norm{w \diamond w'}_M & = \norm{w}_M + \norm{w'}_M + 1,
    \end{align}
\end{subequations}
for all $w, w' \in \cW_{\epsilon}$. 

We see that a free \mbox{(1,2)-magma} generated by a single element is partitioned by the norm function into sets given by the Motzkin numbers only when the norm satisfies \eqref{eq:motNorm0}. Thus we define a Motzkin normed \mbox{(1,2)-magma} to be a free \mbox{(1,2)-magma} with a single generator \textit{along with} a norm function which satisfies \eqref{eq:motNorm0} for the relevant \mbox{(1,2)-magma}.

\begin{definition}[Motzkin normed \mbox{(1,2)-magma}] \label{def:motzkin12magma}
Let \mbox{($\cM$, $f$, $\star$)} be a unique factorisation normed \mbox{(1,2)-magma} with one irreducible element, $\epsilon$. Let \mbox{$\norm{\cdot}_M: \cM \rightarrow \mathbb{N}$} be a norm satisfying
\begin{subequations} \label{eq:motNorm}
    \begin{align}
        \norm{\epsilon}_M & = 1, \\
        \norm{f(m)}_M & = \norm{m}_M + 1, \\
        \norm{m \star m'}_M & = \norm{m}_M + \norm{m'}_M + 1
    \end{align}
\end{subequations}
for all $m, m' \in \cM$. Then \mbox{($\cM$, $f$, $\star$)} with the norm $\norm{\cdot}_M: \cM \rightarrow \mathbb{N}$ is called a \textbf{Motzkin normed \mbox{(1,2)-magma}}.
\end{definition}

If \mbox{($\cM$, $f$, $\star$)} is a Motzkin normed \mbox{(1,2)-magma} with unique generator $\epsilon$, then the base set $\cM$ begins (sorting by norm) by evaluating the following expressions:
\begin{center}
    \begin{tabular}{l l}
        Norm 1: & $\epsilon$. \\
        Norm 2: & $f(\epsilon)$. \\
        Norm 3: & $f(f(\epsilon))$, \quad $\epsilon \star \epsilon$. \\
        Norm 4: & $f(f(f(\epsilon)))$, \quad $f(\epsilon \star \epsilon)$, \quad $\epsilon \star f(\epsilon)$, \quad $f(\epsilon) \star \epsilon$. \\
        Norm 5: & $f(f(f(f(\epsilon))))$, \quad $f(f(\epsilon \star \epsilon))$, \quad $f(\epsilon \star f(\epsilon))$, \quad $f(f(\epsilon) \star \epsilon)$, \\
        &
        $\epsilon \star f(f(\epsilon))$, \quad $\epsilon \star (\epsilon \star \epsilon)$, \quad $f(\epsilon) \star f(\epsilon)$, \quad $f(f(\epsilon)) \star \epsilon$, \\
        &
        $(\epsilon \star \epsilon) \star \epsilon$.
    \end{tabular}
\end{center}

We see that the Motzkin norm $\norm{\cdot}_M: \cM \rightarrow \mathbb{N}$ can be informally stated as follows, where $m \in \cM$:
\begin{equation*}
    \norm{m}_M = (\text{number of $\epsilon$'s}) + (\text{number of $f$'s}) + (\text{number of $\star$'s}).
\end{equation*}

For each Motzkin family the norm is usually a simple function of the conventional size parameter for that family. This can be seen in Appendix \ref{appendix:Motzkin}, where the details of a number of Motzkin normed \mbox{(1,2)-magmas} can be found.

We will consider a number of combinatorial families which are listed in Appendix \ref{appendix:Motzkin}. 
We adopt the following convention for any Motzkin normed \mbox{(1,2)-magma}: (i)   the unary map is denoted $f$ and (ii) the binary map is denoted $\star$ and written using  in-fix notation. 
We make it clear that we are using the maps specific to a certain family by referencing that family in the subscript of the map which is done via the number assigned to that family in the appendix.

We now discuss another example of a Motzkin normed \mbox{(1,2)-magma}, the Motzkin paths \motFamRef{motzkinFamily:motzkinPaths}. 
A Motzkin path of length $n$ is a path from $(0,0)$ to $(n,0)$ using steps $U = (1,1)$, $D = (1,-1)$ and $H = (1,0)$ which remains above the line $y = 0$. 
The number of such paths of length $n$ is given by the $n$th Motzkin number $M_n$. 
For the corresponding Motzkin normed \mbox{(1,2)-magma} \motMag{motzkinFamily:motzkinPaths}, the base set \motBaseRef{motzkinFamily:motzkinPaths} is given by all Motzkin paths from $(0,0)$ to $(n,0)$, for all $n \in \mathbb{N}_0$. The empty path, $n=0$, is taken to be a single vertex.

The generator of this family is  the empty path, 
\def\scl{0.3}
\begin{equation*}
    \epsilon_{\ref{motzkinFamily:motzkinPaths}} = 
    \begin{tikzpicture}[baseline={([yshift=-.5ex]current bounding box.center)}, scale = \scl]
        \fill (0,0) circle (6pt);
    \end{tikzpicture}\,,
\end{equation*}
and the two maps are defined schematically as follows:
\begin{align*}
    f_{\ref{motzkinFamily:motzkinPaths}} \left(
        \begin{tikzpicture}[baseline={([yshift=-2ex]current bounding box.center)}, scale = \scl]
            \draw[fill = \colOne, opacity = \opac] (6,0) -- (0,0) to [bend left = 90, looseness = 1.5] (6,0);
            \node at (3,1.3) {\footnotesize $p$};
        \end{tikzpicture}
    \right)
    & =
    \begin{tikzpicture}[baseline={([yshift=-2ex]current bounding box.center)}, scale = \scl]
        \draw[fill = orange!90!black, opacity = \opac] (6,0) -- (0,0) to [bend left = 90, looseness = 1.5] (6,0);
        \node at (3,1.3) {\footnotesize $p$};
        \draw[thick] (6,0) -- (8,0);
    \end{tikzpicture} \\
    \begin{tikzpicture}[baseline=1.5ex, scale = \scl]
        \draw[fill = \colOne, opacity = \opac] (6,0) -- (0,0) to [bend left = 90, looseness = 1.5] (6,0);
        \node[anchor = mid] at (3,1.3) {\footnotesize $p_1$};
    \end{tikzpicture}
    \ \star_{\ref{motzkinFamily:motzkinPaths}} \ 
    \begin{tikzpicture}[baseline=1.5ex, scale = \scl]
        \draw[fill = \colTwo, opacity = \opac] (6,0) -- (0,0) to [bend left = 90, looseness = 1.5] (6,0);
        \node[anchor = mid] at (3,1.3) {\footnotesize $p_2$};
    \end{tikzpicture}
    & =
    \begin{tikzpicture}[baseline=1.5ex, scale = \scl]
        \draw[fill = \colOne, opacity = \opac] (6,0) -- (0,0) to [bend left = 90, looseness = 1.5] (6,0);
        \node[anchor = mid] at (3,1.3) {\footnotesize $p_1$};
        \draw[thick] (6,0) -- (7.5,1.5);
        \draw[fill = \colTwo, opacity = \opac] (13.5,1.5) -- (7.5,1.5) to [bend left = 90, looseness = 1.5] (13.5,1.5);
        \node at (10.5,2.8) {\footnotesize $p_2$};
        \draw[thick] (13.5,1.5) -- (15,0);
    \end{tikzpicture}
\end{align*}

Thus the unary map   adds a single horizontal step after the path and the binary map concatenates the two paths while adding a pair of up and down steps as shown. 
These are the usual right factorisations of a Motzkin path corresponding to the recursive structure of all Motzkin paths, but now interpreted as maps. 
All Motzkin paths can be constructed using sequences of compositions of these two maps applied to the generator. Therefore this is a unique factorisation \mbox{(1,2)-magma}.

The norm $\norm{m}_M$ of any Motzkin path $m$ is defined as the length of the path + 1. We can immediately see that this norm satisfies \eqref{eq:motNorm}:
\begin{align*}
    \norm{\epsilon}_M & = 1, \\
    \norm{f_{\ref{motzkinFamily:motzkinPaths}}(p_1)}_M & = \norm{p_1}_M + 1, \\
    \norm{p_1 \star_{\ref{motzkinFamily:motzkinPaths}} p_2}_M & = \norm{p_1}_M + \norm{p_2}_M + 1,
\end{align*}
for all Motzkin paths $p_1, p_2 \in \mathMotBaseRef{motzkinFamily:motzkinPaths}$.

Thus we see that \motMag{motzkinFamily:motzkinPaths} is a Motzkin normed (1,2)-magma.

\subsection{Schr\"oder normed \mbox{(1,2)-magmas}}
\label{sec_schroder}
In this  section, we show how the Schr\"oder numbers are related to   free \mbox{(1,2)-magmas} which we do by defining  a different norm on the Cartesian \mbox{(1,2)-magma}.
In order to clearly define the norm in terms of the Cartesian base set, we provide an alternative method for constructing the Cartesian \mbox{(1,2)-magma}.  Note, the resulting  set is the same set as that defined by \eqref{eq:cart12magmaAll} but constructed differently.

Let $X$ be a non-empty finite set, and define the sequence $\cY_n(X)$ of sets of nested 1- and 2-tuples by 
\begin{subequations} \label{eq:cartesian12magmaConstruction}
    \begin{align}
        \cY_1(X) & = X, \\
        \cY_n(X) & = \left[ \cY_{n-1}(X) \right] \union \bigcup_{k = 1}^{n - 1} \left( \cY_k(X) \times \cY_{n - k}(X) \right), \qquad n \geq 2. \label{eq:cartesian12magmaConstruction2}
    \end{align}
\end{subequations}

If $X = \{ \epsilon \}$, then the sequence of sets $\cY_n(X)$ begins as follows:
\begin{align*}
    \cY_1(X) & = \{ \epsilon \}, \\
    \cY_2(X) & = \{ [ \epsilon ], \quad [\epsilon, \epsilon] \}, \\
    \cY_3(X) & = \{ [[ \epsilon ]], \quad [[\epsilon, \epsilon]], \quad [\epsilon, [\epsilon]], \quad [\epsilon, [\epsilon, \epsilon]], \quad [[\epsilon], \epsilon], \quad [[\epsilon, \epsilon], \epsilon] \}, \\
    & \ \, \vdots
\end{align*}

Letting $\cY_X = \bigcup_{n \geq 1} \cY_n(x)$, we have that \mbox{($\cY_X$, $\mu$, $\diamond$)} is equal to the Cartesian \mbox{(1,2)-magma} \mbox{($\cW_X$, $\mu$, $\diamond$)} of Definition \ref{def:cartesian12magma}. This follows from the fact that $\cY_X$ and $\cW_X$ both contain all nested 1- and 2-tuples containing elements of the set $X$. Thus we have the following proposition.

\begin{proposition}
    Let $\cY_X = \bigcup_{n \geq 1} \cY_n(x)$. Define the unary map $\mu: \cW_X \rightarrow \cW_X$ by
    \begin{equation*}
        \mu(w) = [w], \qquad w \in \cW_X,
    \end{equation*}
    and the binary map $\diamond: \cW_X \times \cW_X \rightarrow \cW_X$ by
    \begin{equation*}
        w_1 \diamond w_2 = [w_1, w_2], \qquad w_1, w_2 \in \cW_X.
    \end{equation*}
    Then \mbox{($\cY_X$, $\mu$, $\diamond$)} is the Cartesian \mbox{(1,2)-magma} of Definition \ref{def:cartesian12magma}.
\end{proposition}

The link between the Cartesian \mbox{(1,2)-magma} and the Schr\"oder numbers is given by the following proposition.

\begin{proposition} \label{prop:schroder12magma}
Let \mbox{($\cW_X$, $\mu$, $\diamond$)} be the Cartesian \mbox{(1,2)-magma} generated by the set $X$, where $\card{X} = p$. Define the map $\norm{\cdot}_S:  \cW_{X} \rightarrow \mathbb{N}$ by $\norm{m}_S = n$ when $m \in \cY_n(X)$, where $\cY_n(X)$ is as defined in \eqref{eq:cartesian12magmaConstruction}. If
\begin{equation*}
    N_n = \{ m \in \cW_X: \, \norm{m}_S = n \}, \qquad n \geq 1,
\end{equation*}
then
\begin{equation*}
    \card{N_n} = S_{n - 1}(p), \qquad n \geq 1,
\end{equation*}
where $S_n(p)$ are the $p$-Schr\"oder numbers from Definition \ref{def:pSequences}.
\end{proposition}

\begin{proof}
Since $\norm{m}_S = n$ if and only if $m \in \cY_n(X)$, we have $N_n = \cY_n(X)$. Thus we have $\card{N_1} = \card{\cY_1(X)} = \card{X} = p$. Now, for $n \geq 2$, \eqref{eq:cartesian12magmaConstruction2} gives
\begin{align*}
    \card{N_n} & = \card{\cY_{n - 1}(X) \union \bigcup_{k = 1}^{n - 1} \cY_k(X) \times \cY_{n - k}(X)} \\
    & = \card{N_{n - 1}} + \sum_{k = 1}^{n - 1} \card{N_k} \cdot \card{N_{n - k}}
\end{align*}
This is equivalent to the $p$-Schr\"oder recurrence \eqref{eq:pSchroderRecurrence}.
\end{proof}

\begin{corollary}
Let \mbox{($\cW_{\epsilon}$, $\mu$, $\diamond$)} be the Cartesian \mbox{(1,2)-magma} generated by the single element, $\epsilon$. If $N_n$ and $\norm{\cdot}_S:  \cW_{\epsilon} \rightarrow \mathbb{N}$ are as defined in Proposition \ref{prop:schroder12magma}, then
\begin{equation*}
    \vert N_n \vert = S_{n - 1}, \qquad n \geq 1,
\end{equation*}
where $S_n$ is the Schr\"oder number from Definition \ref{def:pSequences}.
\end{corollary}

Note  the norm function $\norm{\cdot}_S: \cW_{\epsilon} \rightarrow \mathbb{N}$ from the above corollary could equivalently be defined recursively as follows: Let $\norm{\cdot}_S: \cW_{\epsilon} \rightarrow \mathbb{N}$ be such that
\begin{subequations} \label{eq:schNorm0}
    \begin{align}
        \norm{\epsilon}_S & = 1, \\
        \norm{\mu(w)}_S & = \norm{w}_S + 1, \\
        \norm{w \diamond w'}_S & = \norm{w}_S + \norm{w'}_S,
    \end{align}
\end{subequations}
for all $w, w' \in \cW_{\epsilon}$.

This motivates the following definition.

\begin{definition}[Schr\"oder normed \mbox{(1,2)-magma}] \label{def:schroder12magma}
Let \mbox{($\cM$, $f$, $\star$)} be a unique factorisation normed \mbox{(1,2)-magma} with only one irreducible element, $\epsilon$. Let $\norm{\cdot}_S: \cM \rightarrow \mathbb{N}$ be a norm satisfying
\begin{subequations} \label{eq:schNorm}
    \begin{align}
        \norm{\epsilon}_S & = 1, \\
        \norm{f(m)}_S & = \norm{m}_S + 1, \\
        \norm{m \star m'}_S & = \norm{m}_S + \norm{m'}_S
    \end{align}
\end{subequations}
for all $m, m' \in \cM$. Then \mbox{($\cM$, $f$, $\star$)} with the norm $\norm{\cdot}_S: \cM \rightarrow \mathbb{N}$ is called a \textbf{Schr\"oder normed \mbox{(1,2)-magma}}.
\end{definition}

If \mbox{($\cM$, $f$, $\star$)} is a Schr\"oder normed \mbox{(1,2)-magma} with unique generator $\epsilon$, then the base set $\cM$ begins (sorting by norm) by evaluating the following expressions:
\begin{center}
    \begin{tabular}{l l}
        Norm 1: & $\epsilon$. \\
        Norm 2: & $f(\epsilon)$, \quad $\epsilon \star \epsilon$. \\
        Norm 3: & $f(f(\epsilon))$, \quad $f(\epsilon \star \epsilon)$, \quad $\epsilon \star f(\epsilon)$, \quad $\epsilon \star (\epsilon \star \epsilon)$, \\
        & $f(\epsilon) \star \epsilon$, \quad $(\epsilon \star \epsilon) \star \epsilon$.
    \end{tabular}
\end{center}

Informally, the Schr\"oder norm $\norm{\cdot}_S: \cM \rightarrow \mathbb{N}$ can be defined as follows, where $m \in \cM$:
\begin{equation*}
    \norm{m}_S = (\text{number of $f$'s}) + (\text{number of $\star$'s}) + 1.
\end{equation*}

For each Schr\"oder family the norm is usually a simple function of the conventional size parameter for that family. The details of a number of Schr\"oder normed \mbox{(1,2)-magmas} can be found in Appendix \ref{appendix:Schroder}.  

In the remainder of this section, we will reference a number of Schr\"oder families. Further details of these families  are provided in Appendix \ref{appendix:Schroder}. 
We call the unary map $f$ and the binary map $\star$  which will be written in in-fix form. 
We make it clear that we are using the maps specific to a certain family by referencing that family in the subscript of the maps which we do  via the number assigned to that family in the appendix.

As an example of a Schr\"oder normed \mbox{(1,2)-magma}, consider the Schr\"oder family of semi-standard Young Tableaux (SSYT) of shape $n \times 2$, \schFamRef{schroderFamily:SSYT} as defined in Appendix \ref{appendix:Schroder}. 
The $n$th Schr\"oder number $S_n$ is equal to the number of such tableaux.

We construct the Schr\"oder normed \mbox{(1,2)-magma} of SSYT of shape $n \times 2$, \schMag{schroderFamily:SSYT}, as follows:
\begin{itemize}
    \item Take the base set \schBaseRef{schroderFamily:SSYT} to be the set of all semi-standard Young Tableaux of shape $n \times 2$, for every $n \in \mathbb{N}_0$. We take the trivial empty tableau $\emptyset$ to be the only SSYT of shape $0 \times 2$.
    \item Define the unary map $f_{\ref{schroderFamily:SSYT}}$ as follows:
    \begin{equation*}
        f_{\ref{schroderFamily:SSYT}}
        \left(
            \resizebox{\SSYTscl}{!}{\addvbuffer[5pt 3pt]{\begin{tabular}{| c | c |}
            \hline
            \rowcolor{\colOne!\opacTen} \vdots & \vdots \\
            \hline
            \rowcolor{\colOne!50} $a$ & $b$ \\
            \hline
        \end{tabular}}}
        \right) =
        \resizebox{\SSYTscl}{!}{\addvbuffer[3pt 3pt]{\begin{tabular}{| c | c |}
            \hline
            \rowcolor{\colOne!\opacTen} \vdots & \vdots \\
            \hline
            \rowcolor{\colOne!\opacTen} $a$ & $b$ \\
            \hline
            \rowcolor{white} $b + 1$ & $b + 1$ \\
            \hline
        \end{tabular}}}
    \end{equation*}
    with the convention that the empty tableau is considered to have all entries equal to 0. 
    Thus we have
    \begin{equation*}
        f_{\ref{schroderFamily:SSYT}}(\emptyset) =
        \resizebox{\SSYTscl}{!}{\begin{tabular}{| c | c |}
            \hline
            1 & 1 \\
            \hline
        \end{tabular}}
    \end{equation*}
    \item Define the binary map $\star_{\ref{schroderFamily:SSYT}}$ as follows:
    \begin{equation*}
        \resizebox{\SSYTscl}{!}{
        \begin{tabular}{| c | c |}
            \hline
            \rowcolor{\colOne!\opacTen} $a$ & $b$ \\
            \hline
            \rowcolor{\colOne!\opacTen} \vdots & \vdots \\
            \hline
            \rowcolor{\colOne!\opacTen} $c$ & $d$ \\
            \hline
        \end{tabular}
        }
        \ \star_{\ref{schroderFamily:SSYT}} \
        \resizebox{\SSYTscl}{!}{\begin{tabular}{| c | c |}
            \hline
            \rowcolor{\colTwo!\opacTen} $s$ & $t$ \\
            \hline
            \rowcolor{\colTwo!\opacTen} $u$ & $v$ \\
            \hline
            \rowcolor{\colTwo!\opacTen} \vdots & \vdots \\
            \hline
            \rowcolor{\colTwo!\opacTen} $y$ & $z$ \\
            \hline
        \end{tabular}}
        =
        \resizebox{\SSYTscl}{!}{\addvbuffer[3pt 3pt]{\begin{tabular}{| c | c |}
            \hline
            \rowcolor{\colOne!\opacTen} $a$ & $b$ \\
            \hline
            \rowcolor{\colOne!\opacTen} \vdots & \vdots \\
            \hline
            \rowcolor{\colOne!\opacTen} $c$ & $d$ \\
            \hline
            \cellcolor{white} $d + 1$ & \cellcolor{\colTwo!\opacTen} $t + d + 1$ \\
            \hline
            \rowcolor{\colTwo!\opacTen} $s + d + 1$ & $v + d + 1$ \\
            \hline
            \rowcolor{\colTwo!\opacTen} $u + d + 1$ & \vdots \\
            \hline
            \rowcolor{\colTwo!\opacTen} \vdots & $z + d + 1$ \\
            \hline
            \cellcolor{\colTwo!\opacTen} $y + d + 1$ & \cellcolor{white}$z + d + 2$ \\
            \hline
        \end{tabular}}}
    \end{equation*}
    Again, note that when applying this to the empty tableau, we consider any entries of the empty tableau to be equal to 0. 
\end{itemize}

Note that the empty SSYT $\epsilon$ is the only element in the base set which is not in the image of one of the two maps. Thus this is the only generator, so we can define $\epsilon_{\ref{schroderFamily:SSYT}} = \emptyset$. 
Then we can see that the base set \schBaseRef{schroderFamily:SSYT} is generated by $\epsilon_{\ref{schroderFamily:SSYT}}$ via the two maps $f_{\ref{schroderFamily:SSYT}}$ and $\star_{\ref{schroderFamily:SSYT}}$. 

Any SSYT of shape $n \times 2$ factorises uniquely in terms of the generator $\epsilon_{\ref{schroderFamily:SSYT}}$ and the two maps $f_{\ref{schroderFamily:SSYT}}$ and $\star_{\ref{schroderFamily:SSYT}}$.
For example, consider the following SSYT, which we can decompose as follows:
\begin{align*}
    \resizebox{\SSYTscl}{!}{\addvbuffer[3pt 1pt]{\begin{tabular}{| c | c |}
        \hline
        1 & 1 \\
        \hline
        2 & 3 \\
        \hline
        4 & 5 \\
        \hline
        6 & 6 \\
        \hline
    \end{tabular}}}
   & =
    f_{\ref{schroderFamily:SSYT}} \!\left(
        \resizebox{\SSYTscl}{!}{\addvbuffer[3pt 1pt]{\begin{tabular}{| c | c |}
            \hline
            1 & 1 \\
            \hline
            2 & 3 \\
            \hline
            4 & 5 \\
            \hline
        \end{tabular}}}
    \right)
    =
    f_{\ref{schroderFamily:SSYT}} \!\left( 
        \resizebox{\SSYTscl}{!}{\addvbuffer[3pt 1pt]{\begin{tabular}{| c | c |}
            \hline
            1 & 1 \\
            \hline
            2 & 3 \\
            \hline
        \end{tabular}}}
        \star_{\ref{schroderFamily:SSYT}}
        \emptyset
    \right)
    = 
    f_{\ref{schroderFamily:SSYT}} \!\left(\!
        \left(
            \resizebox{\SSYTscl}{!}{\addvbuffer[3pt 1pt]{\begin{tabular}{| c | c |}
                \hline
                1 & 1 \\
                \hline
            \end{tabular}}}
            \star_{\ref{schroderFamily:SSYT}}
        \emptyset
        \right)
        \star_{\ref{schroderFamily:SSYT}}
        \emptyset
    \right)
    =
      f_{\ref{schroderFamily:SSYT}} \left(
        \left(
            f_{\ref{schroderFamily:SSYT}} \left(
                \emptyset
            \right)
            \star_{\ref{schroderFamily:SSYT}}
        \emptyset
        \right)
        \star_{\ref{schroderFamily:SSYT}}
        \emptyset
    \right)
\end{align*}

We define the norm $\norm{\cdot}_S$ of a tableau of shape $n \times 2$ to be $n + 1$. With this definition we obtain:
\begin{align*}
    \norm{\epsilon}_S & = 1, \\
    \norm{f_{\ref{schroderFamily:SSYT}}(t)}_S & = \norm{t}_S + 1, \quad t \in \mathSchBaseRef{schroderFamily:SSYT}, \\
    \norm{t_1 \star_{\ref{schroderFamily:SSYT}} t_2}_S & = \norm{t_1}_S + \norm{t_2}_S, \quad t_1, t_2 \in \mathSchBaseRef{schroderFamily:SSYT}.
\end{align*}

Thus \schMag{schroderFamily:SSYT} is a unique factorisation normed \mbox{(1,2)-magma} with a finite non-empty set of irreducibles (and hence a free \mbox{(1,2)-magma}). The norm $\norm{\cdot}_S: \mathSchBaseRef{schroderFamily:SSYT} \rightarrow \mathbb{N}$ satisfies \eqref{eq:schNorm}. Therefore this is a Schr\"oder normed \mbox{(1,2)-magma}.

\subsection{Free \mbox{(1,2)-magma} isomorphisms and a universal bijection}
\label{sebsec:12unibij}

We now apply Proposition \ref{prop:freenMagmaIsomorphism} which states that there exists a unique \mbox{(1,2)-magma} isomorphism between any free \mbox{(1,2)-magmas} generated by sets of the same size. We demonstrate how this defines a universal bijection between any pair of Motzkin families or any pair of Schr\"oder families.

\begin{definition}[Universal bijection] \label{def:12univBij}
Suppose that \mbox{($\cM$, $f$, $\star$)} and \mbox{($\cN$, $g$, $\ltimes$)} are free \mbox{(1,2)-magmas} with generating sets $X_{\cM}$ and $X_{\cN}$ respectively, with $\card{X_{\cM}} = \card{X_{\cN}}$. Let $\sigma: X_{\cM} \rightarrow X_{\cN}$ be any bijection, and define the map $\Upsilon: \cM \rightarrow \cN$ as follows:\\
\noindent For all $m \in \cM \setminus X_{\cM}$,
\begin{enumerate}[label=(\roman*)]
    \item Decompose $m$ into an expression in terms of generators $\epsilon_i \in X_{\cM}$ and the maps $f$ and $\star$.
    \item In the decomposition of $m$, replace every occurrence of $\epsilon_i$ with $\sigma(\epsilon_i)$, every occurence of $f$ with $g$ and every occurrence of $\star$ with $\ltimes$. Call this expression $\upsilon(m)$.
    \item Define $\Upsilon(m)$ to be $\upsilon(m)$, that is, evaluate all maps in $\upsilon(m)$ to give an element of $\cN$.
\end{enumerate}
\end{definition}

This leads to the following proposition, which follows from the fact that $\Upsilon$ is exactly the map $\Gamma$ from Proposition \ref{prop:freenMagmaIsomorphism}. 

\begin{proposition} \label{prop:12univBij}
Let $\Upsilon: \cM \rightarrow \cN$ be the map of Definition \ref{def:12univBij}. Then $\Upsilon$ is a free \mbox{(1,2)-magma} isomorphism.
\end{proposition}

Schematically, we can write $\Upsilon$ as follows:
\begin{equation}
    m \qquad \overset{\text{decompose}}{\xrightarrow{\hspace*{1.5cm}}} \qquad \underset{\substack{\epsilon_i \rightarrow \, \sigma(\epsilon_i), \ f \rightarrow \, g, \ \star \rightarrow \, \ltimes}}{\text{substitute}} \qquad \overset{\text{evaluate}}{\xrightarrow{\hspace*{1.5cm}}} \qquad n.
\end{equation}

Since $\Upsilon$ is an isomorphism between the free \mbox{(1,2)-magmas} \mbox{($\cM$, $f$, $\star$)} and \mbox{($\cN$, $g$, $\ltimes$)}, we have that $\Upsilon$ defines a bijection between the base sets $\cM$ and $\cN$. 
Furthermore, it gives us that this bijection is recursive: if $m = f(m_0)$, then
\begin{equation*}
    \Upsilon(m) = g(\Upsilon(m_0)),
\end{equation*}
and if $m = m_1 \star m_2$, then
\begin{equation*}
    \Upsilon(m) = \Upsilon(m_1) \ltimes \Upsilon(m_2).
\end{equation*}

Note, $\Upsilon$ preserves the norm when the two \mbox{(1,2)-magmas} are equipped with suitable norms. Suppose that \mbox{($\cM$, $f$, $\star$)} and \mbox{($\cN$, $g$, $\ltimes$)} are free \mbox{(1,2)-magmas} with the same number of generators and that they have respective norms \mbox{$\norm{\cdot}_{\cM}: \cM \rightarrow \mathbb{N}$} and \mbox{$\norm{\cdot}_{\cN}: \cN \rightarrow \mathbb{N}$}. If the following conditions are satisfied:
\begin{enumerate} [label=(\roman*)]
    \item \label{cond:mot1} $\norm{m}_{\cM} = \norm{\sigma(m)}_{\cN}$ for all $m \in X_{\cM}$,
    \item \label{cond:mot2} for $\kappa_1 \in \mathbb{N}$, $\norm{f(m)}_{\cM} = \norm{m}_{\cM} + \kappa_1$ for all $m \in \cM$, and $\norm{g(n)}_{\cN} = \norm{n}_{\cN} + \kappa_1$ for all $n \in \cN$, and
    \item \label{cond:mot3} for $\kappa_2 \in \mathbb{N}_0$, $\norm{m_1 \star m_2}_{\cM} = \norm{m_1}_{\cM} + \norm{m_2}_{\cM} + \kappa_2$ for all $m_1, m_2 \in \cM$, and \\ 
    $\norm{n_1 \ltimes n_2}_{\cN} = \norm{n_1}_{\cN} + \norm{n_2}_{\cN} + \kappa_2$ for all $n_1, n_2 \in \cN$,
\end{enumerate}
then we have
\begin{equation*}
    \norm{m}_{\cM} = \norm{\Upsilon(m)}_{\cN}, \qquad m \in \cM.
\end{equation*}
This follows by considering the decomposed expressions for $m$ and $\Upsilon(m)$, assuming that \ref{cond:mot1}-\ref{cond:mot3} hold.

Both the Motzkin norm, \eqref{eq:motNorm} and the Schr\"oder norm \eqref{eq:schNorm} satisfy (i), (ii) and  (iii) above and thus both norms are invariant under $\Upsilon$.  


\subsubsection{Universal bijections for Motzkin families}

In this section we consider  a number of examples illustrating this universal bijection between free \mbox{(1,2)-magmas}.

Using a simple example, we demonstrate how the universal bijections work for Motzkin normed \mbox{(1,2)-magmas}. We consider the following Motzkin normed \mbox{(1,2)-magmas}:
\begin{itemize}
    \item Motzkin paths \motMag{motzkinFamily:motzkinPaths},
    \item Non-intersecting chords \motMag{motzkinFamily:chords},
    \item Unary-binary trees \motMag{motzkinFamily:motzkinTrees}.
\end{itemize}
See Appendix \ref{appendix:Schroder} for the definitions of these families, as well as details of the relevant (1,2)-magmas and norms.

Take a Motzkin path and decompose it into its factorised form: \def \scl {0.5}
\begin{equation*}
    \begin{tikzpicture}[baseline={([yshift=-.5ex]current bounding box.center)}, scale = \scl]
        \grd[4][1]
        \draw (0,0) \motUp \motAc \motDn \motAc;
    \end{tikzpicture}
    =
    f_{\ref{motzkinFamily:motzkinPaths}} \left(
    \begin{tikzpicture}[baseline={([yshift=-.5ex]current bounding box.center)}, scale = \scl]
        \grd[3][1]
        \draw (0,0) \motUp \motAc \motDn;
        \dummyNodes[0][0][3][1]
    \end{tikzpicture}
    \right)
    = 
    f_{\ref{motzkinFamily:motzkinPaths}} \left(
        \begin{tikzpicture}[baseline={([yshift=-.5ex]current bounding box.center)}, scale = \scl]
            \fill (0,0) \smlnd;
        \end{tikzpicture}
        \ \star_{\ref{motzkinFamily:motzkinPaths}}
        \begin{tikzpicture}[baseline={([yshift=-.5ex]current bounding box.center)}, scale = \scl]
            \grd[1][1]
            \draw (0,0) \motAc;
            \dummyNodes[0][0][1][1]
        \end{tikzpicture}
    \right)
    = 
    f_{\ref{motzkinFamily:motzkinPaths}} \left(
        \begin{tikzpicture}[baseline={([yshift=-.5ex]current bounding box.center)}, scale = \scl]
            \fill (0,0) \smlnd;
        \end{tikzpicture}
        \star_{\ref{motzkinFamily:motzkinPaths}}
        f_{\ref{motzkinFamily:motzkinPaths}} \left(
            \begin{tikzpicture}[baseline={([yshift=-.5ex]current bounding box.center)}, scale = \scl]
                \fill (0,0) \smlnd;
            \end{tikzpicture}
        \right)
\right)
    =
    f_{\ref{motzkinFamily:motzkinPaths}} \left(
        \epsilon_{\ref{motzkinFamily:motzkinPaths}}
        \star_{\ref{motzkinFamily:motzkinPaths}}
        f_{\ref{motzkinFamily:motzkinPaths}} \left(
            \epsilon_{\ref{motzkinFamily:motzkinPaths}}
        \right)
    \right)
\end{equation*}

Now substitute generators and maps then evaluate to obtain an object from the Motzkin family of non-intersecting chords: \def \scl {0.04}
\begin{equation*} 
    f_{\ref{motzkinFamily:chords}} \left(
        \epsilon_{\ref{motzkinFamily:chords}}
        \star_{\ref{motzkinFamily:chords}}
        f_{\ref{motzkinFamily:chords}} \left(
            \epsilon_{\ref{motzkinFamily:chords}}
        \right)
    \right)
    = 
    f_{\ref{motzkinFamily:chords}} \left(
        \begin{tikzpicture}[baseline={([yshift=-.5ex]current bounding box.center)}, scale = \scl]
                \node at (90:10) {\V};
                \draw (0,0) circle [radius = 10];
                \dummyNodesPolar[90:12][270:12]
            \end{tikzpicture}
        \star_{\ref{motzkinFamily:chords}}
        f_{\ref{motzkinFamily:chords}} \left(
            \begin{tikzpicture}[baseline={([yshift=-.5ex]current bounding box.center)}, scale = \scl]
                \node at (90:10) {\V};
                \draw (0,0) circle [radius = 10];
                \dummyNodesPolar[90:12][270:12]
            \end{tikzpicture}
        \right)
    \right)
    = 
    f_{\ref{motzkinFamily:chords}} \left(
        \begin{tikzpicture}[baseline={([yshift=-.5ex]current bounding box.center)}, scale = \scl]
            \node at (90:10) {\V};
            \draw (0,0) circle [radius = 10];
            \dummyNodesPolar[90:12][270:12]
        \end{tikzpicture}
        \star_{\ref{motzkinFamily:chords}}
        \begin{tikzpicture}[baseline={([yshift=-.5ex]current bounding box.center)}, scale = \scl]
            \node at (90:10) {\V};
            \draw (0,0) circle [radius = 10];
            \draw[fill] (270:10) \smlnd;
            \dummyNodesPolar[90:12][270:12]
        \end{tikzpicture}
    \right)
    = 
    f_{\ref{motzkinFamily:chords}} \left(
        \begin{tikzpicture}[baseline={([yshift=-.5ex]current bounding box.center)}, scale = \scl]
            \node at (90:10) {\V};
            \draw (0,0) circle [radius = 10];
            \draw[fill] (30:10) \smlnd;
            \draw[fill] (150:10) \smlnd;
            \draw[fill] (270:10) \smlnd;
            \dummyNodesPolar[90:12][270:12]
            \draw (150:10) to [bend right = 30] (30:10);
        \end{tikzpicture}
    \right) 
    = 
    \begin{tikzpicture}[baseline={([yshift=-.5ex]current bounding box.center)}, scale = \scl]
        \node at (90:10) {\V};
        \draw (0,0) circle [radius = 10];
        \draw[fill] (45:10) \smlnd;
        \draw[fill] (135:10) \smlnd;
        \draw[fill] (-45:10) \smlnd;
        \draw[fill] (-135:10) \smlnd;
        \dummyNodesPolar[90:12][270:12]
        \draw (-45:10) to (135:10);
    \end{tikzpicture}
\end{equation*}

So the universal bijection maps
\begin{equation*}
    \begin{tikzpicture}[baseline={([yshift=-.5ex]current bounding box.center)}, scale = 0.6]
        \grd[4][1]
        \draw (0,0) \motUp \motAc \motDn \motAc;
    \end{tikzpicture}
    \quad \mapsto \quad
    \begin{tikzpicture}[baseline={([yshift=-.5ex]current bounding box.center)}, scale = 0.05]
        \def\scl{0.05}
        \node at (90:10) {\V};
        \draw (0,0) circle [radius = 10];
        \draw[fill] (45:10) \smlnd;
        \draw[fill] (135:10) \smlnd;
        \draw[fill] (-45:10) \smlnd;
        \draw[fill] (-135:10) \smlnd;
        \dummyNodesPolar[90:12][270:12]
        \draw (-45:10) to (135:10);
    \end{tikzpicture}
\end{equation*}

\def\scl{0.5}
If we were instead seeking a bijection from Motzkin paths to Motzkin unary-binary trees, then we would simply replace the maps and generators in the factorised expression for the Motzkin path with the maps and generators from the \mbox{(1,2)-magma} corresponding to Motzkin unary-binary trees  as follows:
\begin{equation*} 
    f_{\ref{motzkinFamily:motzkinTrees}} \left(
        \epsilon_{\ref{motzkinFamily:motzkinTrees}}
        \star_{\ref{motzkinFamily:motzkinTrees}}
        f_{\ref{motzkinFamily:motzkinTrees}} \left(
            \epsilon_{\ref{motzkinFamily:motzkinTrees}}
        \right)
    \right)
    = 
    f_{\ref{motzkinFamily:motzkinTrees}} \left(
        \begin{tikzpicture}[baseline={([yshift=-.5ex]current bounding box.center)}, scale = \scl]
            \fill (0,0) \smlnd;
        \end{tikzpicture}
        \star_{\ref{motzkinFamily:motzkinTrees}}
        f_{\ref{motzkinFamily:motzkinTrees}} \left(
            \begin{tikzpicture}[baseline={([yshift=-.5ex]current bounding box.center)}, scale = \scl]
                \fill (0,0) \smlnd;
            \end{tikzpicture}
        \right)
    \right)
    =
    f_{\ref{motzkinFamily:motzkinTrees}} \left(
        \begin{tikzpicture}[baseline={([yshift=-.5ex]current bounding box.center)}, scale = \scl]
            \fill (0,0) \smlnd;
        \end{tikzpicture}
        \star_{\ref{motzkinFamily:motzkinTrees}}\!
        \begin{tikzpicture}[baseline={([yshift=-.9ex]current bounding box.center)}, scale = \scl]
            \draw (0,0) \un;
            \fill (0,0) \smlunnd;
            \node (0,0) {\smlroot};
        \end{tikzpicture}\!
    \right)
    = 
    f_{\ref{motzkinFamily:motzkinTrees}} \left(
        \begin{tikzpicture}[baseline={([yshift=-.5ex]current bounding box.center)}, scale = \scl]
            \draw (0,0) \bin (1,-1) \un;
            \fill (0,0) \smlbinnd (1,-1) \smlunnd;
            \node (0,0) {\smlroot};
        \end{tikzpicture}
    \, \right) 
    = 
    \begin{tikzpicture}[baseline={([yshift=-.5ex]current bounding box.center)}, scale = \scl]
        \draw (0,0) \un (0,-1) \bin (1,-2) \un;
        \fill (0,0) \smlunnd (0,-1) \smlbinnd (1,-2) \smlunnd;
        \node (0,0) {\smlroot};
    \end{tikzpicture}
\end{equation*}

Thus we see that the universal bijection gives
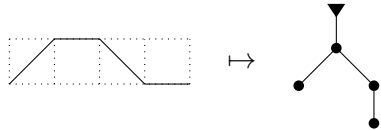
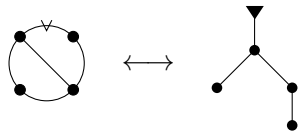
\begin{equation*}
    \begin{tikzpicture}[baseline={([yshift=-.5ex]current bounding box.center)}, scale = 0.6]
        \grd[4][1]
        \draw (0,0) \motUp \motAc \motDn \motAc;
    \end{tikzpicture}
    \quad \mapsto \quad
    \begin{tikzpicture}[baseline={([yshift=-.5ex]current bounding box.center)}, scale = 0.5]
        \draw (0,0) \un (0,-1) \bin (1,-2) \un;
        \fill (0,0) \smlunnd (0,-1) \smlbinnd (1,-2) \smlunnd;
        \node (0,0) {\smlroot};
    \end{tikzpicture}
\end{equation*}
and also
\begin{equation*}
    \begin{tikzpicture}[baseline={([yshift=-.5ex]current bounding box.center)}, scale = 0.05]
        \def \scl {0.05}
        \node at (90:10) {\V};
        \draw (0,0) circle [radius = 10];
        \draw[fill] (45:10) \smlnd;
        \draw[fill] (135:10) \smlnd;
        \draw[fill] (-45:10) \smlnd;
        \draw[fill] (-135:10) \smlnd;
        \dummyNodesPolar[90:12][270:12]
        \draw (-45:10) to (135:10);
    \end{tikzpicture}
    \quad \longleftrightarrow \quad
    \begin{tikzpicture}[baseline={([yshift=-.5ex]current bounding box.center)}, scale = 0.5]
        \draw (0,0) \un (0,-1) \bin (1,-2) \un;
        \fill (0,0) \smlunnd (0,-1) \smlbinnd (1,-2) \smlunnd;
        \node (0,0) {\smlroot};
    \end{tikzpicture}
\end{equation*}

\subsubsection{Universal bijections for Schr\"oder families}

Consider the following Schr\"oder normed \mbox{(1,2)-magmas}:
\begin{itemize}
    \item Rectangulations \schMag{schroderFamily:rectangulations},
    \item Semi-standard Young tableaux of shape $n \times 2$ \schMag{schroderFamily:SSYT},
    \item Unary-binary trees \schMag{schroderFamily:schroderTrees}.
\end{itemize}
See Appendix \ref{appendix:Schroder} for the definitions of these families and details of the relevant (1,2)-magmas and norms.

Take a rectangulation and factorise it:
\begin{equation*}
    \begin{tikzpicture}[baseline={([yshift=-.5ex]current bounding box.center)}, scale = \scl]
        \draw (0,0) rectangle (4,4);
        \draw[thick] (0,2) -- (4,2) (0,3) -- (4,3) (1,0) -- (1,2);
        \fill (1,1) \smlnd (2,2) \smlnd (3,3) \smlnd;
    \end{tikzpicture}
    = 
    \begin{tikzpicture}[baseline={([yshift=-.5ex]current bounding box.center)}, scale = \scl]
        \draw (0,0) rectangle (2,2);
        \draw[thick] (1,0) -- (1,2);
        \fill (1,1) \smlnd;
    \end{tikzpicture}
    \ \star_{\ref{schroderFamily:rectangulations}} \
    \begin{tikzpicture}[baseline={([yshift=-.5ex]current bounding box.center)}, scale = \scl]
        \draw (0,0) rectangle (2,2);
        \draw[thick] (0,1) -- (2,1);
        \fill (1,1) \smlnd;
    \end{tikzpicture}
    =
    f_{\ref{schroderFamily:rectangulations}} \left(\!
        \begin{tikzpicture}[baseline={([yshift=-.5ex]current bounding box.center)}, scale = \scl]
            \draw (0,0) rectangle (1,1);
            \dummyNodes[0][0][1][1]
        \end{tikzpicture}\!
    \right)
    \star_{\ref{schroderFamily:rectangulations}}
    \left(\!
        \begin{tikzpicture}[baseline={([yshift=-.5ex]current bounding box.center)}, scale = \scl]
            \draw (0,0) rectangle (1,1);
            \dummyNodes[0][0][1][1]
        \end{tikzpicture}
        \!\star_{\ref{schroderFamily:rectangulations}}\!
        \begin{tikzpicture}[baseline={([yshift=-.5ex]current bounding box.center)}, scale = \scl]
            \draw (0,0) rectangle (1,1);
            \dummyNodes[0][0][1][1]
        \end{tikzpicture}\!
    \right)
    =
    f_{\ref{schroderFamily:rectangulations}} \left(
        \epsilon_{\ref{schroderFamily:rectangulations}}
    \right)
    \star_{\ref{schroderFamily:rectangulations}}
    \left(
        \epsilon_{\ref{schroderFamily:rectangulations}}
        \star_{\ref{schroderFamily:rectangulations}}
        \epsilon_{\ref{schroderFamily:rectangulations}}
    \right)
\end{equation*}

Now to obtain the image in the family of semi-standard Young tableaux of shape $2 \times n$ via the universal bijection we replace the generators and the maps with the respective generators and maps from the Schr\"oder normed \mbox{(1,2)-magma} of SSYT of shape $2 \times n$:
\begin{equation*}
    f_{\ref{schroderFamily:SSYT}} \left(
        \epsilon_{\ref{schroderFamily:SSYT}}
    \right)
    \star_{\ref{schroderFamily:SSYT}}
    \left(
        \epsilon_{\ref{schroderFamily:SSYT}}
        \star_{\ref{schroderFamily:SSYT}}
        \epsilon_{\ref{schroderFamily:SSYT}}
    \right)
    =
    f_{\ref{schroderFamily:SSYT}} \left(
        \emptyset
    \right)
    \star_{\ref{schroderFamily:SSYT}}
    \left(
        \emptyset
        \star_{\ref{schroderFamily:SSYT}}
        \emptyset
    \right)
    =
    \resizebox{\SSYTscl}{!}{ \begin{tabular}{| c | c |}
        \hline
        1 & 1 \\
        \hline
    \end{tabular}}
    \ \star_{\ref{schroderFamily:SSYT}} \ 
    \resizebox{\SSYTscl}{!}{\begin{tabular}{| c | c |}
        \hline
        1 & 2 \\
        \hline
    \end{tabular}}
    =
    \resizebox{\SSYTscl}{!}{\begin{tabular}{| c | c |}
        \hline
        1 & 1 \\
        \hline
        2 & 4 \\
        \hline 
        3 & 5 \\
        \hline
    \end{tabular}}
\end{equation*}

Finally, to obtain a bijection to Schr\"oder unary-binary trees, replace the maps and generators in the factorised expression for the rectangulation as follows:
\begin{equation*}
    f_{\ref{schroderFamily:schroderTrees}} \left(
        \epsilon_{\ref{schroderFamily:schroderTrees}}
    \right)
    \star_{\ref{schroderFamily:schroderTrees}}
    \left(
        \epsilon_{\ref{schroderFamily:schroderTrees}}
        \star_{\ref{schroderFamily:schroderTrees}}
        \epsilon_{\ref{schroderFamily:schroderTrees}}
    \right)
    =
    f_{\ref{schroderFamily:schroderTrees}} \left(
        \begin{tikzpicture}[baseline={([yshift=-.5ex]current bounding box.center)}, scale = \scl]
            \fill (0,0) \smlnd;
        \end{tikzpicture}
    \right)
    \star_{\ref{schroderFamily:schroderTrees}}
    \left(
        \begin{tikzpicture}[baseline={([yshift=-.5ex]current bounding box.center)}, scale = \scl]
            \fill (0,0) \smlnd;
        \end{tikzpicture}
        \star_{\ref{schroderFamily:schroderTrees}}
        \begin{tikzpicture}[baseline={([yshift=-.5ex]current bounding box.center)}, scale = \scl]
            \fill (0,0) \smlnd;
        \end{tikzpicture}
    \right)
    \ =
    \begin{tikzpicture}[baseline={([yshift=-.5ex]current bounding box.center)}, scale = \scl]
        \draw (0,0) \un;
        \fill (0,0) \smlunnd;
        \node (0,0) {\smlroot};
    \end{tikzpicture}
    \ \star_{\ref{schroderFamily:schroderTrees}} \ 
    \begin{tikzpicture}[baseline={([yshift=-.5ex]current bounding box.center)}, scale = \scl]
        \draw (0,0) \bin;
        \fill (0,0) \smlbinnd;
        \node (0,0) {\smlroot};
    \end{tikzpicture}
    \ = \ 
    \begin{tikzpicture}[baseline={([yshift=-.5ex]current bounding box.center)}, scale = \scl]
        \draw (0,0) \bin (-1,-1) \un (1,-1) \bin;
        \fill (0,0) \smlbinnd (-1,-1) \smlunnd (1,-1) \smlbinnd;
        \node (0,0) {\smlroot};
    \end{tikzpicture}
\end{equation*}

Thus we see that we have the following bijections:
\begin{equation*}
    \begin{tikzpicture}
        \node[anchor = south] at (0,0) (1) {         
            \begin{tikzpicture}[baseline={([yshift=-.5ex]current bounding box.center)}, scale = \scl]
                \draw (0,0) rectangle (4,4);
                \draw[thick] (0,2) -- (4,2) (0,3) -- (4,3) (1,0) -- (1,2);
                \fill (1,1) \smlnd (2,2) \smlnd (3,3) \smlnd;
            \end{tikzpicture} 
        };
        \node[anchor = north east] at (-1,-1) (2) {
            \resizebox{\SSYTscl}{!}{\begin{tabular}{| c | c |}
                \hline
                1 & 1 \\
                \hline
                2 & 4 \\
                \hline 
                3 & 5 \\
                \hline
            \end{tabular}}
        };
        \node[anchor = north west] at (1,-1) (3) {
            {\begin{tikzpicture}[baseline={([yshift=-.5ex]current bounding box.center)}, scale = \scl]
                \draw (0,0) \bin (-1,-1) \un (1,-1) \bin;
                \fill (0,0) \smlbinnd (-1,-1) \smlunnd (1,-1) \smlbinnd;
                \fill (0,-0.75*0.12/\scl) -- (1*0.12/\scl,0.75*0.12/\scl) -- (-1*0.12/\scl,0.75*0.12/\scl) -- (0,-0.75*0.12/\scl);
            \end{tikzpicture}}
        };
        \path[->]
            (-1,0) edge (-1.6,-1)
            (-1.6,-1) edge (-1,0)
            (1,0) edge (1.6,-1)
            (1.6,-1) edge (1,0)
            (-0.9,-1.75) edge (0.9,-1.75)
            (0.9,-1.75) edge (-0.9,-1.75);
    \end{tikzpicture}
\end{equation*}

\section{Fuss-Catalan normed (3)-magmas} \label{section:fussCatalan}

In this section we consider (3)-magmas and discuss how these relate to the order 3 Fuss-Catalan sequence. 
Recall that a (3)-magma \mbox{($\cM$, $t$)} is a set $\cM$ along with a ternary map $t: \cM \times \cM \times \cM \rightarrow \cM$.  
We begin this section by constructing an example of a free (3)-magma, which we call the Cartesian (3)-magma. 

\subsection{Cartesian (3)-magma} \label{3magmaDefinitions}

The Cartesian (3)-magma  is one of the simplest (3)-magmas and is constructed using Cartesian products. 
As in section \ref{section:motzkinSchroder}, we use square parentheses to represent a 3-tuple, so we denote  the ternary Cartesian product as follows:
\begin{equation*}
    X_1 \times X_2 \times X_3 = \left\{ [x_1, x_2, x_3]: \, x_1 \in X_1, \, x_2 \in X_2, \, x_3 \in X_3 \right\}.
\end{equation*}

\begin{definition}[Cartesian (3)-magma] \label{def:cartesian3magma}
Let $X$ be a non-empty finite set. Define the sequence $\cW_n(X)$ of sets of nested 3-tuples by
\begin{subequations}
    \begin{align}
        \cW_1(X) & = X, \\
        \cW_{2n + 1}(X) & = \bigcup_{i = 0}^{n - 1} \ \bigcup_{j = 0}^{n - i - 1} \ \left( \cW_{2i + 1}(X) \times \cW_{2j + 1}(X) \times \cW_{2(n - i - j - 1) + 1}(X) \right), \qquad n \geq 1. \label{eq:cart3}
    \end{align}
\end{subequations}
Let $\cW_X = \bigcup_{n \geq 0} \cW_{2n + 1}(X)$ and $\cW_X^+ = \cW_X \backslash X$. Define $\tau: \cW_X \times \cW_X \times \cW_X \rightarrow \cW_X$ by
\begin{equation*}
    \tau(w_1, w_2, w_3) = [w_1, w_2, w_3], \qquad w_1, w_2, w_3 \in \cW_X.
\end{equation*}
The pair \mbox{($\cW_X$, $\tau$)} is called the \textbf{Cartesian (3)-magma generated by $\bm{X}$}.
\end{definition}

If $X = \{ \epsilon \}$, the sets $\cW_n(X)$ in the above definition begin:
\begin{align*}
    \cW_1(X) = \{ & \epsilon \}, \\
    \cW_3(X) = \{ & [ \epsilon, \epsilon, \epsilon ] \}, \\
    \cW_5(X) = \{ & [ \epsilon, \epsilon, [ \epsilon, \epsilon, \epsilon ] ], \quad [ \epsilon, [ \epsilon, \epsilon, \epsilon ], \epsilon ], \quad [ [ \epsilon, \epsilon, \epsilon ], \epsilon, \epsilon ] \}, \\
    \cW_7(X) = \{ & [ \epsilon, \epsilon, [ \epsilon, \epsilon, [ \epsilon, \epsilon, \epsilon ] ] ], \quad [ \epsilon, \epsilon, [ \epsilon, [ \epsilon, \epsilon, \epsilon ], \epsilon ] ], \quad [ \epsilon, \epsilon, [ [ \epsilon, \epsilon, \epsilon ], \epsilon, \epsilon ] ] \\
    & [ \epsilon, [ \epsilon, \epsilon, \epsilon ], [ \epsilon, \epsilon, \epsilon ], \quad [ \epsilon, [ \epsilon, \epsilon, [ \epsilon, \epsilon, \epsilon ] ], \epsilon ], \quad [ \epsilon, [ \epsilon, [ \epsilon, \epsilon, \epsilon ], \epsilon ], \epsilon ], \\
    & [ \epsilon, [ [ \epsilon, \epsilon, \epsilon ], \epsilon, \epsilon ], \epsilon ], \quad [ [ \epsilon, \epsilon, \epsilon ], \epsilon, [ \epsilon, \epsilon, \epsilon ] ], \quad [ [ \epsilon, \epsilon, \epsilon ], [ \epsilon, \epsilon, \epsilon ], \epsilon ] \quad \\
    & [ [ \epsilon, \epsilon, [ \epsilon, \epsilon, \epsilon ] ], \epsilon, \epsilon ], \quad [ [ \epsilon, [ \epsilon, \epsilon, \epsilon ], \epsilon ], \epsilon, \epsilon ], \quad [ [ [ \epsilon, \epsilon, \epsilon ], \epsilon, \epsilon ], \epsilon, \epsilon ] \\
    \vdots \quad &
\end{align*} 

We now prove that the Cartesian \mbox{(3)-magma} is free.

\begin{theorem} \label{thm:free3magma}
The Cartesian (3)-magma \mbox{($\cW_X$, $\tau$)} is a free (3)-magma.
\end{theorem}

We will show that \mbox{($\cW_X$, $\tau$)} is a unique factorisation normed (3)-magma with set of irreducibles $X$. Then, by Theorem \ref{thm:uniqFactNormedThenFree}, we will have that \mbox{($\cW_X$, $\tau$)} is a free (3)-magma generated by $X$.

\begin{proof}
Suppose that \mbox{$\tau(w_1, w_2, w_3) = \tau(w_1', w_2', w_3')$} for $w_i, w_i' \in \cW_X$, $i = 1, 2, 3$. Therefore 
$[w_1, w_2, w_3] = [w_1', w_2', w_3']$
and hence $(w_1, w_2, w_3) = (w_1', w_2', w_3')$. Thus we have that $\tau$ is injective and so \mbox{($\cW_X$, $\tau$)} is a unique factorisation (3)-magma.
The set of irreducibles is given by $X$ since this is the complement of $\text{Img}(\tau)$.
We have $\norm{\tau(w_1, w_2, w_3)} = \norm{w_1} + \norm{w_2} + \norm{w_3}$ as $[w_1, w_2, w_3] \in \cW_{n_1 + n_2 + n_3}(X)$ if $w_i \in \cW_{n_i}(X)$, for $i \in \{ 1, 2, 3 \}$. Thus \mbox{($\cW_X$, $\tau$)} is a normed \mbox{(3)-magma}.
\end{proof}



Note, the norm we have defined in the above proof differs slightly  to the traditional size parameter for combinatorial families counted by this sequence as it is more natural for us to have only objects with odd norm. 
This ensures that the norm is strictly additive with respect to the ternary map. 
If we instead wanted our base set to be partitioned into `norms' given by the natural numbers, we would require our norm to be sub-additive with respect to the ternary map ie.\ the map would then not be a norm.  

We now show how the norm used in the proof of Theorem \ref{thm:free3magma} partitions the Cartesian (3)-magma into sets whose size is equal to the $p$-Fuss-Catalan numbers.

\begin{proposition} \label{prop:sequence3magma}
Let \mbox{($\cW_X$, $\tau$)} be the Cartesian (3)-magma generated by the set $X$, where $\card{X} = p$, and define the map $\norm{\cdot}_T: \cW_X \rightarrow \mathbb{N}$ by $\norm{m}_T = n$ when $m \in \cW_n(X)$. If
\begin{equation*}
    N_n = \{ m \in \cW_X: \, \norm{m}_T = n \}, \qquad n \geq 1,
\end{equation*}
then
\begin{equation*}
    \card{N_{2n + 1}} = T_n(p), \qquad n \geq 0.
\end{equation*}
\end{proposition}

\begin{proof}
Begin by noting that $\norm{m}_T = n$ if and only if $m \in \cW_n(X)$, and hence we have $N_n = \cW_n(X)$. Therefore $\card{N_1} = \card{\cW_1(X)} = \card{X} = p$. For $n \geq 1$, \eqref{eq:cart3} gives
\begin{align*}
    \card{N_{2n + 1}} & = \card{\bigcup_{i = 0}^{n - 1} \ \bigcup_{j = 0}^{n - i - 1} \ \left( \cW_{2i + 1}(X) \times \cW_{2j + 1}(X) \times \cW_{2(n - i - j - 1) + 1}(X) \right)} \\
    & = \sum_{i = 0}^{n - 1} \sum_{j = 0}^{n - i - 1} \card{N_{2i + 1}} \cdot \card{N_{2j + 1}} \cdot \card{N_{2(n - i - j - 1) + 1}}
\end{align*}
This is equivalent to the $p$-Fuss-Catalan recurrence \eqref{def:pSequences}.
\end{proof}

\begin{corollary}
Let ($\cW_{\epsilon}$, $\tau$) be the Cartesian (3)-magma generated by the single element $\epsilon$. If $N_n$ and $\norm{\cdot}_M: \cW_{\epsilon} \rightarrow \mathbb{N}$ are as defined in Proposition \ref{prop:sequence3magma}, then
\begin{equation*}
    \card{N_{2n + 1}} = T_n, \qquad n \geq 0,
\end{equation*}
with $T_n$ as defined in Definition \ref{eq:pTernaryRecurrence}.
\end{corollary}

Note,   the norm function $\norm{\cdot}_T: \cW_{\epsilon} \rightarrow \mathbb{N}$ from the above corollary could equivalently be defined as follows: Let $\norm{\cdot}_T: \cW_{\epsilon} \rightarrow \mathbb{N}$ be such that
\begin{subequations} \label{eq:fcNorm0}
    \begin{align*}
        \norm{\epsilon}_T & = 1, \\
        \norm{\tau(w_1, w_2, w_3)}_T & = \norm{w_1}_T + \norm{w_2}_T + \norm{w_3}_T, \quad w_1, w_2, w_3 \in \cW_{\epsilon}.
    \end{align*}
\end{subequations}

We have seen that a free \mbox{(3)-magma} generated by a single element with a norm function satisfying \eqref{eq:fcNorm0} is partitioned into sets given by the Fuss-Catalan numbers. Thus we are now ready to characterise when a \mbox{(3)-magma} is associated with the Fuss-Catalan numbers.

\begin{definition}[Fuss-Catalan normed (3)-magma] \label{def:ternary3magma}
Let \mbox{($\cM$, $t$)} be a unique factorisation normed \mbox{(3)-magma} with one irreducible element, $\epsilon$. Let \mbox{$\norm{\cdot}_T: \cM \rightarrow \mathbb{N}$} be a norm which satisfies $\norm{\epsilon}_T = 1$ and
\begin{equation} \label{eq:fcNorm}
    \norm{t(m_1, m_2, m_3)}_T = \norm{m_1}_T + \norm{m_2}_T + \norm{m_3}_T,
\end{equation}
for all $m_1, m_2, m_3 \in \cM$. Then \mbox{($\cM$, $t$)} with the norm $\norm{\cdot}_T: \cM \rightarrow \mathbb{N}$ is called a \textbf{Fuss-Catalan normed (3)-magma}.
\end{definition}

Let ($\cM$, $t$) be a Fuss-Catalan normed (3)-magma. Then the base set $\cM$ begins by evaluating the following expressions:
\begin{center}
    \begin{tabular}{l l}
        Norm 1: & $\epsilon$. \\
        Norm 3: & $t(\epsilon, \, \epsilon, \, \epsilon)$. \\
        Norm 5: & $t(\epsilon, \, \epsilon, \, t(\epsilon, \, \epsilon, \, \epsilon))$, \ $t(\epsilon, \, t(\epsilon, \, \epsilon, \, \epsilon), \, \epsilon)$, \ $t(t(\epsilon, \, \epsilon, \, \epsilon), \, \epsilon, \, \epsilon)$. \\
        Norm 7: & $t(\epsilon, \, t(\epsilon, \, \epsilon, \, \epsilon), \, t(\epsilon, \, \epsilon, \, \epsilon))$, \ $t(t(\epsilon, \, \epsilon, \, \epsilon), \, \epsilon, \, t(\epsilon, \, \epsilon, \, \epsilon))$, \ $t(t(\epsilon, \, \epsilon, \, \epsilon), \, t(\epsilon, \, \epsilon, \, \epsilon), \, \epsilon)$, \\
        & $t(\epsilon, \, \epsilon, \, t(\epsilon, \, \epsilon, \, t(\epsilon, \, \epsilon, \, \epsilon)))$, \ $t(\epsilon, \, \epsilon, \, t(\epsilon, \, t(\epsilon, \, \epsilon, \, \epsilon), \, \epsilon))$, \ $t(\epsilon, \, \epsilon, \, t(t(\epsilon, \, \epsilon, \, \epsilon), \, \epsilon, \, \epsilon))$, \\
        & $t(\epsilon, \, t(\epsilon, \, \epsilon, \, t(\epsilon, \, \epsilon, \, \epsilon)), \, \epsilon)$, \ $t(\epsilon, \, t(\epsilon, \, t(\epsilon, \, \epsilon, \, \epsilon), \, \epsilon), \, \epsilon)$, \ $t(\epsilon, \, t(t(\epsilon, \, \epsilon, \, \epsilon), \, \epsilon, \, \epsilon), \, \epsilon)$, \\
        & $t(t(\epsilon, \, \epsilon, \, t(\epsilon, \, \epsilon, \, \epsilon)), \, \epsilon, \, \epsilon)$, \ $t(t(\epsilon, \, t(\epsilon, \, \epsilon, \, \epsilon), \, \epsilon), \, \epsilon, \, \epsilon)$, \ $t(t(t(\epsilon, \, \epsilon, \, \epsilon), \, \epsilon, \, \epsilon), \, \epsilon, \, \epsilon)$. \\
    \end{tabular}
\end{center}

\def\scl{0.15}
Another example of a Fuss-Catalan normed (3)-magma is the Fuss-Catalan family of ternary trees \fcFamRef{fcFamily:ternaryTrees}. 
The generator is the tree containing no edges which we represent by a single node:
\begin{equation*}
    \epsilon_{\ref{fcFamily:ternaryTrees}} =
    \begin{tikzpicture}[baseline={([yshift=-.5ex]current bounding box.center)}, scale = \scl]
        \fill (0,0) \smlnd;
    \end{tikzpicture}
\end{equation*}
The base set is the set of all ternary trees and the ternary map is defined as follows:
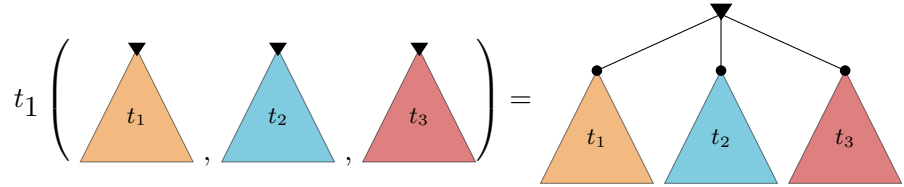
\begin{equation*}
    t_{\ref{fcFamily:ternaryTrees}} \left( 
    \raisebox{-20pt}{
    \begin{tikzpicture}[baseline={([yshift=-5ex]current bounding box.center)}, scale = \scl]
        \draw[fill = \colOne, opacity = \opac] (0,0) -- (-5,-10) -- (5,-10) -- (0,0);
        \node at (0,-2.5) {\smlroot};
        \node at (0,-6) {\footnotesize $t_1$};
    \end{tikzpicture}
    , 
    \begin{tikzpicture}[baseline={([yshift=-5ex]current bounding box.center)}, scale = \scl]
        \draw[fill = \colTwo, opacity = \opac] (0,0) -- (-5,-10) -- (5,-10) -- (0,0);
        \node at (0,-2.5) {\smlroot};
        \node at (0,-6) {\footnotesize $t_2$};
    \end{tikzpicture}
    ,
    \begin{tikzpicture}[baseline={([yshift=-5ex]current bounding box.center)}, scale = \scl]
        \draw[fill = \colThree, opacity = \opac] (0,0) -- (-5,-10) -- (5,-10) -- (0,0);
        \node at (0,-2.5) {\smlroot};
        \node at (0,-6) {\footnotesize $t_3$};
    \end{tikzpicture}}
    \right)
    =
    \begin{tikzpicture}[baseline={([yshift=-1ex]current bounding box.center)}, scale = \scl]
        \draw[fill = \colOne, opacity = \opac] (-11,-5) -- (-16,-15) -- (-6,-15) -- (-11,-5);
        \draw[fill = \colTwo, opacity = \opac] (0,-5) -- (-5,-15) -- (5,-15) -- (0,-5);
        \draw[fill = \colThree, opacity = \opac] (11,-5) -- (16,-15) -- (6,-15) -- (11,-5);
        \draw (0,0) -- (0,-5) (0,0) -- (-11,-5) (0,0) -- (11,-5);
        \node at (0,0) {\root};
        \fill (-11,-5) \smlnd (0,-5) \smlnd (11,-5) \smlnd;
        \node at (-11,-11) {\footnotesize $t_1$};
        \node at (0,-11) {\footnotesize $t_2$};
        \node at (11,-11) {\footnotesize $t_3$};
        \dummyNodes[0][0.2][0][-15.2]
    \end{tikzpicture}
\end{equation*}

The norm of any ternary tree is defined to be the number of leaves in the tree. Notice that this norm is additive as the number of leaves in a ternary tree is equal to the sum of the number of leaves in each of its three factors.

Sorting by norm, the base set  of the ternary tree Fuss-Catalan normed (3)-magma begins as follows:
\begin{center}
    \def\scl{0.6}
    \begin{tabular}{| l l |}
    \hline
    \textit{Norm 1:} &
    $\begin{tikzpicture}[baseline={([yshift=-.5ex]current bounding box.center)}, scale = \scl]
        \fill (0,0) \smlnd;
        \dummyNodes[0][0.2][0][-0.2]
    \end{tikzpicture} 
    = \epsilon$ \\
    \hline
    \textit{Norm 3:} & 
    $\begin{tikzpicture}[baseline={([yshift=-.5ex]current bounding box.center)}, scale = \scl]
        \node at (0,0) {\smlroot};
        \draw (0,0) -- (-1,-1) (0,0) -- (0,-1) (0,0) -- (1,-1);
        \fill (-1,-1) \smlnd  (0,-1) \smlnd (1,-1) \smlnd;
        \dummyNodes[0][0.2][0][-1.2]
    \end{tikzpicture} 
    = t(\epsilon, \epsilon, \epsilon)$ \\
    \hline
    \textit{Norm 5:} &
    $\begin{tikzpicture}[baseline={([yshift=-.5ex]current bounding box.center)}, scale = \scl]
        \node at (0,0) {\smlroot};
        \draw (0,0) -- (-1,-1) (0,0) -- (0,-1) (0,0) -- (1,-1);
        \fill (-1,-1) \smlnd  (0,-1) \smlnd (1,-1) \smlnd;
        \draw (-1,-1) -- (-2,-2) (-1,-1) -- (-1,-2) (-1,-1) -- (0,-2);
        \fill (-2,-2) \smlnd  (-1,-2) \smlnd (0,-2) \smlnd;
        \dummyNodes[0][0.2][0][-2.2]
    \end{tikzpicture} 
    = t(t(\epsilon, \epsilon, \epsilon), \epsilon, \epsilon)$,
    \quad
    $\begin{tikzpicture}[baseline={([yshift=-.5ex]current bounding box.center)}, scale = \scl]
        \node at (0,0) {\smlroot};
        \draw (0,0) -- (-1,-1) (0,0) -- (0,-1) (0,0) -- (1,-1);
        \fill (-1,-1) \smlnd  (0,-1) \smlnd (1,-1) \smlnd;
        \draw (0,-1) -- (-1,-2) (0,-1) -- (0,-2) (0,-1) -- (1,-2);
        \fill (-1,-2) \smlnd  (0,-2) \smlnd (1,-2) \smlnd;
        \dummyNodes[0][0.2][0][-2.2]
    \end{tikzpicture} 
    = t(\epsilon, t(\epsilon, \epsilon, \epsilon), \epsilon)$, \\
    &
    $\begin{tikzpicture}[baseline={([yshift=-.5ex]current bounding box.center)}, scale = \scl]
        \node at (0,0) {\smlroot};
        \draw (0,0) -- (-1,-1) (0,0) -- (0,-1) (0,0) -- (1,-1);
        \fill (-1,-1) \smlnd  (0,-1) \smlnd (1,-1) \smlnd;
        \draw (1,-1) -- (0,-2) (1,-1) -- (1,-2) (1,-1) -- (2,-2);
        \fill (0,-2) \smlnd  (1,-2) \smlnd (2,-2) \smlnd;
        \dummyNodes[0][0.2][0][-2.2]
    \end{tikzpicture}  
    = t(\epsilon, \epsilon, t(\epsilon, \epsilon, \epsilon))$ \\
    \hline
    \end{tabular}
\end{center}

Clearly, this is a unique factorisation \mbox{(3)-magma} with only one irreducible element. This is immediate from the definition of the ternary map, as we see that any ternary tree has unique left, middle and right subtrees. Noting that there exists a norm, we have from Theorem \ref{thm:uniqFactNormedThenFree} that this is a free \mbox{(3)-magma}. Since the norm satisfies \eqref{eq:fcNorm}, we have that this is indeed a Fuss-Catalan normed \mbox{(3)-magma}.

\subsection{Free (3)-magma isomorphisms and a universal bijection}

 Proposition \ref{prop:freenMagmaIsomorphism} tells us that there is a unique \mbox{(3)-magma} ismorphism between any two free \mbox{(3)-magmas} with the same number of generators. This allows us to define a universal bijection between any two Fuss-Catalan normed \mbox{(3)-magmas}.

\begin{definition}[Universal bijection] \label{def:3univBij}
Let ($\cM$, $t$) and ($\cN$, $t'$) be free (3)-magmas with generating sets $X_{\cM}$ and $X_{\cN}$ respectively, with $\card{X_{\cM}} = \card{X_{\cN}}$. Let $\sigma: X_{\cM} \rightarrow X_{\cN}$ be any bijection, and define the map $\Upsilon: \cM \rightarrow \cN$ as follows. For all $m \in \cM \setminus X_{\cM}$,
\begin{enumerate}[label=(\roman*)]
    \item Decompose $m$ into an expression in terms of generators $\epsilon_i \in X_{\cM}$ and the map $t$.
    \item In the decomposition of $m$ replace every occurrence of $\epsilon_i$ with $\sigma(\epsilon_i)$ and every occurrence of $t$ with $t'$. Call this expression $\upsilon(m)$.
    \item Define $\Upsilon(m)$ to be $\upsilon(m)$, that is, evaluate all maps in $\upsilon(m)$ to give an element of $\cN$.
\end{enumerate}
\end{definition}

From the above definition, we immediately get the following proposition. As noted in previous sections, this comes from the fact that the map $\Upsilon$ of Definition \ref{def:3univBij} is equal to the map $\Gamma$ of Proposition \ref{prop:freenMagmaIsomorphism}.

\begin{proposition} \label{prop:3univBij}
Let $\Upsilon: \cM \rightarrow \cN$ be the map of Definition \ref{def:3univBij}. Then $\Upsilon$ is a free (3)-magma isomorphism.
\end{proposition}

Schematically, we can write $\Upsilon$ as follows:
\begin{equation}
    m \qquad \overset{\text{decompose}}{\xrightarrow{\hspace*{1.5cm}}} \qquad \underset{\substack{\epsilon_i \rightarrow \, \sigma(\epsilon_i), \ t \rightarrow \, t'}}{\text{substitute}} \qquad \overset{\text{evaluate}}{\xrightarrow{\hspace*{1.5cm}}} \qquad n.
\end{equation}
\vspace{-.5ex}
$\Upsilon$ is an isomorphism between the free (3)-magmas ($\cM$, $t$) and ($\cN$, $t'$), and therefore defines a bijection between the base sets $\cM$ and $\cN$. It also gives us that this bijection is recursive:
\begin{equation*}
    \Upsilon \left( t(m_1, m_2, m_3) \right) = t' \left( \Upsilon(m_1), \Upsilon(m_2), \Upsilon(m_3) \right).
\end{equation*}
\vspace{-.5ex}
It is also important to note that $\Upsilon$ is a norm-preserving map, subject to some conditions on the norms defined on the two (3)-magmas. 
Suppose that \mbox{($\cM$, $t$)} and \mbox{($\cN$, $t'$)} are free (3)-magmas with generating sets of the same cardinality which have norms \mbox{$\norm{\cdot}_{\cM}: \cM \rightarrow \mathbb{N}$} and \mbox{$\norm{\cdot}_{\cN}: \cN \rightarrow \mathbb{N}$}. 
If the norms $\norm{\cdot}_{\cM}$ and $\norm{\cdot}_{\cN}$ are such that:
\begin{enumerate}[label=(\roman*)]
    \item \label{cond:fc1} $\norm{m}_{\cM} = \norm{\sigma(m)}_{\cN}$ for all $m \in X_{\cM}$, and
    \item \label{cond:fc2} for all \mbox{$m_1, m_2, m_3 \in \cM$},  \mbox{$\norm{t(m_1, m_2, m_3)}_{\cM} = \sum\limits_{i = 1}^3 \norm{m_i}_{\cM} + \kappa$}, and for \\ all $n_1, n_2, n_3 \in \cN$, \mbox{$\norm{t'(n_1, n_2, n_3)}_{\cM} = \sum\limits_{i = 1}^3 \norm{n_i}_{\cM} + \kappa$},  for some $\kappa \in \mathbb{N}_0$,
\end{enumerate}
then we have
\begin{equation*}
    \norm{m}_{\cM} = \norm{\Upsilon(m)}_{\cN}, \qquad m \in \cM.
\end{equation*}
This can be seen by considering the decomposed expressions for $m$ and $\Upsilon(m)$.

This tells us that if \mbox{($\cM$, $t$)} and \mbox{($\cN$, $t'$)} are both Fuss-Catalan normed (3)-magmas, then the norm is preserved under the universal bijection $\Upsilon$.


We now give some examples illustrating this universal bijection between free (3)-magmas generated by a single element. 
As in previous sections, each map and appearance of a generator references a family from Appendix \ref{appendix:FC} via the subscript. 
We consider the following Fuss-Catalan normed (3)-magmas:
\begin{itemize}
    \item Ternary trees \fcMag{fcFamily:ternaryTrees},
    \item Quadrillages \fcMag{fcFamily:quadrillages},
    \item Non-crossing partitions \fcMag{fcFamily:partitions}.
\end{itemize}

Consider the following ternary tree, which we can factorise as shown:
\def\scl{0.4}
\begin{align*} 
    \begin{tikzpicture}[baseline={([yshift=-.5ex]current bounding box.center)}, scale = \scl]
        \draw (0,0) -- (-1,-1) (0,0) -- (0,-1) (0,0) -- (1,-1);
        \fill (-1,-1) \smlnd (0,-1) \smlnd (1,-1) \smlnd;
        \draw (0,-1) -- (-1,-2) (0,-1) -- (0,-2) (0,-1) -- (1,-2);
        \fill (-1,-2) \smlnd (0,-2) \smlnd (1,-2) \smlnd;
        \draw (1,-2) -- (0,-3) (1,-2) -- (1,-3) (1,-2) -- (2,-3);
        \fill (0,-3) \smlnd (1,-3) \smlnd (2,-3) \smlnd;
        \node (0,0) {\smlroot};
    \end{tikzpicture}
    = 
    t_{\ref{fcFamily:ternaryTrees}} \left(
        \begin{tikzpicture}[baseline={([yshift=-.5ex]current bounding box.center)}, scale = \scl]
            \fill (0,0) \smlnd;
        \end{tikzpicture}\spacecomma
        \begin{tikzpicture}[baseline={([yshift=-.5ex]current bounding box.center)}, scale = \scl]
            \draw (0,0) -- (-1,-1) (0,0) -- (0,-1) (0,0) -- (1,-1);
            \fill (-1,-1) \smlnd (0,-1) \smlnd (1,-1) \smlnd;
            \draw (1,-1) -- (0,-2) (1,-1) -- (1,-2) (1,-1) -- (2,-2);
            \fill (0,-2) \smlnd (1,-2) \smlnd (2,-2) \smlnd;
            \node (0,0) {\smlroot};
        \end{tikzpicture}\spacecomma
        \begin{tikzpicture}[baseline={([yshift=-.5ex]current bounding box.center)}, scale = \scl]
            \fill (0,0) \nd;
        \end{tikzpicture}
    \right)
    & =
    t_{\ref{fcFamily:ternaryTrees}} \left(
        \begin{tikzpicture}[baseline={([yshift=-.5ex]current bounding box.center)}, scale = \scl]
            \fill (0,0) \smlnd;
        \end{tikzpicture}\spacecomma
        t_{\ref{fcFamily:ternaryTrees}} \left(
           \begin{tikzpicture}[baseline={([yshift=-.5ex]current bounding box.center)}, scale = \scl]
                \fill (0,0) \smlnd;
            \end{tikzpicture}\spacecomma
                \begin{tikzpicture}[baseline={([yshift=-.5ex]current bounding box.center)}, scale = \scl]
                \fill (0,0) \smlnd;
            \end{tikzpicture} \, ,
            \begin{tikzpicture}[baseline={([yshift=-1.5ex]current bounding box.center)}, scale = \scl]
                \draw (0,0) -- (-1,-1) (0,0) -- (0,-1) (0,0) -- (1,-1);
                \fill (-1,-1) \smlnd (0,-1) \smlnd (1,-1) \smlnd;
                \node (0,0) {\smlroot};
            \end{tikzpicture}
        \right),
        \begin{tikzpicture}[baseline={([yshift=-.5ex]current bounding box.center)}, scale = \scl]
            \fill (0,0) \smlnd;
        \end{tikzpicture}
    \right)
    =
    t_{\ref{fcFamily:ternaryTrees}} \left(
        \begin{tikzpicture}[baseline={([yshift=-.5ex]current bounding box.center)}, scale = \scl]
            \fill (0,0) \smlnd;
        \end{tikzpicture} \,,
        t_{\ref{fcFamily:ternaryTrees}} \left(
            \begin{tikzpicture}[baseline={([yshift=-.5ex]current bounding box.center)}, scale = \scl]
                \fill (0,0) \smlnd;
            \end{tikzpicture} \,,
            \begin{tikzpicture}[baseline={([yshift=-.5ex]current bounding box.center)}, scale = \scl]
                \fill (0,0) \smlnd;
            \end{tikzpicture} \, ,
            t_{\ref{fcFamily:ternaryTrees}} \left(
                \begin{tikzpicture}[baseline={([yshift=-.5ex]current bounding box.center)}, scale = \scl]
                    \fill (0,0) \smlnd;
                \end{tikzpicture} \, ,
                \begin{tikzpicture}[baseline={([yshift=-.5ex]current bounding box.center)}, scale = \scl]
                    \fill (0,0) \smlnd;
                \end{tikzpicture} \, ,
                \begin{tikzpicture}[baseline={([yshift=-.5ex]current bounding box.center)}, scale = \scl]
                    \fill (0,0) \smlnd;
                \end{tikzpicture}
            \right)
        \right),
        \begin{tikzpicture}[baseline={([yshift=-.5ex]current bounding box.center)}, scale = \scl]
            \fill (0,0) \smlnd;
        \end{tikzpicture}
    \right) \\
    & =
    t_{\ref{fcFamily:ternaryTrees}} \left(
        \epsilon_{\ref{fcFamily:ternaryTrees}},
        t_{\ref{fcFamily:ternaryTrees}} \left(
            \epsilon_{\ref{fcFamily:ternaryTrees}},
            \epsilon_{\ref{fcFamily:ternaryTrees}},
            t_{\ref{fcFamily:ternaryTrees}} \left(
                \epsilon_{\ref{fcFamily:ternaryTrees}},
                \epsilon_{\ref{fcFamily:ternaryTrees}},
                \epsilon_{\ref{fcFamily:ternaryTrees}}
            \right)
        \right),
        \epsilon_{\ref{fcFamily:ternaryTrees}}
    \right)
\end{align*}

We wish to obtain the image of this ternary tree under the universal bijection to the Fuss-Catalan normed (3)-magma of quadrillages. 
Thus we replace the generators and the ternary map in this factorised expression with the generators and the ternary map of the quadrillage (3)-magma:
\begin{gather*}
    t_{\ref{fcFamily:quadrillages}} \left(
        \epsilon_{\ref{fcFamily:quadrillages}},
        t_{\ref{fcFamily:quadrillages}} \left(
            \epsilon_{\ref{fcFamily:quadrillages}},
            \epsilon_{\ref{fcFamily:quadrillages}},
            t_{\ref{fcFamily:quadrillages}} \left(
                \epsilon_{\ref{fcFamily:quadrillages}},
                \epsilon_{\ref{fcFamily:quadrillages}},
                \epsilon_{\ref{fcFamily:quadrillages}}
            \right)
        \right),
        \epsilon_{\ref{fcFamily:quadrillages}}
    \right)
    = 
    \def\numPoints{2}
    \def\scl{0.02}
    t_{\ref{fcFamily:quadrillages}} \left(
        \begin{tikzpicture}[baseline={([yshift=-1ex]current bounding box.center)}, scale = \scl]
            \polygonSetup
        \end{tikzpicture},
        t_{\ref{fcFamily:quadrillages}} \left(
            \begin{tikzpicture}[baseline={([yshift=-1ex]current bounding box.center)}, scale = \scl]
                \polygonSetup
            \end{tikzpicture},
            \begin{tikzpicture}[baseline={([yshift=-1ex]current bounding box.center)}, scale = \scl]
                \polygonSetup
            \end{tikzpicture},
            t_{\ref{fcFamily:quadrillages}} \left(
                \begin{tikzpicture}[baseline={([yshift=-1ex]current bounding box.center)}, scale = \scl]
                    \polygonSetup
                \end{tikzpicture},
                \begin{tikzpicture}[baseline={([yshift=-1ex]current bounding box.center)}, scale = \scl]
                    \polygonSetup
                \end{tikzpicture},
                \begin{tikzpicture}[baseline={([yshift=-1ex]current bounding box.center)}, scale = \scl]
                    \polygonSetup
                \end{tikzpicture}
            \right)
        \right),
        \begin{tikzpicture}[baseline={([yshift=-1ex]current bounding box.center)}, scale = \scl]
            \polygonSetup
        \end{tikzpicture}
    \right) \\
    = \def\numPoints{2} \def\scl{0.02}
    t_{\ref{fcFamily:quadrillages}} \left(
        \begin{tikzpicture}[baseline={([yshift=-1ex]current bounding box.center)}, scale = \scl]
            \polygonSetup
        \end{tikzpicture},
        t_{\ref{fcFamily:quadrillages}} \left(\!
            \begin{tikzpicture}[baseline={([yshift=-1ex]current bounding box.center)}, scale = \scl]
                \polygonSetup
            \end{tikzpicture},\!
            \begin{tikzpicture}[baseline={([yshift=-1ex]current bounding box.center)}, scale = \scl]
                \polygonSetup
            \end{tikzpicture},\! \def\scl{0.06}
            \begin{tikzpicture}[baseline={([yshift=-2.2ex]current bounding box.center)}, scale = \scl]
                \def\numPoints{4}
                \polygonSetup
            \end{tikzpicture}
        \right), \def\scl{0.04}
        \begin{tikzpicture}[baseline={([yshift=-1ex]current bounding box.center)}, scale = \scl]
            \def\numPoints{2}
            \polygonSetup
        \end{tikzpicture}
    \right) 
    = \def\numPoints{2} \def\scl{0.02}
    t_{\ref{fcFamily:quadrillages}} \left(\!\!\!
        \raisebox{-7pt}{
        \begin{tikzpicture}[baseline={([yshift=-1ex]current bounding box.center)}, scale = \scl]
            \polygonSetup
        \end{tikzpicture},\!\!\!\!\! \def\numPoints{6} \def\scl{0.08}
        \begin{tikzpicture}[baseline={([yshift=-3ex]current bounding box.center)}, scale = \scl]
            \polygonSetup
            \dissectionArc[1][4]
        \end{tikzpicture},\!\!\!\!\! \def\numPoints{2} \def\scl{0.02}
        \begin{tikzpicture}[baseline={([yshift=-1ex]current bounding box.center)}, scale = \scl]
            \polygonSetup
        \end{tikzpicture}}
    \right)
    =
    \def\numPoints{8} \def\scl{0.12}
    \begin{tikzpicture}[baseline={([yshift=-1ex]current bounding box.center)}, scale = \scl]
        \polygonSetup
        \dissectionArc[2][5]
        \dissectionArc[2][7]
    \end{tikzpicture}
\end{gather*}

So the universal bijection maps
\begin{equation*}
    \begin{tikzpicture}[baseline={([yshift=-.5ex]current bounding box.center)}, scale = 0.5]
        \def\scl{0.5}
        \draw (0,0) -- (-1,-1) (0,0) -- (0,-1) (0,0) -- (1,-1);
        \fill (-1,-1) \smlnd (0,-1) \smlnd (1,-1) \smlnd;
        \draw (0,-1) -- (-1,-2) (0,-1) -- (0,-2) (0,-1) -- (1,-2);
        \fill (-1,-2) \smlnd (0,-2) \smlnd (1,-2) \smlnd;
        \draw (1,-2) -- (0,-3) (1,-2) -- (1,-3) (1,-2) -- (2,-3);
        \fill (0,-3) \smlnd (1,-3) \smlnd (2,-3) \smlnd;
        \node (0,0) {\smlroot};
    \end{tikzpicture}
    \quad \mapsto \quad
    \begin{tikzpicture}[baseline={([yshift=-1ex]current bounding box.center)}, scale = 0.15]
        \def\numPoints{8}
        \polygonSetup
        \dissectionArc[2][5]
        \dissectionArc[2][7]
    \end{tikzpicture}
\end{equation*}

Bijecting the above ternary tree to a non-crossing partition with blocks of even size, we simply replace all occurrences of the generator and the ternary map in the factorised expression for the ternary tree with the generator and the ternary map from the Fuss-Catalan normed (3)-magma of non-crossing partitions with blocks of even size: \def \scl {0.03}
\begin{gather*}
    t_{\ref{fcFamily:partitions}} \left(
        \epsilon_{\ref{fcFamily:partitions}},
        t_{\ref{fcFamily:partitions}} \left(
            \epsilon_{\ref{fcFamily:partitions}},
            \epsilon_{\ref{fcFamily:partitions}},
            t_{\ref{fcFamily:partitions}} \left(
                \epsilon_{\ref{fcFamily:partitions}},
                \epsilon_{\ref{fcFamily:partitions}},
                \epsilon_{\ref{fcFamily:partitions}}
            \right)
        \right),
        \epsilon_{\ref{fcFamily:partitions}}
    \right)
    =
    t_{\ref{fcFamily:partitions}} \left(
        \begin{tikzpicture}[baseline={([yshift=-1ex]current bounding box.center)}, scale = \scl]
            \draw (0,0) circle [radius = 5];
            \dummyNodes[-4][-4][4][4]
        \end{tikzpicture},
        t_{\ref{fcFamily:partitions}} \left(
            \begin{tikzpicture}[baseline={([yshift=-1ex]current bounding box.center)}, scale = \scl]
                \draw (0,0) circle [radius = 5];
                \dummyNodes[-4][-4][4][4]
            \end{tikzpicture},
            \begin{tikzpicture}[baseline={([yshift=-1ex]current bounding box.center)}, scale = \scl]
                \draw (0,0) circle [radius = 5];
                \dummyNodes[-4][-4][4][4]
            \end{tikzpicture},
            t_{\ref{fcFamily:partitions}} \left(
                \begin{tikzpicture}[baseline={([yshift=-1ex]current bounding box.center)}, scale = \scl]
                    \draw (0,0) circle [radius = 5];
                    \dummyNodes[-4][-4][4][4]
                \end{tikzpicture},
                \begin{tikzpicture}[baseline={([yshift=-1ex]current bounding box.center)}, scale = \scl]
                    \draw (0,0) circle [radius = 5];
                    \dummyNodes[-4][-4][4][4]
                \end{tikzpicture},
                \begin{tikzpicture}[baseline={([yshift=-1ex]current bounding box.center)}, scale = \scl]
                    \draw (0,0) circle [radius = 5];
                    \dummyNodes[-4][-4][4][4]
                \end{tikzpicture}
            \right)
        \right),
        \begin{tikzpicture}[baseline={([yshift=-1ex]current bounding box.center)}, scale = \scl]
            \draw (0,0) circle [radius = 5];
            \dummyNodes[-4][-4][4][4]
        \end{tikzpicture}
    \right) \\
    =
    t_{\ref{fcFamily:partitions}} \left(\!\!
        \begin{tikzpicture}[baseline={([yshift=-1ex]current bounding box.center)}, scale = \scl]
            \draw (0,0) circle [radius = 5];
            \dummyNodes[-4][-4][4][4]
        \end{tikzpicture}\!,
        t_{\ref{fcFamily:partitions}} \left(\!
            \begin{tikzpicture}[baseline={([yshift=-1ex]current bounding box.center)}, scale = \scl]
                \draw (0,0) circle [radius = 5];
                \dummyNodes[-4][-4][4][4]
            \end{tikzpicture},\!
            \begin{tikzpicture}[baseline={([yshift=-1ex]current bounding box.center)}, scale = \scl]
                \draw (0,0) circle [radius = 5];
                \dummyNodes[-4][-4][4][4]
            \end{tikzpicture},\! \def\scl{0.06}
            \begin{tikzpicture}[baseline={([yshift=-1ex]current bounding box.center)}, scale = \scl]
                \def\numPoints{2}
                \partitionSetup
                \draw (1) to (2);
            \end{tikzpicture}\!
        \right)\!, \def\scl{0.03}\!
        \begin{tikzpicture}[baseline={([yshift=-1ex]current bounding box.center)}, scale = \scl]
            \draw (0,0) circle [radius = 5];
            \dummyNodes[-4][-4][4][4]
        \end{tikzpicture}\!
    \right)
    = t_{\ref{fcFamily:partitions}} \left(\!
        \begin{tikzpicture}[baseline={([yshift=-1ex]current bounding box.center)}, scale = \scl]
            \draw (0,0) circle [radius = 5];
            \dummyNodes[-4][-4][4][4]
        \end{tikzpicture},\! \def\scl{0.07}
        \begin{tikzpicture}[baseline={([yshift=-1.5ex]current bounding box.center)}, scale = \scl]
            \def\numPoints{4}
            \partitionSetup
            \partitionDrawArc[1][2]
            \partitionDrawArc[3][4]
        \end{tikzpicture}\!,\! \def\scl{0.03}
        \begin{tikzpicture}[baseline={([yshift=-1ex]current bounding box.center)}, scale = \scl]
            \draw (0,0) circle [radius = 5];
            \dummyNodes[-4][-4][4][4]
        \end{tikzpicture}\!
    \right)
    = \def\scl{0.1}
    \begin{tikzpicture}[baseline={([yshift=-.5ex]current bounding box.center)}, scale = \scl]
        \def\numPoints{6}
        \partitionSetup
        \partitionDrawArc[6][1]
        \partitionDrawArc[2][3]
        \partitionDrawArc[4][5]
    \end{tikzpicture}
\end{gather*}

Thus we see that
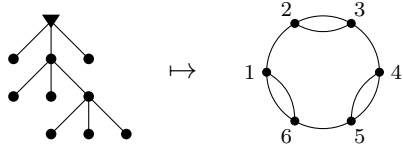
\begin{equation*} \def\scl{0.5}
    \begin{tikzpicture}[baseline={([yshift=-.5ex]current bounding box.center)}, scale = \scl]
        \draw (0,0) -- (-1,-1) (0,0) -- (0,-1) (0,0) -- (1,-1);
        \fill (-1,-1) \smlnd (0,-1) \smlnd (1,-1) \smlnd;
        \draw (0,-1) -- (-1,-2) (0,-1) -- (0,-2) (0,-1) -- (1,-2);
        \fill (-1,-2) \smlnd (0,-2) \smlnd (1,-2) \smlnd;
        \draw (1,-2) -- (0,-3) (1,-2) -- (1,-3) (1,-2) -- (2,-3);
        \fill (0,-3) \smlnd (1,-3) \smlnd (2,-3) \smlnd;
        \node (0,0) {\smlroot};
    \end{tikzpicture}
    \quad \mapsto \quad
    \begin{tikzpicture}[baseline={([yshift=-.5ex]current bounding box.center)}, scale = 0.15]
        \def\numPoints{6}
        \partitionSetupOne
        \partitionDrawArc[6][1]
        \partitionDrawArc[2][3]
        \partitionDrawArc[4][5]
    \end{tikzpicture}
\end{equation*}
under the universal bijection.

\section{Embedded bijections}
In order to use the universal bijection, we are first required to factorise an object. 
This procedure is simple in some families, but for others this procedure can be difficult or slow to do ``by hand''. 
For families where the factorisation process is not straightforward, it may be the case that there exists a simple bijection to a different family which is easy to factorise. 

For families which can be represented geometrically (in contrast to say `pure' sequence families) in certain cases we can give a geometric representation of the product structure. This representation will then additionally provide the factorisation required to apply a universal bijection. To this end we use an idea from category theory where the existence of a (binary) product object is defined  via a product diagram (and a universal mapping principle). Thus if some categorical object $A$ is a product of two other objects $B$ and $C$, the diagram
\begin{center}
   \begin{tikzpicture}
    \path(0,0) node(B)   {$B$} (2,0) node(A)   {$A$}  (4,0) node(C)   {$C$};
    \draw[->] (A) -- (B);
    \draw[->] (A) -- (C);
\end{tikzpicture} 
\end{center}
is used as part of the categorical definition\footnote{We don't provide the full categorical definition as this is not required here - only the diagram.}. 
This idea was used in \cite{BRAK} for Catalan objects which resulted in the embedding of complete binary trees into other Catalan objects. 
Here we consider Motzkin, Schr\"oder and Fuss-Catalan structures which give rise to unary-binary tree  and ternary tree embeddings.
The generalisation to other positive algebraic structures is clear.


\subsection{Embedded bijections for Motzkin and Schr\"oder geometric structures}
\label{sec:MotSchEmbedded}

We can use the (1,2)-magma structure of any Motzkin (respectively Schr\"oder) family to define an embedding of some other Motzkin (Schr\"oder) object inside any object in that family. This occurs in such a way that the recursive structure of these families is respected. 

In Section \ref{sebsec:12unibij}, we showed how unary-binary trees correspond to both Motzkin and Schr\"oder normed \mbox{(1,2)-magmas}. 
Assuming that unary-binary trees correspond to the free \mbox{(1,2)-magma} ($\cM$, $f$, $\star$) with the single generator $\epsilon$, we label each leaf and each internal node of a unary-binary tree as follows:
\def\scl{0.4}
\begin{equation*}
    \begin{tikzpicture}[baseline={([yshift=-2ex]current bounding box.base)}, scale = \scl]
        \fill (0,0) \smlnd;
        \draw (0,0) \un;
    \end{tikzpicture}
    \ \mapsto
    \begin{tikzpicture}[baseline={([yshift=-2ex]current bounding box.base)}, scale = \scl]
        \fill (0,0) \smlnd;
        \draw (0,0) \un;
        \node at ([shift={(-0.8,-0.05)}]0,0) {$f($};
        \node at ([shift={(0.6,-0.05)}]0,0) {$)$};
    \end{tikzpicture}
    ,
    \qquad \qquad
    \begin{tikzpicture}[baseline={([yshift=-2ex]current bounding box.base)}, scale = \scl]
        \fill (0,0) \smlnd;
        \draw (0,0) \bin;
    \end{tikzpicture}
    \, \mapsto
    \begin{tikzpicture}[baseline={([yshift=-2ex]current bounding box.base)}, scale = \scl]
        \fill (0,0) \smlnd;
        \draw (0,0) \bin;
        \node at ([shift={(-0.8,0.1)}]0,0) {$($};
        \node at ([shift={(0.8,0.1)}]0,0) {$)$};
        \node at ([shift={(0,-0.7)}]0,0) {$\star$};
    \end{tikzpicture} \ .
\end{equation*}
The labels given to an internal node are determined by whether that node has out-degree 1 or 2. 
After labelling in this way, counter-clockwise traversal of the tree gives its factorisation.
For example, consider the following unary-binary tree which has been labelled:
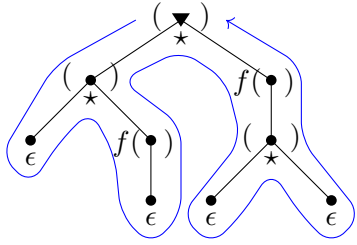
\def\scl{0.8}
\begin{center}
    \begin{tikzpicture}[scale = \scl]
        \draw (0,0) \widebin;
        \draw (-1.5,-1) \bin (1.5,-1) \un;
        \draw (-0.5,-2) \un (1.5,-2) \bin;
        \fill (0,0) \smlwidebinnd (-1.5,-1) \smlbinnd (1.5,-1) \smlunnd (-0.5,-2) \smlunnd (1.5,-2) \smlbinnd;
        \node at (0,0) (root) {\smlroot};
        \binlabel[0][0]
        \binlabel[-1.5][-1]
        \unlabel[1.5][-1]
        \unlabel[-0.5][-2]
        \binlabel[1.5][-2]
        \leaflabel[-2.5][-2]
        \leaflabel[-0.5][-3]
        \leaflabel[0.5][-3]
        \leaflabel[2.5][-3]
        \draw[blue, rounded corners = 7, ->]
            (-0.75,0) -- (-2.1,-0.75) -- (-3,-2.25) -- (-2.5,-2.75) -- (-2,-2) -- (-1.5,-1.5) -- (-1.1,-2.25) -- (-1,-3.25) -- (-0.5,-3.75) -- (0,-3.25) -- (0,-1.75) -- (-1,-1) -- (0,-0.5) -- (0.75,-1) -- (1,-2) -- (0,-3.25) -- (0.5,-3.75) -- (1,-3.25) -- (1.5,-2.5) -- (2,-3.25) -- (2.5,-3.75) -- (3,-3.25) -- (2,-2) -- (2,-0.75) -- (0.75,0);
    \end{tikzpicture}
\end{center}
By traversing this tree counter-clockwise we find that its factorisation is
\begin{equation*}
    \left( \left( \epsilon \star f \left( \epsilon \right) \right) \star f \left( \left( \epsilon \star \epsilon \right) \right) \right).
\end{equation*}

Since we are able to determine the factorisation of a unary-binary tree by simple tree traversal, we focus on embedding unary-binary trees into other geometric families (both Motzkin and Schr\"oder). 

In this section we assume that all free (1,2)-magmas have a single generator and are geometric (meaning they are defined pictorially in some way rather than as sequences). Given that a family is geometric, the details of its (1,2)-magma structure give us the following:
\begin{enumerate}[label=(\roman*)]
    \item We know which part of the geometry is associated with the generator. We will call this the \textbf{generator geometry}.
    \item From the definition of the unary map, we know which part of the geometry is added each time the unary map is applied. We will call this the \textbf{unary map geometry}.
    \item From the defintion of the binary map, we know which part of the geometry is added each time the binary map is applied. We will call this the \textbf{binary map geometry}.
\end{enumerate}

For example, we know that for unary-binary trees, the generators correspond to leaves, and thus the generator geometry is a leaf. From the definitions of the two maps,
\def\scl{0.12}
\begin{equation*}
    f \left(
        \begin{tikzpicture}[baseline={([yshift=-1ex]current bounding box.center)}, scale = \scl]
            \draw[fill = \colOne, opacity = \opac] (0,0) -- (-5,-10) -- (5,-10) -- (0,0);
            \node at (0,0) {\smlroot};
            \node at (0,-6) {\footnotesize $t$};
        \dummyNodes[0][1][0][-11]
        \end{tikzpicture}
    \right)
    =
    \begin{tikzpicture}[baseline={([yshift=-1ex]current bounding box.center)}, scale = \scl]
        \draw[fill = \colOne, opacity = \opac] (0,0) -- (-5,-10) -- (5,-10) -- (0,0);
        \draw (0,0) -- (0,3);
        \node at (0,3) {\smlroot};
        \draw[fill] (0,0)\smlnd;
        \node at (0,-6) {\footnotesize $t$};
        \node at (0,6) {\footnotesize $a$};
        \node at (1,1.5) {\footnotesize $e$};
    \end{tikzpicture}
    \qquad \qquad
    \begin{tikzpicture}[baseline={([yshift=-1ex]current bounding box.center)}, scale = \scl]
        \draw[fill = \colOne, opacity = \opac] (0,0) -- (-5,-10) -- (5,-10) -- (0,0);
        \node at (0,0) {\smlroot};
        \node at (0,-6) {\footnotesize $t_1$};
    \end{tikzpicture}
    \ \star \ 
    \begin{tikzpicture}[baseline={([yshift=-1ex]current bounding box.center)}, scale = \scl]
        \draw[fill = \colTwo, opacity = \opac] (0,0) -- (-5,-10) -- (5,-10) -- (0,0);
        \node at (0,0) {\smlroot};
        \node at (0,-6) {\footnotesize $t_2$};
    \end{tikzpicture}
    =
    \begin{tikzpicture}[baseline={([yshift=-1ex]current bounding box.center)}, scale = \scl]
        \draw[fill = \colOne, opacity = \opac] (0,0) -- (-5,-10) -- (5,-10) -- (0,0);
        \draw[fill = \colTwo, opacity = \opac] (12,0) -- (7,-10) -- (17,-10) -- (12,0);
        \draw (0,0) -- (6,5) -- (12,0);
        \node at (6,5) {\smlroot};
        \fill (0,0) \smlnd (12,0) \smlnd;
        \node at (0,-6) {\footnotesize $t_1$};
        \node at (12,-6) {\footnotesize $t_2$};
        \node at (6,8) {\footnotesize $b$};
        \node at (2,3.5) {\footnotesize $e_1$};
        \node at (10,3.5) {\footnotesize $e_2$};
    \end{tikzpicture}
\end{equation*}
we can see that the unary map geometry is the new root node $a$ and the edge $e$. The binary map geometry is the new root node $b$ and the two edges $e_1$ and $e_2$.

For any geometric family, it is clear that there may be many different ways of drawing the same objects, with each representation differing only slightly and these being trivially in bijection. 
In order to make it clear exactly which parts of an object's geometry correspond to the generator and which arise from the two maps, we will choose a canonical way of drawing an object in a given family. 
In particular, we choose a way of drawing our objects so that the points of any object can be partitioned into disjoint subsets. 
Thus for a (1,2)-magma we require the objects to be drawn such that there are  three disjoint subsets: one subset corresponding to the generator geometry, one to the unary map geometry and one to the binary map geometry. We call this canonical form the \textbf{(1,2)-magma form}. 

\def\scl{0.6}
Let ($\cM$, $f$, $\star$) be a free \mbox{(1,2)-magma} with a single generator. For $m \in \cM$, define the following sets:
\begin{itemize}
    \item Let $\mathbb{G}_{\cM}(m)$ be the set of all elements of generator geometry.
    \item Let $\mathbb{U}_{\cM}(m)$ be the set of all elements of unary map geometry.
    \item Let $\mathbb{B}_{\cM}(m)$ be the set of all elements of binary map geometry.
\end{itemize}
For example, if we choose to draw unary-binary trees as follows, with generator $\begin{tikzpicture}[baseline={([yshift=-.5ex]current bounding box.center)}, scale = \scl]
        \fill (0,0) \smlnd;
    \end{tikzpicture}$,
unary map geometry drawn in blue and binary map geometry drawn in orange, then it is clear that any unary-binary tree can be partitioned into these sets, as shown in the following example:
\begin{equation*}
    \begin{tikzpicture}[baseline={([yshift=-1ex]current bounding box.center)}, scale = \scl]
        \draw (0,0) \widebin;
        \draw (-1.5,-1) \bin;
        \draw (1.5,-1) \un;
        \draw (-0.5,-2) \un;
        \fill (0,0) \smlwidebinnd (-1.5,-1) \smlbinnd (1.5,-1) \smlunnd (-0.5,-2) \smlunnd;
        \node at (0,0) (root) {\smlroot};
        \ndlabel[-0.2][0.1][1]
        \ndlabel[-1.7][-1][2]
        \ndlabel[-2.7][-2][3]
        \ndlabel[-0.7][-2][4]
        \ndlabel[-0.7][-3][5]
        \ndlabel[1.4][-1.2][6]
        \ndlabel[1.4][-2.2][7]
    \end{tikzpicture}
    \qquad \leadsto \qquad
    \begin{tikzpicture}[baseline={([yshift=-1ex]current bounding box.center)}, scale = \scl]
        \node[anchor = mid] at (0,0) {};
        \draw[\colOne, ultra thick] (0,0) \widebin;
        \draw[\colOne, ultra thick] (-1.5,-1) \bin;
        \draw[\colTwo, ultra thick] (1.5,-1) \un;
        \draw[\colTwo, ultra thick] (-0.5,-2) \un;
        \fill[\colOne, ultra thick] (0,0) \smlnd (-1.5,-1) \smlnd (1.5,-2) \smlnd;
        \fill[\colTwo, ultra thick] (-0.5,-2) \smlnd (1.5,-1) \smlnd;
        \fill (-2.5,-2) \smlnd (-0.5,-3) \smlnd (1.5,-2) \smlnd;
    \end{tikzpicture}
\end{equation*}
This gives the partition
\def\scl{0.4}
\begin{equation*}
    \mathbb{G}_{\cM}(m) = \left\{\!\! 
        \begin{tikzpicture}[baseline={([yshift=-1ex]current bounding box.center)}, scale = \scl]
            \fill (0,0) \smlnd;
            \ndlabel[-0.2][0][3]
        \end{tikzpicture},
        \begin{tikzpicture}[baseline={([yshift=-1ex]current bounding box.center)}, scale = \scl]
            \fill (0,0) \smlnd;
            \ndlabel[-0.2][0][5]
        \end{tikzpicture},
        \begin{tikzpicture}[baseline={([yshift=-1ex]current bounding box.center)}, scale = \scl]
            \fill (0,0) \smlnd;
            \ndlabel[-0.2][0][7]
        \end{tikzpicture}
    \right\}, \quad 
    \mathbb{U}_{\cM}(m) = \left\{
        \raisebox{-5pt}{\!\!\!\!
        \begin{tikzpicture}[baseline={([yshift=-2ex]current bounding box.center)}, scale = \scl]
            \fill (0,0) \smlnd;
            \draw (0,0) \un;
            \ndlabel[-0.2][0][4]
        \end{tikzpicture},
        \begin{tikzpicture}[baseline={([yshift=-2ex]current bounding box.center)}, scale = \scl]
            \fill (0,0) \smlnd;
            \draw (0,0) \un;
            \ndlabel[-0.2][0][6]
        \end{tikzpicture}}
    \right\}, \quad
    \mathbb{B}_{\cM}(m)  = \left\{
        \raisebox{-5pt}{
        \begin{tikzpicture}[baseline={([yshift=-2.5ex]current bounding box.base)}, scale = \scl]
            \fill (0,0) \smlnd;
            \draw (0,0) \bin;
            \ndlabel[-0.2][0][1]
        \end{tikzpicture}\spacecomma
        \begin{tikzpicture}[baseline={([yshift=-2.5ex]current bounding box.base)}, scale = \scl]
            \fill (0,0) \smlnd;
            \draw (0,0) \bin;
            \ndlabel[-0.2][0][2]
        \end{tikzpicture}}
    \right\}.
\end{equation*}

For each piece of generator geometry we mark a point, called the \textbf{leaf}. 
For some families, the generator is the empty object. 
In these cases, we enforce that the generator is associated with some geometry by trivially modifying how the object is drawn. 
This ensures that we are able to mark this point. 
For each piece of unary map geometry and each piece of binary map geometry, we also mark a point. In both cases, this marked point is called the \textbf{root}. 

For example, for the family of Motzkin paths \motFamRef{motzkinFamily:motzkinPaths} the generator is the empty path so we choose to represent it by a single node, which is also our marked leaf:
\begin{equation*}
    \epsilon = 
    \begin{tikzpicture}[baseline={([yshift=-.5ex]current bounding box.center)}, scale = \scl]
        \fill (0,0) \nd;
    \end{tikzpicture}
\end{equation*}
We mark the root in the definition of the unary map, noting that this root is part of the unary map geometry:
\def\scl{0.3}
\begin{equation*}
    f \left(
        \begin{tikzpicture}[baseline=1.5ex, scale = \scl]
            \draw[fill = \colOne, opacity = \opac] (6,0) -- (0,0) to [bend left = 90, looseness = 1.5] (6,0);
            \node[anchor = mid] at (3,1.3) {\footnotesize $p$};
        \end{tikzpicture}
    \right)
    =
    \begin{tikzpicture}[baseline=1.5ex, scale = \scl]
        \draw[fill = orange!90!black, opacity = \opac] (6,0) -- (0,0) to [bend left = 90, looseness = 1.5] (6,0);
        \node[anchor = mid] at (3,1.3) {\footnotesize $p$};
        \draw[thick] (6,0) -- (8,0);
        \fill (8,0) \nd;
        \node at ([shift = {(1.5,0)}]8,0) {\footnotesize root};
    \end{tikzpicture}
\end{equation*}

Similarly, we mark the root in the definition of the binary map, again noting that this root is part of the binary map geometry:
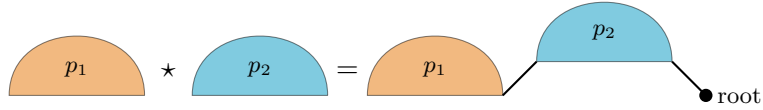
\begin{equation*}
    \begin{tikzpicture}[baseline=1.5ex, scale = \scl]
        \draw[fill = \colOne, opacity = \opac] (6,0) -- (0,0) to [bend left = 90, looseness = 1.5] (6,0);
        \node[anchor = mid] at (3,1.3) {\footnotesize $p_1$};
    \end{tikzpicture}
    \ \star \ 
    \begin{tikzpicture}[baseline=1.5ex, scale = \scl]
        \draw[fill = \colTwo, opacity = \opac] (6,0) -- (0,0) to [bend left = 90, looseness = 1.5] (6,0);
        \node[anchor = mid] at (3,1.3) {\footnotesize $p_2$};
    \end{tikzpicture}
    =
    \begin{tikzpicture}[baseline=1.5ex, scale = \scl]
        \draw[fill = \colOne, opacity = \opac] (6,0) -- (0,0) to [bend left = 90, looseness = 1.5] (6,0);
        \node[anchor = mid] at (3,1.3) {\footnotesize $p_1$};
        \draw[thick] (6,0) -- (7.5,1.5);
        \draw[fill = \colTwo, opacity = \opac] (13.5,1.5) -- (7.5,1.5) to [bend left = 90, looseness = 1.5] (13.5,1.5);
        \node at (10.5,2.8) {\footnotesize $p_2$};
        \draw[thick] (13.5,1.5) -- (15,0);
        \fill(15,0) \nd;
        \node at ([shift = {(1.5,0)}]15,0) {\footnotesize root};
    \end{tikzpicture}
\end{equation*}

Now, for any object $m$ from a (1,2)-magma ($\cM$, $f$, $\star$) such that $m = f(m_0)$, let the \textbf{subroot} of $m$ be the root of $m_0$ if it exists (that is, if $m_0 \neq \epsilon$) and the leaf of $m_0$ if $m_0 = \epsilon$. If $m = m_1 \star m_2$, then define the following:
\begin{itemize}
    \item Let the \textbf{left subroot} of $m$ be the root of $m_1$ if $m_1 \neq \epsilon$ and the leaf of $m_1$ if $m_1 = \epsilon$.
    \item Let the \textbf{right subroot} of $m$ be the root of $m_2$ if $m_2 \neq \epsilon$ and the leaf of $m_2$ if $m_2 = \epsilon$.
\end{itemize}
For example, if we are considering the (1,2)-magma of unary-binary trees (considered as either a Motzkin family or a Schr\"oder family), then the tree
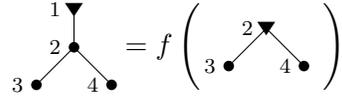
\def\scl{0.5}
\begin{equation*}
     \begin{tikzpicture}[baseline={([yshift=-.5ex]current bounding box.center)}, scale = \scl]
        \draw (0,0) \un (0,-1) \bin;
        \fill (0,0) \smlunnd (0,-1) \smlbinnd;
        \node at (0,0) {\smlroot};
        \ndlabel[-0.2][0][1]
        \ndlabel[-0.2][-1][2]
        \ndlabel[-1.2][-2][3]
        \ndlabel[0.8][-2][4]
    \end{tikzpicture}
    =
    f \left(
        \begin{tikzpicture}[baseline={([yshift=-.5ex]current bounding box.center)}, scale = \scl]
            \clip (-1.75,-1.5) rectangle (1.5,0.5);
            \draw (0,0) \bin;
            \fill (0,0) \smlbinnd;
            \node at (0,0) {\smlroot};
            \ndlabel[-0.2][0][2]
        \ndlabel[-1.2][-1][3]
        \ndlabel[0.8][-1][4]
        \end{tikzpicture}
    \right)
\end{equation*}
has subroot node 2, while for the tree
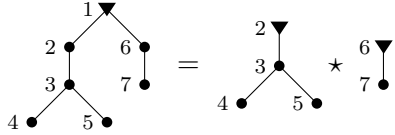
\begin{equation*}
    \begin{tikzpicture}[baseline={([yshift=-.5ex]current bounding box.center)}, scale = \scl]
        \draw (0,0) \bin (-1,-1) \un (-1,-2) \bin (1,-1) \un;
        \fill (0,0) \smlbinnd (-1,-1) \smlunnd (-1,-2) \smlbinnd (1,-1) \smlunnd;
        \node at (0,0) {\smlroot};
        \ndlabel[-0.2][0][1]
        \ndlabel[-1.2][-1][2]
        \ndlabel[-1.2][-2][3]
        \ndlabel[-2.2][-3][4]
        \ndlabel[-0.2][-3][5]
        \ndlabel[0.8][-1][6]
        \ndlabel[0.8][-2][7]
    \end{tikzpicture} \ \
    =
    \begin{tikzpicture}[baseline={([yshift=-.5ex]current bounding box.center)}, scale = \scl]
        \draw (-1,-1) \un (-1,-2) \bin;
        \fill (-1,-1) \smlunnd (-1,-2) \smlbinnd;
        \node at (-1,-1) {\smlroot};
        \ndlabel[-1.2][-1][2]
        \ndlabel[-1.2][-2][3]
        \ndlabel[-2.2][-3][4]
        \ndlabel[-0.2][-3][5]
    \end{tikzpicture}
    \star
    \begin{tikzpicture}[baseline={([yshift=-.5ex]current bounding box.center)}, scale = \scl]
        \draw (1,-1) \un;
        \fill (1,-1) \smlunnd;
        \node at (1,-1) {\smlroot};
        \ndlabel[0.8][-1][6]
        \ndlabel[0.8][-2][7]
    \end{tikzpicture}
\end{equation*}
the left subroot is node 2 and the right subroot is node 6.

We now define our embedded bijections via the following recursive procedure of embedding pairs $(P, \rightarrow)$ and triples $(\overset{L}{\leftarrow}, P, \overset{R}{\rightarrow})$ into some Motzkin or Schr\"oder object $m$. This is done as follows:
\begin{itemize}
    \item If $m = f(m_0)$, then attach $P$ to the root of the unary map $f$. The arrow $\rightarrow$ in the pair $(P, \rightarrow)$ points from $P$ to the subroot of $m$.
    \item If $m = m_1 \star m_2$, then attach $P$ to the root of the binary map $\star$. The left arrow $\overset{L}{\leftarrow}$ of the triple $(\overset{L}{\leftarrow}, P, \overset{R}{\rightarrow})$ points from $P$ to the left subroot and the right arrow $\overset{R}{\rightarrow}$ points from $P$ to the right subroot.
\end{itemize}

We can represent this embedding schematically by drawing the root (of either the unary or binary map) and the subroots in the definition of the maps. For example, for unary-binary trees we have the following:
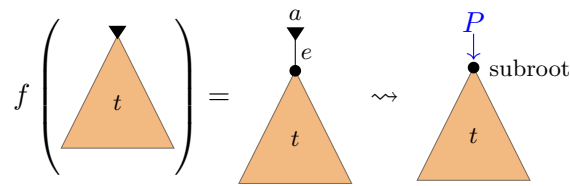
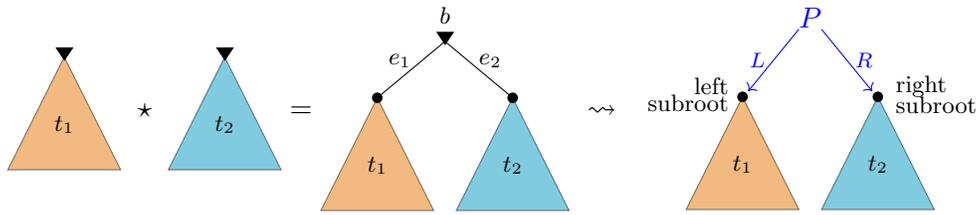
\def\scl{0.15}
\begin{align*}
    f \left(
        \begin{tikzpicture}[baseline={([yshift=-1ex]current bounding box.center)}, scale = \scl]
            \draw[fill = \colOne, opacity = \opac] (0,0) -- (-5,-10) -- (5,-10) -- (0,0);
            \node at (0,0) {\smlroot};
            \node at (0,-6) {\footnotesize $t$};
        \dummyNodes[0][1][0][-11]
        \end{tikzpicture}
    \right)
    & =
    \begin{tikzpicture}[baseline={([yshift=-1ex]current bounding box.center)}, scale = \scl]
        \draw[fill = \colOne, opacity = \opac] (0,0) -- (-5,-10) -- (5,-10) -- (0,0);
        \draw (0,0) -- (0,3);
        \node at (0,3) {\smlroot};
        \draw[fill] (0,0)\smlnd;
        \node at (0,-6) {\footnotesize $t$};
        \node at (0,5) {\footnotesize $a$};
        \node at (1,1.5) {\footnotesize $e$};
    \end{tikzpicture}
    \ \leadsto \
    \begin{tikzpicture}[baseline={([yshift=-1ex]current bounding box.center)}, scale = \scl]
        \draw[fill = \colOne, opacity = \opac] (0,0) -- (-5,-10) -- (5,-10) -- (0,0);
        \P[0][4]
        \draw[fill] (0,0)\smlnd;
        \node at (0,-6) {\footnotesize $t$};
        \node at (5,0) {\footnotesize subroot};
        \path[blue, ->]
            (0,3) edge (0,0.7);
    \end{tikzpicture}
    \\
    \begin{tikzpicture}[baseline={([yshift=-1ex]current bounding box.center)}, scale = \scl]
        \draw[fill = \colOne, opacity = \opac] (0,0) -- (-5,-10) -- (5,-10) -- (0,0);
        \node at (0,0) {\smlroot};
        \node at (0,-6) {\footnotesize $t_1$};
    \end{tikzpicture}
    \ \star \ 
    \begin{tikzpicture}[baseline={([yshift=-1ex]current bounding box.center)}, scale = \scl]
        \draw[fill = \colTwo, opacity = \opac] (0,0) -- (-5,-10) -- (5,-10) -- (0,0);
        \node at (0,0) {\smlroot};
        \node at (0,-6) {\footnotesize $t_2$};
    \end{tikzpicture}
    & =
    \begin{tikzpicture}[baseline={([yshift=-1ex]current bounding box.center)}, scale = \scl]
        \draw[fill = \colOne, opacity = \opac] (0,0) -- (-5,-10) -- (5,-10) -- (0,0);
        \draw[fill = \colTwo, opacity = \opac] (12,0) -- (7,-10) -- (17,-10) -- (12,0);
        \draw (0,0) -- (6,5) -- (12,0);
        \node at (6,5) {\smlroot};
        \fill (0,0) \smlnd (12,0) \smlnd;
        \node at (0,-6) {\footnotesize $t_1$};
        \node at (12,-6) {\footnotesize $t_2$};
        \node at (6,7.3) {\footnotesize $b$};
        \node at (2,3.3) {\footnotesize $e_1$};
        \node at (10,3.3) {\footnotesize $e_2$};
    \end{tikzpicture}
    \ \leadsto \
    \begin{tikzpicture}[baseline={([yshift=-1ex]current bounding box.center)}, scale = \scl]
        \draw[fill = \colOne, opacity = \opac] (0,0) -- (-5,-10) -- (5,-10) -- (0,0);
        \draw[fill = \colTwo, opacity = \opac] (12,0) -- (7,-10) -- (17,-10) -- (12,0);
        \fill (0,0) \smlnd (12,0) \smlnd;
        \node at (0,-6) {\footnotesize $t_1$};
        \node at (12,-6) {\footnotesize $t_2$};
        \P[6][7]
        \node at (-5,0) {\footnotesize $\begin{aligned} \text{left} & \\[-.9em] \text{subroot} &\end{aligned}$};
        \node at (17,0) {\footnotesize $\begin{aligned} & \text{right} \\[-.8em] & \text{subroot}\end{aligned}$};
        \path[blue, ->]
            (5,6) edge node[style = {inner sep = 3pt}, left] {\scriptsize $L$} (0.5,0.5)
            (7,6) edge node[style = {inner sep = 3pt}, right] {\scriptsize $R$} (11.5,0.5);
    \end{tikzpicture}
\end{align*}

Repeating the embedding recursively gives, for example,
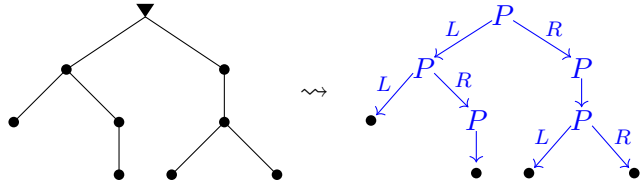
\def\scl{0.7}
\begin{equation*}
    \begin{tikzpicture}[baseline={([yshift=-1ex]current bounding box.center)}, scale = \scl]
        \draw (0,0) \widebin;
        \draw (-1.5,-1) \bin;
        \draw (1.5,-1) \un;
        \draw (-0.5,-2) \un;
        \draw (1.5,-2) \bin;
        \fill (0,0) \smlwidebinnd (-1.5,-1) \smlbinnd (1.5,-1) \smlunnd (-0.5,-2) \smlunnd (1.5,-2) \smlbinnd;
        \node[anchor = mid] at (0,0) (root) {\smlroot};
    \end{tikzpicture}
    \ \leadsto \quad
    \begin{tikzpicture}[baseline={([yshift=-1ex]current bounding box.center)}, scale = \scl]
        \node[anchor = mid] at (0,0) {};
        \fill (-2.5,-2) \smlnd (-0.5,-3) \smlnd (0.5,-3) \smlnd (2.5,-3) \smlnd;
        \P[0][0]
        \P[-1.5][-1]
        \P[1.5][-1]
        \P[-0.5][-2]
        \P[1.5][-2]
        \path[blue, ->]
            (-0.2,-0.1) edge node[style = {inner sep = 1pt}, above left] {\scriptsize $L$} (-1.3,-0.8)
            (0.2,-0.1) edge node[style = {inner sep = 1pt}, above right] {\scriptsize $R$} (1.3,-0.8)
            (-1.7,-1.1) edge node[style = {inner sep = 1pt}, above left] {\scriptsize $L$} (-2.4,-1.9)
            (-1.3,-1.1) edge node[style = {inner sep = 1pt}, above right] {\scriptsize $R$} (-0.65,-1.75)
            (1.5,-1.2) edge (1.5,-1.75)
            (-0.5,-2.2) edge (-0.5,-2.8)
            (1.3,-2.1) edge node[style = {inner sep = 1pt}, above left] {\scriptsize $L$} (0.6,-2.9)
            (1.7,-2.1) edge node[style = {inner sep = 1pt}, above right] {\scriptsize $R$} (2.4,-2.9);
    \end{tikzpicture}
\end{equation*}
It is clear that this embedded representation is trivially an alternative way of drawing a unary-binary tree. Thus by recursively embedding the pair \mbox{$(P, \rightarrow)$} and triple \mbox{$(\overset{L}{\leftarrow}, P, \overset{R}{\rightarrow})$} we have effectively embedded a unary-binary tree inside our object. Note that the unary-binary tree which is embedded is precisely the unary-binary tree which the object maps to under the universal bijection. Thus the object and the embedded unary-binary tree have the same factorisation. This gives us a simple way of decomposing the object, since we can simply traverse the embedded unary-binary tree as illustrated previously. 

We now illustrate how these embedded bijections work with a number of examples, working with both Motzkin normed (1,2)-magmas and Schr\"oder normed (1,2)-magmas.

\subsubsection{Motzkin embedded bijections}

\paragraph{\motFamRef{motzkinFamily:motzkinPaths}: Motzkin paths} \def \scl {0.3}

Consider the family of Motzkin paths. The generator is the empty path which we represent by a single node: $\epsilon = 
\begin{tikzpicture}[baseline={([yshift=-.6ex]current bounding box.center)}, scale = \scl]
    \fill (0,0) \nd;
\end{tikzpicture}$\,. The unary and binary maps are as follows, and we have labelled these diagrams with the roots and details of the embedding:
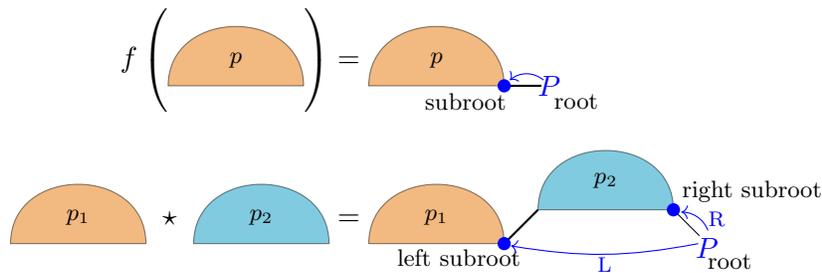
\begin{align*}
    f \left(
        \begin{tikzpicture}[baseline=1.5ex, scale = \scl]
            \draw[fill = \colOne, opacity = \opac] (6,0) -- (0,0) to [bend left = 90, looseness = 1.5] (6,0);
            \node[anchor = mid] at (3,1.3) {\footnotesize $p$};
        \end{tikzpicture}
    \right)
    & =
    \begin{tikzpicture}[baseline=1.5ex, scale = \scl]
        \draw[fill = orange!90!black, opacity = \opac] (6,0) -- (0,0) to [bend left = 90, looseness = 1.5] (6,0);
        \node[anchor = mid] at (3,1.3) {\footnotesize $p$};
        \draw[thick] (6,0) -- (8,0);
        \fill[blue] (6,0) \nd;
        \node at ([shift={(-1.7,-0.7)}]6,0) {\footnotesize subroot};
        \node at ([shift={(1.2,-0.7)}]8,0) {\footnotesize root};
        \draw[blue, ->] (7.75,0.25) to [bend right = 40] (6.25,0.25);
        \fill[white] (8,0) circle (10pt);
        \P[8][0]
    \end{tikzpicture} \\
    \begin{tikzpicture}[baseline=1.5ex, scale = \scl]
        \draw[fill = \colOne, opacity = \opac] (6,0) -- (0,0) to [bend left = 90, looseness = 1.5] (6,0);
        \node[anchor = mid] at (3,1.3) {\footnotesize $p_1$};
    \end{tikzpicture}
    \ \star \ 
    \begin{tikzpicture}[baseline=1.5ex, scale = \scl]
        \draw[fill = \colTwo, opacity = \opac] (6,0) -- (0,0) to [bend left = 90, looseness = 1.5] (6,0);
        \node[anchor = mid] at (3,1.3) {\footnotesize $p_2$};
    \end{tikzpicture}
    & =
    \begin{tikzpicture}[baseline=1.5ex, scale = \scl]
        \draw[fill = \colOne, opacity = \opac] (6,0) -- (0,0) to [bend left = 90, looseness = 1.5] (6,0);
        \node[anchor = mid] at (3,1.3) {\footnotesize $p_1$};
        \draw[thick] (6,0) -- (7.5,1.5);
        \draw[fill = \colTwo, opacity = \opac] (13.5,1.5) -- (7.5,1.5) to [bend left = 90, looseness = 1.5] (13.5,1.5);
        \node at (10.5,2.8) {\footnotesize $p_2$};
        \draw (13.5,1.5) -- (15,0);
        \fill[blue] (6,0) \nd (13.5,1.5) \nd;
        \node at ([shift={(-2,-0.6)}]6,0) {\footnotesize left subroot};
        \node at ([shift={(3.5,0.8)}]13.5,1.5) {\footnotesize right subroot};
        \node at ([shift={(1,-0.8)}]15,0) {\footnotesize root};
        \path[blue, ->] 
            (14.6,0)  edge [bend left = 10]  node[style = {inner sep = 1pt}, below]  {\scriptsize L} (6.3,-0.1)
            (15.1,0.3) edge [bend right = 40] node[style = {inner sep = 3pt}, right] {\scriptsize R} (13.85,1.4);
        \fill[white] (15,0.2) circle (11pt);
        \P[15][-0.2]
    \end{tikzpicture}
\end{align*}

From this, we see that we can embed a unary-binary tree inside a Motzkin path as illustrated in the following example:
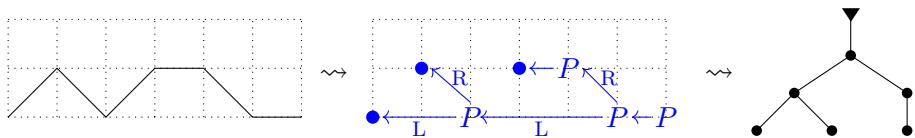
\def\scl {0.65}
\begin{equation*}
    \begin{tikzpicture}[baseline=3ex, scale = \scl]
        \grd[6][2]
        \draw (0,0) \motUp \motDn \motUp \motAc \motDn \motAc;
        \node[anchor = mid] at (1,1) {};
    \end{tikzpicture}
    \ \leadsto \
    \begin{tikzpicture}[baseline=3ex, scale = \scl]
        \grd[6][2]
        \node[anchor = mid] at (1,1) {};
        \fill[blue] (0,0) \nd (1,1) \nd (3,1) \nd;
        \P[2][0]
        \P[4][1]
        \P[5][0]
        \P[6][0]
        \path[blue, ->]
            (5.7,0) edge (5.3,0)
            (5,0.3) edge node[style = {inner sep = 0pt}, above right] {\scriptsize R} (4.3,1)
            (4.7,0) edge node[style = {inner sep = 1.5pt}, below] {\scriptsize L} (2.2,0)
            (3.7,1) edge (3.2,1)
            (2,0.3) edge node[style = {inner sep = 0pt}, above right] {\scriptsize R} (1.2,1)
            (1.7,0) edge node[style = {inner sep = 1.5pt}, below] {\scriptsize L} (0.2,0);
    \end{tikzpicture}
    \ \leadsto \
    \def\scl{0.5}
    \begin{tikzpicture}[baseline={([yshift=-1ex]current bounding box.center)}, scale = \scl]
        \draw (0,0) \un;
        \draw (0,-1) \widebin;
        \draw (-1.5,-2) \bin;
        \draw (1.5,-2) \un;
        \fill (0,0) \smlunnd (0,-1) \smlwidebinnd (-1.5,-2) \smlbinnd (1.5,-2) \smlunnd;
        \node at (0,0) (root) {\smlroot};
    \end{tikzpicture}
\end{equation*}

This is an important example as it demonstrates why we must keep track of which is the left subroot and which is the right. If we had have omitted the labels on the arrows, then it would seem that we have embedded a different tree (namely the tree which is obtained by reflecting the embedded tree across a vertical line travelling through the root).

\paragraph{\motFamRef{motzkinFamily:chords}: Non-intersecting chords} \def \scl {0.08}

\def\numPoints {11}  
\def\numSpaces {12} 
Consider the Motzkin family of non-intersecting chords joining $n$ points on a circle. Note that on each diagram, we place a mark at the top of the circle to fix its orientation. The generator is represented by:
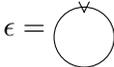
\begin{equation*}
    \epsilon =

\end{equation*}
We need to mark a point on this generator to be the leaf, so we take it to be the marked point shown in blue below:
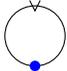
\begin{equation*}
    %
\end{equation*}
Note, the marked (in blue) point is \emph{not} a node of the circle used for connecting chords but, as seen below, is a point between such nodes.
We take the convention that the $n$ points on the circle which we are connecting are drawn in black. The two maps are as follows, with the roots, subroots and embedded arrows all labelled:
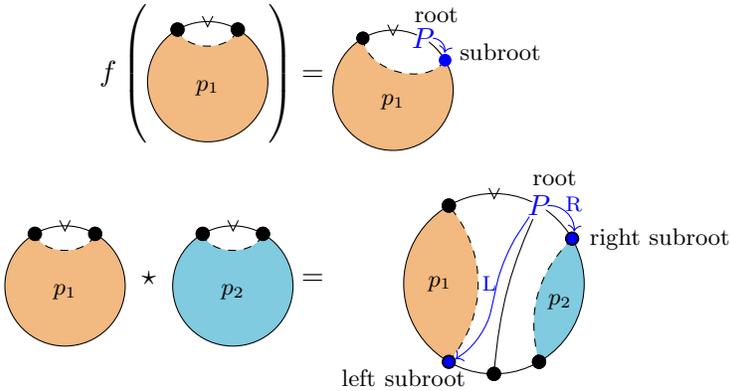
\begin{align*}
    f \left(
    %
\end{align*}

From these, we can see that we can embed a unary-binary tree inside a non-intersecting chord diagram as follows:
\def\scl{0.15}
\def\numPoints {7}  
\def\numSpaces {8} 
\begin{equation*}
    %
\end{equation*}

\subsubsection{Schr\"oder embedded bijections}

\paragraph{\schFamRef{schroderFamily:schroderPaths}: Schr\"oder paths} \def \scl {0.3}

To embed a Schr\"oder unary-binary tree inside a Schr\"oder path. The generator for Schr\"oder paths is taken to be a single vertex: $\epsilon = %
$\,. The embedding of the pairs \mbox{$(P, \rightarrow)$} and triples \mbox{$(\overset{L}{\leftarrow}, P, \overset{R}{\rightarrow})$} is shown in the following diagram:
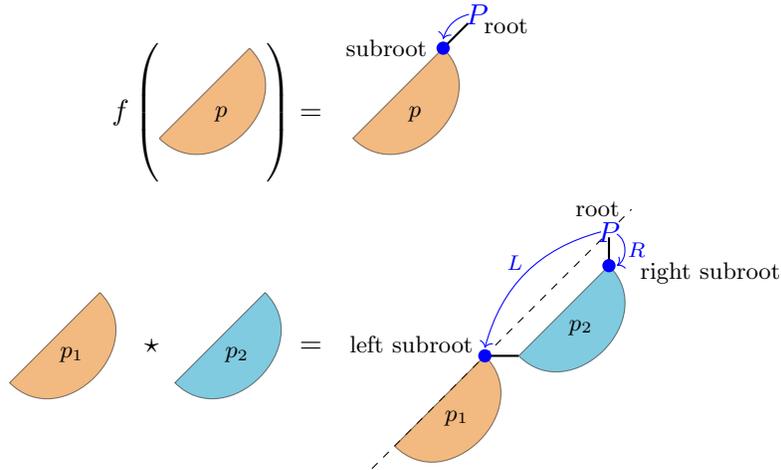
\begin{align*}
    f 
    \left(
    %
\end{align*}
Thus we see that we can embed a unary-binary tree inside a Schr\"oder path as shown in the following example:
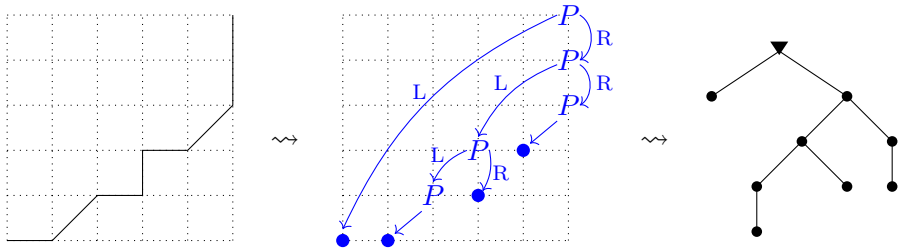
\def\scl{0.6}
\begin{equation*}
    %
\end{equation*}

\paragraph{\schFamRef{schroderFamily:rectangulations}: Rectangulations} \def \scl {0.5}

Finally, we will show how to embed a unary-binary tree inside objects from the Schr\"oder family of rectangulations. The generator is the following, with the marked point shown being the leaf:

\begin{equation*}
    \epsilon = 
    %
\end{equation*}

The unary and binary maps have the following geometry, where LS and RS are used to denote the left subroot and the right subroot respectively:
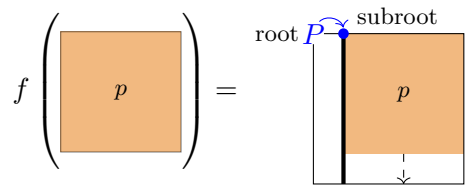
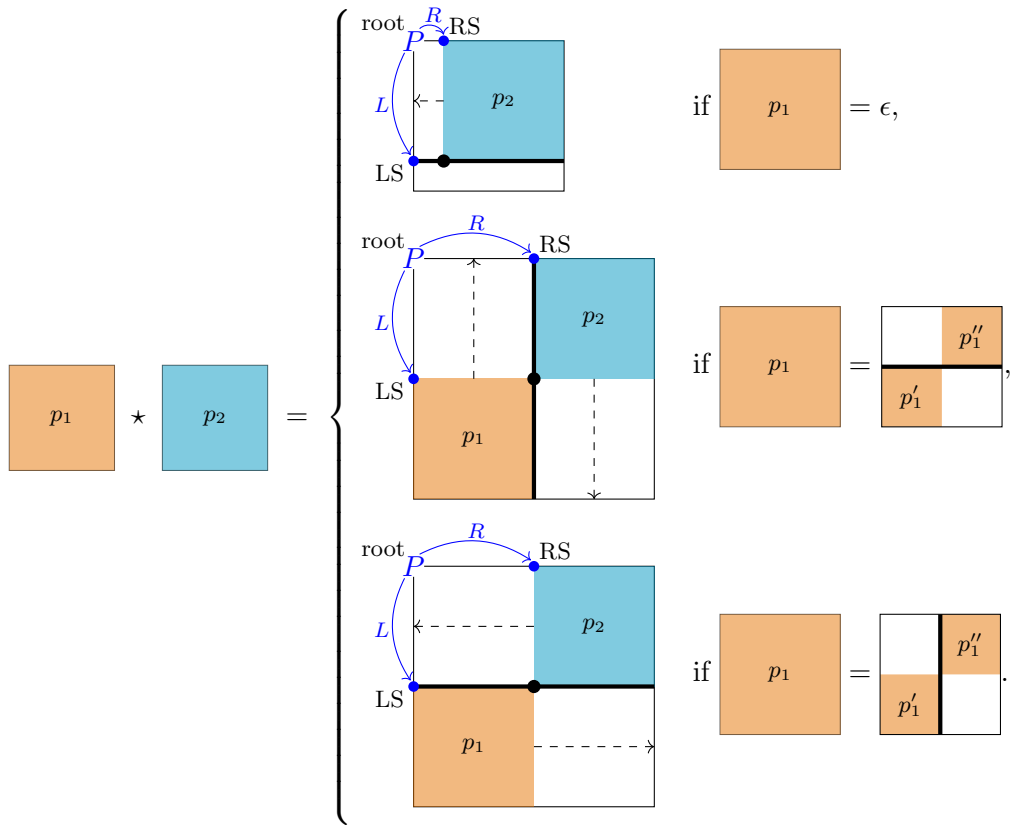
\def\scl {0.8}
\begin{align*}
    & f \left(
    %
$}. \\
    \end{cases}
\end{align*}

Thus we can embed a unary-binary tree inside a rectangulation as demonstrated in the following example:
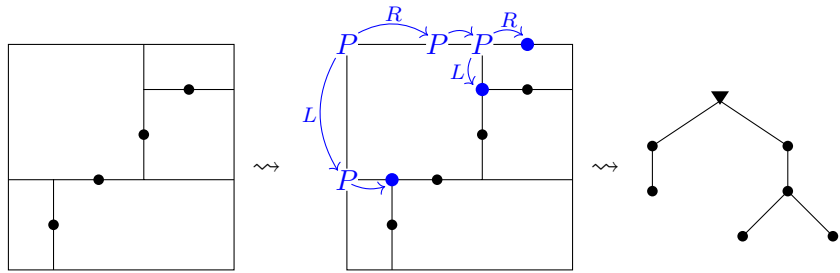
\def\size{5}
\def\scl{0.6}
\begin{equation*}
    %
\end{equation*}


\subsection{Embedded bijections for Fuss-Catalan  (3)-magmas}


In this section we consider how to embed ternary trees inside objects from other Fuss-Catalan families. 

Assume the ternary map in the ternary tree (3)-magma is $t$ and the unique generator is $\epsilon$. 
If we label each leaf and each internal node in a ternary tree as follows:
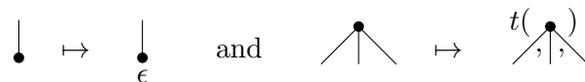
\def\scl{0.5}
\begin{equation*}
    \begin{tikzpicture}[baseline={([yshift=-3ex]current bounding box.base)}, scale = \scl]
        \fill (0,-1) \smlnd;
        \draw (0,0) \un;
    \end{tikzpicture}
    \quad  \mapsto \quad
    \begin{tikzpicture}[baseline={([yshift=-3ex]current bounding box.base)}, scale = \scl]
        \fill (0,-1) \smlnd;
        \draw (0,0) \un;
        \node at ([shift={(0,-0.5)}]0,-1) {$\epsilon$};
    \end{tikzpicture}
    \qquad \text{and}\qquad
    \begin{tikzpicture}[baseline={([yshift=-2.5ex]current bounding box.base)}, scale = \scl]
        \fill (0,0) \smlnd;
        \draw (0,0) \tern;
    \end{tikzpicture}
    \quad  \mapsto \quad 
    \begin{tikzpicture}[baseline={([yshift=-2.5ex]current bounding box.base)}, scale = \scl]
        \fill (0,0) \smlnd;
        \draw (0,0) \tern;
        \node at ([shift={(-0.8,0.1)}]0,0) {$t($};
        \node at ([shift={(0.6,0.1)}]0,0) {$)$};
        \node at ([shift={(-0.3,-0.7)}]0,0) {$,$};
        \node at ([shift={(0.3,-0.7)}]0,0) {$,$};
    \end{tikzpicture}
\end{equation*}
then a counter-clockwise traversal of the tree gives its factorisation, as shown in the following example:
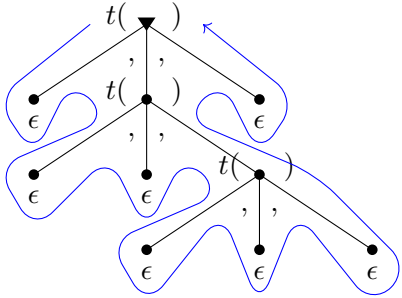
\def\scl{1}
\begin{equation*}
    \begin{tikzpicture}[scale = \scl]
        \node at (0,0) {\smlroot};
        \draw (0,0) \widetern;
        \draw (0,-1) \widetern;
        \draw (1.5,-2) \widetern;
        \fill (0,0) \smlwideternnd (0,-1) \smlwideternnd (1.5,-2) \smlwideternnd;
        \ternlabel[0][0]
        \ternlabel[0][-1]
        \ternlabel[1.5][-2]
        \leaflabel[-1.5][-1]
        \leaflabel[1.5][-1]
        \leaflabel[-1.5][-2]
        \leaflabel[0][-2]
        \leaflabel[0][-3]
        \leaflabel[1.5][-3]
        \leaflabel[3][-3]
        \draw[blue, rounded corners = 10, ->]
            (-0.75,0) -- (-2,-1) -- (-1.5,-1.75) -- (-1,-0.75) -- (-0.5,-1.25) -- (-2,-1.9) -- (-1.5,-2.75) -- (-1,-2.25) -- (-0.5,-1.75) -- (0,-2.75) -- (0.5,-1.75) -- (1,-2.25) -- (-0.5,-2.9) -- (0,-3.75) -- (1,-2.5) -- (1.5,-3.75)  -- (2,-2.5) -- (3,-3.75) -- (3.5,-3.25) -- (2.25,-2) -- (0.5,-1.25) -- (1,-0.75) -- (1.5,-1.75) -- (2,-1) -- (0.75,0);
    \end{tikzpicture}
\end{equation*}

Traversing this tree, we can see that its factorisation is
$
    t \left( \epsilon, t \left( \epsilon, \epsilon, t \left( \epsilon, \epsilon, \epsilon \right) \right) , \epsilon \right).
$

From the definition of any geometric (3)-magma, we have the following information:
\begin{enumerate}[label=(\roman*)]
    \item We know which part of the geometry is associated with the generator. We will call this the \textbf{generator geometry}.
    \item From the definition of the ternary map, we know which part of the geometry is added each time the ternary map is applied. We will call this the \textbf{ternary map geometry}.
\end{enumerate}

Take the family of ternary trees \fcFamRef{fcFamily:ternaryTrees} for example. We can see that the generators correspond to leaves of the tree, while from the ternary map definition,
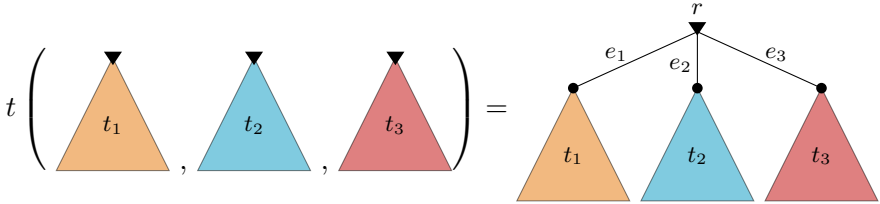
\begin{equation*}  \def\scl{0.15}
    t \left( 
    \raisebox{-20pt}{
    \begin{tikzpicture}[baseline={([yshift=-5ex]current bounding box.center)}, scale = \scl]
        \draw[fill = \colOne, opacity = \opac] (0,0) -- (-5,-10) -- (5,-10) -- (0,0);
        \node at (0,-2.5) {\smlroot};
        \node at (0,-6) {\footnotesize $t_1$};
    \end{tikzpicture}
    , 
    \begin{tikzpicture}[baseline={([yshift=-5ex]current bounding box.center)}, scale = \scl]
        \draw[fill = \colTwo, opacity = \opac] (0,0) -- (-5,-10) -- (5,-10) -- (0,0);
        \node at (0,-2.5) {\smlroot};
        \node at (0,-6) {\footnotesize $t_2$};
    \end{tikzpicture}
    ,
    \begin{tikzpicture}[baseline={([yshift=-5ex]current bounding box.center)}, scale = \scl]
        \draw[fill = \colThree, opacity = \opac] (0,0) -- (-5,-10) -- (5,-10) -- (0,0);
        \node at (0,-2.5) {\smlroot};
        \node at (0,-6) {\footnotesize $t_3$};
    \end{tikzpicture}}
    \right)
    =
    \begin{tikzpicture}[baseline={([yshift=-1ex]current bounding box.center)}, scale = \scl]
        \draw[fill = \colOne, opacity = \opac] (-11,-5) -- (-16,-15) -- (-6,-15) -- (-11,-5);
        \draw[fill = \colTwo, opacity = \opac] (0,-5) -- (-5,-15) -- (5,-15) -- (0,-5);
        \draw[fill = \colThree, opacity = \opac] (11,-5) -- (16,-15) -- (6,-15) -- (11,-5);
        \draw (0,0) -- (0,-5) (0,0) -- (-11,-5) (0,0) -- (11,-5);
        \node at (0,0) {\smlroot};
        \fill (-11,-5) \smlnd (0,-5) \smlnd (11,-5) \smlnd;
        \node at (-11,-11) {\footnotesize $t_1$};
        \node at (0,-11) {\footnotesize $t_2$};
        \node at (11,-11) {\footnotesize $t_3$};
        \dummyNodes[0][0.2][0][-15.2]
        \node at (0,2) {\footnotesize $r$};
        \node at (-7.2,-2) {\footnotesize $e_1$};
        \node at (-1.5,-3) {\footnotesize $e_2$};
        \node at (7,-2) {\footnotesize $e_3$};
    \end{tikzpicture}
\end{equation*}
we observe that the ternary map geometry is the new root $r$ and the three added edges $e_1$, $e_2$ and $e_3$.

For the Fuss-Catalan family of quadrillages \fcFamRef{fcFamily:quadrillages}, the generators correspond to sides of the polygon (with the exception of the single marked side for all polygons with more than one edge). From the ternary map ,
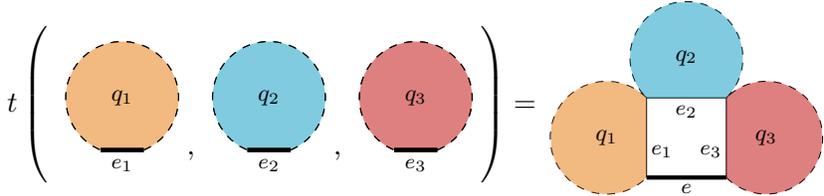
\begin{equation*} \def\numPoints{8} \def\scl{0.15}
    t \left(
    \raisebox{-16pt}{
    \begin{tikzpicture}[baseline={([yshift=-4.5ex]current bounding box.base)}, scale = \scl]
        \blankPolygonSetup
        \fill[\colOne, opacity = \opac] (8) -- (1) (-67.5:5) arc (-67.5:247.5:5);
        \node at (0,0) {\footnotesize $q_1$};
        \node at (-90:6) {\scriptsize $e_1$};
        \blankPolygonSetup
    \end{tikzpicture},
    \begin{tikzpicture}[baseline={([yshift=-4.5ex]current bounding box.base)}, scale = \scl]
        \blankPolygonSetup
        \fill[\colTwo, opacity = \opac] (8) -- (1) (-67.5:5) arc (-67.5:247.5:5);
        \node at (0,0) {\footnotesize $q_2$};
        \node at (-90:6) {\scriptsize $e_2$};
        \blankPolygonSetup
    \end{tikzpicture},
    \begin{tikzpicture}[baseline={([yshift=-4.5ex]current bounding box.base)}, scale = \scl]
        \blankPolygonSetup
        \fill[\colThree, opacity = \opac] (8) -- (1) (-67.5:5) arc (-67.5:247.5:5);
        \node at (0,0) {\footnotesize $q_3$};
        \node at (-90:6) {\scriptsize $e_3$};
        \blankPolygonSetup
    \end{tikzpicture}}
    \right)
    = \def\numPoints{4} \def\scl{0.15}
    \begin{tikzpicture}[baseline={([yshift=-1ex]current bounding box.center)}, scale = \scl]
        \polygonSetup
        \draw[dashed] ([shift={(-7.07,0)}]45:5) arc (45:315:5);
         \begin{scope}
            \clip ([shift={(-7.07,0)}]45:5) arc (45:315:5);;
            \fill[\colOne, opacity = \opac] (3) -- (4) -- (0,-7) -- (-12,-7) -- (-12,7) -- (0,7) -- (3);
        \end{scope}
        \node at (-7.07,0) {\footnotesize $q_1$};
        
        \draw[dashed] ([shift={(0,7.07)}]-45:5) arc (-45:225:5);
        \begin{scope}
            \clip ([shift={(0,7.07)}]-45:5) arc (-45:225:5);
            \fill[\colTwo, opacity = \opac] (2) -- (3) -- (-7,0) -- (-7,12) -- (7,12) -- (7,0) -- (2);
        \end{scope}
        \node at (0,7.07) {\footnotesize $q_2$};
        
        \draw[dashed] ([shift={(7.07,0)}]-135:5) arc (-135:135:5);
        \begin{scope}
            \clip ([shift={(7.07,0)}]-135:5) arc (-135:135:5);
            \fill[\colThree, opacity = \opac] (2) -- (1) -- (0,-7) -- (12,-7) -- (12,7) -- (0,7) -- (2);
        \end{scope}
        \node at (7.07,0) {\footnotesize $q_3$};
        
        \node at (210:2.5) {\scriptsize $e_1$};
        \node at (90:2.3) {\scriptsize $e_2$};
        \node at (-30:2.5) {\scriptsize $e_3$};
        \node at (-90:4.5) {\scriptsize $e$};
    \end{tikzpicture}
\end{equation*}
we can see that the ternary map geometry is just the new marked side $e$.

As in Section \ref{sec:MotSchEmbedded}, we wish to draw our objects in such a way that it is clear exactly which parts of an object's geometry correspond to the generator and which parts arise from the ternary map. We choose to draw our objects so that the points of the object can be partitioned into two disjoint subsets, with one containing all pieces of generator geometry and the other all pieces of ternary map geometry. We will say that families defined in such a way that their objects can be partitioned into these subsets are in \textbf{(3)-magma form}. 

For any free (3)-magma ($\cM$, $t$) with a single generator and every object $m \in \cM$, define the following sets:
\begin{itemize}
    \item Let $\mathbb{G}_{\cM}(m)$ to be the set of all pieces of generator geometry.
    \item Let $\mathbb{T}_{\cM}(m)$ to be the set of all pieces of ternary map geometry.
\end{itemize}
For example, drawing ternary trees with generator \def\scl{1}
$\begin{tikzpicture}[scale = \scl]
\fill (0,0) \smlnd;
\end{tikzpicture}$ \
and ternary tree geometry in orange, we can see how any ternary tree can be partitioned into these subsets. For example, for the ternary tree
\begin{equation*}
    \begin{tikzpicture}[baseline={([yshift=-.5ex]current bounding box.center)}, scale = \scl]
        \draw (0,0) \tern;
        \draw (0,-1) \tern;
        \fill (0,0) \smlternnd (0,-1) \smlternnd;
        \node at (0,0) {\smlroot};
        \ndlabel[0][0][1]
        \ndlabel[-1][-1][2]
        \ndlabel[0][-1][3]
        \ndlabel[1][-1][7]
        \ndlabel[-1][-2][4]
        \ndlabel[0][-2][5]
        \ndlabel[1][-2][6]
    \end{tikzpicture}
    \quad \leadsto \quad
    \begin{tikzpicture}[baseline={([yshift=-.5ex]current bounding box.center)}, scale = \scl]
        \draw[\colOne] (0,0) \tern;
        \draw[\colOne] (0,-1) \tern;
        \fill (0,0) \smlternnd (0,-1) \smlternnd;
        \fill[\colOne] (0,0) \smlnd (0,-1) \smlnd;
    \end{tikzpicture}
\end{equation*}
this partition is
\begin{equation*} \def \scl {0.5}
    \mathbb{G}_{\cM}(m) = \left\{\!\!
        \begin{tikzpicture}[baseline={([yshift=-1ex]current bounding box.center)}, scale = \scl]
            \fill (0,0) \smlnd;
            \ndlabel[-0.1][0][2]
        \end{tikzpicture},
        \begin{tikzpicture}[baseline={([yshift=-1ex]current bounding box.center)}, scale = \scl]
            \fill (0,0) \smlnd;
            \ndlabel[-0.1][0][4]
        \end{tikzpicture},
        \begin{tikzpicture}[baseline={([yshift=-1ex]current bounding box.center)}, scale = \scl]
            \fill (0,0) \smlnd;
            \ndlabel[-0.1][0][5]
        \end{tikzpicture},
        \begin{tikzpicture}[baseline={([yshift=-1ex]current bounding box.center)}, scale = \scl]
            \fill (0,0) \smlnd;
            \ndlabel[-0.1][0][6]
        \end{tikzpicture},
        \begin{tikzpicture}[baseline={([yshift=-1ex]current bounding box.center)}, scale = \scl]
            \fill (0,0) \smlnd;
            \ndlabel[-0.1][0][7]
        \end{tikzpicture}
    \right\}, \quad 
    \mathbb{T}_{\cM}(m) = \left\{ 
        \raisebox{-5pt}{
        \begin{tikzpicture}[baseline={([yshift=-2ex]current bounding box.center)}, scale = \scl]
            \fill (0,0) \smlnd;
            \draw (0,0) \tern;
            \ndlabel[-0.2][0][1]
        \end{tikzpicture}\spacecomma
        \begin{tikzpicture}[baseline={([yshift=-2ex]current bounding box.center)}, scale = \scl]
            \fill (0,0) \smlnd;
            \draw (0,0) \tern;
            \ndlabel[-0.2][0][3]
        \end{tikzpicture}}
    \right\}.
\end{equation*}

Now, for each piece of generator geometry we shall mark a point and call this point the \textbf{leaf}. In some cases, the generator is the empty object. In these instances, we enforce that the generator is associated with some geometry so that we are able to mark this point. We also mark a particular point in each piece of ternary map geometry, calling this point the \textbf{root}. For example, for the Fuss-Catalan family of quadrillages, the generator is a single edge, so we represent this by the following:
\vspace{-1ex}
\begin{equation*} \def\numPoints{2} \def\scl{0.08}
    \epsilon =
    \begin{tikzpicture}[baseline={([yshift=-.9ex]current bounding box.center)}, scale = \scl]
    \polygonSetup
    \fill[blue] (0,0) \nd;
\end{tikzpicture}
\end{equation*}
Notice that we have marked a point on this edge in blue. This point is the leaf. Similarly, we can mark the root in our definition of the ternary map, noting that this root is part of the ternary map geometry. This is shown in blue in the following figure:
\vspace{-1ex}
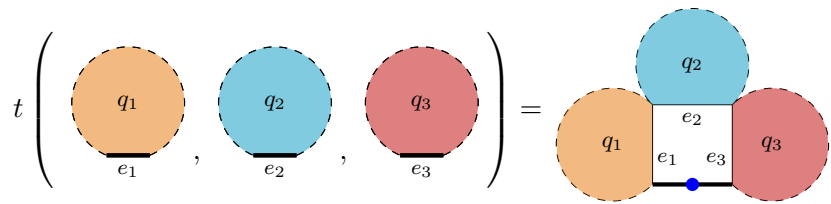
\begin{equation*} \def\numPoints{8} \def\scl{0.15}
    t \left(
    \raisebox{-16pt}{
    \begin{tikzpicture}[baseline={([yshift=-4.5ex]current bounding box.base)}, scale = \scl]
        \blankPolygonSetup
        \fill[\colOne, opacity = \opac] (8) -- (1) (-67.5:5) arc (-67.5:247.5:5);
        \node at (0,0) {\footnotesize $q_1$};
        \node at (-90:6) {\scriptsize $e_1$};
        \blankPolygonSetup
    \end{tikzpicture},
    \begin{tikzpicture}[baseline={([yshift=-4.5ex]current bounding box.base)}, scale = \scl]
        \blankPolygonSetup
        \fill[\colTwo, opacity = \opac] (8) -- (1) (-67.5:5) arc (-67.5:247.5:5);
        \node at (0,0) {\footnotesize $q_2$};
        \node at (-90:6) {\scriptsize $e_2$};
        \blankPolygonSetup
    \end{tikzpicture},
    \begin{tikzpicture}[baseline={([yshift=-4.5ex]current bounding box.base)}, scale = \scl]
        \blankPolygonSetup
        \fill[\colThree, opacity = \opac] (8) -- (1) (-67.5:5) arc (-67.5:247.5:5);
        \node at (0,0) {\footnotesize $q_3$};
        \node at (-90:6) {\scriptsize $e_3$};
        \blankPolygonSetup
    \end{tikzpicture}}
    \right)
    = \def\numPoints{4} \def\scl{0.15}
    \begin{tikzpicture}[baseline={([yshift=-1ex]current bounding box.center)}, scale = \scl]
        \polygonSetup
        \draw[dashed] ([shift={(-7.07,0)}]45:5) arc (45:315:5);
         \begin{scope}
            \clip ([shift={(-7.07,0)}]45:5) arc (45:315:5);;
            \fill[\colOne, opacity = \opac] (3) -- (4) -- (0,-7) -- (-12,-7) -- (-12,7) -- (0,7) -- (3);
        \end{scope}
        \node at (-7.07,0) {\footnotesize $q_1$};
        
        \draw[dashed] ([shift={(0,7.07)}]-45:5) arc (-45:225:5);
        \begin{scope}
            \clip ([shift={(0,7.07)}]-45:5) arc (-45:225:5);
            \fill[\colTwo, opacity = \opac] (2) -- (3) -- (-7,0) -- (-7,12) -- (7,12) -- (7,0) -- (2);
        \end{scope}
        \node at (0,7.07) {\footnotesize $q_2$};
        
        \draw[dashed] ([shift={(7.07,0)}]-135:5) arc (-135:135:5);
        \begin{scope}
            \clip ([shift={(7.07,0)}]-135:5) arc (-135:135:5);
            \fill[\colThree, opacity = \opac] (2) -- (1) -- (0,-7) -- (12,-7) -- (12,7) -- (0,7) -- (2);
        \end{scope}
        \node at (7.07,0) {\footnotesize $q_3$};
        
        \node at (210:2.5) {\scriptsize $e_1$};
        \node at (90:2.3) {\scriptsize $e_2$};
        \node at (-30:2.5) {\scriptsize $e_3$};
        \fill[blue] (-90:3.53553) \nd;
    \end{tikzpicture}
\end{equation*}

Now, for any object $m$ from a Fuss-Catalan family with (3)-magma ($\cM$, $t$) such that \mbox{$m = t(m_1, m_2, m_3)$}, define the following:
\begin{itemize}
    \item Let the \textbf{left subroot} be the root of $m_1$ if it exists (that is, if $m_1 \neq \epsilon$) and the leaf of $m_1$ if $m_1 = \epsilon$.
    \item Let the \textbf{middle subroot} be the root of $m_2$ if $m_2 \neq \epsilon$ and the leaf of $m_2$ if $m_2 = \epsilon$.
    \item Let the \textbf{right subroot} be the root of $m_3$ if $m_3 \neq \epsilon$ and the leaf of $m_3$ if $m_3 = \epsilon$.
\end{itemize}
For example, if we are considering the family of ternary trees and we have
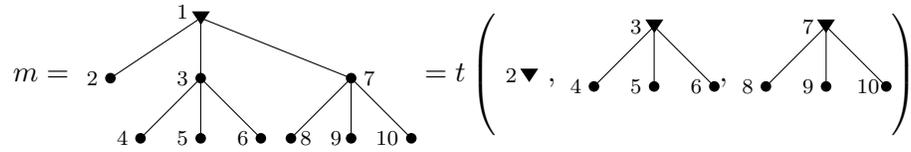
\begin{equation*} \def\scl{0.8}
    m =
    \begin{tikzpicture}[baseline={([yshift=-.5ex]current bounding box.center)}, scale = \scl]
        \draw (0,0) -- (-1.5,-1) (0,0) -- (0,-1) ++(0,1) -- ++(2.5,-1);
        \draw (0,-1) \tern;
        \draw (2.5,-1) \tern;
        \fill (-1.5,-1) \smlnd (0,-1) \smlnd (2.5,-1) \smlnd (0,-1) \smlternnd (2.5,-1) \smlternnd;
        \node at (0,0) {\smlroot};
        \node at ([shift={(-0.3,0.1)}]0,0) {\scriptsize 1};
        \ndlabel[-1.5][-1][2]
        \ndlabel[0][-1][3]
        \ndlabel[-1][-2][4]
        \ndlabel[0][-2][5]
        \ndlabel[1][-2][6]
        \node at ([shift={(0.3,0)}]2.5,-1) {\scriptsize 7};
        \node at ([shift={(0.25,0)}]1.5,-2) {\scriptsize 8};
        \ndlabel[2.55][-2][9]
        \ndlabel[3.4][-2][10]
    \end{tikzpicture}
    = t \left(
    \begin{tikzpicture}[baseline={([yshift=-.5ex]current bounding box.center)}, scale = \scl]
        \ndlabel[0][0][2]
        \node at (0,0) {\smlroot};
    \end{tikzpicture},
    \begin{tikzpicture}[baseline={([yshift=-2ex]current bounding box.center)}, scale = \scl]
        \draw (0,0) \tern;
        \fill (0,0) \smlternnd;
        \node at (0,-0.15) {\smlroot};
        \ndlabel[0][0][3]
        \ndlabel[-1][-1][4]
        \ndlabel[0][-1][5]
        \ndlabel[1][-1][6]
    \end{tikzpicture},
    \begin{tikzpicture}[baseline={([yshift=-2ex]current bounding box.center)}, scale = \scl]
        \draw (0,0) \tern;
        \fill (0,0) \smlternnd;
        \node at (0,-0.15) {\smlroot};
        \ndlabel[0][0][7]
        \ndlabel[-1][-1][8]
        \ndlabel[0][-1][9]
        \ndlabel[1][-1][10]
    \end{tikzpicture}
    \right)
\end{equation*}
then the left subroot is node 2, the middle subroot is node 3 and the right subroot is node 7.

We can now define our embedded bijections via the following recursive procedure of embedding 4-tuples $(P, \overset{L}{\leftarrow}, \downarrow, \overset{R}{\rightarrow})$ inside the Fuss-Catalan object $m$. If $m = t(m_1, m_2, m_3)$, then attach $P$ to the root of the ternary map. The arrow $\overset{L}{\leftarrow}$ points from $P$ to the left subroot, the arrow $\downarrow$ points from $P$ to the middle subroot and the arrow $\overset{R}{\rightarrow}$ points from $P$ to the right subroot.

We can represent this embedding process schematically by drawing the roots, subroots and embedded arrows in the definition of the ternary map. For example, for ternary trees we have the following:
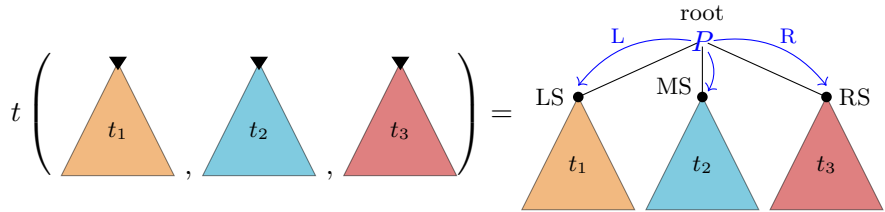
\def \scl {0.15}
\begin{equation*}
    t \left( 
    \raisebox{-20pt}{
    \begin{tikzpicture}[baseline={([yshift=-5ex]current bounding box.center)}, scale = \scl]
        \draw[fill = \colOne, opacity = \opac] (0,0) -- (-5,-10) -- (5,-10) -- (0,0);
        \node at (0,-2.5) {\smlroot};
        \node at (0,-6) {\footnotesize $t_1$};
    \end{tikzpicture}
    , 
    \begin{tikzpicture}[baseline={([yshift=-5ex]current bounding box.center)}, scale = \scl]
        \draw[fill = \colTwo, opacity = \opac] (0,0) -- (-5,-10) -- (5,-10) -- (0,0);
        \node at (0,-2.5) {\smlroot};
        \node at (0,-6) {\footnotesize $t_2$};
    \end{tikzpicture}
    ,
    \begin{tikzpicture}[baseline={([yshift=-5ex]current bounding box.center)}, scale = \scl]
        \draw[fill = \colThree, opacity = \opac] (0,0) -- (-5,-10) -- (5,-10) -- (0,0);
        \node at (0,-2.5) {\smlroot};
        \node at (0,-6) {\footnotesize $t_3$};
    \end{tikzpicture}}
    \right)
    =
    \begin{tikzpicture}[baseline={([yshift=-1ex]current bounding box.center)}, scale = \scl]
        \draw[fill = \colOne, opacity = \opac] (-11,-5) -- (-16,-15) -- (-6,-15) -- (-11,-5);
        \draw[fill = \colTwo, opacity = \opac] (0,-5) -- (-5,-15) -- (5,-15) -- (0,-5);
        \draw[fill = \colThree, opacity = \opac] (11,-5) -- (16,-15) -- (6,-15) -- (11,-5);
        \draw (0,0) -- (0,-5) (0,0) -- (-11,-5) (0,0) -- (11,-5);
        \fill (-11,-5) \smlnd (0,-5) \smlnd (11,-5) \smlnd;
        \node at (-11,-11) {\footnotesize $t_1$};
        \node at (0,-11) {\footnotesize $t_2$};
        \node at (11,-11) {\footnotesize $t_3$};
        \dummyNodes[0][0.2][0][-15.2]
        \node at (0,2.5) {\footnotesize root};
        \node at ([shift = {(-2.5,0)}]-11,-5) {\footnotesize LS};
        \node at ([shift = {(-2.5,1)}]0,-5) {\footnotesize MS};
        \node at ([shift = {(2.5,0)}]11,-5) {\footnotesize RS};
        \path[blue, ->]
            (-1,0) edge [bend right = 30] node[style = {inner sep = 1pt}, above left]  {\scriptsize L} (-11,-4) 
            (0.5,-1) edge [bend left = 30] (0.5,-4.5)
            (1,0) edge [bend left = 30] node[style = {inner sep = 1pt}, above right]  {\scriptsize R} (11,-4);
        \fill[white] (0,0) circle (15pt);
        \P[0][0]
    \end{tikzpicture} 
\end{equation*}
Here, LS, MS and RS denote the left subroot, middle subroot and right subroot respectively.

Repeating this embedding recursively gives, for example,
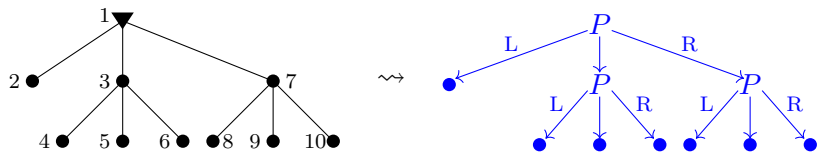
\def\scl{0.8}
\begin{equation*}
    \begin{tikzpicture}[baseline={([yshift=-.5ex]current bounding box.center)}, scale = \scl]
        \draw (0,0) -- (-1.5,-1) (0,0) -- (0,-1) ++(0,1) -- ++(2.5,-1);
        \draw (0,-1) \tern;
        \draw (2.5,-1) \tern;
        \fill (-1.5,-1) \nd (0,-1) \nd (2.5,-1) \nd (0,-1) \ternnd (2.5,-1) \ternnd;
        \node at (0,0) {\root};
        \node at ([shift={(-0.3,0.1)}]0,0) {\scriptsize 1};
        \ndlabel[-1.5][-1][2]
        \ndlabel[0][-1][3]
        \ndlabel[-1][-2][4]
        \ndlabel[0][-2][5]
        \ndlabel[1][-2][6]
        \node at ([shift={(0.3,0)}]2.5,-1) {\scriptsize 7};
        \node at ([shift={(0.25,0)}]1.5,-2) {\scriptsize 8};
        \ndlabel[2.5][-2][9]
        \ndlabel[3.5][-2][10]
    \end{tikzpicture}
    \quad \leadsto \quad
    \begin{tikzpicture}[baseline={([yshift=-.5ex]current bounding box.center)}, scale = \scl]
        \P[0][0]
        \P[0][-1]
        \P[2.5][-1]
        \fill[blue] (-2.5,-1) \nd (0,-1) \ternnd (2.5,-1) \ternnd;
        \path[blue, ->]
            (-0.2,-0.1) edge node[style = {inner sep = 1pt}, above left]  {\scriptsize L} (-2.4,-0.9)
            (0,-0.2) edge (0,-0.8)
            (0.2,-0.1) edge node[style = {inner sep = 1pt}, above right]  {\scriptsize R} (2.4,-0.9)
            (-0.2,-1.1) edge node[style = {inner sep = 1pt}, above left]  {\scriptsize L} (-0.9,-1.9)
            (0,-1.2) edge (0,-1.9)
            (0.2,-1.1) edge node[style = {inner sep = 1pt}, above right]  {\scriptsize R} (0.9,-1.9)
            (2.3,-1.1) edge node[style = {inner sep = 1pt}, above left]  {\scriptsize L} (1.6,-1.9)
            (2.5,-1.2) edge (2.5,-1.9)
            (2.7,-1.1) edge node[style = {inner sep = 1pt}, above right]  {\scriptsize R} (3.4,-1.9);
    \end{tikzpicture}
\end{equation*}
This clearly defines another tree which trivially differs from a ternary tree. We see that by recursively embedding 4-tuples $(P, \overset{L}{\leftarrow}, \downarrow, \overset{R}{\rightarrow})$ inside an object we are effectively embedding ternary trees. Moreover, we are embedding precisely the ternary tree which is in bijection with that object via the universal bijection. As a result we are then able to easily factorise the original object simply by factorising the embedded ternary tree.

We now demonstrate some explicit embedded bijections for different Fuss-Catalan families.

\paragraph{\fcFamRef{fcFamily:quadrillages}: Quadrillages}

The generator of this family is a single edge, drawn here with a marked point for the leaf:
\begin{equation*} \def\numPoints{2} \def\scl{0.08}
    \epsilon =
    \begin{tikzpicture}[baseline={([yshift=-.9ex]current bounding box.center)}, scale = \scl]
    \polygonSetup
    \fill[blue] (0,0) \nd;
\end{tikzpicture}
\end{equation*}
The ternary map is as follows, with details of the embedding drawn in:
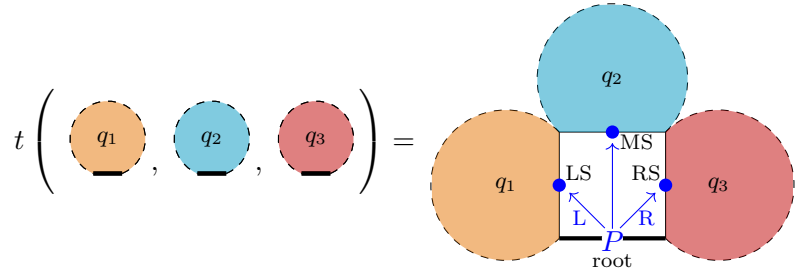
\begin{equation*} \def \numPoints {8} \def\scl{0.1}
    t \left(
    \raisebox{-8pt}{
    \begin{tikzpicture}[baseline={([yshift=-2.5ex]current bounding box.base)}, scale = \scl]
        \blankPolygonSetup
        \fill[\colOne, opacity = \opac] (8) -- (1) (-67.5:5) arc (-67.5:247.5:5);
        \node at (0,0) {\footnotesize $q_1$};
        \blankPolygonSetup
    \end{tikzpicture},
    \begin{tikzpicture}[baseline={([yshift=-2.5ex]current bounding box.base)}, scale = \scl]
        \blankPolygonSetup
        \fill[\colTwo, opacity = \opac] (8) -- (1) (-67.5:5) arc (-67.5:247.5:5);
        \node at (0,0) {\footnotesize $q_2$};
        \blankPolygonSetup
    \end{tikzpicture},
    \begin{tikzpicture}[baseline={([yshift=-2.5ex]current bounding box.base)}, scale = \scl]
        \blankPolygonSetup
        \fill[\colThree, opacity = \opac] (8) -- (1) (-67.5:5) arc (-67.5:247.5:5);
        \node at (0,0) {\footnotesize $q_3$};
        \blankPolygonSetup
    \end{tikzpicture}}
    \right)
    = \def\numPoints{4} \def \scl {0.2}
    \begin{tikzpicture}[baseline={([yshift=-1ex]current bounding box.center)}, scale = \scl]
        \polygonSetup
        \draw[dashed] ([shift={(-7.07,0)}]45:5) arc (45:315:5);
         \begin{scope}
            \clip ([shift={(-7.07,0)}]45:5) arc (45:315:5);;
            \fill[\colOne, opacity = \opac] (3) -- (4) -- (0,-7) -- (-12,-7) -- (-12,7) -- (0,7) -- (3);
        \end{scope}
        \node at (-7.07,0) {\footnotesize $q_1$};
        
        \draw[dashed] ([shift={(0,7.07)}]-45:5) arc (-45:225:5);
        \begin{scope}
            \clip ([shift={(0,7.07)}]-45:5) arc (-45:225:5);
            \fill[\colTwo, opacity = \opac] (2) -- (3) -- (-7,0) -- (-7,12) -- (7,12) -- (7,0) -- (2);
        \end{scope}
        \node at (0,7.07) {\footnotesize $q_2$};
        
        \draw[dashed] ([shift={(7.07,0)}]-135:5) arc (-135:135:5);
        \begin{scope}
            \clip ([shift={(7.07,0)}]-135:5) arc (-135:135:5);
            \fill[\colThree, opacity = \opac] (2) -- (1) -- (0,-7) -- (12,-7) -- (12,7) -- (0,7) -- (2);
        \end{scope}
        \node at (7.07,0) {\footnotesize $q_3$};
        
        \fill[white] (-90:3.53553) circle (20pt);
        \polarP[-90][3.7]
        \node at (-90:5) {\scriptsize root};
        \fill[blue] (0:3.53553) \nd;
        \node at ([shift = {(0.5,-0.2)}]30:2) {\scriptsize RS};
        \fill[blue] (90:3.53553) \nd;
        \node at ([shift = {(-0.4,-0.6)}]60:4) {\scriptsize MS};
        \fill[blue] (180:3.53553) \nd;
        \node at ([shift = {(-0.5,-0.2)}]150:2) {\scriptsize LS};
        \path[blue, ->]
            (-100:2.9) edge  node[style = {inner sep = 0pt}, below left]  {\scriptsize L} (190:2.9)
            (-90:2.9) edge (90:2.8)
            (-80:2.9) edge node[style = {inner sep = 0pt}, below right]  {\scriptsize R} (-10:2.9);
    \end{tikzpicture}
\end{equation*}

We can see from this how we can embed a ternary tree inside a quadrillage. The following is an example of this embedding:
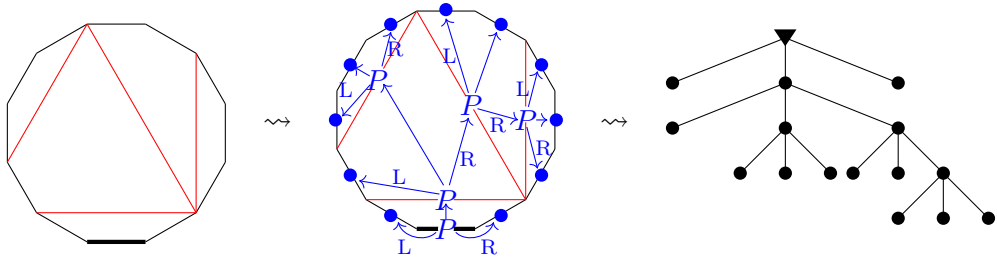
\def \numPoints {12}
\def \scl {0.3}
\begin{equation*}
    \begin{tikzpicture}[baseline={([yshift=-1ex]current bounding box.center)}, scale = \scl]
        \polygonSetup
        \dissectionArc[2][11]
        \dissectionArc[2][7]
        \dissectionArc[2][5]
        \dissectionArc[7][10]
    \end{tikzpicture}
    \quad \leadsto \quad
    \begin{tikzpicture}[baseline={([yshift=-1ex]current bounding box.center)}, scale = \scl]
        \polygonSetup
        \dissectionArc[2][11]
        \dissectionArc[2][7]
        \dissectionArc[2][5]
        \dissectionArc[7][10]
        \foreach \i in {1,...,11}
            {
            \fill[blue] (-90 + 30*\i:4.892963) \nd;
            }
        \fill[white] (-90:4.82963) circle (10pt);
        \polarP[-90][4.82963]
        \fill[white] (4:3.53553) circle (10pt);
        \polarP[0][3.53553]
        \fill[white] (-90:3.53553) circle (10pt);
        \polarP[-90][3.53553]
        \fill[white] (149:3.53553) circle (15pt);
        \polarP[150][3.53553]
        \fill[white] (1.12072,0.8) circle (10pt);
        \P[1.12072][0.647048]
        \path[blue, ->] 
            (-95:5) edge [bend left = 60] node[style = {inner sep = 1pt}, below left]  {\scriptsize L} (-115:5)
            (-90:4.4) edge  (-90:3.7)
            (-85:5) edge [bend right = 60] node[style = {inner sep = 1pt}, below right]  {\scriptsize R} (-65:5)
            (-95:3.3) edge node[style = {inner sep = 1pt}, above]  {\scriptsize L} (-145:4.7)
            (-92:2.9) edge (150:3.2)
            (-87:2.9) edge node[style = {inner sep = 0pt}, below right]  {\scriptsize R} (1,0.1)
            (155:3.7) edge node[style = {inner sep = 1pt}, above left]  {\scriptsize L} (178:4.6)
            (150:3.9) edge (150:4.7)
            (140:3.7) edge node[style = {inner sep = 0pt}, right]  {\scriptsize R} (122:4.6)
            (1,1.2) edge node[style = {inner sep = 0pt}, left]  {\scriptsize L} (90:4.5)
            (1.3,1.2) edge (60:4.5)
            (1.4,0.5) edge node[style = {inner sep = 1pt}, below]  {\scriptsize R} (0:3.2)
            (9:3.7) edge node[style = {inner sep = 1pt}, left]  {\scriptsize L} (28:4.6)
            (0:4) edge (0:4.5)
            (-5:3.6) edge node[style = {inner sep = 1pt}, right]  {\scriptsize R} (-28:4.6);
    \end{tikzpicture} 
    \quad \leadsto \quad
    \def \scl {0.6}
    \begin{tikzpicture}[baseline={([yshift=-.5ex]current bounding box.center)}, scale = \scl]
        \draw (0,0) \verywidetern;
        \draw (0,-1) \verywidetern;
        \draw (0,-2) \tern;
        \draw (2.5,-2) \tern;
        \draw (3.5,-3) \tern;
        \fill (0,0) \verywideternnd (0,-1) \verywideternnd (0,-2) \ternnd (2.5,-2) \ternnd (3.5,-3) \ternnd;
        \node at (0,0) {\root};
    \end{tikzpicture}
\end{equation*}

\paragraph{\fcFamRef{fcFamily:latticePaths}: Lattice paths} \def \scl {0.4}

We will show how to embed a ternary tree inside a lattice path of the type enumerated by the Fuss-Catalan numbers. These are paths from $(0,0)$ to $(n, 2n)$ consisting of $n$ East steps $(1,0)$ and $2n$ North steps $(0,1)$ that lie weakly below the line $y = 2x$. The generator for these paths is taken to be a single vertex: $\epsilon = \begin{tikzpicture}[baseline={([yshift=-.6ex]current bounding box.center)}, scale = \scl]
    \fill (0,0) \smlnd;
\end{tikzpicture}$\,. The ternary map is as follows, with the root, subroots and embedded arrows all labelled:
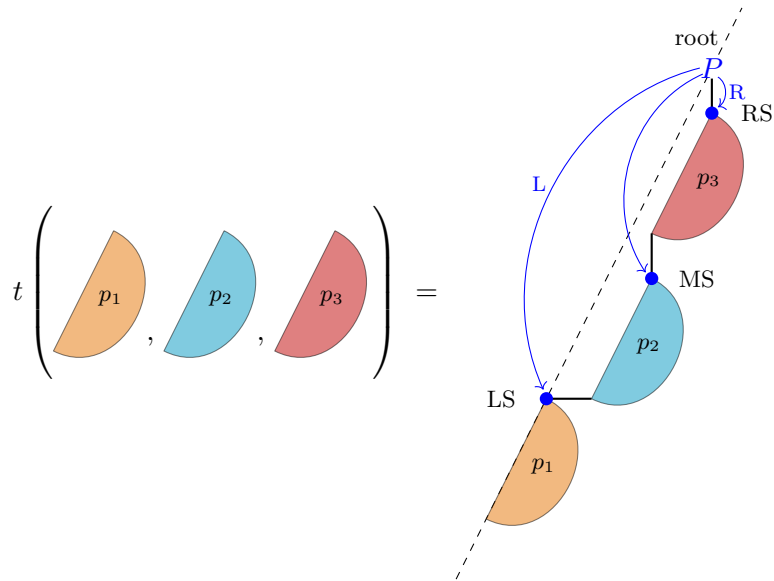
\begin{align*}
    t \left(
    \raisebox{-15pt}{\!\!
    \begin{tikzpicture}[baseline={([yshift=-3ex]current bounding box.center)}, scale = \scl]
        \clip (0,-1) rectangle (3.25,4.25);
        \draw[fill = \colOne, opacity = \opac] (2,4) -- (0,0) to [bend right = 90, looseness = 1.5] (2,4);
        \node at (1.9,1.75) {\footnotesize $p_1$};
    \end{tikzpicture}\!,
    \begin{tikzpicture}[baseline={([yshift=-3ex]current bounding box.center)}, scale = \scl]
        \clip (0,-1) rectangle (3.25,4.25);
        \draw[fill = \colTwo, opacity = \opac] (2,4) -- (0,0) to [bend right = 90, looseness = 1.5] (2,4);
        \node at (1.9,1.75) {\footnotesize $p_2$};
    \end{tikzpicture}\!,
    \begin{tikzpicture}[baseline={([yshift=-3ex]current bounding box.center)}, scale = \scl]
        \clip (0,-1) rectangle (3.25,4.25);
        \draw[fill = \colThree, opacity = \opac] (2,4) -- (0,0) to [bend right = 90, looseness = 1.5] (2,4);
        \node at (1.9,1.75) {\footnotesize $p_3$};
    \end{tikzpicture}}
    \right)
    \ = \
    \begin{tikzpicture}[baseline={([yshift=-.5ex]current bounding box.center)}, scale = \scl]
        \draw[dashed] (-1,-2) -- (8.5,17);
        \draw[fill = \colOne, opacity = \opac] (2,4) -- (0,0) to [bend right = 90, looseness = 1.5] (2,4);
        \node at (1.9,1.75) {\footnotesize $p_1$};
        \draw[thick] (2,4) -- (3.5,4);
        \draw[fill = \colTwo, opacity = \opac] (5.5,8) -- (3.5,4) to [bend right = 90, looseness = 1.5] (5.5,8);
        \node at (5.4,5.75) {\footnotesize $p_2$};
        \draw[thick] (5.5,8) -- (5.5,9.5);
        \draw[fill = \colThree, opacity = \opac] (7.5,13.5) -- (5.5,9.5) to [bend right = 90, looseness = 1.5] (7.5,13.5);
        \node at (7.4,11.25) {\footnotesize $p_3$};
        \draw[thick] (7.5,13.5) -- (7.5,15);
        \fill[white] (7.5,15) circle (10pt);
        \P[7.5][15]
        \node at ([shift = {(-0.5,1)}]7.5,15) {\footnotesize root};
        \node at ([shift = {(1.5,0)}]7.5,13.5) {\footnotesize RS};
        \node at ([shift = {(1.5,0)}]5.5,8) {\footnotesize MS};
        \node at ([shift = {(-1.5,0)}]2,4) {\footnotesize LS};
        \fill[blue] (2,4) \nd (5.5,8) \nd (7.5,13.5) \nd;
        \path[blue, ->]
            (7.1,15) edge [bend right = 50] node[style = {inner sep = 1pt}, above left]  {\scriptsize L} (1.9,4.3)
            (7.2,14.8) edge [bend right = 50] (5.3,8.2)
            (7.7,14.7) edge [bend left = 60] node[style = {inner sep = 1pt}, right]  {\scriptsize R} (7.7,13.7);
    \end{tikzpicture}
\end{align*}
The dotted line in the above schematic diagram is the line $y = 2x$.

We see that we can embed a ternary tree inside a lattice path of this type as shown in the following example:
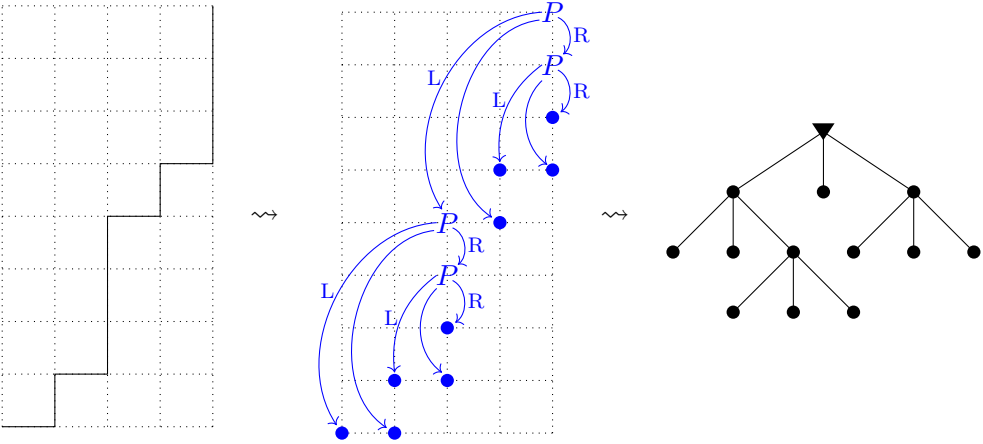
\begin{equation*} \def \scl {0.7}
    \begin{tikzpicture}[baseline={([yshift=-.5ex]current bounding box.center)}, scale = \scl]
        \grd[4][8]
        \draw (0,0) \schAc \schUp \schAc \schUp \schUp \schUp \schAc \schUp \schAc \schUp \schUp \schUp;
    \end{tikzpicture}
    \quad \leadsto
    \begin{tikzpicture}[baseline={([yshift=-.5ex]current bounding box.center)}, scale = \scl]
        \grd[4][8]
        \P[4][8]
        \P[4][7]
        \P[2][4]
        \P[2][3]
        \fill[blue] (0,0) \nd (1,0) \nd (1,1) \nd (2,1) \nd (2,2) \nd (3,4) \nd (3,5) \nd (4,5) \nd (4,6) \nd;
        \path[blue, ->]
            (4.1,7.9) edge [bend left = 60] node[style = {inner sep = 1pt}, right]  {\scriptsize R} (4.2,7.2)
            (3.75,7.85) edge [bend right = 70] (2.85,4.1)
            (3.8,8) edge [bend right = 60] node[style = {inner sep = 0pt}, above left]  {\scriptsize L} (1.9,4.25)
            (4.1,6.9) edge [bend left = 60] node[style = {inner sep = 1pt}, right]  {\scriptsize R} (4.15,6.1)
            (3.8,6.7) edge [bend right = 50] (3.9,5.1)
            (3.8,7) edge [bend right = 30] node[style = {inner sep = 0pt}, above left]  {\scriptsize L} (3,5.15)
            (2.1,3.9) edge [bend left = 60] node[style = {inner sep = 1pt}, right]  {\scriptsize R} (2.2,3.2)
            (1.75,3.85) edge [bend right = 70] (0.85,0.1)
            (1.8,4) edge [bend right = 60] node[style = {inner sep = 0pt}, above left]  {\scriptsize L} (-0.1,0.15)
            (2.1,2.9) edge [bend left = 60] node[style = {inner sep = 1pt}, right]  {\scriptsize R} (2.15,2.1)
            (1.8,2.75) edge [bend right = 50] (1.9,1.15)
            (1.8,3) edge [bend right = 30] node[style = {inner sep = 1pt}, left]  {\scriptsize L} (1,1.15);
    \end{tikzpicture}
    \leadsto \quad 
    \def \scl {0.8}
    \begin{tikzpicture}[baseline={([yshift=-.5ex]current bounding box.center)}, scale = \scl]
        \draw (0,0) \widetern;
        \draw (-1.5,-1) \tern;
        \draw (1.5,-1) \tern;
        \draw (-0.5,-2) \tern;
        \fill (0,0) \wideternnd (-1.5,-1) \ternnd (1.5,-1) \ternnd (-0.5,-2) \ternnd;
        \node at (0,0) {\root};
    \end{tikzpicture}
\end{equation*}



\newpage

\section{Appendix}


In the following appendices, we list a number of families which are enumerated by each of the Fibonacci numbers (Section \ref{appendix:Fibonacci}), Motzkin numbers (Section \ref{appendix:Motzkin}), Schr\"oder numbers (Section \ref{appendix:Schroder}) and order 3 Fuss-Catalan numbers (Section \ref{appendix:FC}). For each family, we provide a reference and give a brief definition before detailing the following:
\begin{itemize}
    \item The generator of the relevant $\bm{n}$-magma.
    \item A definition of all of the relevant $\bm{n}$-magma maps. These are in most cases schematic diagrams.
    \item A definition of the norm, in terms of some natural parameter of the objects of that family.
    \item A list of some elements of the corresponding $\bm{n}$-magma. In each family, we list the elements with the smallest norms alongside their decomposition in terms of the generator and the $\bm{n}$-magma maps. This information provides concrete examples of how to apply the maps.
\end{itemize}

One can then use the information provided to obtain a bijection between any two families which are enumerated by the same integer sequence. This can be done using the relevant universal bijection (Definition \ref{def:11univBij} for Fibonacci families, Definition \ref{def:12univBij} for Motzkin or Schr\"oder families and Definition \ref{def:3univBij} for Fuss-Catalan families).

\pagebreak

\subsection{Fibonacci Families} \label{appendix:Fibonacci}
 
\addcontentsline{toca}{section}{Fibonacci Families}

We present a number of Fibonacci normed (1,1)-magmas for well-known families of objects enumerated by the Fibonacci numbers. We take the convention that each generator is called $\epsilon$ and the two unary maps are called $f$ and $g$. These are chosen in such a way that
\begin{equation*}
    \norm{f(m)} = \norm{m} + 1, \qquad \norm{g(m)} = \norm{m} + 2,
\end{equation*}
for all $m \in \cM$, where $\cM$ is the base set of the (1,1)-magma. Despite the same names being used for each family presented, it is clear that these are all different maps and that the generators are all different. The choice to use the same names was made for the sake of clarity and simplicity. At any point in this paper where these maps or generators are referenced, they appear with a subscript indicating which family they correspond to.

\addcontentsline{toca}{subsection}{\fibFamRef{fibFamily:tilings}: Fibonacci tilings}

\begin{fibFamily}
[Fibonacci tilings \cite{benjamin_quinn_2003,BenjaminQuinnSu}] \label{fibFamily:tilings}

\def\scl{0.45}

The Fibonacci number $F_{n + 1}$ is the number of ways to tile a $1 \times n$ board using $1 \times 1$ squares and $1 \times 2$ dominoes. We will represent a $1 \times 1$ square and a $1 \times 2$ domino by
\begin{equation*}
    \begin{tikzpicture}[baseline={([yshift=-.5ex]current bounding box.center)}, scale = \scl]
        \draw (0,0) \onetile \finaledge;
    \end{tikzpicture}
    \qquad \text{and} \qquad
    \begin{tikzpicture}[baseline={([yshift=-.5ex]current bounding box.center)}, scale = \scl]
        \draw (0,0) \twotiles \finaledge;
    \end{tikzpicture}
\end{equation*}
respectively.

\textit{Generator:} The empty tiling:
\begin{equation*}
    \epsilon = \emptyset.
\end{equation*}

\textit{Unary maps:} \vspace{-1ex}
\begin{align*}
    f \left(
    \begin{tikzpicture}[baseline={([yshift=-.5ex]current bounding box.center)}, scale = \scl]
        \dummyNodes[0.5][-0.1][0.5][1.1]
        \draw[fill = \colOne, opacity = \opac] (0,0) rectangle (6,1);
        \node at (3,0.5) {$t$};
    \end{tikzpicture}
    \right)
    & =
    \begin{tikzpicture}[baseline={([yshift=-.5ex]current bounding box.center)}, scale = \scl]
        \dummyNodes[0.5][-0.1][0.5][1.1]
        \draw[fill = \colOne, opacity = \opac] (0,0) rectangle (6,1);
        \node at (3,0.5) {$t$};
        \draw (6,0) \onetile \finaledge;
    \end{tikzpicture} \\
    g \left(
    \begin{tikzpicture}[baseline={([yshift=-.5ex]current bounding box.center)}, scale = \scl]
        \dummyNodes[0.5][-0.1][0.5][1.1]
        \draw[fill = \colOne, opacity = \opac] (0,0) rectangle (6,1);
        \node at (3,0.5) {$t$};
    \end{tikzpicture}
    \right)
    & =
    \begin{tikzpicture}[baseline={([yshift=-.5ex]current bounding box.center)}, scale = \scl]
        \dummyNodes[0.5][-0.1][0.5][1.1]
        \draw[fill = \colOne, opacity = \opac] (0,0) rectangle (6,1);
        \node at (3,0.5) {$t$};
        \draw (6,0) \twotiles \finaledge;
    \end{tikzpicture}
\end{align*}

\textit{Norm:} If $t$ is a tiling of a $1 \times n$ board, then $\norm{t} = n + 1$.

\textit{(1,1)-magma:} The (1,1)-magma begins (sorting by norm) as follows:
\def\scl{0.48}
\vspace{-1em}
\begin{center}
    \begin{tabular}{| l l |}
    \hline
    \textit{Norm 1:} & 
    $\emptyset = \epsilon$ \\
    \hline
    \textit{Norm 2:} & 
    $\begin{tikzpicture}[baseline={([yshift=-.5ex]current bounding box.center)}, scale = \scl]
        \dummyNodes[0.5][-0.1][0.5][1.1]
        \draw (0,0) \onetile \finaledge;
    \end{tikzpicture} 
    = f(\epsilon)$ \\
    \hline
    \textit{Norm 3:} & 
    $\begin{tikzpicture}[baseline={([yshift=-.5ex]current bounding box.center)}, scale = \scl]
        \dummyNodes[0.5][-0.1][0.5][1.1]
        \draw (0,0) \onetile \onetile \finaledge;
    \end{tikzpicture} 
    = f(f(\epsilon))$, 
    \quad
    $\begin{tikzpicture}[baseline={([yshift=-.5ex]current bounding box.center)}, scale = \scl]
        \dummyNodes[0.5][-0.1][0.5][1.1]
        \draw (0,0) \twotiles \finaledge;
    \end{tikzpicture} 
    = g(\epsilon)$ \\
    \hline
    \textit{Norm 4:} & 
    $\begin{tikzpicture}[baseline={([yshift=-.5ex]current bounding box.center)}, scale = \scl]
        \dummyNodes[0.5][-0.1][0.5][1.1]
        \draw (0,0) \onetile \onetile \onetile \finaledge;
    \end{tikzpicture} 
    = f(f(f(\epsilon)))$, 
    \quad
    $\begin{tikzpicture}[baseline={([yshift=-.5ex]current bounding box.center)}, scale = \scl]
        \dummyNodes[0.5][-0.1][0.5][1.1]
        \draw (0,0) \onetile \twotiles \finaledge;
    \end{tikzpicture} 
    = g(f(\epsilon))$, 
    \quad
    $\begin{tikzpicture}[baseline={([yshift=-.5ex]current bounding box.center)}, scale = \scl]
        \dummyNodes[0.5][-0.1][0.5][1.1]
        \draw (0,0) \twotiles \onetile \finaledge;
    \end{tikzpicture} 
    = f(g(\epsilon))$ \\
    \hline
    \textit{Norm 5:} & 
    $\begin{tikzpicture}[baseline={([yshift=-.5ex]current bounding box.center)}, scale = \scl]
        \dummyNodes[0.5][-0.1][0.5][1.1]
        \draw (0,0) \onetile \onetile \onetile \onetile \finaledge;
    \end{tikzpicture}
    = f(f(f(f(\epsilon))))$, 
    \quad
    $\begin{tikzpicture}[baseline={([yshift=-.5ex]current bounding box.center)}, scale = \scl]
        \dummyNodes[0.5][-0.1][0.5][1.1]
        \draw (0,0) \twotiles \onetile \onetile \finaledge;
    \end{tikzpicture} 
    = f(f(g(\epsilon)))$, \\
    & 
    $\begin{tikzpicture}[baseline={([yshift=-.5ex]current bounding box.center)}, scale = \scl]
        \dummyNodes[0.5][-0.1][0.5][1.1]
        \draw (0,0) \onetile \twotiles \onetile \finaledge;
    \end{tikzpicture}
    = f(g(f(\epsilon)))$,
    \ \
    $\begin{tikzpicture}[baseline={([yshift=-.5ex]current bounding box.center)}, scale = \scl]
        \dummyNodes[0.5][-0.1][0.5][1.1]
        \draw (0,0) \onetile \onetile \twotiles \finaledge;
    \end{tikzpicture} 
    = g(f(f(\epsilon)))$, 
    \ \
    $\begin{tikzpicture}[baseline={([yshift=-.5ex]current bounding box.center)}, scale = \scl]
        \dummyNodes[0.5][-0.1][0.5][1.1]
        \draw (0,0) \twotiles \twotiles \finaledge;
    \end{tikzpicture} 
    = g(g(\epsilon))$ \\
    \hline
    \end{tabular}
\end{center}
\end{fibFamily}

\addcontentsline{toca}{subsection}{\fibFamRef{fibFamily:pathGraphMatchings}: Path graph matchings}

\begin{fibFamily}
[Path graph matchings \cite{farrell1986occurrences}] \label{fibFamily:pathGraphMatchings} 

\def\scl{0.6}

$F_{n + 1}$ is the number of matchings in a path graph on $n$ vertices, $P_n$. This is a tree with two nodes of degree 1 and the other $n - 2$ nodes of degree 2.

\textit{Generator:} The empty matching on zero vertices:
\begin{equation*}
    \epsilon = \emptyset.
\end{equation*}

\textit{Unary maps:} Let
$g = 
\begin{tikzpicture}[baseline={([yshift=-.6ex]current bounding box.center)}, scale = \scl]
        \fill (0,0) \halfnd \pathStep \halfnd \pathStep \halfnd ++(2.5,0) \halfnd ;
        \node at (3.25,0) {$\cdots$};
    \end{tikzpicture}$ 
\ be a matching on a path graph on $n$ vertices. Define two unary maps as follows:
\begin{align*}
    f \left( \,
    \begin{tikzpicture}[baseline={([yshift=-.6ex]current bounding box.center)}, scale = \scl]
        \fill (0,0) \halfnd \pathStep \halfnd \pathStep \halfnd ++(2.5,0) \halfnd ;
        \node at (3.25,0) {$\cdots$};
    \end{tikzpicture}
    \, \right)
    \ & = \
    \begin{tikzpicture}[baseline={([yshift=-.6ex]current bounding box.center)}, scale = \scl]
        \fill (0,0) \halfnd \pathStep \halfnd \pathStep \halfnd ++(2.5,0) \halfnd ;
        \node at (3.25,0) {$\cdots$};
        \fill[red] (5.5,0) \halfnd;
    \end{tikzpicture} \\
    g \left( \,
    \begin{tikzpicture}[baseline={([yshift=-.6ex]current bounding box.center)}, scale = \scl]
        \fill (0,0) \halfnd \pathStep \halfnd \pathStep \halfnd \halfnd ++(2.5,0) \halfnd ;
        \node at (3.25,0) {$\cdots$};
    \end{tikzpicture}
    \, \right)
    \ & = \
    \begin{tikzpicture}[baseline={([yshift=-.6ex]current bounding box.center)}, scale = \scl]
        \fill (0,0) \halfnd \pathStep \halfnd \pathStep \halfnd \halfnd ++(2.5,0) \halfnd ;
        \node at (3.25,0) {$\cdots$};
        \fill[red] (5.5,0) \halfnd (6.55,0) \halfnd;
        \draw[red] (5.5,0) -- (6.5,0);
    \end{tikzpicture}
\end{align*}

\textit{Norm:} If $g$ is a matching in a path graph on $n$ vertices, then $ \norm{g} = n + 1$.

\textit{(1,1)-magma:} The (1,1) magma begins (sorting by norm) as follows:
\vspace{-.5em}
\begin{center}
    \begin{tabular}{| l l |}
    \hline
    \textit{Norm 1:} & 
    $\emptyset = \epsilon$ \\
    \hline
    \textit{Norm 2:} & 
    $\begin{tikzpicture}[baseline={([yshift=-.5ex]current bounding box.center)}, scale = \scl]
        \fill (0,0) \halfnd;
    \end{tikzpicture}
    \ = f(\epsilon)$ \\
    \hline
    \textit{Norm 3:} & 
    $\begin{tikzpicture}[baseline={([yshift=-.5ex]current bounding box.center)}, scale = \scl]
        \fill (0,0) \halfnd \pathStep \halfnd;
    \end{tikzpicture}
    \ = f(f(\epsilon)) $, 
    \quad
    $\begin{tikzpicture}[baseline={([yshift=-.5ex]current bounding box.center)}, scale = \scl]
        \fill (0,0) \halfnd \pathStep \halfnd;
        \draw (0,0) \pathEdge;
    \end{tikzpicture}
    \ = g(\epsilon)$ \\
    \hline
    \textit{Norm 4:} & 
    $\begin{tikzpicture}[baseline={([yshift=-.5ex]current bounding box.center)}, scale = \scl]
        \fill (0,0) \halfnd \pathStep \halfnd \pathStep \halfnd;
    \end{tikzpicture}
    \ = f(f(f(\epsilon)))$, 
    \quad
    $\begin{tikzpicture}[baseline={([yshift=-.5ex]current bounding box.center)}, scale = \scl]
        \fill (0,0) \halfnd \pathStep \halfnd \pathStep \halfnd;
        \draw (0,0) \pathEdge \pathNoEdge;
    \end{tikzpicture}
    \ = f(g(\epsilon))$, \\
    &
    $\begin{tikzpicture}[baseline={([yshift=-.5ex]current bounding box.center)}, scale = \scl]
        \fill (0,0) \halfnd \pathStep \halfnd \pathStep \halfnd;
        \draw (0,0) \pathNoEdge \pathEdge;
    \end{tikzpicture}
    \ = g(f(\epsilon))$ \\
    \hline
    \textit{Norm 5:} & 
    $\begin{tikzpicture}[baseline={([yshift=-.5ex]current bounding box.center)}, scale = \scl]
        \fill (0,0) \halfnd \pathStep \halfnd \pathStep \halfnd \pathStep \halfnd;
    \end{tikzpicture}
    \ = f(f(f(f(\epsilon))))$, 
    \quad
    $\begin{tikzpicture}[baseline={([yshift=-.5ex]current bounding box.center)}, scale = \scl]
        \fill (0,0) \halfnd \pathStep \halfnd \pathStep \halfnd \pathStep \halfnd;
        \draw (0,0) \pathEdge;
    \end{tikzpicture}
    \ = f(f(g(\epsilon)))$, \\
    &
    $\begin{tikzpicture}[baseline={([yshift=-.5ex]current bounding box.center)}, scale = \scl]
        \fill (0,0) \halfnd \pathStep \halfnd \pathStep \halfnd \pathStep \halfnd;
        \draw (0,0) \pathNoEdge \pathEdge;
    \end{tikzpicture}
    \ = f(g(f(\epsilon)))$, 
    \quad
    $\begin{tikzpicture}[baseline={([yshift=-.5ex]current bounding box.center)}, scale = \scl]
        \fill (0,0) \halfnd \pathStep \halfnd \pathStep \halfnd \pathStep \halfnd;
        \draw (0,0) \pathNoEdge \pathNoEdge \pathEdge;
    \end{tikzpicture}
    \ = g(f(f(\epsilon)))$, \\
    &
    $\begin{tikzpicture}[baseline={([yshift=-.5ex]current bounding box.center)}, scale = \scl]
        \fill (0,0) \halfnd \pathStep \halfnd \pathStep \halfnd \pathStep \halfnd;
        \draw (0,0) \pathEdge \pathNoEdge \pathEdge;
    \end{tikzpicture}
    \ = g(g(\epsilon))$ \\
    \hline
    \end{tabular}
\end{center}
\end{fibFamily}

\addcontentsline{toca}{subsection}{\fibFamRef{fibFamily:ladderGraphMatchings}: Perfect matchings in a ladder graph}
\begin{fibFamily}
[Perfect matchings in a ladder graph \cite{FibonacciOEIS}] \label{fibFamily:ladderGraphMatchings} 

\def\scl{0.6}

$F_{n + 1}$ is the number of perfect matchings in the ladder graph $L_n = P_2 \times P_n$.

\textit{Generator:} The empty matching in the ladder graph on zero vertices:
\begin{equation*}
    \epsilon = \emptyset.
\end{equation*}

\textit{Unary maps:}
\begin{align*}
    f \left(
    \begin{tikzpicture}[baseline={([yshift=-.5ex]current bounding box.center)}, scale = \scl]
        \dummyNodes[0.5][0.2][0.5][-3.2]
        \fill (0,0) \halfnd \ladderSS \halfnd \ladderNodes \ladderNodes ++(0,-0.75) \ladderNodes;
        \node at (0.5,-2.25) {$\cdots$};
    \end{tikzpicture}
    \right)
    \ = \
    \begin{tikzpicture}[baseline={([yshift=-.5ex]current bounding box.center)}, scale = \scl]
        \fill (0,0) \halfnd \ladderSS \halfnd \ladderNodes \ladderNodes ++(0,-0.75) \ladderNodes;
        \fill[red] (0,-3.75) \halfnd \ladderSS \halfnd;
        \draw[red] (0,-3.75) \hor;
        \node at (0.5,-2.25) {$\cdots$};
    \end{tikzpicture} \qquad \qquad \qquad
    g \left(
    \begin{tikzpicture}[baseline={([yshift=-.5ex]current bounding box.center)}, scale = \scl]
        \dummyNodes[0.5][0.2][0.5][-3.2]
        \fill (0,0) \halfnd \ladderSS \halfnd \ladderNodes \ladderNodes ++(0,-0.75) \ladderNodes;
        \node at (0.5,-2.25) {$\cdots$};
    \end{tikzpicture}
    \right)
    \ = \
    \begin{tikzpicture}[baseline={([yshift=-.5ex]current bounding box.center)}, scale = \scl]
        \fill (0,0) \halfnd \ladderSS \halfnd \ladderNodes \ladderNodes ++(0,-0.75) \ladderNodes;
        \fill[red] (0,-3.75) \halfnd \ladderSS \halfnd \ladderNodes;
        \draw[red] (0,-3.75) \ver;
        \node at (0.5,-2.25) {$\cdots$};
    \end{tikzpicture}
\end{align*}

\textit{Norm:} If $M$ is a perfect matching in $L_n$, then $\norm{M} = n + 1$.

\textit{(1,1)-magma:} The (1,1) magma begins (sorting by norm) as follows:
\vspace{-.5em}
\def\scl{0.5}
\begin{center}
    \begin{tabular}{| l l |}
    \hline
    \textit{Norm 1:} & 
    $\emptyset = \epsilon$ \\
    \hline
    \textit{Norm 2:} & 
    $\begin{tikzpicture}[baseline={([yshift=-.5ex]current bounding box.center)}, scale = \scl]
        \dummyNodes[0.5][0.2][0.5][-0.5]
        \fill (0,0) \halfnd \ladderSS \halfnd;
        \draw (0,0) \finalhor;
    \end{tikzpicture}
    \ = f(\epsilon)$ \\
    \hline
    \textit{Norm 3:} & 
    $\begin{tikzpicture}[baseline={([yshift=-.5ex]current bounding box.center)}, scale = \scl]
        \dummyNodes[0.5][0.2][0.5][-1.25]
        \fill (0,0) \halfnd \ladderSS \halfnd \ladderNodes;
        \draw (0,0) \hor \finalhor;
    \end{tikzpicture}
    \ = f(f(\epsilon)) $, 
    \quad
    $\begin{tikzpicture}[baseline={([yshift=-.5ex]current bounding box.center)}, scale = \scl]
        \dummyNodes[0.5][0.2][0.5][-1.25]
        \fill (0,0) \halfnd \ladderSS \halfnd \ladderNodes;
        \draw (0,0) \finalver;
    \end{tikzpicture}
    \ = g(\epsilon)$ \\
    \hline
    \textit{Norm 4:} & 
    $\begin{tikzpicture}[baseline={([yshift=-.5ex]current bounding box.center)}, scale = \scl]
        \dummyNodes[0.5][0.2][0.5][-2]
        \fill (0,0) \halfnd \ladderSS \halfnd \ladderNodes \ladderNodes;
        \draw (0,0) \hor \hor \finalhor;
    \end{tikzpicture}
    \ = f(f(f(\epsilon)))$, 
    \quad
    $\begin{tikzpicture}[baseline={([yshift=-.5ex]current bounding box.center)}, scale = \scl]
        \dummyNodes[0.5][0.2][0.5][-2]
        \fill (0,0) \halfnd \ladderSS \halfnd \ladderNodes \ladderNodes;
        \draw (0,0) \ver \finalhor;
    \end{tikzpicture}
    \ = f(g(\epsilon))$, 
    \quad
    $\begin{tikzpicture}[baseline={([yshift=-.5ex]current bounding box.center)}, scale = \scl]
        \dummyNodes[0.5][0.2][0.5][-2]
        \fill (0,0) \halfnd \ladderSS \halfnd \ladderNodes \ladderNodes;
        \draw (0,0) \hor \finalver;
    \end{tikzpicture}
    \ = g(f(\epsilon))$ \\
    \hline
    \textit{Norm 5:} & 
    $\begin{tikzpicture}[baseline={([yshift=-.5ex]current bounding box.center)}, scale = \scl]
        \dummyNodes[0.5][0.2][0.5][-3]
        \fill (0,0) \halfnd \ladderSS \halfnd \ladderNodes \ladderNodes \ladderNodes;
        \draw (0,0) \hor \hor \hor \finalhor;
    \end{tikzpicture}
    \ = f(f(f(f(\epsilon))))$, 
    \quad
    $\begin{tikzpicture}[baseline={([yshift=-.5ex]current bounding box.center)}, scale = \scl]
        \dummyNodes[0.5][0.2][0.5][-3]
        \fill (0,0) \halfnd \ladderSS \halfnd \ladderNodes \ladderNodes \ladderNodes;
        \draw (0,0) \ver \hor \finalhor;
    \end{tikzpicture}
    \ = f(f(g(\epsilon)))$, \\
    &
    $\begin{tikzpicture}[baseline={([yshift=-.5ex]current bounding box.center)}, scale = \scl]
        \dummyNodes[0.75][0.2][0.5][-3]
        \fill (0,0) \halfnd \ladderSS \halfnd \ladderNodes \ladderNodes \ladderNodes;
        \draw (0,0) \hor \ver \finalhor;
    \end{tikzpicture}
    \ = f(g(f(\epsilon)))$,
    \quad
    $\begin{tikzpicture}[baseline={([yshift=-.5ex]current bounding box.center)}, scale = \scl]
        \dummyNodes[0.75][0.2][0.5][-3]
        \fill (0,0) \halfnd \ladderSS \halfnd \ladderNodes \ladderNodes \ladderNodes;
        \draw (0,0) \hor \hor \finalver;
    \end{tikzpicture}
    \ = g(f(f(\epsilon)))$,
    \quad
    $\begin{tikzpicture}[baseline={([yshift=-.5ex]current bounding box.center)}, scale = \scl]
        \dummyNodes[0.75][0.2][0.5][-2.75]
        \fill (0,0) \halfnd \ladderSS \halfnd \ladderNodes \ladderNodes \ladderNodes;
        \draw (0,0) \ver \finalver;
    \end{tikzpicture}
    \ = g(g(\epsilon))$ \\
    \hline
    \end{tabular}
\end{center}
\end{fibFamily}

\addcontentsline{toca}{subsection}{\fibFamRef{fibFamily:compositionsNo1s}: Compositions with no 1's}
\begin{fibFamily}
[Compositions with no 1's \cite{Stanley:2011:ECV:2124415}] \label{fibFamily:compositionsNo1s}

$F_n$ is the number of compositions of $n + 1$ with no part equal to 1. A composition of an integer $n$ is a way of writing $n$ as the sum of a sequence of strictly positive integers. Two sequences that differ in the order of their terms define different compositions of their sum.

\textit{Generator:} The composition of 2 into one part:
\begin{equation*}
    \epsilon = 2.
\end{equation*}

\textit{Unary maps:} Let $\alpha_1 + \cdots + \alpha_k$ be a composition of $n$, where \mbox{$\alpha_i \in \mathbb{N} \backslash \{ 1 \}$} for each $i \in \{ 1, \hdots, k \}$. Then define two unary maps as follows:
\begin{align*}
    f(\alpha_1 + \cdots + \alpha_k) & = \alpha_1 + \cdots + (\alpha_k + 1), \\
    g(\alpha_1 + \cdots + \alpha_k) & = \alpha_1 + \cdots + \alpha_k + 2.
\end{align*}

\textit{Norm:} If $c$ is a composition of the integer $n$, then $\norm{c} = n - 1$.

\textit{(1,1)-magma:} The (1,1) magma begins (sorting by norm) as follows:
\vspace{-.5em}
\begin{center}
    \begin{tabular}{| l l |}
    \hline
    \textit{Norm 1:} & 
    $2 = \epsilon$ \\
    \hline
    \textit{Norm 2:} & 
    $3 = f(\epsilon)$ \\
    \hline
    \textit{Norm 3:} 
    & $4 = f(f(\epsilon)) $, \quad
    $2 + 2 = g(\epsilon)$ \\
    \hline
    \textit{Norm 4:} & 
    $5 = f(f(f(\epsilon)))$, \quad
    $2 + 3 = f(g(\epsilon))$, \quad
    $3 + 2 = g(f(\epsilon))$ \\
    \hline
    \textit{Norm 5:} & 
    $6 = f(f(f(f(\epsilon))))$, \quad
    $2 + 4 = f(f(g(\epsilon)))$, \quad
    $3 + 3 = f(g(f(\epsilon)))$, \\ &
    $4 + 2 = g(f(f(\epsilon)))$, \quad
    $2 + 2 + 2 = g(g(\epsilon))$ \\
    \hline
    \end{tabular}
\end{center}
\end{fibFamily}

\addcontentsline{toca}{subsection}{\fibFamRef{fibFamily:compositionsLeq2}: Compositions with no part greater than 2}
\begin{fibFamily}
[Compositions with no part greater than 2 \cite{moser2}] \label{fibFamily:compositionsLeq2}

$F_n$ is the number of compositions of $n - 1$ with no part greater than 2.

\textit{Generator:} The empty composition:
\begin{equation*}
    \epsilon = \emptyset.
\end{equation*}

\textit{Unary maps:} Let $c$ be a composition, and define two unary maps as follows:
\begin{align*}
    f(c) & = c + 1, \\
    g(c) & = c + 2.
\end{align*}

\textit{Norm:} If $c$ is a composition of the integer $n$, then $\norm{c} = n + 1$.

\textit{(1,1)-magma:} The (1,1) magma begins (sorting by norm) as follows:
\vspace{-.5em}
\begin{center}
    \begin{tabular}{| l l |}
    \hline
    \textit{Norm 1:} & 
    $\emptyset = \epsilon$ \\
    \hline
    \textit{Norm 2:} & 
    $1 = f(\epsilon)$ \\
    \hline
    \textit{Norm 3:} & 
    $1 + 1 = f(f(\epsilon)) $, \quad
    $2 = g(\epsilon)$ \\
    \hline
    \textit{Norm 4:} & 
    $1 + 1 + 1 = f(f(f(\epsilon)))$, \quad
    $2 + 1 = f(g(\epsilon))$, \\ &
    $1 + 2 = g(f(\epsilon))$ \\
    \hline
    \textit{Norm 5:} & 
    $1 + 1 + 1 + 1 = f(f(f(f(\epsilon))))$, \quad
    $2 + 1 + 1 = f(f(g(\epsilon)))$, \\ &
    $1 + 2 + 1 = f(g(f(\epsilon)))$, \quad 
    $1 + 1 + 2 = g(f(f(\epsilon)))$, \\ & 
    $2 + 2  = g(g(\epsilon))$ \\
    \hline
    \end{tabular}
\end{center}
\end{fibFamily}

\addcontentsline{toca}{subsection}{\fibFamRef{fibFamily:compositionsOdd}: Compositions using odd parts}
\begin{fibFamily}
[Compositions using odd parts \cite{Stanley:2011:ECV:2124415}] \label{fibFamily:compositionsOdd}

$F_n$ is the number of compositions of $n$ into odd parts.

\textit{Generator:} The composition of 1 into one part:
\begin{equation*}
    \epsilon = 1.
\end{equation*}

\textit{Unary maps:} Let $\alpha_1 + \hdots + \alpha_k$ be a composition, and define two unary maps as follows:
\begin{align*}
    f(\alpha_1 + \hdots + \alpha_k) & = \alpha_1 + \hdots + \alpha_k + 1, \\
    g(\alpha_1 + \hdots + \alpha_k) & = \alpha_1 + \hdots + (\alpha_k + 2).
\end{align*}

\textit{Norm:} If $c$ is a composition of the integer $n$, then $\norm{c} = n$.

\pagebreak
\textit{(1,1)-magma:} The (1,1) magma begins (sorting by norm) as follows:
\vspace{-.5em}
\begin{center}
    \begin{tabular}{| l l |}
    \hline
    \textit{Norm 1:} & 
    $1 = \epsilon$ \\
    \hline
    \textit{Norm 2:} & 
    $1 + 1 = f(\epsilon)$ \\
    \hline
    \textit{Norm 3:} & 
    $1 + 1 + 1 = f(f(\epsilon)) $, \quad
    $3 = g(\epsilon)$ \\
    \hline
    \textit{Norm 4:} & 
    $1 + 1 + 1 + 1 = f(f(f(\epsilon)))$, \quad
    $3 + 1 = f(g(\epsilon))$, \quad
    $1 + 3 = g(f(\epsilon))$ \\
    \hline
    \textit{Norm 5:} & 
    $1 + 1 + 1 + 1 + 1 = f(f(f(f(\epsilon))))$, \quad
     $3 + 1 + 1 = f(f(g(\epsilon)))$, \\ &
     $1 + 3 + 1 = f(g(f(\epsilon)))$, \quad
    $1 + 1 + 3 = g(f(f(\epsilon)))$, \quad
    $5  = g(g(\epsilon))$ \\
    \hline
    \end{tabular}
\end{center}
\end{fibFamily}

\addcontentsline{toca}{subsection}{\fibFamRef{fibFamily:binaryWordsOddRuns}: Binary words with odd run lengths}
\begin{fibFamily}
[Binary words with odd run lengths \cite{FibonacciOEIS}] \label{fibFamily:binaryWordsOddRuns}

$F_n$ is the number of binary words (words in the alphabet $\{ 0, 1 \}$) of length $n$ beginning with 0 and having all run lengths odd. A run is a subword containing only 0's or only 1's which is maximal (meaning that we cannot extend the subword and still have the property that it contains only 0's or only 1's).

\textit{Generator:} 
\begin{equation*}
    \epsilon = 0.
\end{equation*}

\textit{Unary maps:} Let $w$ be a binary word. Define two unary maps as follows:
\begin{align*}
    f(w) & = 
    \begin{cases}
        w0, & \text{if length of $w$ is even}, \\
        w1, & \text{if length of $w$ is odd}, \\
    \end{cases} \\
    g(w) & = 
    \begin{cases}
        w11, & \text{if length of $w$ is even}, \\
        w00, & \text{if length of $w$ is odd}. \\
    \end{cases}
\end{align*}

\textit{Norm:} If $w$ is a binary word of length $n$, then $\norm{w} = n$.

\textit{(1,1)-magma:} The (1,1) magma begins (sorting by norm) as follows:
\vspace{-.5em}
\begin{center}
    \begin{tabular}{| l l |}
    \hline
    \textit{Norm 1:} & 
    $0 = \epsilon$ \\
    \hline
    \textit{Norm 2:} & 
    $01 = f(\epsilon)$ \\
    \hline
    \textit{Norm 3:} & 
    $010 = f(f(\epsilon)) $, \quad
    $000 = g(\epsilon)$ \\
    \hline
    \textit{Norm 4:} & 
    $0101 = f(f(f(\epsilon)))$, \quad
    $0001 = f(g(\epsilon))$, \quad
    $0111 = g(f(\epsilon))$ \\
    \hline
    \textit{Norm 5:} & 
    $01010 = f(f(f(f(\epsilon))))$, \quad
    $00010 = f(f(g(\epsilon)))$, \quad
    $01110 = f(g(f(\epsilon)))$, \\ &
    $01000 = g(f(f(\epsilon)))$, \quad
    $00000  = g(g(\epsilon))$ \\
    \hline
    \end{tabular}
\end{center}
\end{fibFamily}

\addcontentsline{toca}{subsection}{\fibFamRef{fibFamily:permutations}: Permutations with $\vert p_k - k \vert \leq 1$}
\begin{fibFamily}
[Permutations with $\bm{\vert p_k - k \vert \leq 1}$ \cite{SIMION1985383}] \label{fibFamily:permutations}

$F_{n + 1}$ is the number of permutations $p_1 p_2 \cdots p_n$ of $\{ 1, \hdots, n \}$ such that
\begin{align*}
    \vert p_k - k \vert \leq 1, \qquad k = 1, \hdots, n.
\end{align*}

\textit{Generator:} The empty permutation ($n=0$):
\begin{equation*}
    \epsilon = \emptyset.
\end{equation*}

\textit{Unary maps:} Let $p_1 p_2 \cdots p_n$ be a permutation of $\{ 1, \hdots, n \}$. Define two unary maps as follows:
\begin{align*}
    f(p_1 p_2 \cdots p_n) & = p_1 p_2 \cdots p_n (n + 1), \\
    g(p_1 p_2 \cdots p_n) & = p_1 p_2 \cdots p_n (n + 2) (n + 1).
\end{align*}

\textit{Norm:} If $p$ is a permutation of $\{ 1, \hdots, n \}$, then $\norm{p} = n + 1$.

\textit{(1,1)-magma:} The (1,1) magma begins (sorting by norm) as follows:
\vspace{-.5em}
\begin{center}
    \begin{tabular}{| l l |}
    \hline
    \textit{Norm 1:} & 
    $\emptyset = \epsilon$ \\
    \hline
    \textit{Norm 2:} & 
    $1 = f(\epsilon)$ \\
    \hline
    \textit{Norm 3:} & 
    $12 = f(f(\epsilon)) $, \quad
    $21 = g(\epsilon)$ \\
    \hline
    \textit{Norm 4:} & 
    $123 = f(f(f(\epsilon)))$, \quad
    $213 = f(g(\epsilon))$, \quad
    $132 = g(f(\epsilon))$ \\
    \hline
    \textit{Norm 5:} & 
    $1234 = f(f(f(f(\epsilon))))$, \quad
    $2134 = f(f(g(\epsilon)))$, \quad
    $1324 = f(g(f(\epsilon)))$, \quad \\ & 
    $1243 = g(f(f(\epsilon)))$, \quad
    $2143 = g(g(\epsilon))$ \\
    \hline
    \end{tabular}
\end{center}
\end{fibFamily}

\addcontentsline{toca}{subsection}{\fibFamRef{fibFamily:binarySeqs}: Binary sequences with no consecutive 1's}
\begin{fibFamily}
[Binary sequences with no consecutive 1's \cite{nyblom2012enumerating}] \label{fibFamily:binarySeqs}

$F_{n + 2}$ is equal to the number of binary sequences (words in the alphabet $\{0,1\}$) of length $n$ that have no consecutive 1's.

\textit{Generator:} For this family there does not exist a natural representation for the generator $\epsilon$. This is due to the fact that norm $n$ objects correspond to words of length $n - 2$ and thus there is no natural way to describe a norm 1 object. For this reason, we describe the (1,1)-magma corresponding to this Fibonacci family in the manner described at the end of Section \ref{section:fib11magmas}. We define the generator simply to be $\epsilon$ (with no further meaning associated with it) and define the following objects:
\begin{align*}
    f(\epsilon) & = \emptyset \quad \text{(the empty word)}, \\
    g(\epsilon) & = 1.
\end{align*}
Having defined these objects along with the unary and binary maps which follow, we have thus completely specified the relevant (1,1)-magma.

\textit{Unary maps:} Let $w$ be a binary sequence of length $n$ containing no consecutive 1's. Then define two unary maps as follows:
\begin{align*}
    f(w) & = w 0, \\
    g(w) & = w 0 1.
\end{align*}

\textit{Norm:} If $w$ is a binary sequence of length $n$, then $\norm{w} = n + 2$.

\textit{(1,1)-magma:} The (1,1) magma begins (sorting by norm) as follows:
\vspace{-.5em}
\begin{center}
    \begin{tabular}{| l l |}
    \hline
    \textit{Norm 1:} & 
    $\epsilon$ \\
    \hline
    \textit{Norm 2:} & 
    $\emptyset = f(\epsilon)$ \\
    \hline
    \textit{Norm 3:} & 
    $0 = f(f(\epsilon))$, \quad
    $1 = g(\epsilon)$ \\
    \hline
    \textit{Norm 4:} & 
    $00 = f(f(f(\epsilon)))$, \quad
    $10 = f(g(\epsilon))$, \quad
    $01 = g(f(\epsilon))$ \\
    \hline
    \textit{Norm 5:} & 
    $000 = f(f(f(f(\epsilon))))$, \quad
    $010 = f(g(f(\epsilon)))$, \quad
    $100 = f(f(g(\epsilon)))$, \\
    & 
    $001 = g(f(f(\epsilon)))$, \quad
    $101 = g(g(\epsilon))$ \\
    \hline
    \end{tabular}
\end{center}
\end{fibFamily}

\addcontentsline{toca}{subsection}{\fibFamRef{fibFamily:glassReflections}: Reflections across two glass plates}
\begin{fibFamily}
[Reflections across two glass plates \cite{HOGGATT10016530370, moser}] \label{fibFamily:glassReflections} 

\def\scl{0.3}

$F_{n+2}$ is equal to the number of paths through two plates of glass with $n$ reflections (where reflections can occur at plate/plate or plate/air interfaces). These are represented schematically as in \cite{HOGGATT10016530370}.

\textit{Generator:} We represent the generator as follows:
\begin{equation*}
    \epsilon =
    \begin{tikzpicture}[baseline={([yshift=-.5ex]current bounding box.center)}, scale = \scl]
        \plates[3]
    \end{tikzpicture}
\end{equation*}

\textit{Unary maps:}
For the generator $\epsilon$,
\begin{equation*}
    f(\epsilon) = 
    \begin{tikzpicture}[baseline={([yshift=-.5ex]current bounding box.center)}, scale = \scl]
        \plates[5]
        \glassDummyNodes
        \draw[red, ->] (0.5,1) -- (4.5,-3);
    \end{tikzpicture}
    \qquad \qquad \qquad
    g(\epsilon) = 
    \begin{tikzpicture}[baseline={([yshift=-.5ex]current bounding box.center)}, scale = \scl]
        \plates[5]
        \glassDummyNodes
        \draw[red, ->] (0.5,1) -- (2.5,-1) -- (4.5,1);
    \end{tikzpicture}
\end{equation*}

For all elements of the base set other than $\epsilon$, the unary maps can be illustrated schematically as follows:
\begin{align*}
    f \left(
    \begin{tikzpicture}[baseline={([yshift=-.5ex]current bounding box.center)}, decoration = {snake, segment length = 2.3mm, amplitude = 0.3mm}, scale = \scl]
        \plates[5]
        \glassDummyNodes
        \path [decorate, draw = red] (1.5,0) -- (3.5,-2);
        \draw[red, ->] (0.5,1) -- (1.5,0) (3.5,-2) -- (4.5,-3);
    \end{tikzpicture}
    \right)
    \ = \
    \begin{tikzpicture}[baseline={([yshift=-.5ex]current bounding box.center)}, decoration = {snake, segment length = 2.3mm, amplitude = 0.3mm}, scale = \scl]
        \plates[7]
        \glassDummyNodes
        \path [decorate, draw = red] (1.5,0) -- (3.5,-2);
        \draw [red, ->] (0.5,1) -- (1.5,0) (3.5,-2) -- (6.5,1);
    \end{tikzpicture},
    & \qquad
    f \left(
    \begin{tikzpicture}[baseline={([yshift=-.5ex]current bounding box.center)}, decoration = {snake, segment length = 2.3mm, amplitude = 0.3mm}, scale = \scl]
        \plates[5]
        \glassDummyNodes
        \path [decorate, draw = red] (1.5,0) to [bend right = 30] (2.5,-1.5) to [bend right = 30] (3.5,0);
        \draw[red, ->] (0.5,1) -- (1.5,0) (3.5,0) -- (4.5,1);
    \end{tikzpicture}
    \right)
    \ = \
    \begin{tikzpicture}[baseline={([yshift=-.5ex]current bounding box.center)}, decoration = {snake, segment length = 2.3mm, amplitude = 0.3mm}, scale = \scl]
        \plates[7]
        \glassDummyNodes
        \path [decorate, draw = red] (1.5,0) to [bend right = 30] (2.5,-1.5) to [bend right = 30] (3.5,0);
        \draw[red, ->] (0.5,1) -- (1.5,0) (3.5,0) -- (6.5,-3);
    \end{tikzpicture}, \\
    g \left(
    \begin{tikzpicture}[baseline={([yshift=-.5ex]current bounding box.center)}, decoration = {snake, segment length = 2.3mm, amplitude = 0.3mm}, scale = \scl]
        \plates[5]
        \glassDummyNodes
        \path [decorate, draw = red] (1.5,0) -- (3.5,-2);
        \draw[red, ->] (0.5,1) -- (1.5,0) (3.5,-2) -- (4.5,-3);
    \end{tikzpicture}
    \right)
    \ = \
    \begin{tikzpicture}[baseline={([yshift=-.5ex]current bounding box.center)}, decoration = {snake, segment length = 2.3mm, amplitude = 0.3mm}, scale = \scl]
        \plates[7]
        \glassDummyNodes
        \path [decorate, draw = red] (1.5,0) -- (3.5,-2);
        \draw [red, ->] (0.5,1) -- (1.5,0) (3.5,-2) -- (4.5,-1) -- (6.5,-3);
    \end{tikzpicture}, 
    & \qquad 
    g \left(
    \begin{tikzpicture}[baseline={([yshift=-.5ex]current bounding box.center)}, decoration = {snake, segment length = 2.3mm, amplitude = 0.3mm}, scale = \scl]
        \plates[5]
        \glassDummyNodes
        \path [decorate, draw = red] (1.5,0) to [bend right = 30] (2.5,-1.5) to [bend right = 30] (3.5,0);
        \draw[red, ->] (0.5,1) -- (1.5,0) (3.5,0) -- (4.5,1);
    \end{tikzpicture}
    \right)
    \ = \
    \begin{tikzpicture}[baseline={([yshift=-.5ex]current bounding box.center)}, decoration = {snake, segment length = 2.3mm, amplitude = 0.3mm}, scale = \scl]
        \plates[7]
        \glassDummyNodes
        \path [decorate, draw = red] (1.5,0) to [bend right = 30] (2.5,-1.5) to [bend right = 30] (3.5,0);
        \draw[red, ->] (0.5,1) -- (1.5,0) (3.5,0) -- (4.5,-1) -- (6.5,1);
    \end{tikzpicture}. \\
\end{align*}
The map $f$ simply adds one reflection by taking the exiting ray and reflecting it as it leaves the bottom plate. The map $g$ adds two reflections by taking the exiting ray and reflecting it twice, with the second reflection occurring at the centre plate/plate interface.
Thus $f$ changes the direction the ray exits the plates whilst $g$ does not change the direction.

\textit{Norm:} If $p$ is a path with $n$ reflections, then $\norm{p} = n + 2$.

\textit{(1,1)-magma:} The (1,1) magma begins (sorting by norm) as follows: \def \scl {0.3}
\vspace{-.5em}
\begin{center}
    \begin{tabular}{| l l |}
    \hline
    \textit{Norm 1:} & 
    $\begin{tikzpicture}[baseline={([yshift=-.5ex]current bounding box.center)}, scale = \scl]
        \plates[3]
        \glassDummyNodes
    \end{tikzpicture}
    = \epsilon$ \\
    \hline
    \textit{Norm 2:} & 
    $\begin{tikzpicture}[baseline={([yshift=-.5ex]current bounding box.center)}, scale = \scl]
        \plates[3]
        \glassDummyNodes
        \draw[red, ->] (0.5,1) -- (2.5,-3);
    \end{tikzpicture}
    = f(\epsilon)$ \\
    \hline
    \textit{Norm 3:} & 
    $\begin{tikzpicture}[baseline={([yshift=-.5ex]current bounding box.center)}, scale = \scl]
        \plates[4]
        \glassDummyNodes
        \draw[red] (1,0) \glassDn \glassDn \glassUp \glassUp;
        \draw[red, ->] (0.5,1) -- (1,0) (3,0) -- (3.5,1);
    \end{tikzpicture}
    = f(f(\epsilon)) $, \quad
    $\begin{tikzpicture}[baseline={([yshift=-.5ex]current bounding box.center)}, scale = \scl]
        \plates[3]
        \glassDummyNodes
        \draw[red, ->] (0.5,1) -- (1.5,-1) -- (2.5,1);
    \end{tikzpicture}
    = g(\epsilon)$ \\
    \hline
    \textit{Norm 4:} & 
    $\begin{tikzpicture}[baseline={([yshift=-.5ex]current bounding box.center)}, scale = \scl]
        \plates[5]
        \glassDummyNodes
        \draw[red] (1,0) \glassDn \glassDn \glassUp \glassUp \glassDn \glassDn;
        \draw[red, ->] (0.5,1) -- (1,0) (4,-2) -- (4.5,-3);
    \end{tikzpicture}
    = f(f(f(\epsilon)))$, \quad
    $\begin{tikzpicture}[baseline={([yshift=-.5ex]current bounding box.center)}, scale = \scl]
        \plates[4]
        \glassDummyNodes
        \draw[red] (1,0) \glassDn \glassUp \glassDn \glassDn;
        \draw[red, ->] (0.5,1) -- (1,0) (3,-2) -- (3.5,-3);
    \end{tikzpicture}
    = f(g(\epsilon))$, \quad
    $\begin{tikzpicture}[baseline={([yshift=-.5ex]current bounding box.center)}, scale = \scl]
        \plates[4]
        \glassDummyNodes
        \draw[red] (1,0) \glassDn \glassDn \glassUp \glassDn;
        \draw[red, ->] (0.5,1) -- (1,0) (3,-2) -- (3.5,-3);
    \end{tikzpicture}
    = g(f(\epsilon))$ \\
    \hline
    \textit{Norm 5:} & 
    $\begin{tikzpicture}[baseline={([yshift=-.5ex]current bounding box.center)}, scale = \scl]
        \plates[6]
        \glassDummyNodes
        \draw[red] (1,0) \glassDn \glassDn \glassUp \glassUp \glassDn \glassDn \glassUp \glassUp;
        \draw[red, ->] (0.5,1) -- (1,0) (5,0) -- (5.5,1);
    \end{tikzpicture}
    = f(f(f(f(\epsilon))))$, \quad
    $\begin{tikzpicture}[baseline={([yshift=-.5ex]current bounding box.center)}, scale = \scl]
        \plates[5]
        \glassDummyNodes
        \draw[red] (1,0) \glassDn \glassDn \glassUp \glassUp \glassDn \glassUp;
        \draw[red, ->] (0.5,1) -- (1,0) (4,0) -- (4.5,1);
    \end{tikzpicture}
    = g(f(f(\epsilon)))$, \\
    &
    $\begin{tikzpicture}[baseline={([yshift=-.5ex]current bounding box.center)}, scale = \scl]
        \plates[5]
        \glassDummyNodes
        \draw[red] (1,0) \glassDn \glassDn \glassUp \glassDn \glassUp \glassUp;
        \draw[red, ->] (0.5,1) -- (1,0) (4,0) -- (4.5,1);
    \end{tikzpicture}
    = f(g(f(\epsilon)))$, \quad 
    $\begin{tikzpicture}[baseline={([yshift=-.5ex]current bounding box.center)}, scale = \scl]
        \plates[5]
        \glassDummyNodes
        \draw[red] (1,0) \glassDn \glassUp \glassDn \glassDn \glassUp \glassUp;
        \draw[red, ->] (0.5,1) -- (1,0) (4,0) -- (4.5,1);
    \end{tikzpicture} = f(f(g(\epsilon)))$, \\
    &
    $\begin{tikzpicture}[baseline={([yshift=-.5ex]current bounding box.center)}, scale = \scl]
        \plates[4]
        \glassDummyNodes
        \draw[red] (1,0) \glassDn \glassUp \glassDn \glassUp;
        \draw[red, ->] (0.5,1) -- (1,0) (3,0) -- (3.5,1);
    \end{tikzpicture}
    = g(g(\epsilon))$ \\
    \hline
    \end{tabular}
\end{center}
\end{fibFamily}

\addcontentsline{toca}{subsection}{\fibFamRef{fibFamily:subsets}: Subsets with no consecutive integers}
\begin{fibFamily}
[Subsets with no consecutive integers \cite{Stanley:2011:ECV:2124415}] \label{fibFamily:subsets}

$F_{n + 2}$ is equal to the number of subsets of $\{ 1, 2, \hdots, n \}$ that contain no consecutive integers. 

To distinguish between two equal subsets that arise as subsets of two different sized sets we consider  \fibFamRef{fibFamily:subsets} to be pairs $(n,S)$ where $S \subseteq \{ 1, 2, \hdots, n \}$. 
Thus $(n,S)$ and $(m,S)$ are only equal if $n=m$ (even though both subsets are the same).


\textit{Generator:} For this family there does not exist a natural representation for the generator $\epsilon$. This is due simply to the nature of the family, since norm $n$ objects correspond to subsets of $\{ 1, 2, \hdots, n - 2 \}$ and thus there is no natural way to describe a norm 1 object. For this reason, we describe the (1,1)-magma corresponding to this Fibonacci family in the manner described at the end of Section \ref{section:fib11magmas}. We define the generator simply to be $\epsilon$ (with no further meaning associated with it) and define the following objects:
\begin{align*}
    f(\epsilon) & = (0,\emptyset), \\
    g(\epsilon) & = (1,\{ 1 \}).
\end{align*}
Having defined these objects along with the unary and binary maps which follow, we have thus completely specified the relevant (1,1)-magma.

\textit{Unary maps:} Let $S$ be a subset of $\{ 1, 2, \hdots, n \}$ containing no consecutive integers. Define two unary maps as follows:
\begin{align*}
    f(n,S) & = (n+1,S)  \\
    g(n,S) & = (n+2,S \cup \{ n + 2 \}) .
\end{align*}

\textit{Norm:}   $\norm{(n,S)} = n + 2$.

\textit{(1,1)-magma:} The (1,1) magma begins (sorting by norm) as follows:
\vspace{-.5em}
\begin{center}
    \begin{tabular}{| l l |}
    \hline
    \textit{Norm 1:} & 
    $\epsilon$ \\
    \hline
    \textit{Norm 2:} & 
    $f(\epsilon) = (0,\emptyset)$ \\
    \hline
    \textit{Norm 3:} & 
    $f(f(\epsilon)) = (1,\emptyset)$, \quad $g(\epsilon) = (1,\{ 1 \})$ \\
    \hline
    \textit{Norm 4:} & 
    $f(f(f(\epsilon))) = (2,\emptyset)  $, \quad
    $f(g(\epsilon)) = (2,\{ 1 \})  $, \\
    &
    $g(f(\epsilon)) =(2, \{ 2 \})  $ \\
    \hline
    \textit{Norm 5:} & 
    $f(f(f(f(\epsilon)))) = (3,\emptyset)  $, \quad 
    $f(f(g(\epsilon))) = (3,\{ 1 \} )$, \\
    & 
    $f(g(f(\epsilon))) =(3, \{2 \} )$, \quad
    $g(f(f(\epsilon))) = (3,\{ 3 \})$, \\
    & 
    $g(g(\epsilon)) = (3,\{ 1, 3 \} )$ \\
    \hline
    \end{tabular}
\end{center}
\end{fibFamily}

    \pagebreak
    
\subsection{Motzkin Families} \label{appendix:Motzkin}

\addcontentsline{toca}{section}{Motzkin families}

We present a number of Motzkin normed (1,2)-magmas. For each \mbox{(1,2)-magma}, we take the convention that the generator is $\epsilon$, the unary map is $f$ and the binary map is $\star$ (and this is always written as an in-fix operator). Despite the same names being used for each \mbox{(1,2)-magma} presented, it is clear that these are all different maps and that the generators are all different.

\addcontentsline{toca}{subsection}{\motFamRef{motzkinFamily:motzkinPaths}: Motzkin paths}
\begin{motFamily} 
[Motzkin paths \cite{DONAGHEY1977291}] \label{motzkinFamily:motzkinPaths}

\def\scl {0.3}

$M_n$ is the number of paths from $(0,0)$ to $(n,0)$ using steps $U = (1,1)$, $D = (1,-1)$ and $H = (1,0)$ which remain above the line $y = 0$.

\textit{Generator:} The empty path which we represent by a single vertex:
\begin{equation*}
    \epsilon = 
    \begin{tikzpicture}[baseline={([yshift=-.5ex]current bounding box.center)}, scale = \scl]
        \fill (0,0) \smlnd;
    \end{tikzpicture}
\end{equation*}

\textit{Unary map:} 
\begin{equation*}
    f \left(
        \begin{tikzpicture}[baseline={([yshift=-2ex]current bounding box.center)}, scale = \scl]
            \draw[fill = \colOne, opacity = \opac] (6,0) -- (0,0) to [bend left = 90, looseness = 1.5] (6,0);
            \node at (3,1.3) {\footnotesize $p$};
        \end{tikzpicture}
    \right)
    =
    \begin{tikzpicture}[baseline={([yshift=-2ex]current bounding box.center)}, scale = \scl]
        \draw[fill = orange!90!black, opacity = \opac] (6,0) -- (0,0) to [bend left = 90, looseness = 1.5] (6,0);
        \node at (3,1.3) {\footnotesize $p$};
        \draw[thick] (6,0) -- (8,0);
    \end{tikzpicture}
\end{equation*}

\textit{Binary map:} 
\begin{equation*}
    \begin{tikzpicture}[baseline=1.5ex, scale = \scl]
        \draw[fill = \colOne, opacity = \opac] (6,0) -- (0,0) to [bend left = 90, looseness = 1.5] (6,0);
        \node[anchor = mid] at (3,1.3) {\footnotesize $p_1$};
    \end{tikzpicture}
    \ \star \ 
    \begin{tikzpicture}[baseline=1.5ex, scale = \scl]
        \draw[fill = \colTwo, opacity = \opac] (6,0) -- (0,0) to [bend left = 90, looseness = 1.5] (6,0);
        \node[anchor = mid] at (3,1.3) {\footnotesize $p_2$};
    \end{tikzpicture}
    =
    \begin{tikzpicture}[baseline=1.5ex, scale = \scl]
        \draw[fill = \colOne, opacity = \opac] (6,0) -- (0,0) to [bend left = 90, looseness = 1.5] (6,0);
        \node[anchor = mid] at (3,1.3) {\footnotesize $p_1$};
        \draw[thick] (6,0) -- (7.5,1.5);
        \draw[fill = \colTwo, opacity = \opac] (13.5,1.5) -- (7.5,1.5) to [bend left = 90, looseness = 1.5] (13.5,1.5);
        \node at (10.5,2.8) {\footnotesize $p_2$};
        \draw[thick] (13.5,1.5) -- (15,0);
    \end{tikzpicture}
\end{equation*}

\textit{Norm:} If $p$ is a Motzkin path from $(0,0)$ to $(n,0)$, then $\norm{p} = n + 1$.

\textit{(1,2)-magma:} The (1,2)-magma begins (sorting by norm) as follows:
\vspace{-.5em}
\begin{center}
    \def\scl {0.45}
    \begin{tabular}{| l l |}
    \hline
    \textit{Norm 1:} & 
    $\begin{tikzpicture}[baseline={([yshift=-.7ex]current bounding box.center)}, scale = \scl]
        \fill (0,0) \smlnd;
    \end{tikzpicture}
    = \epsilon$ \\
    \hline
    \textit{Norm 2:} & 
    $\begin{tikzpicture}[baseline={([yshift=-.5ex]current bounding box.center)}, scale = \scl]
        \grd[1][1]
        \draw[thick] (0,0) \motAc;
        \dummyNodes[1][-0.2][1][1.2]
    \end{tikzpicture} 
    = f(\epsilon)$ \\
    \hline
    \textit{Norm 3:} & 
    $\begin{tikzpicture}[baseline={([yshift=-.5ex]current bounding box.center)}, scale = \scl]
        \grd[2][1]
        \draw[thick] (0,0) \motAc \motAc;
        \dummyNodes[1][-0.2][1][1.2]
    \end{tikzpicture} 
    = f(f(\epsilon))$, 
    \quad
    $\begin{tikzpicture}[baseline={([yshift=-.5ex]current bounding box.center)}, scale = \scl]
        \grd[2][1]
        \draw[thick] (0,0) \motUp \motDn;
        \dummyNodes[1][-0.2][1][1.2]
    \end{tikzpicture}
    = \epsilon \star \epsilon$ \\
    \hline
    \textit{Norm 4:} & 
    $\begin{tikzpicture}[baseline={([yshift=-.5ex]current bounding box.center)}, scale = \scl]
        \grd[3][1]
        \draw[thick] (0,0) \motAc \motAc \motAc;
        \dummyNodes[1][-0.2][1][1.2]
    \end{tikzpicture} 
    = f(f(f(\epsilon)))$, 
    \quad
    $\begin{tikzpicture}[baseline={([yshift=-.5ex]current bounding box.center)}, scale = \scl]
        \grd[3][1]
        \draw[thick] (0,0) \motUp \motDn \motAc;
        \dummyNodes[1][-0.2][1][1.2]
    \end{tikzpicture} 
    = f(\epsilon \star \epsilon)$, 
    \quad
    $\begin{tikzpicture}[baseline={([yshift=-.5ex]current bounding box.center)}, scale = \scl]
        \grd[3][1]
        \draw[thick] (0,0) \motAc \motUp \motDn;
        \dummyNodes[1][-0.2][1][1.2]
    \end{tikzpicture} 
    = f(\epsilon) \star \epsilon$, \\
    &
    $\begin{tikzpicture}[baseline={([yshift=-.5ex]current bounding box.center)}, scale = \scl]
        \grd[3][1]
        \draw[thick] (0,0) \motUp \motAc \motDn;
        \dummyNodes[1][-0.2][1][1.2]
    \end{tikzpicture} 
    = \epsilon \star f(\epsilon)$ \\
    \hline
    \end{tabular}
\end{center}
\end{motFamily}

\pagebreak

\addcontentsline{toca}{subsection}{\motFamRef{motzkinFamily:chords}: Non-intersecting chords}
\begin{motFamily} 
[Non-intersecting chords \cite{motzkin1948}] \label{motzkinFamily:chords}

$M_n$ counts the number of ways of drawing any number of non-intersecting chords joining up to $n$ distinct points on a circle.  These $n$ points are called chord points.
Note that not all $n$ chord points must be incident with a chord. 
We mark, with a $\V{}$ , a unique point at the top of each circle. 
This fixes the orientation of the circle and distinguishes non-intersecting chord diagrams which only differ by a rotation.

\def\scl{0.02}
\def\numPoints {11}  
\def\numSpaces {12} 

\textit{Generator:} The trivial way of joining zero chord points on a circle:
\begin{equation*}
    \epsilon = \def\scl{0.02}

\end{equation*}

\textit{Norm:} If $p$ contains $n$ chord points, then $\norm{p} = n + 1$.

\textit{(1,2)-magma:} The (1,2)-magma begins (sorting by norm) as follows:
\vspace{-.5em}
\begin{center}
    \begin{tabular}{| l l |}
    \hline
    \textit{Norm 1:} & \def\scl{0.02}
    $\begin{tikzpicture}[baseline={([yshift=-.5ex]current bounding box.center)}, scale = \scl]
        \node at (90:10) {\V};
        \draw (0,0) circle [radius = 10];
        \dummyNodesPolar[90:12][270:12]
    \end{tikzpicture}
    = \epsilon$ \\
    \hline 
    \textit{Norm 2:} & \def\scl{0.03}
    $\begin{tikzpicture}[baseline={([yshift=-.5ex]current bounding box.center)}, scale = \scl]
        \node at (90:10) {\V};
        \draw (0,0) circle [radius = 10];
        \fill (270:10) \smlnd;
        \dummyNodesPolar[90:12][270:12]
    \end{tikzpicture}
    = f(\epsilon)$ \\
    \hline 
    \textit{Norm 3:} & \def\scl{0.04}
    $\begin{tikzpicture}[baseline={([yshift=-.5ex]current bounding box.center)}, scale = \scl]
        \node at (90:10) {\V};
        \draw (0,0) circle [radius = 10];
        \fill (0:10) \smlnd;
        \fill (180:10) \smlnd;
        \dummyNodesPolar[90:12][270:12]
    \end{tikzpicture}
    = f(f(\epsilon))$,
    \quad
    $\begin{tikzpicture}[baseline={([yshift=-.5ex]current bounding box.center)}, scale = \scl]
        \node at (90:10) {\V};
        \draw (0,0) circle [radius = 10];
        \draw[fill] (0:10) \smlnd;
        \draw[fill] (180:10) \smlnd;
        \draw (0:10) -- (180:10);
        \dummyNodesPolar[90:12][270:12]
    \end{tikzpicture}
    = \epsilon \star \epsilon$ \\
\hline
    \end{tabular}
\end{center}
\end{motFamily}

\pagebreak

\addcontentsline{toca}{subsection}{\motFamRef{motzkinFamily:motzkinTrees}: Motzkin unary-binary trees}
\begin{motFamily} 
[Motzkin unary-binary trees \cite{DONAGHEY1977291}] \label{motzkinFamily:motzkinTrees}

\def\scl{0.15}

$M_n$ is equal to the number of rooted trees with $n$ edges in which every vertex has degree at most 3, and in which the root has degree at most 2.

\textit{Generator:} The generator is a single vertex, that is, the only tree containing no edges:
\begin{equation*}
    \epsilon =
    %
\end{equation*}

\textit{Norm:} If $t$ is a tree with $n$ edges, then $\norm{t} = n + 1$.

\textit{(1,2)-magma:} The (1,2)-magma begins (sorting by norm) as follows:
\vspace{-.5em}
\begin{center}
    \def\scl{0.5}
    \begin{tabular}{| l l |}
    \hline
    \textit{Norm 1:} & 
    \ $\begin{tikzpicture}[baseline={([yshift=-.5ex]current bounding box.center)}, scale = \scl]
        \fill (0,0) \smlnd;
    \end{tikzpicture}
    = \epsilon$ \\
    \hline
    \textit{Norm 2:} & 
    $\begin{tikzpicture}[baseline={([yshift=-.5ex]current bounding box.center)}, scale = \scl]
        \draw (0,0) \un;
        \fill (0,0) \smlunnd;
        \node at (0,0) {\smlroot};
        \dummyNodes[0][0.2][0][-1.2]
    \end{tikzpicture} 
    = f(\epsilon)$ \\
    \hline
    \textit{Norm 3:} & 
    $\begin{tikzpicture}[baseline={([yshift=-.5ex]current bounding box.center)}, scale = \scl]
        \draw (0,0) \un (0,-1) \un;
        \fill (0,0) \smlunnd (0,-1) \smlunnd;
        \node at (0,0) {\smlroot};
        \dummyNodes[0][0.2][0][-2.2]
    \end{tikzpicture} 
    = f(f(\epsilon))$, \quad
    $\begin{tikzpicture}[baseline={([yshift=-.5ex]current bounding box.center)}, scale = \scl]
        \draw (0,0) \bin;
        \fill (0,0) \smlbinnd;
        \node at (0,0) {\smlroot};
        \dummyNodes[0][0.2][0][-1.2]
    \end{tikzpicture} 
    = \epsilon \star \epsilon$ \\
\hline
    \end{tabular}
\end{center}
\end{motFamily}

\pagebreak

\addcontentsline{toca}{subsection}{\motFamRef{motzkinFamily:brushes}: Bushes}
\begin{motFamily} 
[Bushes \cite{DONAGHEY1977291}] \label{motzkinFamily:brushes}

$M_{n - 1}$ is the number of rooted planar trees with $n$ edges and no degree two nodes, except possibly for the root. Such a tree is called a bush.

\def\scl{0.4}

\textit{Generator:} \begin{equation*}
    \epsilon =
    %
 & \text{otherwise.}
    \end{cases}
\end{equation*}

\textit{Norm:} If $b$ is a bush with $n$ edges, then $\norm{b} = n$.

\textit{(1,2)-magma:} The (1,2)-magma begins (sorting by norm) as follows:
\vspace{-.5em}
\begin{center}
    \def\scl{0.5}
    \begin{tabular}{| l l |}
    \hline
    \textit{Norm 1:} & 
    $\begin{tikzpicture}[baseline={([yshift=-.5ex]current bounding box.center)}, scale = \scl]
        \node at (0,0) (root) {\smlroot};
        \draw (0,0) -- (0,-1);
        \fill (0,-1) \smlnd;
        \dummyNodes[0][0.2][0][-1.2]
    \end{tikzpicture}
    = \epsilon$ \\
    \hline
    \textit{Norm 2:} & 
    $\begin{tikzpicture}[baseline={([yshift=-.5ex]current bounding box.center)}, scale = \scl]
        \node at (0,0) (root) {\smlroot};
        \draw (0,0) -- (-0.5,-1) (0,0) -- (0.5,-1);
        \fill (-0.5,-1) \smlnd (0.5,-1) \smlnd;
        \dummyNodes[0][0.2][0][-1.2]
    \end{tikzpicture}
    = f(\epsilon)$ \\
    \hline
    \textit{Norm 3:} & 
    $\begin{tikzpicture}[baseline={([yshift=-.5ex]current bounding box.center)}, scale = \scl]
        \node at (0,0) (root) {\smlroot};
        \draw (0,0) -- (-1,-1) (0,0) -- (0,-1) (0,0) -- (1,-1);
        \fill (-1,-1) \smlnd (0,-1) \smlnd (1,-1) \smlnd;
        \dummyNodes[0][0.2][0][-1.2]
    \end{tikzpicture}
    = f(f(\epsilon))$,
    \quad
    $\begin{tikzpicture}[baseline={([yshift=-.5ex]current bounding box.center)}, scale = \scl]
        \node at (0,0) (root) {\smlroot};
        \draw (0,0) -- (0,-1) (0,-1) -- (-0.5,-2) (0,-1) -- (0.5,-2);
        \fill (0,-1) \smlnd (-0.5,-2) \smlnd (0.5,-2) \smlnd;
        \dummyNodes[0][0.2][0][-2.2]
    \end{tikzpicture}
    = \epsilon \star \epsilon$ \\
\hline
    \end{tabular}
\end{center}
\end{motFamily}

\addcontentsline{toca}{subsection}{\motFamRef{motzkinFamily:dyckPathsUUU}: UUU-avoiding Dyck paths}
\begin{motFamily} 
[UUU-avoiding Dyck paths \cite{MotzkinOEIS}] \label{motzkinFamily:dyckPathsUUU}

$M_n$ is the number of Dyck paths of length $2n$ with no sequence of three or more consecutive up steps. A Dyck path is a path from $(0,0)$ to $(2n,0)$ using steps $U = (1,1)$ and $D = (1,-1)$ which remain above the line $y = 0$.

\def\scl{0.3}

\textit{Generator:} The empty path which we represent by a single vertex:
\begin{equation*}
    \epsilon = 
    %
\end{equation*}

\textit{Norm:} If $p$ is a Dyck path from $(0,0)$ to $(2n, 0)$, then $\norm{p} = n + 1$.

\textit{(1,2)-magma:} The (1,2)-magma begins (sorting by norm) as follows:
\vspace{-.5em}
\begin{center}
    \def\scl{0.45}
    \begin{tabular}{| l l |}
    \hline
    \textit{Norm 1:} & 
    $\begin{tikzpicture}[baseline={([yshift=-.5ex]current bounding box.center)}, scale = \scl]
        \fill (0,0) \smlnd;
    \end{tikzpicture}
    = \epsilon$ \\
    \hline
    \textit{Norm 2:} & 
    $\begin{tikzpicture}[baseline={([yshift=-.5ex]current bounding box.center)}, scale = \scl]
        \grd[2][1]
        \draw[thick] (0,0) \dyckUp \dyckDn;
        \node at (1,1.2) (A) {};
        \node at (1,-0.2) (B) {};
    \end{tikzpicture} 
    = f(\epsilon)$ \\
    \hline
    \textit{Norm 3:} & 
    $\begin{tikzpicture}[baseline={([yshift=-.5ex]current bounding box.center)}, scale = \scl]
        \grd[4][2]
        \draw[thick] (0,0) \dyckUp \dyckDn \dyckUp \dyckDn;
        \node at (1,2.2) (A) {};
        \node at (1,-0.2) (B) {};
    \end{tikzpicture} 
    = f(f(\epsilon))$, 
    \quad
    $\begin{tikzpicture}[baseline={([yshift=-.5ex]current bounding box.center)}, scale = \scl]
        \grd[4][2]
        \draw[thick] (0,0) \dyckUp \dyckUp \dyckDn \dyckDn;
        \node at (1,2.2) (A) {};
        \node at (1,-0.2) (B) {};
    \end{tikzpicture} 
    = \epsilon \star \epsilon$ \\
    \hline
    \textit{Norm 4:} & 
    $\begin{tikzpicture}[baseline={([yshift=-.5ex]current bounding box.center)}, scale = \scl]
        \grd[6][2]
        \draw[thick] (0,0) \dyckUp \dyckDn \dyckUp \dyckDn \dyckUp \dyckDn;
        \node at (1,2.2) (A) {};
        \node at (1,-0.2) (B) {};
    \end{tikzpicture}  
    = f(f(f(\epsilon)))$, 
    \quad
    $\begin{tikzpicture}[baseline={([yshift=-.5ex]current bounding box.center)}, scale = \scl]
        \grd[6][2]
        \draw[thick] (0,0) \dyckUp \dyckUp \dyckDn \dyckDn \dyckUp \dyckDn;
        \node at (1,2.2) (A) {};
        \node at (1,-0.2) (B) {};
    \end{tikzpicture} 
    = f(\epsilon \star \epsilon)$, \\
\hline
    \end{tabular}
\end{center}
\end{motFamily}

\addcontentsline{toca}{subsection}{\motFamRef{motzkinFamily:dyckPathsUDU}: UDU-avoiding Dyck paths}
\begin{motFamily} 
[UDU-avoiding Dyck paths \cite{MotzkinOEIS}] \label{motzkinFamily:dyckPathsUDU}

$M_{n-1}$ is the number of Dyck paths of length $2n$ with no sequence of UDU steps.

\def\scl{0.45}

\textit{Generator:} 
\begin{equation*}
    \epsilon = 
    %
\end{align*}

\textit{Norm:} Let $p$ be a Dyck path from $(0,0)$ to $(2n,0)$. Then $\norm{p} = n$.

\textit{(1,2)-magma:} The (1,2)-magma begins (sorting by norm) as follows:
\vspace{-.5em}
\def\scl{0.45}
\begin{center}
    \begin{tabular}{| l l |}
    \hline
    \textit{Norm 1:} & 
    $\begin{tikzpicture}[baseline={([yshift=-.5ex]current bounding box.center)}, scale = \scl]
        \grd[2][1]
        \draw[thick] (0,0) \dyckUp \dyckDn;
        \dummyNodes[0][-0.2][0][1.2]
    \end{tikzpicture} 
    = \epsilon$ \\
    \hline
    \textit{Norm 2:} & 
    $\begin{tikzpicture}[baseline={([yshift=-.5ex]current bounding box.center)}, scale = \scl]
        \grd[4][2]
        \draw[thick] (0,0) \dyckUp \dyckUp \dyckDn \dyckDn;
        \dummyNodes[0][-0.2][0][2.2]
    \end{tikzpicture} 
    = f(\epsilon)$ \\
    \hline
    \textit{Norm 3:} & 
    $\begin{tikzpicture}[baseline={([yshift=-.5ex]current bounding box.center)}, scale = \scl]
        \grd[6][3]
        \draw[thick] (0,0) \dyckUp \dyckUp \dyckUp \dyckDn \dyckDn \dyckDn;
        \dummyNodes[0][-0.2][0][3.2]
    \end{tikzpicture} 
    = f(f(\epsilon))$, 
    \quad
    $\begin{tikzpicture}[baseline={([yshift=-.5ex]current bounding box.center)}, scale = \scl]
        \grd[6][3]
        \draw[thick] (0,0) \dyckUp \dyckUp \dyckDn \dyckDn \dyckUp \dyckDn;
        \dummyNodes[0][-0.2][0][3.2]
    \end{tikzpicture} 
    = \epsilon \star \epsilon$ \\
    \hline
    \textit{Norm 4:} & 
    $\begin{tikzpicture}[baseline={([yshift=-.5ex]current bounding box.center)}, scale = \scl]
        \grd[8][4]
        \draw[thick] (0,0) \dyckUp \dyckUp \dyckUp \dyckUp \dyckDn \dyckDn \dyckDn \dyckDn;
        \dummyNodes[0][-0.2][0][4.2]
    \end{tikzpicture} 
    = f(f(f(\epsilon)))$, 
    \quad
    $\begin{tikzpicture}[baseline={([yshift=-.5ex]current bounding box.center)}, scale = \scl]
        \grd[8][4]
        \draw[thick] (0,0) \dyckUp \dyckUp \dyckUp \dyckDn \dyckDn \dyckUp \dyckDn \dyckDn;
        \dummyNodes[0][-0.2][0][4.2]
    \end{tikzpicture} 
    = f(\epsilon \star \epsilon)$, \\
\hline
    \end{tabular}
\end{center}
\end{motFamily}

\addcontentsline{toca}{subsection}{\motFamRef{motzkinFamily:bracketings}: Recursive set of bracketings}
\begin{motFamily} 
[Recursive set of bracketings \cite{MotzkinOEIS}] \label{motzkinFamily:bracketings}


$M_{n - 1}$ is the number of strings of length $2n$ from the following recursively defined set: $L$ contains the string $[ \, ]$ and, for any strings $a$ and $b$ in $L$, we also find $[ \, a \, ]$ and $[ \, ab \, ]$ in $L$.

\textit{Generator:}
\begin{equation*}
    \epsilon = [ \, ].
\end{equation*}

\textit{Unary map:} Let $a$ be a bracketing. Then define the unary map as
\begin{align*}
f(a) = [ \, a \, ].
\end{align*}

\textit{Binary map:} Let $a$ and $b$ be two bracketings. Then define the binary map as
\begin{align*}
a \star b = [ \, ab \, ].
\end{align*}

\textit{Norm:} If $a$ is a bracketing of length $2n$, then $\norm{a} = n$.

\textit{(1,2)-magma:} The (1,2)-magma begins (sorting by norm) as follows:
\vspace{-.5em}
\begin{center}
    \begin{tabular}{| l l |}
    \hline
    \textit{Norm 1:} & 
    $[ \, ]
    = \epsilon$ \\
    \hline
    \textit{Norm 2:} & 
    $[ \, [ \, ] \, ] 
    = f(\epsilon)$ \\
    \hline
    \textit{Norm 3:} & 
    $[ \, [ \, [ \, ] \, ] \, ] 
    = f(f(\epsilon))$, 
    \quad
    $[ \, [ \, ] \, [ \, ] \, ]  
    = \epsilon \star \epsilon$ \\
    \hline
    \textit{Norm 4:} & 
    $[ \, [ \, [ \, [ \, ] \, ] \, ] \, ] 
    = f(f(f(\epsilon)))$, 
    \quad
    $[ \, [ \, [ \, ] \, [ \, ] \, ] \, ]
    = f(\epsilon \star \epsilon)$, \\
    &
    $[ \, [ \, ] \, [ \, [ \, ] \, ] \, ] 
    = \epsilon \star f(\epsilon)$,
    \quad
    $[ \, [ \, [ \, ] \, ] \, [ \, ] \, ]
    = f(\epsilon) \star \epsilon$ \\
    \hline
    \end{tabular}
\end{center}
\end{motFamily}

\addcontentsline{toca}{subsection}{\motFamRef{motzkinFamily:dyckPathsEven}: Dyck paths with even valleys}
\begin{motFamily} 
[Dyck paths with even valleys \cite{MotzkinOEIS}] \label{motzkinFamily:dyckPathsEven}

\def\scl{0.3}

$M_n$ is the number of length $2n$ Dyck paths whose valleys all have even $x$-coordinates. A valley is a path vertex preceded by a down step and followed by an up step.

\textit{Generator:} The empty path which we represent by a single vertex:
\begin{equation*}
    \epsilon = 
    \begin{tikzpicture}[baseline={([yshift=-.5ex]current bounding box.center)}, scale = \scl]
        \fill (0,0) \smlnd;
    \end{tikzpicture}
\end{equation*}

\textit{Unary map:}
\begin{align*}
    f \left(
    \begin{tikzpicture}[baseline={([yshift=-.5ex]current bounding box.center)}, scale = \scl]
        \draw[fill = \colOne, opacity = \opac] (6,0) -- (0,0) to [bend left = 90, looseness = 1.5] (6,0);
        \node at (3,1.3) {\footnotesize $p$};
        \dummyNodes[3][-0.3][3][-0.3]
    \end{tikzpicture}
    \right) =
    \begin{tikzpicture}[baseline={([yshift=-.5ex]current bounding box.center)}, scale = \scl]
        \draw[fill = \colOne, opacity = \opac] (6,0) -- (0,0) to [bend left = 90, looseness = 1.5] (6,0);
        \node at (3,1.3) {\footnotesize $p$};
        \draw[thick] (6,0) -- (7.5,1.5) -- (9,0);
        \dummyNodes[3][-0.3][3][-0.3]
    \end{tikzpicture}
\end{align*}

\textit{Binary map:} 
\begin{align*}
    \begin{tikzpicture}[baseline=1.5ex, scale = \scl]
        \draw[fill = \colOne, opacity = \opac] (6,0) -- (0,0) to [bend left = 90, looseness = 1.5] (6,0);
        \node[anchor = mid] at (3,1.3) {\footnotesize $p_1$};
    \end{tikzpicture}
    \ \star \
    \begin{tikzpicture}[baseline=1.5ex, scale = \scl]
        \draw[fill = \colTwo, opacity = \opac] (6,0) -- (0,0) to [bend left = 90, looseness = 1.5] (6,0);
        \node[anchor = mid] at (3,1.3) {\footnotesize $p_2$};
    \end{tikzpicture}
    \ = \
    \begin{tikzpicture}[baseline=1.5ex, scale = \scl]
        \draw[fill = \colOne, opacity = \opac] (6,0) -- (0,0) to [bend left = 90, looseness = 1.5] (6,0);
        \node[anchor = mid] at (3,1.3) {\footnotesize $p_1$};
        \draw[thick] (6,0) -- (9,3);
        \draw[fill = \colTwo, opacity = \opac] (15,3) -- (9,3) to [bend left = 90, looseness = 1.5] (15,3);
        \node at (12,4.3) {\footnotesize $p_2$};
        \draw[thick] (15,3) -- (18,0);
        \fill[white] (7.5,1.5) circle (5pt) (16.5,1.5) circle (5pt);
    \end{tikzpicture}
\end{align*}

\textit{Norm:} If $p$ is a Dyck path from $(0,0)$ to $(2n,0)$, then $\norm{p} = n + 1$.

\vbox{
\textit{(1,2)-magma:} The (1,2)-magma begins (sorting by norm) as follows:
\vspace{-.5em}
\def\scl{0.4}
\begin{center}
\def\scl{0.45}
    \begin{tabular}{| l l |}
    \hline
    \textit{Norm 1:} & 
    $\begin{tikzpicture}[baseline={([yshift=-.5ex]current bounding box.center)}, scale = \scl]
        \fill (0,0) \smlnd;
    \end{tikzpicture}
    = \epsilon$ \\
    \hline
    \textit{Norm 2:} & 
    $\begin{tikzpicture}[baseline={([yshift=-.5ex]current bounding box.center)}, scale = \scl]
        \grd[2][1]
        \draw[thick] (0,0) \dyckUp \dyckDn;
        \dummyNodes[0][-0.2][0][1.2]
    \end{tikzpicture} 
    = f(\epsilon)$ \\
    \hline
    \textit{Norm 3:} & 
    $\begin{tikzpicture}[baseline={([yshift=-.5ex]current bounding box.center)}, scale = \scl]
        \grd[4][2]
        \draw[thick] (0,0) \dyckUp \dyckDn \dyckUp \dyckDn;
        \dummyNodes[0][-0.2][0][2.2]
    \end{tikzpicture}  
    = f(f(\epsilon))$, 
    \quad
    $\begin{tikzpicture}[baseline={([yshift=-.5ex]current bounding box.center)}, scale = \scl]
        \grd[4][2]
        \draw[thick] (0,0) \dyckUp \dyckUp \dyckDn \dyckDn;
        \dummyNodes[0][-0.2][0][2.2]
    \end{tikzpicture}  
    = \epsilon \star \epsilon$ \\
    \hline
    \textit{Norm 4:} & 
    $\begin{tikzpicture}[baseline={([yshift=-.5ex]current bounding box.center)}, scale = \scl]
        \grd[6][3]
        \draw[thick] (0,0) \dyckUp \dyckDn \dyckUp \dyckDn \dyckUp \dyckDn;
        \dummyNodes[0][-0.2][0][3.2]
    \end{tikzpicture}  
    = f(f(f(\epsilon)))$, 
    \quad
    $\begin{tikzpicture}[baseline={([yshift=-.5ex]current bounding box.center)}, scale = \scl]
        \grd[6][3]
        \draw[thick] (0,0) \dyckUp \dyckUp \dyckDn \dyckDn \dyckUp \dyckDn;
        \dummyNodes[0][-0.2][0][3.2]
    \end{tikzpicture}  
    = f(\epsilon \star \epsilon)$, \\
    &
    $\begin{tikzpicture}[baseline={([yshift=-.5ex]current bounding box.center)}, scale = \scl]
        \grd[6][3]
        \draw[thick] (0,0) \dyckUp \dyckUp \dyckUp \dyckDn \dyckDn \dyckDn;
        \dummyNodes[0][-0.2][0][3.2]
    \end{tikzpicture}  
    = \epsilon \star f(\epsilon)$, 
    \quad
    $\begin{tikzpicture}[baseline={([yshift=-.5ex]current bounding box.center)}, scale = \scl]
        \grd[6][3]
        \draw[thick] (0,0) \dyckUp \dyckDn \dyckUp \dyckUp \dyckDn \dyckDn;
        \dummyNodes[0][-0.2][0][3.2]
    \end{tikzpicture}  
    = f(\epsilon) \star \epsilon$ \\
    \hline
    \end{tabular}
\end{center}
}
\end{motFamily}

\addcontentsline{toca}{subsection}{\motFamRef{motzkinFamily:RNA}: RNA shapes}
\begin{motFamily} 
[RNA shapes \cite{lorenz2008asymptotics}] \label{motzkinFamily:RNA}

$M_n$ is the number of RNA shapes of size $2n + 2$. RNA shapes are Dyck words without ``directly nested'' motifs of the form $A \,[ \, [ \, B \, ] \, ] \, C$ for $A$, $B$ and $C$ Dyck words. A Dyck word is a word in the alphabet $\{ \, [ \, , \, ] \, \}$ with an equal number of left and right parentheses with the property that, reading from left to right, the number of right parentheses never exceeds the number of left parentheses.

\textit{Generator:}
\begin{equation*}
    \epsilon = [ \, ].
\end{equation*}

\textit{Unary map:} Let $w$ be an RNA shape. Then define the unary map as
\begin{align*}
    f(w) = w \, [ \, ].
\end{align*}

\textit{Binary map:} Let $w_1$ and $w_2$ be RNA shapes. Then define the binary map as
\begin{align*}
    w_1 \star w_2 = w_1' \, [ \, ] \, w_2 \, ],
\end{align*}
where $w_1 = w_1' \, ]$.

\textit{Norm:} If $w$ is an RNA shape of length $2n$, then $\norm{w} = n$.

\textit{(1,2)-magma:} The (1,2)-magma begins (sorting by norm) as follows:
\vspace{-.5em}
\begin{center}
    \begin{tabular}{| l l |}
    \hline
    \textit{Norm 1:} & 
    $[ \, ] = \epsilon$ \\
    \hline
    \textit{Norm 2:} & 
    $[ \, ] \, [ \, ] = f(\epsilon)$ \\
    \hline
    \textit{Norm 3:} & 
    $[ \, ] \, [ \, ] \, [ \, ] = f(f(\epsilon))$, \quad
    $[ \, [ \, ] \, [ \, ] \, ] = \epsilon \star \epsilon$ \\
    \hline
    \textit{Norm 4:} & 
    $[ \, ] \, [ \, ] \, [ \, ] \, [ \, ] = f(f(f(\epsilon)))$, \quad
    $[ \, [ \, ] \, [ \, ] \, ] \, [ \, ] = f(\epsilon \star \epsilon)$, \\ &
    $[ \, [ \, ] \, [ \, ] \, [ \, ] \, ] = \epsilon \star f(\epsilon)$, \quad
    $[ \, ] \, [ \, [ \, ] \, [ \, ] \, ] =  f(\epsilon) \star \epsilon$ \\
    \hline
    \end{tabular}
\end{center}
\end{motFamily}

    \pagebreak
    
\subsection{Schr\"oder Families} \label{appendix:Schroder}

\addcontentsline{toca}{section}{Schr\"oder families}

We present a number of Schr\"oder normed \mbox{(1,2)-magmas}. For each \mbox{(1,2)-magma}, we take the convention that the generator is $\epsilon$, the unary map is $f$ and the binary map is $\star$ (and that this is always written as an in-fix operator). Despite the same names being used for each Schr\"oder family presented, it is clear that these are all different maps and the generators are all different. 

\addcontentsline{toca}{subsection}{\schFamRef{schroderFamily:schroderPaths}: Schr\"oder paths}
\begin{schFamily}
[Schr\"oder paths \cite{BONIN199335}] \label{schroderFamily:schroderPaths}

\def \scl {0.3}

$S_n$ is the number of paths from $(0,0)$ to $(n,n)$ using only steps $(1,0)$, $(0,1)$ and $(1,1)$ which lie weakly below the line $y = x$. The diagonal $y = x$ is shown as a dashed line in the schematic diagram defining the binary map.

\textit{Generator:} The empty path which we represent by a single vertex:
\vspace{-.5ex}
\begin{equation*}
    \epsilon = 
    \begin{tikzpicture}[baseline={([yshift=-.5ex]current bounding box.center)}, scale = \scl]
        \fill (0,0) \smlnd;
    \end{tikzpicture}
\end{equation*}

\textit{Unary map:}
\vspace{-2em}
\begin{align*}
    f 
    \left(
    \begin{tikzpicture}[baseline=1.5ex, scale = \scl]
        \clip (0,-1) rectangle (4.75,4.25);
        \draw[fill = \colOne, opacity = \opac] (4,4) -- (0,0) to [bend right = 90, looseness = 1.5] (4,4);
        \node[anchor = mid] at (2.75,1.25) {\footnotesize $p$};
    \end{tikzpicture}
    \right)
    & =
    \begin{tikzpicture}[baseline=1.5ex, scale = \scl]
        \clip (-1,-1) rectangle (6.5,6.5);
        \draw[fill = \colOne, opacity = \opac] (4,4) -- (0,0) to [bend right = 90, looseness = 1.5] (4,4);
        \node[anchor = mid] at (2.75,1.25) {\footnotesize $p$};
        \draw[thick] (4,4) -- (5.5,5.5);
    \end{tikzpicture}
\end{align*}

\textit{Binary map:}
\vspace{-3ex}
\begin{align*}
    \begin{tikzpicture}[baseline={([yshift=-.5ex]current bounding box.center)}, scale = \scl]
        \clip (0,-1) rectangle (5.25,4.25);
        \draw[fill = \colOne, opacity = \opac] (4,4) -- (0,0) to [bend right = 90, looseness = 1.5] (4,4);
        \node at (2.75,1.25) {\footnotesize $p_1$};
    \end{tikzpicture}
    \ \star \
    \begin{tikzpicture}[baseline={([yshift=-.5ex]current bounding box.center)}, scale = \scl]
        \clip (0,-1) rectangle (4.75,4.25);
        \draw[fill = \colTwo, opacity = \opac] (4,4) -- (0,0) to [bend right = 90, looseness = 1.5] (4,4);
        \node at (2.75,1.25) {\footnotesize $p_2$};
    \end{tikzpicture}
    \ = \
    \begin{tikzpicture}[baseline={([yshift=-.5ex]current bounding box.center)}, scale = \scl]
        \clip (-1,-1) rectangle (10.5,10);
        \draw[dashed] (-1,-1) -- (10.5,10.5);
        \draw[fill = \colOne, opacity = \opac] (4,4) -- (0,0) to [bend right = 90, looseness = 1.5] (4,4);
        \node at (2.75,1.25) {\footnotesize $p_1$};
        \draw[thick] (4,4) -- (5.5,4);
        \draw[fill = \colTwo, opacity = \opac] (9.5,8) -- (5.5,4) to [bend right = 90, looseness = 1.5] (9.5,8);
        \node at (8.25,5.25) {\footnotesize $p_2$};
        \draw[thick] (9.5,8) -- (9.5,9.5);
    \end{tikzpicture}
\end{align*}

\textit{Norm:} If $p$ is a Schr\"oder path from $(0,0)$ to $(n,n)$, then $\norm{p} = n + 1$.

\vbox{
\textit{(1,2)-magma:} The (1,2)-magma begins (sorting by norm) as follows:
\vspace{-.5em}
\begin{center}
    \def\scl {0.45}
    \begin{tabular}{| l l |}
    \hline
    \textit{Norm 1:} & 
    $\begin{tikzpicture}[baseline={([yshift=-.5ex]current bounding box.center)}, scale = \scl]
        \fill (0,0) \smlnd;
    \end{tikzpicture}
    = \epsilon$ \\
    \hline
    \textit{Norm 2:} & 
    $\begin{tikzpicture}[baseline={([yshift=-.5ex]current bounding box.center)}, scale = \scl]
        \grd[1][1]
        \draw[thick] (0,0) \schDiag;
        \dummyNodes[1][-0.2][1][1.2]
    \end{tikzpicture} 
    = f(\epsilon)$, 
    \quad
    $\begin{tikzpicture}[baseline={([yshift=-.5ex]current bounding box.center)}, scale = \scl]
        \grd[1][1]
        \draw[thick] (0,0) \schAc \schUp;
        \dummyNodes[1][-0.2][1][1.2]
    \end{tikzpicture} 
    = \epsilon \star \epsilon$ \\
    \hline
    \textit{Norm 3:} & 
    $\begin{tikzpicture}[baseline={([yshift=-.5ex]current bounding box.center)}, scale = \scl]
        \grd[2][2]
        \draw[thick] (0,0) \schDiag \schDiag;
        \dummyNodes[1][-0.2][1][2.2]
    \end{tikzpicture} 
    = f(f(\epsilon))$, 
    \quad
    $\begin{tikzpicture}[baseline={([yshift=-.5ex]current bounding box.center)}, scale = \scl]
        \grd[2][2]
        \draw[thick] (0,0) \schAc \schUp \schDiag;
        \dummyNodes[1][-0.2][1][2.2]
    \end{tikzpicture}  
    = f(\epsilon \star \epsilon)$, 
    \quad
    $\begin{tikzpicture}[baseline={([yshift=-.5ex]current bounding box.center)}, scale = \scl]
        \grd[2][2]
        \draw[thick] (0,0) \schAc \schDiag \schUp;
        \dummyNodes[1][-0.2][1][2.2]
    \end{tikzpicture} 
    = \epsilon \star f(\epsilon)$, \\
    &
    $\begin{tikzpicture}[baseline={([yshift=-.5ex]current bounding box.center)}, scale = \scl]
        \grd[2][2]
        \draw[thick] (0,0) \schAc \schAc \schUp \schUp;
        \dummyNodes[1][-0.2][1][2.2]
    \end{tikzpicture} 
    = \epsilon \star (\epsilon \star \epsilon)$, 
    \quad
    $\begin{tikzpicture}[baseline={([yshift=-.5ex]current bounding box.center)}, scale = \scl]
        \grd[2][2]
        \draw[thick] (0,0) \schDiag \schAc \schUp;
        \dummyNodes[1][-0.2][1][2.2]
    \end{tikzpicture} 
    = f(\epsilon) \star \epsilon$, 
    \quad
    $\begin{tikzpicture}[baseline={([yshift=-.5ex]current bounding box.center)}, scale = \scl]
        \grd[2][2]
        \draw[thick] (0,0) \schAc \schUp \schAc \schUp;
        \dummyNodes[1][-0.2][1][2.2]
    \end{tikzpicture} 
    = (\epsilon \star \epsilon) \star \epsilon$ \\
    \hline
    \end{tabular}
\end{center}
}
\end{schFamily}

\addcontentsline{toca}{subsection}{\schFamRef{schroderFamily:dyckPathsColoured}: Dyck paths with coloured peaks}
\begin{schFamily}
[Dyck paths with coloured peaks \cite{Deutsch2001ABP}] \label{schroderFamily:dyckPathsColoured} 

\def\scl{0.3}

$S_n$ is the number of Dyck paths from $(0,0)$ to $(2n,0)$ with each peak coloured black or white. A peak is a point preceded by an up step and followed by a down step.

\textit{Generator:} The empty path which we represent by a single vertex which is coloured black:
\begin{equation*}
    \epsilon = 
    \begin{tikzpicture}[baseline={([yshift=-.7ex]current bounding box.center)}, scale = \scl]
        \fill (0,0) \smlnd;
    \end{tikzpicture}
\end{equation*}
Note that when applying the unary map to $\epsilon$ or applying the binary map with $\epsilon$ as the left factor that this is simply the empty path since this will not correspond to a peak and thus the node need not be coloured. 

\textit{Unary map:} 
\vspace{-1ex}
\begin{align*}
    f \left(
    \begin{tikzpicture}[baseline={([yshift=-2ex]current bounding box.center)}, scale = \scl]
        \draw[fill = \colOne, opacity = \opac] (6,0) -- (0,0) to [bend left = 90, looseness = 1.5] (6,0);
        \node at (3,1.3) {\footnotesize $p$};
    \end{tikzpicture}
    \right)
    & =
    \begin{tikzpicture}[baseline={([yshift=-2ex]current bounding box.center)}, scale = \scl]
        \draw[fill = \colOne, opacity = \opac] (6,0) -- (0,0) to [bend left = 90, looseness = 1.5] (6,0);
        \node at (3,1.3) {\footnotesize $p$};
        \draw[thick] (6,0) -- (7.5,1.5) -- (9,0);
        \draw[fill = white] (7.5,1.5) \smlnd;
    \end{tikzpicture}
\end{align*}

\textit{Binary map:}
\vspace{-1em}
\begin{align*}
    \begin{tikzpicture}[baseline=1.5ex, scale = \scl]
        \draw[fill = \colOne, opacity = \opac] (6,0) -- (0,0) to [bend left = 90, looseness = 1.5] (6,0);
        \node[anchor = mid] at (3,1.3) {\footnotesize $p_1$};
    \end{tikzpicture}
    \ \star \
    \begin{tikzpicture}[baseline=1.5ex, scale = \scl]
        \draw[fill = \colTwo, opacity = \opac] (6,0) -- (0,0) to [bend left = 90, looseness = 1.5] (6,0);
        \node[anchor = mid] at (3,1.3) {\footnotesize $p_2$};
    \end{tikzpicture}
    \ = \
    \begin{cases}    
   \begin{tikzpicture}[baseline=1.5ex, scale = \scl]
        \draw[fill = \colOne, opacity = \opac] (6,0) -- (0,0) to [bend left = 90, looseness = 1.5] (6,0);
        \node[anchor = mid] at (3,1.3) {\footnotesize $p_1$};
        \draw[thick] (6,0) -- (7.5,1.5);
         \draw[thick] (6,0) -- (7.5,1.5) -- (9,0);
        \draw[fill = black] (7.5,1.5) \smlnd;
    \end{tikzpicture} & \text{if $p_2=\epsilon$ }\\
    & \\
    \begin{tikzpicture}[baseline=1.5ex, scale = \scl]
        \draw[fill = \colOne, opacity = \opac] (6,0) -- (0,0) to [bend left = 90, looseness = 1.5] (6,0);
        \node[anchor = mid] at (3,1.3) {\footnotesize $p_1$};
        \draw[thick] (6,0) -- (7.5,1.5);
        \draw[fill = \colTwo, opacity = \opac] (13.5,1.5) -- (7.5,1.5) to [bend left = 90, looseness = 1.5] (13.5,1.5);
        \node at (10.5,2.8) {\footnotesize $p_2$};
        \draw[thick] (13.5,1.5) -- (15,0);
    \end{tikzpicture} & \text{if $p_2\ne\epsilon$ }
     \end{cases}
\end{align*}

\textit{Norm:} If $p$ is a Dyck path with coloured peaks from $(0,0)$ to $(2n,0)$, then $\norm{p} = n + 1$.

\textit{(1,2)-magma:} The (1,2)-magma begins (sorting by norm) as follows:
\vspace{-.5em}
\begin{center}
    \def\scl{0.45}
    \begin{tabular}{| l l |}
    \hline
    \textit{Norm 1:} & 
    $\begin{tikzpicture}[baseline={([yshift=-.5ex]current bounding box.center)}, scale = \scl]
        \blkNd[(0,0)]
    \end{tikzpicture} 
    = \epsilon$ \\
    \hline
    \textit{Norm 2:} & 
    $\begin{tikzpicture}[baseline={([yshift=-.5ex]current bounding box.center)}, scale = \scl]
        \grd[2][1]
        \draw[thick] (0,0) \dyckUp \dyckDn;
        \whtNd[(1,1)]
        \dummyNodes[1][-0.2][1][1.2]
    \end{tikzpicture} 
    = f(\epsilon)$, 
    \quad
    $\begin{tikzpicture}[baseline={([yshift=-.5ex]current bounding box.center)}, scale = \scl]
        \grd[2][1]
        \draw[thick] (0,0) \dyckUp \dyckDn;
        \blkNd[(1,1)]
        \dummyNodes[1][-0.2][1][1.2]
    \end{tikzpicture} 
    = \epsilon \star \epsilon$ \\
    \hline
    \textit{Norm 3:} & 
    $\begin{tikzpicture}[baseline={([yshift=-.5ex]current bounding box.center)}, scale = \scl]
        \grd[4][2]
        \draw[thick] (0,0) \dyckUp \dyckDn \dyckUp \dyckDn;
        \whtNd[(1,1)]
        \whtNd[(3,1)]
        \dummyNodes[1][-0.2][1][2.2]
    \end{tikzpicture} 
    = f(f(\epsilon))$, 
    \quad
    $\begin{tikzpicture}[baseline={([yshift=-.5ex]current bounding box.center)}, scale = \scl]
        \grd[4][2]
        \draw[thick] (0,0) \dyckUp \dyckDn \dyckUp \dyckDn;
        \blkNd[(1,1)]
        \whtNd[(3,1)]
        \dummyNodes[1][-0.2][1][2.2]
    \end{tikzpicture} 
    = f(\epsilon \star \epsilon)$, 
    \quad
    $\begin{tikzpicture}[baseline={([yshift=-.5ex]current bounding box.center)}, scale = \scl]
        \grd[4][2]
        \draw[thick] (0,0) \dyckUp \dyckUp \dyckDn \dyckDn;
        \whtNd[(2,2)]
        \dummyNodes[1][-0.2][1][2.2]
    \end{tikzpicture}   
    = \epsilon \star f(\epsilon)$, \\
    &
    $\begin{tikzpicture}[baseline={([yshift=-.5ex]current bounding box.center)}, scale = \scl]
        \grd[4][2]
        \draw[thick] (0,0) \dyckUp \dyckUp \dyckDn \dyckDn;
        \blkNd[(2,2)]
        \dummyNodes[1][-0.2][1][2.2]
    \end{tikzpicture}  
    = \epsilon \star (\epsilon \star \epsilon)$, 
    \quad
    $\begin{tikzpicture}[baseline={([yshift=-.5ex]current bounding box.center)}, scale = \scl]
        \grd[4][2]
        \draw[thick] (0,0) \dyckUp \dyckDn \dyckUp \dyckDn;
        \whtNd[(1,1)]
        \blkNd[(3,1)]
        \dummyNodes[1][-0.2][1][2.2]
    \end{tikzpicture}  
    = f(\epsilon) \star \epsilon$, 
    \quad
    $\begin{tikzpicture}[baseline={([yshift=-.5ex]current bounding box.center)}, scale = \scl]
        \grd[4][2]
        \draw[thick] (0,0) \dyckUp \dyckDn \dyckUp \dyckDn;
        \blkNd[(1,1)]
        \blkNd[(3,1)]
        \dummyNodes[1][-0.2][1][2.2]
    \end{tikzpicture}  
    = (\epsilon \star \epsilon) \star \epsilon$ \\
    \hline
    \end{tabular}
\end{center}
\end{schFamily}

\addcontentsline{toca}{subsection}{\schFamRef{schroderFamily:SSYT}: Semi-standard Young tableaux of shape $n \times 2$}
\begin{schFamily}
[Semi-standard Young tableaux of shape $\bm{n \times 2}$ \cite{SchroderOEIS}] \label{schroderFamily:SSYT}


$S_n$ is the number of semi-standard Young tableaux (SSYT) of shape $n \times 2$. Each tableau is filled with entries from $\{ 1, \hdots, r \}$ for some $r \in \mathbb{N}$, and we require that each number in this set appears in at least one cell. The rows must be weakly increasing from left to right and the columns must be strictly increasing from top to bottom.

\textit{Generator:} The empty SSYT: 
\begin{equation*}
    \epsilon = \emptyset.
\end{equation*}
When applying either of the two maps to the generator, every entry of the empty tableau is taken to be 0. In the map definitions which follow, we explicitly state what happens when applying the maps to the generator in order to make this clear.

\textit{Unary map:} 
\begin{align*}
    f \left(
    \resizebox{\SSYTscl}{!}{\addvbuffer[5pt 3pt]{\begin{tabular}{| c | c |}
        \hline
        \rowcolor{\colOne!\opacTen} \vdots & \vdots \\
        \hline
        \rowcolor{\colOne!50} $a$ & $b$ \\
        \hline
    \end{tabular}}}
    \right) =
    \resizebox{\SSYTscl}{!}{\addvbuffer[3pt 3pt]{\begin{tabular}{| c | c |}
        \hline
        \rowcolor{\colOne!\opacTen} \vdots & \vdots \\
        \hline
        \rowcolor{\colOne!\opacTen} $a$ & $b$ \\
        \hline
        \rowcolor{white} $b + 1$ & $b + 1$ \\
        \hline
    \end{tabular}}}
\end{align*}

Note that we have
\begin{equation*}
    f(\epsilon) =
    \resizebox{\SSYTscl}{!}{\begin{tabular}{| c | c |}
        \hline
        1 & 1 \\
        \hline
    \end{tabular}}
\end{equation*}

\textit{Binary map:}
\vspace{-1em}
\begin{align*}
    \resizebox{\SSYTscl}{!}{
    \begin{tabular}{| c | c |}
        \hline
        \rowcolor{\colOne!\opacTen} $a$ & $b$ \\
        \hline
        \rowcolor{\colOne!\opacTen} \vdots & \vdots \\
        \hline
        \rowcolor{\colOne!\opacTen} $c$ & $d$ \\
        \hline
    \end{tabular}
    }
    \ \star \
    \resizebox{\SSYTscl}{!}{\begin{tabular}{| c | c |}
        \hline
        \rowcolor{\colTwo!\opacTen} $s$ & $t$ \\
        \hline
        \rowcolor{\colTwo!\opacTen} $u$ & $v$ \\
        \hline
        \rowcolor{\colTwo!\opacTen} \vdots & \vdots \\
        \hline
        \rowcolor{\colTwo!\opacTen} $y$ & $z$ \\
        \hline
    \end{tabular}}
    =
    \resizebox{\SSYTscl}{!}{\addvbuffer[3pt 3pt]{\begin{tabular}{| c | c |}
        \hline
        \rowcolor{\colOne!\opacTen} $a$ & $b$ \\
        \hline
        \rowcolor{\colOne!\opacTen} \vdots & \vdots \\
        \hline
        \rowcolor{\colOne!\opacTen} $c$ & $d$ \\
        \hline
        \cellcolor{white} $d + 1$ & \cellcolor{\colTwo!\opacTen} $t + d + 1$ \\
        \hline
        \rowcolor{\colTwo!\opacTen} $s + d + 1$ & $v + d + 1$ \\
        \hline
        \rowcolor{\colTwo!\opacTen} $u + d + 1$ & \vdots \\
        \hline
        \rowcolor{\colTwo!\opacTen} \vdots & $z + d + 1$ \\
        \hline
        \cellcolor{\colTwo!\opacTen} $y + d + 1$ & \cellcolor{white}$z + d + 2$ \\
        \hline
    \end{tabular}}}
\end{align*}

If we apply the binary map with the generator, we have the following:
\begin{equation*}
    \epsilon \star \epsilon
    =
    \resizebox{\SSYTscl}{!}{
    \begin{tabular}{| c | c |}
        \hline
        1 & 2 \\
        \hline
    \end{tabular}
    }
    \qquad \quad
    \resizebox{\SSYTscl}{!}{
    \begin{tabular}{| c | c |}
        \hline
        \rowcolor{\colOne!\opacTen} $a$ & $b$ \\
        \hline
        \rowcolor{\colOne!\opacTen} \vdots & \vdots \\
        \hline
        \rowcolor{\colOne!\opacTen} $c$ & $d$ \\
        \hline
    \end{tabular}
    }
    \star \,
    \epsilon
    =
    \resizebox{\SSYTscl}{!}{
    \addvbuffer[3pt 3pt]{\begin{tabular}{| c | c |}
        \hline
        \rowcolor{\colOne!\opacTen} $a$ & $b$ \\
        \hline
        \rowcolor{\colOne!\opacTen} \vdots & \vdots \\
        \hline
        \rowcolor{\colOne!\opacTen} $c$ & $d$ \\
        \hline
        \rowcolor{white} $d + 1$ & $ d + 2$ \\
        \hline
    \end{tabular}}}
    \qquad \quad
    \epsilon
    \, \star \,
    \resizebox{\SSYTscl}{!}{\begin{tabular}{| c | c |}
        \hline
        \rowcolor{\colTwo!\opacTen} $s$ & $t$ \\
        \hline
        \rowcolor{\colTwo!\opacTen} $u$ & $v$ \\
        \hline
        \rowcolor{\colTwo!\opacTen} \vdots & \vdots \\
        \hline
        \rowcolor{\colTwo!\opacTen} $y$ & $z$ \\
        \hline
    \end{tabular}}
    =
    \resizebox{\SSYTscl}{!}{\addvbuffer[3pt 3pt]{\begin{tabular}{| c | c |}
        \hline
        \cellcolor{white} $1$ & \cellcolor{\colTwo!\opacTen} $t + 1$ \\
        \hline
        \rowcolor{\colTwo!\opacTen} $s + 1$ & $v + 1$ \\
        \hline
        \rowcolor{\colTwo!\opacTen} $u + 1$ & \vdots \\
        \hline
        \rowcolor{\colTwo!\opacTen} \vdots & $z + 1$ \\
        \hline
        \cellcolor{\colTwo!\opacTen} $y + 1$ & \cellcolor{white}$z + 2$ \\
        \hline
    \end{tabular}}}
\end{equation*}

\textit{Norm:} If $t$ is a semi-standard Young tableau of size $n \times 2$, then $\norm{t} = n + 1$.

\textit{(1,2)-magma:} The (1,2)-magma begins (sorting by norm) as follows:
\vspace{-.5em}
\begin{center}
    \begin{tabular}{| l l |}
    \hline
    \textit{Norm 1:} & 
    $\emptyset = \epsilon$ \\
    \hline
    \textit{Norm 2:} & 
    $\addvbuffer[3pt 3pt]{\begin{tabular}{| c | c |}
        \hline
        1 & 1 \\
        \hline
    \end{tabular}} 
    = f(\epsilon)$, 
    \quad
    $\addvbuffer[3pt 3pt]{\begin{tabular}{| c | c |}
        \hline
        1 & 2 \\
        \hline
    \end{tabular}} 
    = \epsilon \star \epsilon$ \\
    \hline
    \textit{Norm 3:} & 
    $\addvbuffer[3pt 3pt]{\begin{tabular}{| c | c |}
        \hline
        1 & 1 \\
        \hline
        2 & 2 \\
        \hline
    \end{tabular}} 
    = f(f(\epsilon))$, 
    \quad
    $\addvbuffer[3pt 3pt]{\begin{tabular}{| c | c |}
        \hline
        1 & 2 \\
        \hline
        3 & 3 \\
        \hline
    \end{tabular}} 
    = f(\epsilon \star \epsilon)$, 
    \quad
    $\addvbuffer[3pt 3pt]{\begin{tabular}{| c | c |}
        \hline
        1 & 2 \\
        \hline
        2 & 3 \\
        \hline
    \end{tabular}} 
    = \epsilon \star f(\epsilon)$, \\
    &
    $\addvbuffer[3pt 3pt]{\begin{tabular}{| c | c |}
        \hline
        1 & 3 \\
        \hline
        2 & 4 \\
        \hline
    \end{tabular}} 
    = \epsilon \star (\epsilon \star \epsilon)$, 
    \quad
    $\addvbuffer[3pt 3pt]{\begin{tabular}{| c | c |}
        \hline
        1 & 1 \\
        \hline
        2 & 3 \\
        \hline
    \end{tabular}}  
    = f(\epsilon) \star \epsilon$, 
    \quad
    $\addvbuffer[3pt 3pt]{\begin{tabular}{| c | c |}
        \hline
        1 & 2 \\
        \hline
        3 & 4 \\
        \hline
    \end{tabular}}  
    = (\epsilon \star \epsilon) \star \epsilon$ \\
    \hline
    \end{tabular}
\end{center}
\end{schFamily}

\addcontentsline{toca}{subsection}{\schFamRef{schroderFamily:rectangulations}: Rectangulations}
\begin{schFamily}
[Rectangulations \cite{Ackerman:2004:NRP:982792.982904}] \label{schroderFamily:rectangulations} 

\def\scl{0.6}

$S_n$ is the number of ways to divide a rectangle into $n + 1$ smaller rectangles using $n$ cuts through $n$ points placed equidistant from each other   inside the rectangle along the diagonal joining the bottom left corner and the top right corner. Each cut intersects one of the points and divides only a single rectangle in two. 

Curiously, this combinatorial family appears on the Schr\"oder numbers Wikipedia page with no citation. After some investigation, it was discovered that these are equivalent to ``point-constrained rectangular guillotine partitions'' as defined in \cite{Ackerman:2004:NRP:982792.982904}.

\textit{Generator:} The generator is taken to be the trivial way of dividing a rectangle into one rectangle using zero cuts:
\begin{equation*}
    \epsilon = 
    \begin{tikzpicture}[baseline={([yshift=-1ex]current bounding box.center)}, scale = 1.5*\scl]
        \draw (0,0) rectangle (0.5,0.5);
    \end{tikzpicture}
\end{equation*}

\textit{Unary map:} 
\begin{equation*}
    f \left(
    \begin{tikzpicture}[baseline={([yshift=-.5ex]current bounding box.center)}, scale = \scl]
        \draw[fill = \colOne, opacity = \opac] (0,0) rectangle (2,2);
        \node at (1,1) {\footnotesize $p$};
        \dummyNodes[1][-0.2][1][2.2]
    \end{tikzpicture}
    \right) = 
    \begin{tikzpicture}[baseline={([yshift=-.5ex]current bounding box.center)}, scale = \scl]
        \fill[\colOne, opacity = \opac] (0,0) rectangle (2,2);
        \node at (1,1) {\footnotesize $p$};
        \draw (-0.5,-0.5) rectangle (2,2);
        \draw[ultra thick] (0,2) -- (0,-0.5);
        \fill (0,0) \smlnd;
        \draw[thin, dashed, ->] (1,0) -- (1,-0.5);
    \end{tikzpicture}
\end{equation*}
After applying the unary map, any vertical cuts in $p$ are extended downwards until they reach the lower boundary. This is represented by the dashed arrow.

\textit{Binary map:} 
\def\scl{0.5}
\begin{align*}
    \begin{tikzpicture}[baseline={([yshift=-.5ex]current bounding box.center)}, scale = \scl]
        \draw[fill = \colOne, opacity = \opac] (0,0) rectangle (2,2);
        \node at (1,1) {\footnotesize $p_1$};
    \end{tikzpicture}
    \ \star \
    \begin{tikzpicture}[baseline={([yshift=-.5ex]current bounding box.center)}, scale = \scl]
        \draw[fill = \colTwo, opacity = \opac] (0,0) rectangle (2,2);
        \node at (1,1) {\footnotesize $p_2$};
    \end{tikzpicture}
    \ = \
    \begin{cases}
        \begin{tikzpicture}[baseline={([yshift=-.5ex]current bounding box.center)}, scale = \scl]
            \draw (0,0) rectangle (2.5,2.5);
            \fill[\colTwo, opacity = \opac] (0.5,0.5) rectangle (2.5,2.5);
            \node at (1.5,1.5) {\footnotesize $p_2$};
            \draw[ultra thick] (0,0.5) -- (2.5,0.5);
            \fill (0.5,0.5) \smlnd;
            \draw[thin, dashed, ->] (0.5,1.5) -- (0,1.5);
            \dummyNodes[1][-0.2][1][2.8]
            \dummyNodes[0][1][2.5][1]
        \end{tikzpicture}
        & \text{if 
            $\begin{tikzpicture}[baseline={([yshift=-.5ex]current bounding box.center)}, scale = \scl]
                \draw[fill = \colOne, opacity = \opac] (0,0) rectangle (2,2);
                \node at (1,1) {\footnotesize $p_1$};
            \end{tikzpicture}
            =
            \epsilon$,} \\
        \begin{tikzpicture}[baseline={([yshift=-.5ex]current bounding box.center)}, scale = \scl]
            \draw (0,0) rectangle (4,4);
            \fill[\colOne, opacity = \opac] (0,0) rectangle (2,2);
            \node at (1,1) {\footnotesize $p_1$};
            \fill[\colTwo, opacity = \opac] (2,2) rectangle (4,4);
            \node at (3,3) {\footnotesize $p_2$};
            \draw[ultra thick] (2,0) -- (2,4);
            \fill (2,2) \smlnd;
            \draw[thin, dashed, ->] (3,2) -- (3,0);
            \draw[thin, dashed, ->] (1,2) -- (1,4);
            \dummyNodes[1][-0.2][1][4.3]
            \dummyNodes[0][1][4][1]
        \end{tikzpicture}
        & \text{if 
            $\begin{tikzpicture}[baseline={([yshift=-.5ex]current bounding box.center)}, scale = \scl]
                \draw[fill = \colOne, opacity = \opac] (0,0) rectangle (2,2);
                \node at (1,1) {\footnotesize $p_1$};
            \end{tikzpicture}
            =
            \begin{tikzpicture}[baseline={([yshift=-.5ex]current bounding box.center)}, scale = \scl]
                \fill[\colOne, opacity = \opac] (0,0) rectangle (1,1) rectangle (2,2);
                \draw (0,0) rectangle (2,2);
                \draw[ultra thick] (0,1) -- (2,1);
                \node at (0.5,0.5) (p1d) {\footnotesize $p_1'$};
                \node at (1.5,1.5) (p1dd) {\footnotesize $p_1''$};
            \end{tikzpicture}$,} \\
        \begin{tikzpicture}[baseline={([yshift=-.5ex]current bounding box.center)}, scale = \scl]
            \draw (0,0) rectangle (4,4);
            \fill[\colOne, opacity = \opac] (0,0) rectangle (2,2);
            \node at (1,1) {\footnotesize $p_1$};
            \fill[\colTwo, opacity = \opac] (2,2) rectangle (4,4);
            \node at (3,3) {\footnotesize $p_2$};
            \draw[ultra thick] (0,2) -- (4,2);
            \fill (2,2) \smlnd;
            \draw[thin, dashed, ->] (2,3) -- (0,3);
            \draw[thin, dashed, ->] (2,1) -- (4,1);
            \dummyNodes[1][-0.2][1][4.3]
            \dummyNodes[0][1][4][1]
        \end{tikzpicture}
        & \text{if 
            $\begin{tikzpicture}[baseline={([yshift=-.5ex]current bounding box.center)}, scale = \scl]
                \draw[fill = \colOne, opacity = \opac] (0,0) rectangle (2,2);
                \node at (1,1) {\footnotesize $p_1$};
            \end{tikzpicture}
            =
            \begin{tikzpicture}[baseline={([yshift=-.5ex]current bounding box.center)}, scale = \scl]
                \fill[\colOne, opacity = \opac] (0,0) rectangle (1,1) rectangle (2,2);
                \draw (0,0) rectangle (2,2);
                \draw[ultra thick] (1,0) -- (1,2);
                \node at (0.5,0.5) (p1d) {\footnotesize $p_1'$};
                \node at (1.5,1.5) (p1dd) {\footnotesize $p_1''$};
            \end{tikzpicture}$.} \\
    \end{cases}
\end{align*}
After joining the two rectangulations as shown in the binary map, we then extend all cuts until they meet the boundary. This is represented by the dashed arrows.

Note that $p_1'$ and $p_1''$ above are not necessarily unique, however the product rule is well defined. In any rectangulation, there will be either a cut which joins the left and right boundaries or there will be a cut which joins the top and bottom boundaries. This follows from the constraint that each cut may divide only a single rectangle in two. It is the unique orientation of any longest cut which determines how we apply the product rule.

\textit{Norm:} If $p$ is a rectangle dissection using $n$ internal lines (so the rectangle consists of $n + 1$ smaller rectangles), then $\norm{p} = n + 1$.

\textit{(1,2)-magma:} The (1,2)-magma begins (sorting by norm) as follows:
\vspace{-.5em}
\begin{center}
    \def\scl {0.4}
    \begin{tabular}{| l l |}
    \hline
    \textit{Norm 1:} & 
    $\begin{tikzpicture}[baseline={([yshift=-.5ex]current bounding box.center)}, scale = \scl]
        \draw (0,0) rectangle (1,1);
        \dummyNodes[1][-0.2][1][1.2]
    \end{tikzpicture} 
    = \epsilon$ \\
    \hline
    \textit{Norm 2:} & 
    $\begin{tikzpicture}[baseline={([yshift=-.5ex]current bounding box.center)}, scale = \scl]
        \draw (0,0) rectangle (2,2);
        \draw[thick] (1,0) -- (1,2);
        \fill (1,1) \smlnd;
        \dummyNodes[1][-0.2][1][2.3]
    \end{tikzpicture} 
    = f(\epsilon)$, 
    \quad
    $\begin{tikzpicture}[baseline={([yshift=-.5ex]current bounding box.center)}, scale = \scl]
        \draw (0,0) rectangle (2,2);
        \draw[thick] (0,1) -- (2,1);
        \fill (1,1) \smlnd;
        \dummyNodes[1][-0.2][1][2.3]
    \end{tikzpicture} 
    = \epsilon \star \epsilon$ \\
    \hline
    \textit{Norm 3:} & 
    $\begin{tikzpicture}[baseline={([yshift=-.5ex]current bounding box.center)}, scale = \scl]
        \draw (0,0) rectangle (3,3);
        \draw[thick] (1,0) -- (1,3);
        \draw[thick] (2,0) -- (2,3);
        \fill (1,1) \smlnd (2,2) \smlnd;
        \dummyNodes[1][-0.2][1][3.3]
    \end{tikzpicture} 
    = f(f(\epsilon))$, 
    \quad
    $\begin{tikzpicture}[baseline={([yshift=-.5ex]current bounding box.center)}, scale = \scl]
        \draw (0,0) rectangle (3,3);
        \draw[thick] (1,0) -- (1,3);
        \draw[thick] (1,2) -- (3,2);
        \fill (1,1) \smlnd (2,2) \smlnd;
        \dummyNodes[1][-0.2][1][3.3]
    \end{tikzpicture} 
    = f(\epsilon \star \epsilon)$, 
    \quad
    $\begin{tikzpicture}[baseline={([yshift=-.5ex]current bounding box.center)}, scale = \scl]
        \draw (0,0) rectangle (3,3);
        \draw[thick] (0,1) -- (3,1);
        \draw[thick] (2,1) -- (2,3);
        \fill (1,1) \smlnd (2,2) \smlnd;
        \dummyNodes[1][-0.2][1][3.3]
    \end{tikzpicture} 
    = \epsilon \star f(\epsilon)$, \\
    &
    $\begin{tikzpicture}[baseline={([yshift=-.5ex]current bounding box.center)}, scale = \scl]
        \draw (0,0) rectangle (3,3);
        \draw[thick] (0,1) -- (3,1);
        \draw[thick] (0,2) -- (3,2);
        \fill (1,1) \smlnd (2,2) \smlnd;
        \dummyNodes[1][-0.2][1][3.3]
    \end{tikzpicture} 
    = \epsilon \star (\epsilon \star \epsilon)$, 
    \quad
    $\begin{tikzpicture}[baseline={([yshift=-.5ex]current bounding box.center)}, scale = \scl]
        \draw (0,0) rectangle (3,3);
        \draw[thick] (0,2) -- (3,2);
        \draw[thick] (1,0) -- (1,2);
        \fill (1,1) \smlnd (2,2) \smlnd;
        \dummyNodes[1][-0.2][1][3.3]
    \end{tikzpicture} 
    = f(\epsilon) \star \epsilon$, 
    \quad
    $\begin{tikzpicture}[baseline={([yshift=-.5ex]current bounding box.center)}, scale = \scl]
        \draw (0,0) rectangle (3,3);
        \draw[thick] (2,0) -- (2,3);
        \draw[thick] (0,1) -- (2,1);
        \fill (1,1) \smlnd (2,2) \smlnd;
        \dummyNodes[1][-0.2][1][3.3]
    \end{tikzpicture}
    = (\epsilon \star \epsilon) \star \epsilon$ \\
    \hline
    \end{tabular}
\end{center}
\end{schFamily}

\addcontentsline{toca}{subsection}{\schFamRef{schroderFamily:schroderTrees}: Schr\"oder unary-binary trees}
\begin{schFamily}
[Schr\"oder unary-binary trees \cite{SchroderOEIS}] \label{schroderFamily:schroderTrees} 

\def\scl{0.15}
$S_n$ is equal to the number of rooted trees with $n$ non-leaf nodes in which every vertex has degree at most 3, and in which the root has degree at most 2.

\textit{Generator:} The generator is a single vertex:
\begin{equation*}
    \epsilon =
    \begin{tikzpicture}[baseline={([yshift=-.5ex]current bounding box.center)}, scale = \scl]
        \fill (0,0) \smlnd;
    \end{tikzpicture} 
\end{equation*}

\textit{Unary map:} 
\begin{equation*}
    f \left(
        \begin{tikzpicture}[baseline={([yshift=-1ex]current bounding box.center)}, scale = \scl]
            \draw[fill = \colOne, opacity = \opac] (0,0) -- (-5,-10) -- (5,-10) -- (0,0);
            \node at (0,0) {\smlroot};
            \node at (0,-6) {\footnotesize $t$};
        \dummyNodes[0][1][0][-11]
        \end{tikzpicture}
    \right)
    =
    \begin{tikzpicture}[baseline={([yshift=-1ex]current bounding box.center)}, scale = \scl]
        \draw[fill = \colOne, opacity = \opac] (0,0) -- (-5,-10) -- (5,-10) -- (0,0);
        \draw (0,0) -- (0,3);
        \node at (0,3) {\smlroot};
        \draw[fill] (0,0)\smlnd;
        \node at (0,-6) {\footnotesize $t$};
    \end{tikzpicture}
\end{equation*}

\textit{Binary map:}
\begin{equation*}
    \begin{tikzpicture}[baseline={([yshift=-1ex]current bounding box.center)}, scale = \scl]
        \draw[fill = \colOne, opacity = \opac] (0,0) -- (-5,-10) -- (5,-10) -- (0,0);
        \node at (0,0) {\smlroot};
        \node at (0,-6) {\footnotesize $t_1$};
    \end{tikzpicture}
    \ \star \ 
    \begin{tikzpicture}[baseline={([yshift=-1ex]current bounding box.center)}, scale = \scl]
        \draw[fill = \colTwo, opacity = \opac] (0,0) -- (-5,-10) -- (5,-10) -- (0,0);
        \node at (0,0) {\smlroot};
        \node at (0,-6) {\footnotesize $t_2$};
    \end{tikzpicture}
    =
    \begin{tikzpicture}[baseline={([yshift=-1ex]current bounding box.center)}, scale = \scl]
        \draw[fill = \colOne, opacity = \opac] (0,0) -- (-5,-10) -- (5,-10) -- (0,0);
        \draw[fill = \colTwo, opacity = \opac] (12,0) -- (7,-10) -- (17,-10) -- (12,0);
        \draw (0,0) -- (6,5) -- (12,0);
        \node at (6,5) {\smlroot};
        \fill (0,0) \smlnd (12,0) \smlnd;
        \node at (0,-6) {\footnotesize $t_1$};
        \node at (12,-6) {\footnotesize $t_2$};
    \end{tikzpicture}
\end{equation*}

\textit{Norm:} If $t$ is a tree with $n$ non-leaf nodes, then $\norm{t} = n + 1$. Note that we consider the only node in $\epsilon$ to be a leaf node, and hence $\norm{\epsilon} = 1$.

\pagebreak
\textit{(1,2)-magma:} The (1,2)-magma begins (sorting by norm) as follows:
\vspace{-.5em}
\begin{center}
    \def\scl{0.5}
    \begin{tabular}{| l l |}
    \hline
    \textit{Norm 1:} & 
    $\begin{tikzpicture}[baseline={([yshift=-.5ex]current bounding box.center)}, scale = \scl]
        \fill (0,0) \smlnd;
    \end{tikzpicture}
    = \epsilon$ \\
    \hline
    \textit{Norm 2:} & 
    $\begin{tikzpicture}[baseline={([yshift=-.5ex]current bounding box.center)}, scale = \scl]
        \draw (0,0) \un;
        \fill (0,0) \smlunnd;
        \node at (0,0) (root) {\smlroot};
        \dummyNodes[0][0.2][0][-1.2]
    \end{tikzpicture} 
    = f(\epsilon)$, 
    \quad
    $\begin{tikzpicture}[baseline={([yshift=-.5ex]current bounding box.center)}, scale = \scl]
        \draw (0,0) \bin;
        \fill (0,0) \smlbinnd;
        \node at (0,0) (root) {\smlroot};
        \dummyNodes[0][0.2][0][-1.2]
    \end{tikzpicture} 
    = \epsilon \star \epsilon$ \\
    \hline
    \textit{Norm 3:} & 
    $\begin{tikzpicture}[baseline={([yshift=-.5ex]current bounding box.center)}, scale = \scl]
        \draw (0,0) \un (0,-1) \un;
        \fill (0,0) \smlunnd (0,-1) \smlunnd;
        \node at (0,0) (root) {\smlroot};
        \dummyNodes[0][0.2][0][-2.2]
    \end{tikzpicture} 
    = f(f(\epsilon))$, 
    \quad
    $\begin{tikzpicture}[baseline={([yshift=-.5ex]current bounding box.center)}, scale = \scl]
        \draw (0,0) \un (0,-1) \bin;
        \fill (0,0) \smlunnd (0,-1) \smlbinnd;
        \node at (0,0) (root) {\smlroot};
        \dummyNodes[0][0.2][0][-2.2]
    \end{tikzpicture} 
    = f(\epsilon \star \epsilon)$, 
    \quad
    $\begin{tikzpicture}[baseline={([yshift=-.5ex]current bounding box.center)}, scale = \scl]
        \draw (0,0) \bin (1,-1) \un;
        \fill (0,0) \smlbinnd (1,-1) \smlunnd;
        \node at (0,0) (root) {\smlroot};
        \dummyNodes[0][0.2][0][-2.2]
    \end{tikzpicture} 
    = \epsilon \star f(\epsilon)$, \\ 
    &
    $\begin{tikzpicture}[baseline={([yshift=-.5ex]current bounding box.center)}, scale = \scl]
        \draw (0,0) \bin (1,-1) \bin;
        \fill (0,0) \smlbinnd (1,-1) \smlbinnd;
        \node at (0,0) (root) {\smlroot};
        \dummyNodes[0][0.2][0][-2.2]
    \end{tikzpicture}
    = \epsilon \star (\epsilon \star \epsilon)$, 
    \quad
    $\begin{tikzpicture}[baseline={([yshift=-.5ex]current bounding box.center)}, scale = \scl]
        \draw (0,0) \bin (-1,-1) \un;
        \fill (0,0) \smlbinnd (-1,-1) \smlunnd;
        \node at (0,0) (root) {\smlroot};
        \dummyNodes[0][0.2][0][-2.2]
    \end{tikzpicture} 
    = f(\epsilon) \star \epsilon$, 
    \quad
    $\begin{tikzpicture}[baseline={([yshift=-.5ex]current bounding box.center)}, scale = \scl]
        \draw (0,0) \bin (-1,-1) \bin;
        \fill (0,0) \smlbinnd (-1,-1) \smlbinnd;
        \node at (0,0) (root) {\smlroot};
        \dummyNodes[0][0.2][0][-2.2]
    \end{tikzpicture}
    = (\epsilon \star \epsilon) \star \epsilon$ \\
    \hline
    \end{tabular}
\end{center}
\end{schFamily}

\addcontentsline{toca}{subsection}{\schFamRef{schroderFamily:polygonDissections}: Polygon dissections}
\begin{schFamily}
[Polygon dissections \cite{SchroderOEIS}] \label{schroderFamily:polygonDissections} 

\def\scl{0.03}
\def\numPoints{4}

$S_n$ is the number of dissections of a regular $(n + 4)$-gon by diagonals that do not touch the base. A diagonal is a straight line joining two non-consecutive vertices and dissection means that diagonals are non-crossing though they may share an endpoint. We draw all polygons with a marked side which is denoted by a thick black line. This fixes the orientation of the polygon and distinguishes polygons which only differ by a rotation.

\textit{Generator:} The trivial dissection of a square obtained by placing no diagonals:
\begin{equation*}
    \epsilon =
    \begin{tikzpicture}[baseline={([yshift=-1ex]current bounding box.center)}, scale = \scl]
        \foreach \i in {1,...,\numPoints}
            {
            \draw[fill] (-90 - 180/\numPoints + 360/\numPoints * \i:10) circle [radius = 0.2];
            \coordinate (\i) at (-90 - 180/\numPoints + 360/\numPoints * \i:10);
            }
        
        \draw[ultra thick] (1) to (\numPoints);    
        \foreach \i [evaluate = \i as \ieval using \i - 1] in {2,...,\numPoints}
            {
            \draw (\i) to (\ieval);
            }
        
        \dummyNodes[0][-11][0][11]
    \end{tikzpicture}
\end{equation*}

\textit{Unary map:}

\def\numPoints{8}
\def\scl {0.08}

\begin{equation*}
    f \left(
    \begin{tikzpicture}[baseline={([yshift=-.5ex]current bounding box.center)}, scale = \scl]
        \foreach \i in {1,...,\numPoints}
            {
            \coordinate (\i) at (-90 - 180/\numPoints + 360/\numPoints * \i:10);
            }
            
        \fill[\colOne, opacity = \opac] (7) -- (8) -- (1) (0,0) (-70:10) arc (-70:203:10);
        
        \draw[ultra thick] (1) to (\numPoints);    
        \draw (7) to (8);
        \draw[dashed] (0,0) (-70:10) arc (-70:203:10);
        \node at (0,0) (p) {\footnotesize $p$};
        \node at (225:11) (e) {\scriptsize $e$};
        \dummyNodesPolar[-90:12][90:12]
    \end{tikzpicture}
    \right)
    =
    \begin{tikzpicture}[baseline={([yshift=-.5ex]current bounding box.center)}, scale = \scl]
        \def\numPoints{9}
        \foreach \i in {1,...,\numPoints}
            {
            \coordinate (\i) at (-90 - 180/\numPoints + 360/\numPoints * \i:10);
            }
        \begin{scope}
        \clip (0,0) circle [radius = 10];
        \fill[\colOne, opacity = \opac] (7) -- (8) to [bend left = 90] (9) -- (10,-10) -- (10,10) -- (-10,10) -- (7);
        \end{scope}
        \draw[ultra thick] (1) to (\numPoints);    
        \draw (7) to (8);
        \draw[red] (8) to (9);
        \draw[dotted] (8) to [bend left = 90] (9);
        \draw[dashed] (0,0) (-70:10) arc (-70:170:10);
        \node at (0,0) (p) {\footnotesize $p$};
        \node at (190:11) (e) {\scriptsize $e$};
        \dummyNodesPolar[-90:12][90:12]
    \end{tikzpicture}
\end{equation*}

\textit{Binary map:}
\def\numPoints{12}
\begin{equation*}
    \begin{tikzpicture}[baseline={([yshift=-.5ex]current bounding box.center)}, scale = \scl]
        \foreach \i in {1,...,\numPoints}
            {
            \coordinate (\i) at (-90 - 180/\numPoints + 360/\numPoints * \i:10);
            }

        \fill[\colOne, opacity = \opac] (11) -- (12) -- (1) -- (2) (0,0) (-45:10) arc (-45:225:10);
        
        \draw[ultra thick] (1) to (\numPoints);    
        \draw (1) to (2);
        \draw (11) to (12);
        
        \draw[dashed] (0,0) (-45:10) arc (-45:225:10);
        \node at (0,0) (p) {\footnotesize $p_1$};
        \node at (240:11.5) (f1) {\scriptsize $e_1$};
        \node at (-90:11.5) (f2) {\scriptsize $e_2$};
        \node at (-60:11.5) (f3) {\scriptsize $e_3$};
        \dummyNodesPolar[-90:12][90:12]
    \end{tikzpicture}
    \ \star \ 
    \begin{tikzpicture}[baseline={([yshift=-.5ex]current bounding box.center)}, scale = \scl]
        \foreach \i in {1,...,\numPoints}
            {
            \coordinate (\i) at (-90 - 180/\numPoints + 360/\numPoints * \i:10);
            }
        
        \fill[\colTwo, opacity = \opac] (10) -- (11) -- (12) -- (1) -- (2) (0,0) (-45:10) arc (-45:195:10);
        
        \draw[ultra thick] (1) to (\numPoints);    
        \draw (1) to (2);
        \draw (10) to (11);
        \draw (11) to (12);
        \draw[dashed] (0,0) (-45:10) arc (-45:195:10);
        \node at (0,0) (p) {\footnotesize $p_2$};
        \node at (210:11.7) (e1) {\scriptsize $f_1$};
        \node at (240:11.7) (e2) {\scriptsize $f_2$};
        \node at (-90:11.7) (e3) {\scriptsize $f_3$};
        \node at (-60:11.7) (e4) {\scriptsize $f_4$};
        \dummyNodesPolar[-90:12][90:12]
    \end{tikzpicture}
    =
    \begin{tikzpicture}[baseline={([yshift=-.5ex]current bounding box.center)}, scale = 1.5*\scl]
        \foreach \i in {1,...,\numPoints}
            {
            \coordinate (\i) at (-90 - 180/\numPoints + 360/\numPoints * \i:10);
            }
            
        \fill[\colTwo, opacity = \opac] (4) -- (11) -- (12) -- (1) -- (2) (0,0) (-45:10) arc (-45:15:10);
        \fill[\colOne, opacity = \opac] (10) -- (224:9.9) -- (16:9.9) (0,0) (15:10) arc (15:195:10);
        
        \draw[ultra thick] (1) to (\numPoints);
        \draw (224:9.9) to (16:9.9);
        \draw (1) to (2);
        \draw[ultra thick] (-7.1934,-6.84658) to (1.23293,-1.9792);
        \draw (10) to (224:9.9);
        \draw (4) to (11);
        \draw (11) to (12);
        \draw[dashed] (0,0) (15:10) arc (15:195:10);
        \draw[dashed] (0,0) (-45:10) arc (-45:15:10);
        \node at (3,-6) (p1) {\footnotesize $p_2$};
        \node at (2,-3.3) (e1) {\scriptsize $f_1$};
        \node at (240:11) (e2) {\scriptsize $f_2$};
        \node at (-90:11.2) (e3) {\scriptsize $f_3$};
        \node at (-60:11) (e4) {\scriptsize $f_4$};
        \node at (-2,3) (p2) {\footnotesize $p_1$};
        \node at (210:11) (f1) {\scriptsize $e_1$};
        \node at (-4,-3) (f2) {\scriptsize $e_2$};
        \node at (4,1.5) (f3) {\scriptsize $e_3$};
        \dummyNodesPolar[-90:12][90:12]
    \end{tikzpicture}
\end{equation*}
where the edge $f_1$ coincides with the edges $e_2e_3$ to form a single diagonal.

\textit{Norm:} If $p$ is a dissection of an $n$-gon, then $\norm{p} = n - 3$.

\textit{(1,2)-magma:} The (1,2)-magma begins (sorting by norm) as follows:
\vspace{-.5em}
\begin{center}
    \def\scl{0.03}
    \begin{tabular}{| l l |}
    \hline
    \textit{Norm 1:} & \def\numPoints {4}
    $\begin{tikzpicture}[baseline={([yshift=-.5ex]current bounding box.center)}, scale = \scl]
        \foreach \i in {1,...,\numPoints}
            {
            \draw[fill] (-90 - 180/\numPoints + 360/\numPoints * \i:10) circle [radius = 0.2];
            \coordinate (\i) at (-90 - 180/\numPoints + 360/\numPoints * \i:10);
            }
        
        \draw[ultra thick] (1) to (\numPoints);    
        \foreach \i [evaluate = \i as \ieval using \i - 1] in {2,...,\numPoints}
            {
            \draw (\i) to (\ieval);
            }
        
        \dummyNodes[0][-11][0][11]
    \end{tikzpicture}
    = \epsilon$ \\
    \hline
    \textit{Norm 2:} & \def\numPoints{5} \def\scl{0.04}
    $\begin{tikzpicture}[baseline={([yshift=-.5ex]current bounding box.center)}, scale = \scl]
        \foreach \i in {1,...,\numPoints}
            {
            \draw[fill] (-90 - 180/\numPoints + 360/\numPoints * \i:10) circle [radius = 0.2];
            \coordinate (\i) at (-90 - 180/\numPoints + 360/\numPoints * \i:10);
            }
        
        \draw[ultra thick] (1) to (\numPoints);    
        \foreach \i [evaluate = \i as \ieval using \i - 1] in {2,...,\numPoints}
            {
            \draw (\i) to (\ieval);
            }
        
        \dummyNodes[0][-11][0][11]
    \end{tikzpicture} = f(\epsilon)$, 
    \quad
    $\begin{tikzpicture}[baseline={([yshift=-.5ex]current bounding box.center)}, scale = \scl]
        \foreach \i in {1,...,\numPoints}
            {
            \draw[fill] (-90 - 180/\numPoints + 360/\numPoints * \i:10) circle [radius = 0.2];
            \coordinate (\i) at (-90 - 180/\numPoints + 360/\numPoints * \i:10);
            }
        
        \draw[ultra thick] (1) to (\numPoints);    
        \foreach \i [evaluate = \i as \ieval using \i - 1] in {2,...,\numPoints}
            {
            \draw (\i) to (\ieval);
            }
            
        \dissectionArc[2][4]
        
        \dummyNodes[0][-11][0][11]
    \end{tikzpicture} = \epsilon \star \epsilon$ \\
    \hline
    \textit{Norm 3:} & \def\numPoints{6} \def\scl{0.05}
    $\begin{tikzpicture}[baseline={([yshift=-.5ex]current bounding box.center)}, scale = \scl]
        \foreach \i in {1,...,\numPoints}
            {
            \draw[fill] (-90 - 180/\numPoints + 360/\numPoints * \i:10) circle [radius = 0.2];
            \coordinate (\i) at (-90 - 180/\numPoints + 360/\numPoints * \i:10);
            }
        
        \draw[ultra thick] (1) to (\numPoints);    
        \foreach \i [evaluate = \i as \ieval using \i - 1] in {2,...,\numPoints}
            {
            \draw (\i) to (\ieval);
            }
        
        \dummyNodes[0][-11][0][11]
    \end{tikzpicture} = f(f(\epsilon))$, \quad
    $\begin{tikzpicture}[baseline={([yshift=-.5ex]current bounding box.center)}, scale = \scl]
        \foreach \i in {1,...,\numPoints}
            {
            \draw[fill] (-90 - 180/\numPoints + 360/\numPoints * \i:10) circle [radius = 0.2];
            \coordinate (\i) at (-90 - 180/\numPoints + 360/\numPoints * \i:10);
            }
        
        \draw[ultra thick] (1) to (\numPoints);    
        \foreach \i [evaluate = \i as \ieval using \i - 1] in {2,...,\numPoints}
            {
            \draw (\i) to (\ieval);
            }
        
        \dissectionArc[2][4]
        
        \dummyNodes[0][-11][0][11]
    \end{tikzpicture} = f(\epsilon \star \epsilon)$, \quad
    $\begin{tikzpicture}[baseline={([yshift=-.5ex]current bounding box.center)}, scale = \scl]
        \foreach \i in {1,...,\numPoints}
            {
            \draw[fill] (-90 - 180/\numPoints + 360/\numPoints * \i:10) circle [radius = 0.2];
            \coordinate (\i) at (-90 - 180/\numPoints + 360/\numPoints * \i:10);
            }
        
        \draw[ultra thick] (1) to (\numPoints);    
        \foreach \i [evaluate = \i as \ieval using \i - 1] in {2,...,\numPoints}
            {
            \draw (\i) to (\ieval);
            }
        
        \dissectionArc[3][5]
        
        \dummyNodes[0][-11][0][11]
    \end{tikzpicture}
    = \epsilon \star f(\epsilon)$, \\ & \def\numPoints{6} \def\scl{0.05}
    $\begin{tikzpicture}[baseline={([yshift=-.5ex]current bounding box.center)}, scale = \scl]
        \foreach \i in {1,...,\numPoints}
            {
            \draw[fill] (-90 - 180/\numPoints + 360/\numPoints * \i:10) circle [radius = 0.2];
            \coordinate (\i) at (-90 - 180/\numPoints + 360/\numPoints * \i:10);
            }
        
        \draw[ultra thick] (1) to (\numPoints);    
        \foreach \i [evaluate = \i as \ieval using \i - 1] in {2,...,\numPoints}
            {
            \draw (\i) to (\ieval);
            }
        
        \dissectionArc[2][5]
        \dissectionArc[3][5]
        
        \dummyNodes[0][-11][0][11]
    \end{tikzpicture}
    = \epsilon \star (\epsilon \star \epsilon)$, \quad
    $\begin{tikzpicture}[baseline={([yshift=-.5ex]current bounding box.center)}, scale = \scl]
        \foreach \i in {1,...,\numPoints}
            {
            \draw[fill] (-90 - 180/\numPoints + 360/\numPoints * \i:10) circle [radius = 0.2];
            \coordinate (\i) at (-90 - 180/\numPoints + 360/\numPoints * \i:10);
            }
        
        \draw[ultra thick] (1) to (\numPoints);    
        \foreach \i [evaluate = \i as \ieval using \i - 1] in {2,...,\numPoints}
            {
            \draw (\i) to (\ieval);
            }
        \dissectionArc[2][5]
        
        \dummyNodes[0][-11][0][11]
    \end{tikzpicture} = f(\epsilon) \star \epsilon$, \quad
    $\begin{tikzpicture}[baseline={([yshift=-.5ex]current bounding box.center)}, scale = \scl]
        \foreach \i in {1,...,\numPoints}
            {
            \draw[fill] (-90 - 180/\numPoints + 360/\numPoints * \i:10) circle [radius = 0.2];
            \coordinate (\i) at (-90 - 180/\numPoints + 360/\numPoints * \i:10);
            }
        
        \draw[ultra thick] (1) to (\numPoints);    
        \foreach \i [evaluate = \i as \ieval using \i - 1] in {2,...,\numPoints}
            {
            \draw (\i) to (\ieval);
            }
        
        \dissectionArc[2][5]
        \dissectionArc[2][4]
        
        \dummyNodes[0][-11][0][11]
    \end{tikzpicture} = (\epsilon \star \epsilon) \star \epsilon$ \\
    \hline
    \end{tabular}
\end{center}
\end{schFamily}

\addcontentsline{toca}{subsection}{\schFamRef{schroderFamily:triangularGridMatchings}: Perfect matchings in an Aztec triangle}
\begin{schFamily}
[Perfect matchings in an Aztec triangle \cite{Ciucu1996}] \label{schroderFamily:triangularGridMatchings} 

\def\scl{0.25}

$S_n$ is the number of perfect matchings in an Aztec triangle, which is a triangular grid of $n^2$ squares. This grid is formed by starting with one square, then placing a centred row of 3 squares beneath it, followed by a centred row of 5 squares beneath this and so on. We then place a node at the corner of each square. We allow only perfect matchings in which no two edges cross when all edges are drawn as straight lines connecting two nodes. This enforces that any node is matched with a node placed on an adjacent corner of the same square ie.\ all matching edges are either a horizontal or vertical and unit length.

\textit{Generator:} The trivial empty matching in a triangular grid of zero squares:
\begin{equation*}
    \epsilon = \emptyset.
\end{equation*}

\textit{Unary map:} 

\begin{equation*}
    f \left(
    \begin{tikzpicture}[baseline={([yshift=-.5ex]current bounding box.center)}, scale = \scl]
        \fill[\colOne, opacity = \opac] (0,0) -- (11,0) -- (11,1) -- (10,1) -- (6,4) -- (5,4) -- (1,1) -- (0,1) -- (0,0);
        \draw (0,0) -- (11,0) -- (11,1) -- (10,1) -- (6,4) -- (5,4) -- (1,1) -- (0,1) -- (0,0);
        \dummyNodes[1][-0.3][1][4.3]
        \node at (5.5,1.75) {\footnotesize $m$};
    \end{tikzpicture}
    \right)
    =
    \begin{tikzpicture}[baseline={([yshift=-.5ex]current bounding box.center)}, scale = \scl]
        \fill[\colOne, opacity = \opac] (0,0) -- (11,0) -- (11,1) -- (10,1) -- (6,4) -- (5,4) -- (1,1) -- (0,1) -- (0,0);
        \draw (0,0) -- (11,0) -- (11,1) -- (10,1) -- (6,4) -- (5,4) -- (1,1) -- (0,1) -- (0,0);
        \fill[red] (6,5) \halfnd (7,5) \halfnd (7,4) \halfnd (8,4) \halfnd (8,3) \halfnd (9,3) \halfnd (11,2) \halfnd (12,2) \halfnd (12,1) \halfnd (13,1) \halfnd (12,0) \halfnd (13,0) \halfnd;
        \draw[red] (6,5) -- (7,5) (7,4) -- (8,4) (8,3) -- (9,3) (11,2) -- (12,2) (12,1) -- (13,1) (12,0) -- (13,0);
        \node[rotate = -20] at (10,2.3) (D) {\color{red}{\tiny $\cdots$}};
        \node at (5.5,1.75) {\footnotesize $m$};
    \end{tikzpicture}
\end{equation*}

\textit{Binary map:}
\def\scl{0.2}
\begin{align*}
    \begin{tikzpicture}[baseline={([yshift=-.5ex]current bounding box.center)}, scale = \scl]
        \fill[\colOne, opacity = \opac] (0,0) -- (11,0) -- (11,1) -- (10,1) -- (6,4) -- (5,4) -- (1,1) -- (0,1) -- (0,0);
        \draw (0,0) -- (11,0) -- (11,1) -- (10,1) -- (6,4) -- (5,4) -- (1,1) -- (0,1) -- (0,0);
        \node at (5.5,1.75) {\footnotesize $m_1$};
    \end{tikzpicture}
    \ \star \
    \begin{tikzpicture}[baseline={([yshift=-.5ex]current bounding box.center)}, scale = \scl]
        \fill[\colTwo, opacity = \opac] (0,0) -- (11,0) -- (11,1) -- (10,1) -- (6,4) -- (5,4) -- (1,1) -- (0,1) -- (0,0);
        \draw (0,0) -- (11,0) -- (11,1) -- (10,1) -- (6,4) -- (5,4) -- (1,1) -- (0,1) -- (0,0);
        \node at (5.5,1.75) {\footnotesize $m_2$};
    \end{tikzpicture}
    \ = \ \def\scl{0.3}
    \begin{tikzpicture}[baseline={([yshift=-.5ex]current bounding box.center)}, scale = \scl]
        \fill[\colOne, opacity = \opac] (0,0) -- (11,0) -- (11,1) -- (10,1) -- (6,4) -- (5,4) -- (1,1) -- (0,1) -- (0,0);
        \draw (0,0) -- (11,0) -- (11,1) -- (10,1) -- (6,4) -- (5,4) -- (1,1) -- (0,1) -- (0,0);
        \fill[red] (12,0) \halfnd (12,1) \halfnd (13,0) \halfnd (14,0) \halfnd (15,0) \halfnd (16,0) \halfnd (21,0) \halfnd (22,0) \halfnd (23,0) \halfnd (24,0) \halfnd (25,0) \halfnd (25,1) \halfnd;
        \draw[red] (12,0) -- (12,1) (13,0) -- (14,0) (15,0) -- (16,0) (21,0) -- (22,0) (23,0) -- (24,0) (25,0) -- (25,1);
        \node at (18.5,0) (a) {\color{red}{\tiny $\cdots$}};
        \fill[red] (11,2) \halfnd (12,2) \halfnd (10,3) \halfnd (11,3) \halfnd (12,3) \halfnd (13,3) \halfnd;
        \draw[red] (11,2) -- (12,2) (10,3) -- (11,3) (12,3) -- (13,3);
        \fill[red] (7,4.25) \halfnd (8,4.25) \halfnd (6,5) \halfnd (7,5) \halfnd (7,5.75) \halfnd (8,5.75) \halfnd (16,5) \halfnd (17,5) \halfnd (17,5.75) \halfnd (18,5.75) \halfnd (16,6.5) \halfnd (17,6.5) \halfnd (12.17,9.375) \halfnd (13.17,9.375) \halfnd (11.17,8.625) \halfnd (12.17,8.625) \halfnd (13.17,8.625) \halfnd (14.17,8.625) \halfnd (10.17,7.875) \halfnd (11.17,7.875) \halfnd (12.17,7.875) \halfnd (13.17,7.875) \halfnd (14.17,7.875) \halfnd (15.17,7.875) \halfnd;
        \draw[red] (7,4.25) -- (8,4.25) (6,5) -- (7,5) (7,5.75) -- (8,5.75) (16,5) -- (17,5) (16,6.5) -- (17,6.5) (17,5.75) -- (18,5.75) (12.17,9.375) -- (13.17,9.375) (11.17,8.625) -- (12.17,8.625) (13.17,8.625) -- (14.17,8.625) (10.17,7.875) -- (11.17,7.875) (12.17,7.875) -- (13.17,7.875) (14.17,7.875) -- (15.17,7.875);
        \node[rotate = -30] at (9,3.5) (b) {\color{red}{\tiny $\cdots$}};
        \node[rotate = 30] at (14.5,4) (c) {\color{red}{\tiny $\cdots$}};
        \node[rotate = 90] at (12.5,5.75) (d) {\color{red}{\tiny $\cdots$}};
        \node[rotate = 30] at (9,7) (e) {\color{red}{\tiny $\cdots$}};
        \node[rotate = -30] at (16,7.15) (f) {\color{red}{\tiny $\cdots$}};
        \fill[\colTwo, opacity = \opac] (13,1) -- (24,1) -- (24,2) -- (23,2) -- (19,5) -- (18,5) -- (14,2) -- (13,2) -- (13,1);
        \draw (13,1) -- (24,1) -- (24,2) -- (23,2) -- (19,5) -- (18,5) -- (14,2) -- (13,2) -- (13,1);
        \node at (5.5,1.75) {\footnotesize $m_1$};
        \node at (18.5,2.75) {\footnotesize $m_2$};
    \end{tikzpicture}
\end{align*}

\textit{Norm:} If $m$ is a perfect matching in an Aztec triangle of $n^2$ squares, then $\norm{m} = n + 1$.

\pagebreak
\textit{(1,2)-magma:} The (1,2)-magma begins (sorting by norm) as follows:
\vspace{-.5em}
\begin{center}
    \def\scl{0.5}
    \begin{tabular}{| l l |}
    \hline
    \textit{Norm 1:} & 
    $\emptyset = \epsilon$ \\
    \hline
    \textit{Norm 2:} & \def\numRows {2}
    $\begin{tikzpicture}[baseline={([yshift=-.5ex]current bounding box.center)}, scale = \scl]
        \foreach \i [evaluate = \i as \ieval using 2*(\i - 1)] in {2,...,\numRows}
            {
            \foreach \j in {1,...,\ieval}
                {
                \fill (\numRows - \i + \j - 1,\numRows - \i + 1) \halfnd;
                }
            }
        \foreach \i [evaluate = \i as \ieval using 2*(\i - 1)] in {\numRows}
            {
            \foreach \j in {1,...,\ieval}
                {
                \fill (\j - 1,0) \halfnd;
                }
            }
        \draw (0,0) -- (1,0) (0,1) -- (1,1);
        \dummyNodes[1][-0.3][1][1.2]
    \end{tikzpicture} 
    = f(\epsilon)$, 
    \quad
    $\begin{tikzpicture}[baseline={([yshift=-.5ex]current bounding box.center)}, scale = \scl]
        \foreach \i [evaluate = \i as \ieval using 2*(\i - 1)] in {2,...,\numRows}
            {
            \foreach \j in {1,...,\ieval}
                {
                \fill (\numRows - \i + \j - 1,\numRows - \i + 1) \halfnd;
                }
            }
        \foreach \i [evaluate = \i as \ieval using 2*(\i - 1)] in {\numRows}
            {
            \foreach \j in {1,...,\ieval}
                {
                \fill (\j - 1,0) \halfnd;
                }
            }
        \draw (0,0) -- (0,1) (1,0) -- (1,1);
        \dummyNodes[1][-0.3][1][1.2]
    \end{tikzpicture} 
    = \epsilon \star \epsilon$ \\
    \hline
    \textit{Norm 3:} & \def\numRows{3}
    $\begin{tikzpicture}[baseline={([yshift=-.5ex]current bounding box.center)}, scale = \scl]
        \foreach \i [evaluate = \i as \ieval using 2*(\i - 1)] in {2,...,\numRows}
            {
            \foreach \j in {1,...,\ieval}
                {
                \fill (\numRows - \i + \j - 1,\numRows - \i + 1) \halfnd;
                }
            }
        \foreach \i [evaluate = \i as \ieval using 2*(\i - 1)] in {\numRows}
            {
            \foreach \j in {1,...,\ieval}
                {
                \fill (\j - 1,0) \halfnd;
                }
            } 
        \draw (0,0) -- (1,0) (2,0) -- (3,0) (0,1) -- (1,1) (2,1) -- (3,1) (1,2) -- (2,2);
        \dummyNodes[1][-0.3][1][2.2]
    \end{tikzpicture}
    = f(f(\epsilon))$, 
    \quad
    $\begin{tikzpicture}[baseline={([yshift=-.5ex]current bounding box.center)}, scale = \scl]
        \foreach \i [evaluate = \i as \ieval using 2*(\i - 1)] in {2,...,\numRows}
            {
            \foreach \j in {1,...,\ieval}
                {
                \fill (\numRows - \i + \j - 1,\numRows - \i + 1) \halfnd;
                }
            }
        \foreach \i [evaluate = \i as \ieval using 2*(\i - 1)] in {\numRows}
            {
            \foreach \j in {1,...,\ieval}
                {
                \fill (\j - 1,0) \halfnd;
                }
            } 
        \draw (0,0) -- (0,1) (2,0) -- (3,0) (1,0) -- (1,1) (2,1) -- (3,1) (1,2) -- (2,2);
        \dummyNodes[1][-0.3][1][2.2]
    \end{tikzpicture} 
    = f(\epsilon \star \epsilon)$, 
    \quad
    $\begin{tikzpicture}[baseline={([yshift=-.5ex]current bounding box.center)}, scale = \scl]
        \foreach \i [evaluate = \i as \ieval using 2*(\i - 1)] in {2,...,\numRows}
            {
            \foreach \j in {1,...,\ieval}
                {
                \fill (\numRows - \i + \j - 1,\numRows - \i + 1) \halfnd;
                }
            }
        \foreach \i [evaluate = \i as \ieval using 2*(\i - 1)] in {\numRows}
            {
            \foreach \j in {1,...,\ieval}
                {
                \fill (\j - 1,0) \halfnd;
                }
            } 
        \draw (0,0) -- (0,1) (1,0) -- (2,0) (1,1) -- (2,1) (3,0) -- (3,1) (1,2) -- (2,2);
        \dummyNodes[1][-0.5][1][2.2]
    \end{tikzpicture}
    = \epsilon \star f(\epsilon)$, \\ 
    & \def\numRows{3}
    $\begin{tikzpicture}[baseline={([yshift=-.5ex]current bounding box.center)}, scale = \scl]
        \foreach \i [evaluate = \i as \ieval using 2*(\i - 1)] in {2,...,\numRows}
            {
            \foreach \j in {1,...,\ieval}
                {
                \fill (\numRows - \i + \j - 1,\numRows - \i + 1) \halfnd;
                }
            }
        \foreach \i [evaluate = \i as \ieval using 2*(\i - 1)] in {\numRows}
            {
            \foreach \j in {1,...,\ieval}
                {
                \fill (\j - 1,0) \halfnd;
                }
            } 
        \draw (0,0) -- (0,1) (1,0) -- (2,0) (1,1) -- (1,2) (3,0) -- (3,1) (2,1) -- (2,2);
        \dummyNodes[1][-0.3][1][2.8]
    \end{tikzpicture}
    = \epsilon \star (\epsilon \star \epsilon)$, 
    \quad
    $\begin{tikzpicture}[baseline={([yshift=-.5ex]current bounding box.center)}, scale = \scl]
        \foreach \i [evaluate = \i as \ieval using 2*(\i - 1)] in {2,...,\numRows}
            {
            \foreach \j in {1,...,\ieval}
                {
                \fill (\numRows - \i + \j - 1,\numRows - \i + 1) \halfnd;
                }
            }
        \foreach \i [evaluate = \i as \ieval using 2*(\i - 1)] in {\numRows}
            {
            \foreach \j in {1,...,\ieval}
                {
                \fill (\j - 1,0) \halfnd;
                }
            } 
        \draw (0,0) -- (1,0) (0,1) -- (1,1) (2,0) -- (2,1) (3,0) -- (3,1) (1,2) -- (2,2);
        \dummyNodes[1][-0.3][1][2.2]
    \end{tikzpicture}
    = f(\epsilon) \star \epsilon$, 
    \quad
    $\begin{tikzpicture}[baseline={([yshift=-.5ex]current bounding box.center)}, scale = \scl]
        \foreach \i [evaluate = \i as \ieval using 2*(\i - 1)] in {2,...,\numRows}
            {
            \foreach \j in {1,...,\ieval}
                {
                \fill (\numRows - \i + \j - 1,\numRows - \i + 1) \halfnd;
                }
            }
        \foreach \i [evaluate = \i as \ieval using 2*(\i - 1)] in {\numRows}
            {
            \foreach \j in {1,...,\ieval}
                {
                \fill (\j - 1,0) \halfnd;
                }
            } 
        \draw (0,0) -- (0,1) (1,0) -- (1,1) (2,0) -- (2,1) (3,0) -- (3,1) (1,2) -- (2,2);
        \dummyNodes[1][-0.3][1][2.2]
    \end{tikzpicture} 
    = (\epsilon \star \epsilon) \star \epsilon$ \\
    \hline
    \end{tabular}
\end{center}
\end{schFamily}

\addcontentsline{toca}{subsection}{\schFamRef{schroderFamily:zebras}: Coloured parallelogram polyominoes (zebras)}
\begin{schFamily}
[Coloured parallelogram polyominoes (zebras) \cite{Pergola98schrodertriangles}] \label{schroderFamily:zebras} 

\def\scl{0.34}

$S_n$ is the number of parallelogram polyominoes of perimeter $2n + 2$ with each column coloured black or white. A parallelogram polyomino is a translation invariant array of unit squares bounded by two lattice paths that use the steps $(0,1)$ and $(1,0)$ and that intersect only at their first and last vertices.

\textit{Generator:} The empty parallelogram polyomino, which we choose to represent as follows:
\begin{equation*}
    \epsilon = \, \begin{tikzpicture}[baseline={([yshift=-.7ex]current bounding box.center)}, scale = \scl]
        \draw (0,0) -- (0,1);
    \end{tikzpicture}
\end{equation*}
We take this to have perimeter 2.

\textit{Unary map:} Add a single white square to the top right of the coloured parallelogram polyomino:
\begin{equation*}
    f \left( \
    \begin{tikzpicture}[baseline={([yshift=-.5ex]current bounding box.center)}, scale = \scl]
        \fill[\colOne, opacity = \opac] (0,1) -- (0,0) -- (2,0) -- (6,3) -- (6,4) -- (4,4) -- (0,1);
        \draw (0,1) -- (0,0) -- (2,0) (6,3) -- (6,4) -- (4,4);
        \draw (2,0) -- (6,3) (4,4) -- (0,1);
        \dummyNodes[0.5][-0.2][0.5][4.2]
        \node at (3,2) {\footnotesize $p$};
    \end{tikzpicture}
    \ \right)
    =
    \begin{tikzpicture}[baseline={([yshift=-.5ex]current bounding box.center)}, scale = \scl]
        \fill[\colOne, opacity = \opac] (0,1) -- (0,0) -- (2,0) -- (6,3) -- (6,4) -- (4,4) -- (0,1);
        \draw (0,1) -- (0,0) -- (2,0) (6,3) -- (6,4) -- (4,4);
        \draw (2,0) -- (6,3) (4,4) -- (0,1);
        \draw (6,3) -- (7,3) -- (7,4) -- (6,4);
        \node at (3,2) {\footnotesize $p$};
    \end{tikzpicture}
\end{equation*}

\textit{Binary map:}
\begin{equation*}
    \begin{tikzpicture}[baseline={([yshift=-.5ex]current bounding box.center)}, scale = \scl]
        \fill[\colOne, opacity = \opac] (0,1) -- (0,0) -- (2,0) -- (6,3) -- (6,4) -- (4,4) -- (0,1);
        \draw (0,1) -- (0,0) -- (2,0) (6,3) -- (6,4) -- (4,4);
        \draw (2,0) -- (6,3) (4,4) -- (0,1);
        \node at (3,2) (p1) {\footnotesize $p_1$};
    \end{tikzpicture}
    \ \star \ 
    \begin{tikzpicture}[baseline={([yshift=-.5ex]current bounding box.center)}, scale = \scl]
        \fill[\colTwo, opacity = \opac] (0,1) -- (0,0) -- (2,0) -- (6,3) -- (6,4) -- (4,4) -- (0,1);
        \draw (0,1) -- (0,0) -- (2,0) (6,3) -- (6,4) -- (4,4);
        \draw (2,0) -- (6,3) (4,4) -- (0,1);
        \node at (3,2) (p2) {\footnotesize $p_2$};
    \end{tikzpicture}
    =
    \begin{cases}
        \begin{tikzpicture}[baseline={([yshift=-.5ex]current bounding box.center)}, scale = \scl]
            \fill[\colOne, opacity = \opac] (0,1) -- (0,0) -- (2,0) -- (6,3) -- (6,4) -- (4,4) -- (0,1);
            \draw[fill = gray] (6,3) -- (7,3) -- (7,4) -- (6,4);
            \draw (0,1) -- (0,0) -- (2,0) (6,3) -- (6,4) -- (4,4);
            \draw (2,0) -- (6,3) (4,4) -- (0,1);
            \node at (3,2) (p1) {\footnotesize $p_1$};
        \end{tikzpicture} & \text{if $p_2 = \epsilon$,} \\
        \begin{tikzpicture}[baseline={([yshift=-.5ex]current bounding box.center)}, scale = \scl]
            \fill[\colOne, opacity = \opac] (0,1) -- (0,0) -- (2,0) -- (6,3) -- (6,4) -- (4,4) -- (0,1);
            \fill[red, opacity = 0.1] (6,3) -- (8,3) -- (12,6) -- (12,7) -- (8,4) -- (6,4) -- (6,3);
            \draw (0,1) -- (0,0) -- (2,0) (6,3) -- (6,4) -- (4,4);
            \draw (2,0) -- (6,3) (4,4) -- (0,1);
            \node at (3,2) (p1) {\footnotesize $p_1$};
            \fill[\colTwo, opacity = \opac] (6,5) -- (6,4) -- (8,4) -- (12,7) -- (12,8) -- (10,8) -- (6,5);
            \draw (6,5) -- (6,4) -- (8,4) (12,7) -- (12,8) -- (10,8);
            \draw (8,4) -- (12,7) (10,8) -- (6,5);
            \node at (9,6) (p2) {\footnotesize $p_2$};
            \draw[red] (6,3) -- (8,3) -- (12,6) -- (12,7);
        \end{tikzpicture} & \text{otherwise.} \\
    \end{cases}
\end{equation*}
The red section means that we join $p_1$ and $p_2$ as shown and then add one cell at the bottom of each column of $p_2$, with this cell being given the colour of the column it is being added to. 

\textit{Norm:} If $p$ is a coloured parallelogram polyomino of perimeter $2n + 2$, then $\norm{p} = n + 1$.

\textit{(1,2)-magma:} The (1,2)-magma begins (sorting by norm) as follows:
\vspace{-.5em}
\begin{center}
    \def\scl{0.4}
    \begin{tabular}{| l l |}
    \hline
    \textit{Norm 1:} & 
    $\emptyset = \epsilon$ \\
    \hline
    \textit{Norm 2:} & 
    $\begin{tikzpicture}[baseline={([yshift=-.5ex]current bounding box.center)}, scale = \scl]
        \draw (0,0) rectangle (1,1);
        \dummyNodes[0.5][0][0.5][1]
    \end{tikzpicture} 
    = f(\epsilon)$, 
    \quad
    $\begin{tikzpicture}[baseline={([yshift=-.5ex]current bounding box.center)}, scale = \scl]
        \draw[fill = gray] (0,0) rectangle (1,1);
        \dummyNodes[0.5][0][0.5][1]
    \end{tikzpicture}
    = \epsilon \star \epsilon$ \\
    \hline
    \textit{Norm 3:} & 
    $\begin{tikzpicture}[baseline={([yshift=-.5ex]current bounding box.center)}, scale = \scl]
        \draw (0,0) rectangle (2,1);
        \draw (1,0) -- (1,1);
        \dummyNodes[0.5][0][0.5][1]
    \end{tikzpicture} 
    = f(f(\epsilon))$, 
    \quad
    $\begin{tikzpicture}[baseline={([yshift=-.5ex]current bounding box.center)}, scale = \scl]
        \fill[gray] (0,0) rectangle (1,1);
        \draw (0,0) rectangle (2,1);
        \draw (1,0) -- (1,1);
        \dummyNodes[0.5][0][0.5][1]
    \end{tikzpicture} 
    = f(\epsilon \star \epsilon)$, 
    \quad
    $\begin{tikzpicture}[baseline={([yshift=-.5ex]current bounding box.center)}, scale = \scl]
        \draw (0,0) rectangle (1,2);
        \draw (0,1) -- (1,1);
        \dummyNodes[0.5][-0.2][0.5][2]
    \end{tikzpicture} 
    = \epsilon \star f(\epsilon)$, \\
    &
    $\begin{tikzpicture}[baseline={([yshift=-.5ex]current bounding box.center)}, scale = \scl]
        \fill[gray] (0,0) rectangle (1,2);
        \draw (0,0) rectangle (1,2);
        \draw (0,1) -- (1,1);
        \dummyNodes[0.5][0][0.5][2]
    \end{tikzpicture}
    = \epsilon \star (\epsilon \star \epsilon)$, 
    \quad
    $\begin{tikzpicture}[baseline={([yshift=-.5ex]current bounding box.center)}, scale = \scl]
        \fill[gray] (1,0) rectangle (2,1);
        \draw (0,0) rectangle (2,1);
        \draw (1,0) -- (1,1);
        \dummyNodes[0.5][0][0.5][1]
    \end{tikzpicture}  
    = f(\epsilon) \star \epsilon$, 
    \quad
    $\begin{tikzpicture}[baseline={([yshift=-.5ex]current bounding box.center)}, scale = \scl]
        \fill[gray] (0,0) rectangle (2,1);
        \draw (0,0) rectangle (2,1);
        \draw (1,0) -- (1,1);
        \dummyNodes[0.5][-0.2][0.5][1.2]
    \end{tikzpicture} 
    = (\epsilon \star \epsilon) \star \epsilon$ \\
    \hline
    \end{tabular}
\end{center}
\end{schFamily}
 
    \pagebreak
    
\subsection{Fuss-Catalan Families} \label{appendix:FC}

\addcontentsline{toca}{section}{Fuss-Catalan families}

We present a number of Fuss-Catalan normed (3)-magmas for combinatorial families enumerated by the order 3 Fuss-Catalan numbers. We take the convention that each generator is called $\epsilon$ and that each ternary map is called $t$. Despite the same names being used for each \mbox{(3)-magma} presented, it is clear that these are all different maps and the generators are all different. 

\addcontentsline{toca}{subsection}{\fcFamRef{fcFamily:ternaryTrees}: Ternary trees}
\begin{fcFamily}
[Ternary trees \cite{AVAL20084660}] \label{fcFamily:ternaryTrees} 

\def \scl {0.15}

$T_n$ is the number of complete ternary trees with $2n + 1$ leaves. A complete ternary tree is a rooted tree such that every non-leaf node has three children.

\textit{Generator:} The generator is a single vertex:
\begin{equation*}
    \epsilon =
    \begin{tikzpicture}[baseline={([yshift=-.5ex]current bounding box.center)}, scale = \scl]
        \fill (0,0) \smlnd;
    \end{tikzpicture} 
\end{equation*}

\textit{Ternary map:} 
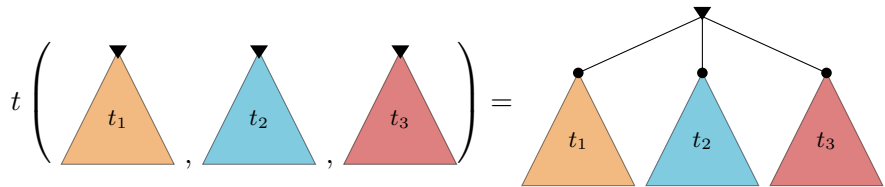
\begin{equation*}
    t \left( 
    \raisebox{-20pt}{
    \begin{tikzpicture}[baseline={([yshift=-5ex]current bounding box.center)}, scale = \scl]
        \draw[fill = \colOne, opacity = \opac] (0,0) -- (-5,-10) -- (5,-10) -- (0,0);
        \node at (0,-2.5) {\smlroot};
        \node at (0,-6) {\footnotesize $t_1$};
    \end{tikzpicture}
    , 
    \begin{tikzpicture}[baseline={([yshift=-5ex]current bounding box.center)}, scale = \scl]
        \draw[fill = \colTwo, opacity = \opac] (0,0) -- (-5,-10) -- (5,-10) -- (0,0);
        \node at (0,-2.5) {\smlroot};
        \node at (0,-6) {\footnotesize $t_2$};
    \end{tikzpicture}
    ,
    \begin{tikzpicture}[baseline={([yshift=-5ex]current bounding box.center)}, scale = \scl]
        \draw[fill = \colThree, opacity = \opac] (0,0) -- (-5,-10) -- (5,-10) -- (0,0);
        \node at (0,-2.5) {\smlroot};
        \node at (0,-6) {\footnotesize $t_3$};
    \end{tikzpicture}}
    \right)
    =
    \begin{tikzpicture}[baseline={([yshift=-1ex]current bounding box.center)}, scale = \scl]
        \draw[fill = \colOne, opacity = \opac] (-11,-5) -- (-16,-15) -- (-6,-15) -- (-11,-5);
        \draw[fill = \colTwo, opacity = \opac] (0,-5) -- (-5,-15) -- (5,-15) -- (0,-5);
        \draw[fill = \colThree, opacity = \opac] (11,-5) -- (16,-15) -- (6,-15) -- (11,-5);
        \draw (0,0) -- (0,-5) (0,0) -- (-11,-5) (0,0) -- (11,-5);
        \node at (0,0) {\smlroot};
        \fill (-11,-5) \smlnd (0,-5) \smlnd (11,-5) \smlnd;
        \node at (-11,-11) {\footnotesize $t_1$};
        \node at (0,-11) {\footnotesize $t_2$};
        \node at (11,-11) {\footnotesize $t_3$};
        \dummyNodes[0][0.2][0][-15.2]
    \end{tikzpicture}
\end{equation*}

\textit{Norm:} If $t$ is a ternary tree with $n$ leaves, then $\norm{t} = n$. Note that we consider the only node in $\epsilon$ to be a leaf node, and hence $\norm{\epsilon} = 1$.

\textit{(3)-magma:} The (3)-magma begins (sorting by norm) as follows:
\vspace{-.5em}
\begin{center}
    \def\scl{0.6}
    \begin{tabular}{| l l |}
    \hline
    \textit{Norm 1:} &
    $\begin{tikzpicture}[baseline={([yshift=-.5ex]current bounding box.center)}, scale = \scl]
        \fill (0,0) \smlnd;
        \dummyNodes[0][0.2][0][-0.2]
    \end{tikzpicture} 
    = \epsilon$ \\
    \hline
    \textit{Norm 3:} & 
    $\begin{tikzpicture}[baseline={([yshift=-.5ex]current bounding box.center)}, scale = \scl]
        \node at (0,0) {\smlroot};
        \draw (0,0) -- (-1,-1) (0,0) -- (0,-1) (0,0) -- (1,-1);
        \fill (-1,-1) \smlnd  (0,-1) \smlnd (1,-1) \smlnd;
        \dummyNodes[0][0.2][0][-1.2]
    \end{tikzpicture} 
    = t(\epsilon, \epsilon, \epsilon)$ \\
    \hline
    \textit{Norm 5:} &
    $\begin{tikzpicture}[baseline={([yshift=-.5ex]current bounding box.center)}, scale = \scl]
        \node at (0,0) {\smlroot};
        \draw (0,0) -- (-1,-1) (0,0) -- (0,-1) (0,0) -- (1,-1);
        \fill (-1,-1) \smlnd  (0,-1) \smlnd (1,-1) \smlnd;
        \draw (-1,-1) -- (-2,-2) (-1,-1) -- (-1,-2) (-1,-1) -- (0,-2);
        \fill (-2,-2) \smlnd  (-1,-2) \smlnd (0,-2) \smlnd;
        \dummyNodes[0][0.2][0][-2.2]
    \end{tikzpicture} 
    = t(t(\epsilon, \epsilon, \epsilon), \epsilon, \epsilon)$,
    \quad
    $\begin{tikzpicture}[baseline={([yshift=-.5ex]current bounding box.center)}, scale = \scl]
        \node at (0,0) {\smlroot};
        \draw (0,0) -- (-1,-1) (0,0) -- (0,-1) (0,0) -- (1,-1);
        \fill (-1,-1) \smlnd  (0,-1) \smlnd (1,-1) \smlnd;
        \draw (0,-1) -- (-1,-2) (0,-1) -- (0,-2) (0,-1) -- (1,-2);
        \fill (-1,-2) \smlnd  (0,-2) \smlnd (1,-2) \smlnd;
        \dummyNodes[0][0.2][0][-2.2]
    \end{tikzpicture} 
    = t(\epsilon, t(\epsilon, \epsilon, \epsilon), \epsilon)$, \\
    &
    $\begin{tikzpicture}[baseline={([yshift=-.5ex]current bounding box.center)}, scale = \scl]
        \node at (0,0) {\smlroot};
        \draw (0,0) -- (-1,-1) (0,0) -- (0,-1) (0,0) -- (1,-1);
        \fill (-1,-1) \smlnd  (0,-1) \smlnd (1,-1) \smlnd;
        \draw (1,-1) -- (0,-2) (1,-1) -- (1,-2) (1,-1) -- (2,-2);
        \fill (0,-2) \smlnd  (1,-2) \smlnd (2,-2) \smlnd;
        \dummyNodes[0][0.2][0][-2.2]
    \end{tikzpicture}  
    = t(\epsilon, \epsilon, t(\epsilon, \epsilon, \epsilon))$ \\
    \hline
    \end{tabular}
\end{center}
\end{fcFamily}

\pagebreak

\addcontentsline{toca}{subsection}{\fcFamRef{fcFamily:evenTrees}: Even trees}
\begin{fcFamily}
[Even trees \cite{FussCatalanOEIS}] \label{fcFamily:evenTrees} 

\def\scl{0.15}

$T_n$ is the number of rooted plane trees with $2n$ edges, where every vertex has even out-degree.

\textit{Generator:} The generator is a single vertex:
\begin{equation*}
    \epsilon =
    \begin{tikzpicture}[baseline={([yshift=-.5ex]current bounding box.center)}, scale = \scl]
        \fill (0,0) \smlnd;
    \end{tikzpicture} 
\end{equation*}

\textit{Ternary map:} 
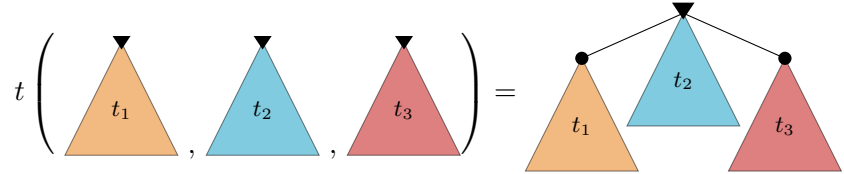
\begin{equation*}
    t \left( 
    \raisebox{-20pt}{
    \begin{tikzpicture}[baseline={([yshift=-5ex]current bounding box.center)}, scale = \scl]
        \draw[fill = \colOne, opacity = \opac] (0,0) -- (-5,-10) -- (5,-10) -- (0,0);
        \node at (0,-2.5) {\smlroot};
        \node at (0,-6) {\footnotesize $t_1$};
    \end{tikzpicture}
    , 
    \begin{tikzpicture}[baseline={([yshift=-5ex]current bounding box.center)}, scale = \scl]
        \draw[fill = \colTwo, opacity = \opac] (0,0) -- (-5,-10) -- (5,-10) -- (0,0);
        \node at (0,-2.5) {\smlroot};
        \node at (0,-6) {\footnotesize $t_2$};
    \end{tikzpicture}
    ,
    \begin{tikzpicture}[baseline={([yshift=-5ex]current bounding box.center)}, scale = \scl]
        \draw[fill = \colThree, opacity = \opac] (0,0) -- (-5,-10) -- (5,-10) -- (0,0);
        \node at (0,-2.5) {\smlroot};
        \node at (0,-6) {\footnotesize $t_3$};
    \end{tikzpicture}}
    \right)
    =
    \begin{tikzpicture}[baseline={([yshift=-1ex]current bounding box.center)}, scale = \scl]
        \draw[fill = \colOne, opacity = \opac] (-9,-4) -- (-14,-14) -- (-4,-14) -- (-9,-4);
        \draw[fill = \colTwo, opacity = \opac] (0,0) -- (-5,-10) -- (5,-10) -- (0,0);
        \draw[fill = \colThree, opacity = \opac] (9,-4) -- (14,-14) -- (4,-14) -- (9,-4);
        \node at (0,0) {\root};
        \draw (0,0) -- (-9,-4) (0,0) -- (9,-4);
        \fill (-9,-4) \nd (9,-4) \nd;
        \node at (-9,-10) {\footnotesize $t_1$};
        \node at (0,-6) {\footnotesize $t_2$};
        \node at (9,-10) {\footnotesize $t_3$};
        \dummyNodes[0][0.2][0][-14.2]
    \end{tikzpicture}
\end{equation*}

\textit{Norm:} If $t$ is a tree with $n$ edges, then $\norm{t} = n + 1$.

\textit{(3)-magma:} The (3)-magma begins (sorting by norm) as follows:
\vspace{-.5em}
\begin{center}
    \def\scl{0.6}
    \begin{tabular}{| l l |}
    \hline
    \textit{Norm 1:} &
    $\begin{tikzpicture}[baseline={([yshift=-.5ex]current bounding box.center)}, scale = \scl]
        \fill (0,0) \smlnd;
    \end{tikzpicture} 
    = \epsilon$ \\
    \hline
    \textit{Norm 3:} &
    $\begin{tikzpicture}[baseline={([yshift=-.5ex]current bounding box.center)}, scale = \scl]
        \node at (0,0) (root) {\smlroot};
        \draw (0,0) -- (-1,-1) (0,0) -- (1,-1);
        \fill (-1,-1) \smlnd (1,-1) \smlnd;
        \dummyNodes[0][0.2][0][-1.2]
    \end{tikzpicture}
    = t(\epsilon, \epsilon, \epsilon)$ \\
    \hline
    \textit{Norm 5:} &
    $\begin{tikzpicture}[baseline={([yshift=-.5ex]current bounding box.center)}, scale = \scl]
        \node at (0,0) (root) {\smlroot};
        \draw (0,0) -- (-1,-1) (0,0) -- (1,-1) (-1,-1) -- (-2,-2) (-1,-1) -- (0,-2);
        \fill (-1,-1) \smlnd (1,-1) \smlnd (-2,-2) \smlnd (0,-2) \smlnd;
        \dummyNodes[0][0.2][0][-2.2]
    \end{tikzpicture} 
    = t(t(\epsilon, \epsilon, \epsilon), \epsilon, \epsilon)$,
    \quad
    $\begin{tikzpicture}[baseline={([yshift=-.5ex]current bounding box.center)}, scale = \scl]
        \node at (0,0) (root) {\smlroot};
        \draw (0,0) -- (-1.5,-1) (0,0) -- (-0.5,-1) (0,0) -- (0.5,-1) (0,0) -- (1.5,-1);
        \fill (-1.5,-1) \smlnd (-0.5,-1) \smlnd (0.5,-1) \smlnd (1.5,-1) \smlnd;
        \dummyNodes[0][0.2][0][-1.2]
    \end{tikzpicture} 
    = t(\epsilon, t(\epsilon, \epsilon, \epsilon), \epsilon)$, \\
    &
    $\begin{tikzpicture}[baseline={([yshift=-.5ex]current bounding box.center)}, scale = \scl]
        \node at (0,0) (root) {\smlroot};
        \draw (0,0) -- (-1,-1) (0,0) -- (1,-1) (1,-1) -- (2,-2) (1,-1) -- (0,-2);
        \fill (-1,-1) \smlnd (1,-1) \smlnd (2,-2) \smlnd (0,-2) \smlnd;
        \dummyNodes[0][0.2][0][-2.2]
    \end{tikzpicture}  
    = t(\epsilon, \epsilon, t(\epsilon, \epsilon, \epsilon))$ \\
    \hline
    \end{tabular}
\end{center}
\end{fcFamily}

\addcontentsline{toca}{subsection}{\fcFamRef{fcFamily:partitions}: Non-crossing partitions with blocks of even size}
\begin{fcFamily}
[Non-crossing partitions with blocks of even size \cite{FussCatalanOEIS}] \label{fcFamily:partitions} 

\def\scl{0.05}

$T_n$ is the number of non-crossing partitions of $\{ 1, 2, \hdots, 2n \}$ with all blocks of even size. We represent such a partition schematically by $n$ marked points on a circle, with a chord connecting points which are in the same block.

\textit{Generator:} The generator is the trivially empty partition of the empty set: 
\begin{equation*}
    \epsilon = 
    \begin{tikzpicture}[baseline={([yshift=-.5ex]current bounding box.center)}, scale = \scl]
        \draw (0,0) circle [radius = 5];
        \dummyNodes[0][-6][0][6]
    \end{tikzpicture}
\end{equation*}

\def\scl{0.1}
\textit{Ternary map:} 
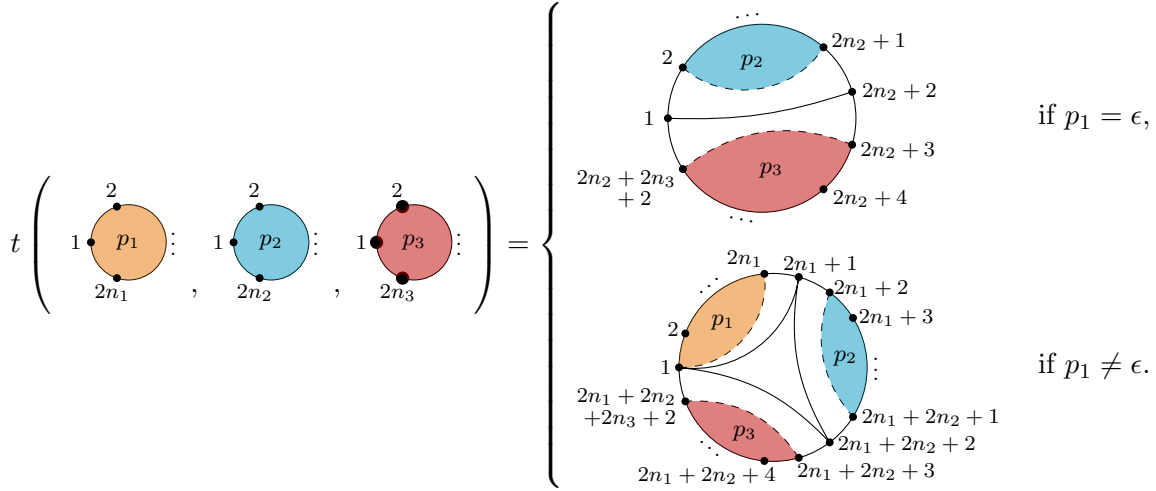
\begin{equation*}
    \def\numPoints{5}
    t \left(
    \raisebox{-15pt}{
    \begin{tikzpicture}[baseline={([yshift=-4ex]current bounding box.center)}, scale = \scl]
        \blankPartitionSetup
        \node at (180:7) (11) {\scriptsize $1$};
        \node at (108:7.3) (22) {\scriptsize $2$};
        \node at (-108:7) (n) {\scriptsize $2n_1$};
        \node[rotate = 90] at (0:6) {\scriptsize $\cdots$};
        \fill[\colOne, opacity = \opac] (0,0) circle [radius = 5];
        \fill (1) \halfnd (2) \halfnd (5) \halfnd;
        \node at (0,0) {\footnotesize $p_1$};
    \end{tikzpicture}\spacecomma
    \begin{tikzpicture}[baseline={([yshift=-4ex]current bounding box.center)}, scale = \scl]
        \blankPartitionSetup
        \node at (180:7) (11) {\scriptsize $1$};
        \node at (108:7.3) (22) {\scriptsize $2$};
        \node at (-108:7) (n) {\scriptsize $2n_2$};
        \node[rotate = 90] at (0:6) {\scriptsize $\cdots$};
        \fill[\colTwo, opacity = \opac] (0,0) circle [radius = 5];
        \fill (1) \halfnd (2) \halfnd (5) \halfnd;
        \node at (0,0) {\footnotesize $p_2$};
    \end{tikzpicture}\spacecomma
    \begin{tikzpicture}[baseline={([yshift=-4ex]current bounding box.center)}, scale = \scl]
        \blankPartitionSetup
        \fill (1) \nd (2) \nd (5) \nd;
        \node at (180:7) (11) {\scriptsize $1$};
        \node at (108:7.3) (22) {\scriptsize $2$};
        \node at (-108:7) (n) {\scriptsize $2n_3$};
        \node[rotate = 90] at (0:6) {\scriptsize $\cdots$};
        \fill[\colThree, opacity = \opac] (0,0) circle [radius = 5];
        \fill (1) \halfnd (2) \halfnd (5) \halfnd;
        \node at (0,0) {\footnotesize $p_3$};
    \end{tikzpicture}}
    \right)
    = \def\numPoints {11} \def\scl{0.25} 
    \begin{cases}
        \begin{tikzpicture}[baseline={([yshift=-.5ex]current bounding box.center)}, scale = \scl]
            \blankPartitionSetup
            \node at (180:6) {\scriptsize $1$};
            \node at (147:6) {\scriptsize $2$};
            \node[rotate = 10] at (97:5.5) {\scriptsize $\cdots$};
            \node at ([shift = {(2,0)}]48:5.5) {\scriptsize $2n_2 + 1$};
            \node at ([shift = {(2,0)}]15:5.5) {\scriptsize $2n_2 + 2$};
            \node at ([shift = {(2,0)}]-18:5.5) {\scriptsize $2n_2 + 3$};
            \node at ([shift = {(2.3,0)}]-51:5.5) {\scriptsize $2n_2 + 4$};
            \node[rotate = -10] at (-100.5:5.5) {\scriptsize $\cdots$};
            \node at ([shift = {(-2.7,-0.8)}]213:5.5) {\scriptsize $\begin{aligned}2n_2 & + 2n_3 \\[-.5em] & + 2\end{aligned}$};
            \draw[dashed] (2) to [bend right = 50] (5) (7) to [bend right = 30] (11);
            \begin{scope}
                \clip (0,0) circle [radius = 5];
                \fill[\colTwo, opacity = \opac] (2) to [bend right = 50] (5) -- (7,7) -- (-7,7) -- (2);
            \end{scope}
            \begin{scope}
                \clip (0,0) circle [radius = 5];
                \fill[\colThree, opacity = \opac] (7) to [bend right = 30] (11) -- (-7,-7) -- (7,-7) -- (7);
            \end{scope}
            \fill (1) \halfnd (2) \halfnd (5) \halfnd (6) \halfnd (7) \halfnd (8) \halfnd (11) \halfnd;
            \draw (1) to [bend right = 10] (6);
            \node at (-0.5,3) {\footnotesize $p_2$};
            \node at (0.5,-2.75) {\footnotesize $p_3$};
        \end{tikzpicture}
        & \text{if $p_1 = \epsilon$,} \\
        \def\numPoints {17}
        \begin{tikzpicture}[baseline={([yshift=-.5ex]current bounding box.center)}, scale = \scl]
            \blankPartitionSetup
            \node at (180:5.7) {\scriptsize $1$};
            \node at (159:5.7) {\scriptsize $2$};
            \node[rotate = 40] at (127:5.5) {\scriptsize $\cdots$};
            \node at ([shift={(-1.5,0.3)}]90:5.5) {\scriptsize $2n_1$};
            \node at ([shift={(1,0.2)}]74:5.5) {\scriptsize $2n_1 + 1$};
            \node at ([shift={(2.2,0.2)}]53:5) {\scriptsize $2n_1 + 2$};
            \node at ([shift={(2.3,0)}]32:5) {\scriptsize $2n_1 + 3$};
            \node[rotate = 90] at (0:5.5) {\scriptsize $\cdots$};
            \node at ([shift={(3.3,0.5)}]-32:6) {\scriptsize $2n_1 + 2n_2 + 1$};
            \node at ([shift={(3.5,0.7)}]-53:6) {\scriptsize $2n_1 + 2n_2 + 2$};
            \node at ([shift={(3.3,0.2)}]-74:6) {\scriptsize $2n_1 + 2n_2 + 3$};
            \node at ([shift={(-3,0.2)}]-95:6) {\scriptsize $2n_1 + 2n_2 + 4$};
            \node[rotate = -40] at (-127:5.5) {\scriptsize $\cdots$};
            \node at ([shift={(-2.3,0)}]-159:6) {\scriptsize $\begin{aligned}2n_1 + 2n_2 \\[-.5em] + 2n_3 + 2\end{aligned}$};
            \draw[dashed] (1) to [bend right = 50] (5) (7) to [bend right = 30] (11) (13) to [bend right = 30] (17);
            \begin{scope}
                \clip (0,0) circle [radius = 5];
                \fill[\colOne, opacity = \opac] (1) to [bend right = 50] (5) -- (-7,7) -- (1);
            \end{scope}
            \begin{scope}
                \clip (0,0) circle [radius = 5];
                \fill[\colTwo, opacity = \opac] (7) to [bend right = 30] (11) -- (7,-7) -- (7,7) -- (7);
            \end{scope}
            \begin{scope}
                \clip (0,0) circle [radius = 5];
                \fill[\colThree, opacity = \opac] (13) to [bend right = 30] (17) -- (-7,-7) -- (7,-7) -- (13);
            \end{scope}
            \fill (1) \halfnd (2) \halfnd (5) \halfnd (6) \halfnd (7) \halfnd (8) \halfnd (11) \halfnd (12) \halfnd (13) \halfnd (14) \halfnd (17) \halfnd;
            \draw (1) to [bend right = 45] (6) to [bend right = 20] (12) to [bend right = 20] (1);
            \node at (-2.65,2.4) {\footnotesize $p_1$};
            \node at (3.85,0.5) {\footnotesize $p_2$};
            \node at (-1.5,-3.5) {\footnotesize $p_3$};
        \end{tikzpicture}
        & \text{if $p_1 \neq \epsilon$.} \\
    \end{cases}
\end{equation*}
In the second case, the chord shown connecting node 1 with node $2n_1 + 1$ means that we add node $2n_1 + 1$ to the block containing node 1. Note that this will not always result in the chord shown.

\textit{Norm:} If $p$ is a partition of $\{ 1, 2, \hdots, 2n \}$, then $\norm{p} = 2n + 1$.

\textit{(3)-magma:} The (3)-magma begins (sorting by norm) as follows:
\vspace{-.5em}
\begin{center}
    \begin{tabular}{| l l |}
    \hline
    \textit{Norm 1:} &  \def\scl{0.05}
    \ $\begin{tikzpicture}[baseline={([yshift=-.5ex]current bounding box.center)}, scale = \scl]
        \draw (0,0) circle [radius = 5];
        \dummyNodes[0][-6][0][6]
    \end{tikzpicture} = \epsilon$ \\
    \hline
    \textit{Norm 3:} & \def\numPoints{2} \def\scl{0.07}
    $\begin{tikzpicture}[baseline={([yshift=-.5ex]current bounding box.center)}, scale = \scl]
        \clip (-8,-7) rectangle (8,7);
        \partitionSetup
        \draw (1) to (2);
    \end{tikzpicture} = t(\epsilon, \epsilon, \epsilon)$ \\
    \hline
    \textit{Norm 5:} & \def\numPoints {4} \def\scl{0.09}
    $\begin{tikzpicture}[baseline={([yshift=-.5ex]current bounding box.center)}, scale = \scl]
        \clip (-8,-9) rectangle (8,9);
        \partitionSetup
        \partitionDrawArc[1][2]
        \partitionDrawArc[2][3]
        \partitionDrawArc[3][4]
        \partitionDrawArc[4][1]
    \end{tikzpicture} = t(t(\epsilon, \epsilon, \epsilon), \epsilon, \epsilon)$, 
    \quad
    $\begin{tikzpicture}[baseline={([yshift=-.5ex]current bounding box.center)}, scale = \scl]
        \clip (-8,-9) rectangle (8,9);
        \partitionSetup
        \partitionDrawArc[2][3]
        \partitionDrawArc[4][1]
    \end{tikzpicture} = t(\epsilon, t(\epsilon, \epsilon, \epsilon), \epsilon)$, \\
    & \def\numPoints{4} \def\scl{0.09}
    $\begin{tikzpicture}[baseline={([yshift=-.5ex]current bounding box.center)}, scale = \scl]
        \clip (-8,-9) rectangle (8,9);
        \partitionSetup
        \partitionDrawArc[1][2]
        \partitionDrawArc[3][4]
    \end{tikzpicture} = t(\epsilon, \epsilon, t(\epsilon, \epsilon, \epsilon))$
    \\
    \hline
    \end{tabular}
\end{center}
\end{fcFamily}

\addcontentsline{toca}{subsection}{\fcFamRef{fcFamily:quadrillages}: Quadrillages}
\begin{fcFamily}
[Quadrillages \cite{Baryshnikov2001}] \label{fcFamily:quadrillages} 

\def\scl{0.04}
\def\numPoints{2}

$T_n$ is the number of quadrillages of a $(2n + 2)$-gon. A quadrillage is a dissection of a polygon such that all sub-objects have four sides. We draw all polygons with a marked side which is denoted by a thick black line. This fixes the orientation of the polygon and distinguishes polygons which only differ by a rotation.

\textit{Generator:} The generator is taken to be a single edge, i.e.\ a 2-gon:
\begin{equation*}
    \epsilon = \begin{tikzpicture}[baseline={([yshift=-.7ex]current bounding box.center)}, scale = \scl]
    \polygonSetup
\end{tikzpicture}
\end{equation*}

\def\scl{0.12}
\textit{Ternary map:} 
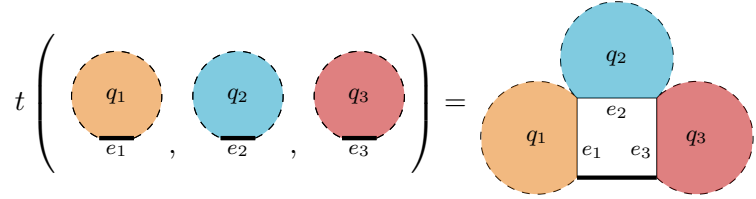
\begin{equation*} \def\numPoints{8}
    t \left(
    \raisebox{-16pt}{
    \begin{tikzpicture}[baseline={([yshift=-4.5ex]current bounding box.base)}, scale = \scl]
        \blankPolygonSetup
        \fill[\colOne, opacity = \opac] (8) -- (1) (-67.5:5) arc (-67.5:247.5:5);
        \node at (0,0) {\footnotesize $q_1$};
        \node at (-90:6) {\scriptsize $e_1$};
        \blankPolygonSetup
    \end{tikzpicture},
    \begin{tikzpicture}[baseline={([yshift=-4.5ex]current bounding box.base)}, scale = \scl]
        \blankPolygonSetup
        \fill[\colTwo, opacity = \opac] (8) -- (1) (-67.5:5) arc (-67.5:247.5:5);
        \node at (0,0) {\footnotesize $q_2$};
        \node at (-90:6) {\scriptsize $e_2$};
        \blankPolygonSetup
    \end{tikzpicture},
    \begin{tikzpicture}[baseline={([yshift=-4.5ex]current bounding box.base)}, scale = \scl]
        \blankPolygonSetup
        \fill[\colThree, opacity = \opac] (8) -- (1) (-67.5:5) arc (-67.5:247.5:5);
        \node at (0,0) {\footnotesize $q_3$};
        \node at (-90:6) {\scriptsize $e_3$};
        \blankPolygonSetup
    \end{tikzpicture}}
    \right)
    = \def\numPoints{4} \def\scl{0.15}
    \begin{tikzpicture}[baseline={([yshift=-1ex]current bounding box.center)}, scale = \scl]
        \polygonSetup
        \draw[dashed] ([shift={(-7.07,0)}]45:5) arc (45:315:5);
         \begin{scope}
            \clip ([shift={(-7.07,0)}]45:5) arc (45:315:5);;
            \fill[\colOne, opacity = \opac] (3) -- (4) -- (0,-7) -- (-12,-7) -- (-12,7) -- (0,7) -- (3);
        \end{scope}
        \node at (-7.07,0) {\footnotesize $q_1$};
        
        \draw[dashed] ([shift={(0,7.07)}]-45:5) arc (-45:225:5);
        \begin{scope}
            \clip ([shift={(0,7.07)}]-45:5) arc (-45:225:5);
            \fill[\colTwo, opacity = \opac] (2) -- (3) -- (-7,0) -- (-7,12) -- (7,12) -- (7,0) -- (2);
        \end{scope}
        \node at (0,7.07) {\footnotesize $q_2$};
        
        \draw[dashed] ([shift={(7.07,0)}]-135:5) arc (-135:135:5);
        \begin{scope}
            \clip ([shift={(7.07,0)}]-135:5) arc (-135:135:5);
            \fill[\colThree, opacity = \opac] (2) -- (1) -- (0,-7) -- (12,-7) -- (12,7) -- (0,7) -- (2);
        \end{scope}
        \node at (7.07,0) {\footnotesize $q_3$};
        
        \node at (210:2.5) {\scriptsize $e_1$};
        \node at (90:2.3) {\scriptsize $e_2$};
        \node at (-30:2.5) {\scriptsize $e_3$};
    \end{tikzpicture}
\end{equation*}

\textit{Norm:} If $q$ is a quadrillage of a $(2n + 2)$-gon, then $\norm{q} = 2n + 1$.

\textit{(3)-magma:} The (3)-magma begins (sorting by norm) as follows:
\vspace{-.5em}
\begin{center}
    \begin{tabular}{| l l |}
    \hline
    \textit{Norm 1:} & \def\numPoints{2} \def\scl{0.04}
    $\begin{tikzpicture}[baseline={([yshift=-.7ex]current bounding box.center)}, scale = \scl]
        \polygonSetup
    \end{tikzpicture} \empty = \epsilon$ \\
    \hline
    \textit{Norm 3:} & \def\numPoints{4} \def\scl{0.08}
    $\begin{tikzpicture}[baseline={([yshift=-.5ex]current bounding box.center)}, scale = \scl]
        \polygonSetup
    \end{tikzpicture} 
    = t(\epsilon, \epsilon, \epsilon) $
    \\
    \hline
    \textit{Norm 5:} & \def\numPoints{6} \def\scl{0.1}
    $\begin{tikzpicture}[baseline={([yshift=-.5ex]current bounding box.center)}, scale = \scl]
        \polygonSetup
        \dissectionArc[3][6]
    \end{tikzpicture}
    = t(t(\epsilon, \epsilon, \epsilon), \epsilon, \epsilon)$, 
    \quad
    $\begin{tikzpicture}[baseline={([yshift=-.5ex]current bounding box.center)}, scale = \scl]
        \polygonSetup
        \dissectionArc[2][5]
    \end{tikzpicture}
    = t(\epsilon, t(\epsilon, \epsilon, \epsilon), \epsilon)$, \\
    & \def\numPoints{6} \def\scl{0.1}
    $\begin{tikzpicture}[baseline={([yshift=-.5ex]current bounding box.center)}, scale = \scl]
        \polygonSetup
        \dissectionArc[1][4]
    \end{tikzpicture}
    = t(\epsilon, \epsilon, t(\epsilon, \epsilon, \epsilon))$
    \\ \hline
    \end{tabular}
\end{center}
\end{fcFamily}

\addcontentsline{toca}{subsection}{\fcFamRef{fcFamily:polygonDissections}: Polygon dissection}
\begin{fcFamily}
[Polygon dissection \cite{FussCatalanOEIS}] \label{fcFamily:polygonDissections}

\def\scl{0.04}
\def\numPoints{2}

$T_n$ is the number of dissections of some convex polygon by non-intersecting chords into polygons with an odd number of sides and having a total number of $2n + 1$ edges (sides and diagonals). We draw all polygons with a marked side which is denoted by a thick black line. This fixes the orientation of the polygon and distinguishes polygons which only differ by a rotation.

\textit{Generator:} The generator is taken to be a single edge: 
\begin{equation*}
    \epsilon = 
    \begin{tikzpicture}[baseline={([yshift=-.7ex]current bounding box.center)}, scale = \scl]
        \polygonSetup
    \end{tikzpicture}
\end{equation*}

\def\scl{0.12}
\textit{Ternary map:} \def\numPoints{8}
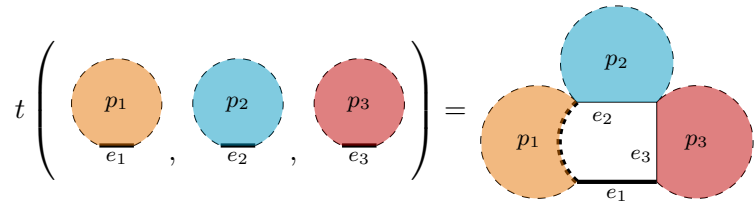
\begin{equation*}
    t \left(
    \raisebox{-16pt}{
    \begin{tikzpicture}[baseline={([yshift=-4.5ex]current bounding box.base)}, scale = \scl]
        \blankPolygonSetup
        \fill[\colOne, opacity = \opac] (8) -- (1) (-67.5:5) arc (-67.5:247.5:5);
        \node at (0,0) {\footnotesize $p_1$};
        \node at (-90:6) {\scriptsize $e_1$};
    \end{tikzpicture},
    \begin{tikzpicture}[baseline={([yshift=-4.5ex]current bounding box.base)}, scale = \scl]
        \blankPolygonSetup
        \fill[\colTwo, opacity = \opac] (8) -- (1) (-67.5:5) arc (-67.5:247.5:5);
        \node at (0,0) {\footnotesize $p_2$};
        \node at (-90:6) {\scriptsize $e_2$};
    \end{tikzpicture},
    \begin{tikzpicture}[baseline={([yshift=-4.5ex]current bounding box.base)}, scale = \scl]
        \blankPolygonSetup
        \fill[\colThree, opacity = \opac] (8) -- (1) (-67.5:5) arc (-67.5:247.5:5);
        \node at (0,0) {\footnotesize $p_3$};
        \node at (-90:6) {\scriptsize $e_3$};
    \end{tikzpicture}}
    \right)
    = \def\numPoints{4} \def\scl{0.15}
    \begin{tikzpicture}[baseline={([yshift=-1ex]current bounding box.center)}, scale = \scl]
            \trulyBlankPolygonSetup
            \draw (1) -- (2) -- (3);
            
            \draw[dashed] ([shift={(-7.07,0)}]45:5) arc (45:315:5);
            \draw[dotted,line width=2pt] (3) to [bend right = 50] (4);
            \begin{scope}
                \clip ([shift={(-7.07,0)}]45:5) arc (45:315:5);;
                \fill[\colOne, opacity = \opac] (3) to [bend right = 50] (4) -- (0,-7) -- (-12,-7) -- (-12,7) -- (0,7) -- (3);
            \end{scope}
            \node at (-7.8,0) {\footnotesize $p_1$};
            
            \draw[dashed] ([shift={(0,7.07)}]-45:5) arc (-45:225:5);
            \begin{scope}
                \clip ([shift={(0,7.07)}]-45:5) arc (-45:225:5);
                \fill[\colTwo, opacity = \opac] (2) -- (3) -- (-7,0) -- (-7,12) -- (7,12) -- (7,0) -- (2);
            \end{scope}
            \node at (0,7.07) {\footnotesize $p_2$};
            
            \draw[dashed] ([shift={(7.07,0)}]-135:5) arc (-135:135:5);
            \begin{scope}
                \clip ([shift={(7.07,0)}]-135:5) arc (-135:135:5);
                \fill[\colThree, opacity = \opac] (2) -- (1) -- (0,-7) -- (12,-7) -- (12,7) -- (0,7) -- (2);
            \end{scope}
            \node at (7.07,0) {\footnotesize $p_3$};
            
            \node at (-90:4.7) {\scriptsize $e_1$};
            \node at (120:2.5) {\scriptsize $e_2$};
            \node at (-30:2.5) {\scriptsize $e_3$};
        \end{tikzpicture}
\end{equation*}

In the above figure, the dotted internal line for $p_1$ denotes the fact that we ``open up'' the polygon $p_1$ by disconnecting it at the vertex on the right hand end of the marked thick edge. We do this in such a way that any diagonals adjacent to this vertex remain connected to the endpoint of the non-thick edge. For example, we ``open up'' the following polygon dissection as shown:
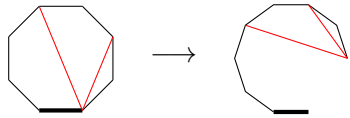
\begin{equation*} \def\scl{0.15}
    \begin{tikzpicture}[baseline={([yshift=-1ex]current bounding box.center)}, scale = \scl]
        \def\numPoints{8}
        \polygonSetup
        \dissectionArc[1][3]
        \dissectionArc[1][5]
    \end{tikzpicture}
    \quad \longrightarrow \quad
    \begin{tikzpicture}[baseline={([yshift=-1ex]current bounding box.center)}, scale = \scl]
        \def\numPoints{10}
        \trulyBlankPolygonSetup
        \draw (3) -- (4) -- (5) -- (6) -- (7) -- (8) -- (9) -- (10);
        \dissectionArc[3][5]
        \dissectionArc[3][7]
    \end{tikzpicture}
\end{equation*}

\def\scl{0.1}
An example of this map is the following:
\begin{align*}
    t \left(
    \raisebox{-7pt}{
    \begin{tikzpicture}[baseline={([yshift=-2ex]current bounding box.base)}, scale = \scl]
        \def\numPoints{6}
        \polygonSetup
        \dissectionArc[1][5]
    \end{tikzpicture}, \
    \begin{tikzpicture}[baseline={([yshift=-1ex]current bounding box.base)}, scale = \scl]
        \def\numPoints{3}
        \polygonSetup
    \end{tikzpicture}\, , \
    \begin{tikzpicture}[baseline={([yshift=-1.5ex]current bounding box.base)}, scale = \scl]
        \def\numPoints{4}
        \polygonSetup
        \dissectionArc[2][4]
    \end{tikzpicture}}
    \right)
    & =
    t \left(
    \raisebox{-7pt}{
    \begin{tikzpicture}[baseline={([yshift=-2ex]current bounding box.base)}, scale = \scl]
        \def\numPoints{8}
        \trulyBlankPolygonSetup
        \draw (3) -- (4) -- (5) -- (6) -- (7) -- (8);
        \dissectionArc[3][7]
    \end{tikzpicture}, \
    \begin{tikzpicture}[baseline={([yshift=-1ex]current bounding box.base)}, scale = \scl]
        \def\numPoints{3}
        \polygonSetup
    \end{tikzpicture}\, , \
    \begin{tikzpicture}[baseline={([yshift=-1.5ex]current bounding box.base)}, scale = \scl]
        \def\numPoints{4}
        \polygonSetup
        \dissectionArc[2][4]
    \end{tikzpicture}}
    \right) = \def\scl{0.18}
    \begin{tikzpicture}[baseline={([yshift=-1.5ex]current bounding box.center)}, scale = \scl]
        \def\numPoints{11}
        \polygonSetup
        \dissectionArc[1][4]
        \dissectionArc[2][4]
        \dissectionArc[4][6]
        \dissectionArc[6][10]
    \end{tikzpicture}
\end{align*}

\textit{Norm:} If $p$ is a polygon dissection with a total of $n$ of edges (sides and diagonals), then $\norm{p} = n$.

\textit{(3)-magma:} The (3)-magma begins (sorting by norm) as follows:
\vspace{-.5em}
\begin{center}
    \def\scl{0.04}
    \begin{tabular}{| l l |}
    \hline
    \textit{Norm 1:} & \def\numPoints{2}
    $\begin{tikzpicture}[baseline={([yshift=-.7ex]current bounding box.center)}, scale = \scl]
        \polygonSetup
    \end{tikzpicture}
    = \epsilon$\\
    \hline
    \textit{Norm 3:} & \def\numPoints{3} \def\scl{0.07}
    $\begin{tikzpicture}[baseline={([yshift=-1ex]current bounding box.center)}, scale = \scl]
        \polygonSetup
    \end{tikzpicture}
    = t(\epsilon, \epsilon, \epsilon)$ \\
    \hline
    \textit{Norm 5:} & \def\numPoints{5} \def\scl{0.1}
    $\begin{tikzpicture}[baseline={([yshift=-1ex]current bounding box.center)}, scale = \scl]
        \polygonSetup
    \end{tikzpicture}
    = t(t(\epsilon, \epsilon, \epsilon), \epsilon, \epsilon)$,
    \quad \def\numPoints{4}
    $\begin{tikzpicture}[baseline={([yshift=-1ex]current bounding box.center)}, scale = \scl]
        \polygonSetup
        \dissectionArc[2][4]
    \end{tikzpicture}
    = t(\epsilon, t(\epsilon, \epsilon, \epsilon), \epsilon)$,
    \\ & \def\scl{0.1} \def\numPoints{4}
    $\begin{tikzpicture}[baseline={([yshift=-1ex]current bounding box.center)}, scale = \scl]
        \polygonSetup
        \dissectionArc[1][3]
    \end{tikzpicture}
    = t(\epsilon, \epsilon, t(\epsilon, \epsilon, \epsilon))$
    \\
    \hline
    \end{tabular}
\end{center}
\end{fcFamily}

\addcontentsline{toca}{subsection}{\fcFamRef{fcFamily:latticePaths}: Lattice paths}
\begin{fcFamily}
[Lattice paths \cite{FussCatalanLin}] \label{fcFamily:latticePaths} 

\def\scl{0.4}

$T_n$ is the number of lattice paths from $(0,0)$ to $(n, 2n)$ consisting of $n$ East steps $(1,0)$ and $2n$ North steps $(0,1)$ that lie weakly below the line $y = 2x$. The line $y = 2x$ is shown as a dashed line in the schematic diagram defining the map.

\textit{Generator:} The empty path which we represent by a single vertex:
\begin{equation*}
    \epsilon = 
    \begin{tikzpicture}[baseline={([yshift=-.7ex]current bounding box.center)}, scale = \scl]
        \fill (0,0) \smlnd;
    \end{tikzpicture}
\end{equation*}

\vbox{
\textit{Ternary map:}
\vspace{-1em}
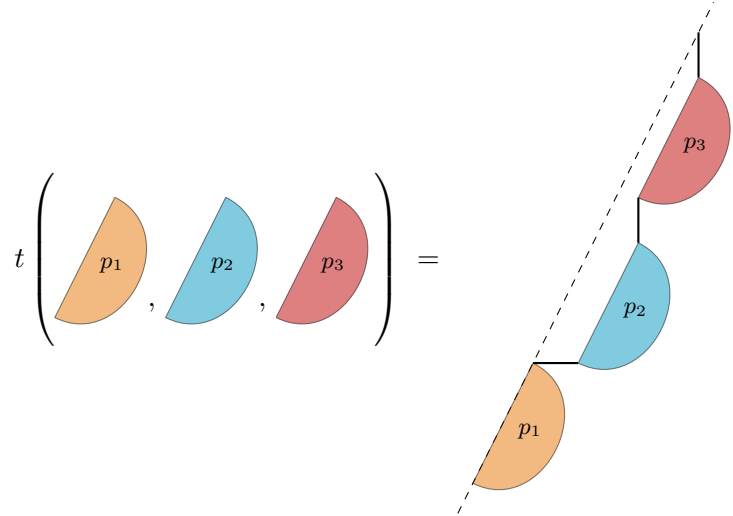
\begin{align*}
    t \left(
    \raisebox{-15pt}{\!\!
    \begin{tikzpicture}[baseline={([yshift=-3ex]current bounding box.center)}, scale = \scl]
        \clip (0,-1) rectangle (3.25,4.25);
        \draw[fill = \colOne, opacity = \opac] (2,4) -- (0,0) to [bend right = 90, looseness = 1.5] (2,4);
        \node at (1.9,1.75) {\footnotesize $p_1$};
    \end{tikzpicture}\!,
    \begin{tikzpicture}[baseline={([yshift=-3ex]current bounding box.center)}, scale = \scl]
        \clip (0,-1) rectangle (3.25,4.25);
        \draw[fill = \colTwo, opacity = \opac] (2,4) -- (0,0) to [bend right = 90, looseness = 1.5] (2,4);
        \node at (1.9,1.75) {\footnotesize $p_2$};
    \end{tikzpicture}\!,
    \begin{tikzpicture}[baseline={([yshift=-3ex]current bounding box.center)}, scale = \scl]
        \clip (0,-1) rectangle (3.25,4.25);
        \draw[fill = \colThree, opacity = \opac] (2,4) -- (0,0) to [bend right = 90, looseness = 1.5] (2,4);
        \node at (1.9,1.75) {\footnotesize $p_3$};
    \end{tikzpicture}}
    \right)
    \ = \
    \begin{tikzpicture}[baseline={([yshift=-.5ex]current bounding box.center)}, scale = \scl]
        \draw[dashed] (-0.5,-1) -- (8,16);
        \draw[fill = \colOne, opacity = \opac] (2,4) -- (0,0) to [bend right = 90, looseness = 1.5] (2,4);
        \node at (1.9,1.75) {\footnotesize $p_1$};
        \draw[thick] (2,4) -- (3.5,4);
        \draw[fill = \colTwo, opacity = \opac] (5.5,8) -- (3.5,4) to [bend right = 90, looseness = 1.5] (5.5,8);
        \node at (5.4,5.75) {\footnotesize $p_2$};
        \draw[thick] (5.5,8) -- (5.5,9.5);
        \draw[fill = \colThree, opacity = \opac] (7.5,13.5) -- (5.5,9.5) to [bend right = 90, looseness = 1.5] (7.5,13.5);
        \node at (7.4,11.25) {\footnotesize $p_3$};
        \draw[thick] (7.5,13.5) -- (7.5,15);
    \end{tikzpicture}
\end{align*}
}

\textit{Norm:} If $p$ is a path from $(0,0)$ to $(n,2n)$, then $\norm{p} = 2n + 1$.

\textit{(3)-magma:} The (3)-magma begins (sorting by norm) as follows:
\vspace{-.5em}
\begin{center}
    \def\scl{0.42}
    \begin{tabular}{| l l |}
    \hline
    \textit{Norm 1:} &
    $\begin{tikzpicture}[baseline={([yshift=-.5ex]current bounding box.center)}, scale = \scl]
        \fill (0,0) circle (4pt);
    \end{tikzpicture} = \epsilon$ \\
    \hline
    \textit{Norm 3:} &
    $\begin{tikzpicture}[baseline={([yshift=-.5ex]current bounding box.center)}, scale = \scl]
        \grd[1][2]
        \draw[thick] (0,0) \schAc \schUp \schUp;
        \dummyNodes[1][-0.2][1][2.2]
    \end{tikzpicture}
    = t(\epsilon, \epsilon, \epsilon)$ \\
    \hline
    \textit{Norm 5:} &
    $\begin{tikzpicture}[baseline={([yshift=-.5ex]current bounding box.center)}, scale = \scl]
        \grd[2][4]
        \draw[thick] (0,0) \schAc \schUp \schUp \schAc \schUp \schUp;
        \dummyNodes[1][-0.2][1][4.2]
    \end{tikzpicture}
    = t(t(\epsilon, \epsilon, \epsilon), \epsilon, \epsilon)$,
    \quad
    $\begin{tikzpicture}[baseline={([yshift=-.5ex]current bounding box.center)}, scale = \scl]
        \grd[2][4]
        \draw[thick] (0,0) \schAc \schAc \schUp \schUp \schUp \schUp;
        \dummyNodes[1][-0.2][1][4.2]
    \end{tikzpicture}
    = t(\epsilon, t(\epsilon, \epsilon, \epsilon), \epsilon)$,
    \quad
    $\begin{tikzpicture}[baseline={([yshift=-.5ex]current bounding box.center)}, scale = \scl]
        \grd[2][4]
        \draw[thick] (0,0) \schAc \schUp \schAc \schUp \schUp \schUp;
        \dummyNodes[1][-0.2][1][4.2]
    \end{tikzpicture}
    = t(\epsilon, \epsilon, t(\epsilon, \epsilon, \epsilon))$\\
    \hline
    \end{tabular}
\end{center}
\end{fcFamily}

\addcontentsline{toca}{subsection}{\fcFamRef{fcFamily:twoDyckPaths}: 2-Dyck paths}
\begin{fcFamily}
[2-Dyck paths \cite{AVAL20084660}] \label{fcFamily:twoDyckPaths}

\def\scl{0.3}

$T_n$ is the number of paths from $(0,0)$ to $(3n,0)$ with steps $(1,1)$ and $(1,-2)$ which remain above the line $y = 0$.

\textit{Generator:} The empty path which we represent by a single vertex:
\begin{equation*}
    \epsilon = \begin{tikzpicture}[baseline={([yshift=-.5ex]current bounding box.center)}, scale = \scl]
    \fill (0,0) \smlnd;
\end{tikzpicture}
\end{equation*} 

\textit{Ternary map:}
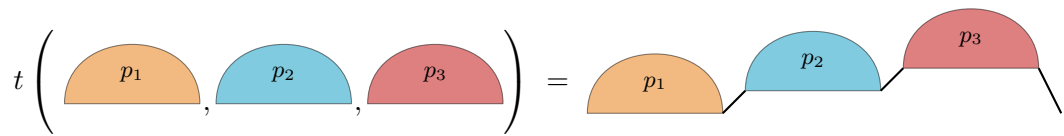
\begin{align*}
    t \left(
    \raisebox{-8pt}{
    \begin{tikzpicture}[baseline={([yshift=-.5ex]current bounding box.base)}, scale = \scl]
        \draw[fill = \colOne, opacity = \opac] (6,0) -- (0,0) to [bend left = 90, looseness = 1.5] (6,0);
        \node at (3,1.3) {\footnotesize $p_1$};
        \node at (1,-0.3) {};
    \end{tikzpicture}\spacecomma
    \begin{tikzpicture}[baseline={([yshift=-.5ex]current bounding box.base)}, scale = \scl]
        \draw[fill = \colTwo, opacity = \opac] (6,0) -- (0,0) to [bend left = 90, looseness = 1.5] (6,0);
        \node at (3,1.3) {\footnotesize $p_2$};
        \node at (1,-0.3) {};
    \end{tikzpicture}\spacecomma 
    \begin{tikzpicture}[baseline={([yshift=-.5ex]current bounding box.base)}, scale = \scl]
        \draw[fill = \colThree, opacity = \opac] (6,0) -- (0,0) to [bend left = 90, looseness = 1.5] (6,0);
        \node at (3,1.3) {\footnotesize $p_3$};
        \node at (1,-0.3) {};
    \end{tikzpicture}}
    \right)
    \ = \
    \begin{tikzpicture}[baseline={([yshift=-3ex]current bounding box.center)}, scale = \scl]
        \draw[fill = \colOne, opacity = \opac] (6,0) -- (0,0) to [bend left = 90, looseness = 1.5] (6,0);
        \node at (3,1.3) {\footnotesize $p_1$};
        \draw[thick] (6,0) -- (7,1);
        \draw[fill = \colTwo, opacity = \opac] (13,1) -- (7,1) to [bend left = 90, looseness = 1.5] (13,1);
        \node at (10,2.3) {\footnotesize $p_2$};
        \draw[thick] (13,1) -- (14,2);
        \draw[fill = \colThree, opacity = \opac] (20,2) -- (14,2) to [bend left = 90, looseness = 1.5] (20,2);
        \node at (17,3.3) {\footnotesize $p_3$};
        \draw[thick] (20,2) -- (21,0);
    \end{tikzpicture}
\end{align*}

\textit{Norm:} If $p$ is a 2-Dyck path from $(0,0)$ to $(3n,0)$, then $\norm{p} = 2n + 1$.

\textit{(3)-magma:} The (3)-magma begins (sorting by norm) as follows:
\vspace{-.5em}
\begin{center}
    \def\scl{0.45}
    \begin{tabular}{| l l |}
    \hline
    \textit{Norm 1:} &
    $\begin{tikzpicture}[baseline={([yshift=-.5ex]current bounding box.center)}, scale = \scl]
        \fill (0,0) \smlnd;
    \end{tikzpicture} = \epsilon$ \\
    \hline
    \textit{Norm 3:} &
    $\begin{tikzpicture}[baseline={([yshift=-.5ex]current bounding box.center)}, scale = \scl]
        \grd[3][2]
        \draw[thick] (0,0) \dyckUp \dyckUp \dyckTwoDn;
        \dummyNodes[1][-0.2][1][2.2]
    \end{tikzpicture}
    = t(\epsilon, \epsilon, \epsilon)$ \\
    \hline
    \textit{Norm 5:} &
    $\begin{tikzpicture}[baseline={([yshift=-.5ex]current bounding box.center)}, scale = \scl]
        \grd[6][4]
        \draw[thick] (0,0) \dyckUp \dyckUp \dyckTwoDn \dyckUp \dyckUp \dyckTwoDn;
        \dummyNodes[1][-0.2][1][4.2]
    \end{tikzpicture}
    = t(t(\epsilon, \epsilon, \epsilon), \epsilon, \epsilon)$,
    \quad
    $\begin{tikzpicture}[baseline={([yshift=-.5ex]current bounding box.center)}, scale = \scl]
        \grd[6][4]
        \draw[thick] (0,0) \dyckUp \dyckUp \dyckUp \dyckTwoDn \dyckUp \dyckTwoDn;
        \dummyNodes[1][-0.2][1][4.2]
    \end{tikzpicture}
    = t(\epsilon, t(\epsilon, \epsilon, \epsilon), \epsilon)$, \\
    &
    $\begin{tikzpicture}[baseline={([yshift=-.5ex]current bounding box.center)}, scale = \scl]
        \grd[6][4]
        \draw[thick] (0,0) \dyckUp \dyckUp \dyckUp \dyckUp \dyckTwoDn \dyckTwoDn;
        \dummyNodes[1][-0.2][1][4.2]
    \end{tikzpicture}
    = t(\epsilon, \epsilon, t(\epsilon, \epsilon, \epsilon))$\\
    \hline
    \end{tabular}
\end{center}
\end{fcFamily}

 \pagebreak


\bibliographystyle{apa}
\bibliography{./Biblio}

\end{document}